\documentclass{article}

\usepackage[a4paper,bindingoffset=0.2in,left=1in,right=1in,top=1in,bottom=1in,footskip=.25in,
]{geometry}  \usepackage[utf8]{inputenc}
\usepackage[T1]{fontenc}
\usepackage{amsmath,amssymb}
\usepackage{amsthm}\usepackage{tikz}
\usepackage{subcaption}
\usepackage{mathtools}
\usepackage{booktabs}
\usepackage{microtype}
\usepackage{bm}
\usepackage{pifont} \usepackage[section]{placeins} 
\usepackage{relsize}  
\usepackage{vruler}
\usepackage{etoolbox} 
\usepackage{authblk} \usepackage[numbers,sort&compress]{natbib}

\usepackage{xr-hyper}
\usepackage[colorlinks,citecolor=blue]{hyperref}
\usepackage{doi}

\usetikzlibrary{shapes,positioning,fit,matrix,arrows,arrows.meta}
\usetikzlibrary{decorations.pathmorphing,decorations.pathreplacing,decorations.markings}
\usetikzlibrary{calc}

\DeclarePairedDelimiter{\ceil}{\lceil}{\rceil}

\newcommand{\mystrut}{\rule[-0.55\baselineskip]{0pt}{\baselineskip}}

\DeclareMathSymbol{\shortminus}{\mathbin}{AMSa}{"39}

\newcommand{\xmark}{\text{\ding{55}}}

\tikzset{
dot/.style={draw,fill,circle,inner sep = 0pt,minimum size = 3pt},
bigdot/.style={dot,minimum size = 4pt},
terminal/.style={draw,circle, inner sep=2.5pt},
vcolour/.style={draw,inner sep=1.5pt,font=\scriptsize,label distance=2pt},
22box/.style={draw,minimum width=2cm,minimum height=1.5cm,font=\LARGE,node contents=\( \Downarrow \)},
22boxsmall/.style={22box,minimum width=1.5cm, minimum height=1cm},
Vset/.style={draw=black!30,ellipse,minimum width=1.5cm,minimum height=2.5cm},
petal/.style={
    decoration = {markings, mark=at position 1.0 with {
\coordinate (centre) at (0,0);
\path (centre)+(120:2) node(v)[bigdot]{};
\path (centre)+(240:2) node(u)[bigdot]{};
\path (centre)+(0:2) node(w)[bigdot]{};
\draw (v)--(u)--(w);\draw (v) arc (120:-120:2);
\path (v)--(u) node(v1)[pos=0.4][bigdot]{} node(u-v)[pos=0.6][bigdot]{};
\path (u)--(w) node(u-w)[pos=0.4][bigdot]{} node(w-u)[pos=0.6][bigdot]{};
\draw (w)--(v) node(w-v)[pos=0.4][bigdot]{} node(x1)[pos=0.6][bigdot]{};
\draw (u-w)--(u-v)--node[pos=0.3][bigdot]{} (w-v)--(w-u)--node[pos=0.3][bigdot]{} (v1)--(x1)--node[pos=0.3][bigdot]{} (u-w);
}},
    postaction = {decorate}
    },
}

\newtheorem{theorem}{Theorem}
\newtheorem{corollary}{Corollary}
\newtheorem{lemma}{Lemma}
\newtheorem{observation}{Observation}
\newtheorem{conjecture}{Conjecture}
\newtheorem{problem}{Problem}

\newtheoremstyle{freethm}{3pt}{3pt}{}{}{\bfseries}{.\\}{.5em}{\thmname{#1}\thmnumber{ #2}\thmnote{ (#3)}}

\theoremstyle{freethm}
\newtheorem{construct}{Construction}

\newtheoremstyle{thmpart}{3pt}{3pt}{}{}{\bfseries}{:}{.5em}{\thmname{#1}\thmnumber{ #2}\thmnote{ (#3)}}

\theoremstyle{thmpart}
\newtheorem{claim}{Claim}\theoremstyle{plain}

\theoremstyle{freethm}

\newcommand{\pagetarget}[2]{\phantomsection \label{#1}\hypertarget{#1}{#2}}

\providetoggle{forThesis}  \togglefalse{forThesis}

\title{Hardness Transitions of Star Colouring\\ and Restricted Star Colouring\footnote{A preliminary version of Section~\ref{sec:star colouring} appeared in CALDAM 2022 \cite{shalu_cyriac4}}}
\author[1]{Shalu M.A.}
\author[2]{Cyriac Antony}
\affil[1]{Indian Institute of Information Technology, Design \& Manufacturing\\ (IIITDM)~Kancheepuram, Chennai, India; Email: \href{mailto:shalu@iiitdm.ac.in}{\texttt{shalu@iiitdm.ac.in}}}
\affil[2]{IIT Madras, Chennai, India; Email: \href{mailto:ma23r004@smail.iitm.ac.in}{\texttt{ma23r004@smail.iitm.ac.in}}}
\date{}

\begin{document}

\maketitle

\providetoggle{forThesis}

\providetoggle{extended}

\providetoggle{iiitdmV} 

\begin{abstract}
We study how the complexity of the graph colouring problems star colouring and restricted star colouring vary with the maximum degree of the graph. 
Restricted star colouring (in short, rs colouring) is a variant of star colouring, as the name implies. 
For \( k\in \mathbb{N} \), a \( k\)-colouring of a graph \( G \) is a function \( f\colon V(G)\to \mathbb{Z}_k \) such that \( f(u)\neq f(v) \) for every edge \( uv \) of \( G \). 
A \( k \)-colouring of \( G \) is called a \( k \)-star colouring of \( G \) if there is no path \( u,v,w,x \) in \( G \) with \( f(u)=f(w) \) and \( f(v)=f(x) \). 
A \( k \)-colouring of \( G \) is called a \( k \)-rs colouring of \( G \) if there is no path \( u,v,w \) in \( G \) with \( f(v)>f(u)=f(w) \). 
For \( k\in \mathbb{N} \), the problem \textsc{\( k \)-Star Colourability} takes a graph \( G \) as input and asks whether \( G \) admits a \( k \)-star colouring. 
The problem \textsc{\( k \)-RS Colourability} is defined similarly. 
Recently, Brause et al.~(Electron.\ J.\ Comb., 2022) investigated the complexity of 3-star colouring with respect to the graph diameter. 
We study the complexity of \( k \)-star colouring and \( k \)-rs colouring with respect to the maximum degree for all \( k\geq 3 \). 
For \( k\geq 3 \), let us denote the least integer \( d \) such that \textsc{\( k \)-Star Colourability} (resp.\ \textsc{\( k \)-RS Colourability}) is NP-complete for graphs of maximum degree \( d \) by \( L_s^{(k)} \) (resp.\  \( L_{rs}^{(k)} \)). 

We prove that for \( k=5 \) and \( k\geq 7 \), \textsc{\( k \)-Star Colourability} is NP-complete for graphs of maximum degree \( k-1 \) (i.e., \( L_s^{(k)}\leq k-1 \)). 
We also show that \textsc{\( 4 \)-RS Colourability} is NP-complete for planar 3-regular graphs of girth 5 and \textsc{\( k \)-RS Colourability} is NP-complete for triangle-free graphs of maximum degree \( k-1 \) for \( k\geq 5 \) (i.e., \( L_{rs}^{(k)}\leq k-1 \)). 
Using these results, we prove the following: 
(i) for \( k\geq 4 \) and \( d\leq k-1 \), \textsc{\( k \)\nobreakdash-Star Colourability} is NP-complete for \( d \)-regular graphs if and only if \( d\geq L_s^{(k)} \); and (ii) for \( k\geq 4 \), \textsc{\( k \)\nobreakdash-RS Colourability} is NP-complete for \( d \)-regular graphs if and only if \( L_{rs}^{(k)}\leq d\leq k-1 \). 
It is not known whether result~(ii) has a star colouring analogue. 
\end{abstract}

\section{Introduction and Definitions}\label{sec:intro}
Star colouring is a colouring variant introduced by Gr\"{u}nbaum~\cite{grunbaum}, which is used in the estimation of sparse Hessian matrices~\cite{gebremedhin2005}. 
Restricted star colouring (abbreviated rs colouring) is a variant of star colouring first introduced specifically for this application~\cite{gebremedhin2007}. 
It was later introduced independently under the names independent set star partition~\cite{shalu_sandhya}, and with the order of colours reversed, unique-superior colouring~\cite{karpas} and 2-ranking~\cite{almeter,bose}. 
The complexity of star colouring is studied in various graph classes \cite{coleman_more,albertson,gebremedhin2007,lyons,linhares-sales,yue,harshita,lei,omoomi,bokPreprint,shalu_cyriac2,shalu_cyriac3,shalu_cyriac4,brause,bhyravarapu_reddy}. Although rs colouring is not as widely known, there are already five papers that focus on rs colouring~\cite{karpas,shalu_sandhya,almeter,shalu_cyriac2,bose}. 
Interestingly, rs colouring can be defined in terms of locally constrained graph homomorphisms (see Section~\ref{sec:rs in terms of OIH} for details).

Brause et al.~\cite{brause} investigated the complexity of 3-star colouring with respect to the graph diameter. 
For \( k\geq 3 \), we study the complexity of \( k \)-star colouring and \( k \)-rs colouring with respect to the maximum degree focusing on graphs of maximum degree \( d \) and \( d \)-regular graphs. 
Our interest is in the values of \( d \) for which the complexity of \( k \)-star colouring in graphs of maximum degree \( d \) (resp.\  \( d \)-regular graphs) differ drastically from that in graphs of maximum degree \( d-1 \) (resp.\  \( (d-1) \)-regular graphs); we call such values of \( d \) as \emph{hardness transitions} of \( k \)-star colouring with respect to maximum degree for the class of graphs of maximum degree \( d \) (resp.\  \( d \)-regular graphs); see Section~\ref{sec:intro hardness transitions} for details. 
Hardness transitions of rs colouring are defined similarly.

This paper is organised as follows. 
Subsections~\ref{sec:def}, \ref{sec:intro hardness transitions} and \ref{sec:our results} present basic definitions, definitions related to hardness transitions, and an overview of our results, respectively. 
This is followed by Section~\ref{sec:star colouring}, devoted to star colouring and Section~\ref{sec:rs colouring}, devoted to rs colouring. 

Section~\ref{sec:star colouring} is subdivided into subsections~\ref{sec:star intro}, \ref{sec:star colouring hardness transitions} and \ref{sec:star colouring points of hardness transition}, which present (i)~an introduction and literature survey on star colouring, (ii)~details of our results on hardness transitions of star colouring and (iii)~consequences of our results on the values of \( L_s^{(k)} \) and two similar parameters, respectively. 

Similarly, Section~\ref{sec:rs colouring} is subdivided into subsections~\ref{sec:rs intro}, \ref{sec:rs in terms of OIH}, \ref{sec:rs colouring hardness transitions} and \ref{sec:rs colouring points of hardness transition}, which present (i)~an introduction and literature survey on rs colouring, (ii)~characterisation of rs colouring in terms of graph homomorphisms, (iii)~details of our results on hardness transitions of rs colouring and (iv)~consequences of our results on the value of \( L_{rs}^{(k)} \), respectively. 
We conclude with Section~\ref{sec:conclusion}.

\iftoggle{forThesis}
{
}
{
\subsection{Basic Definitions}\label{sec:def}
All graphs considered in this paper are finite, simple and undirected unless otherwise specificed. 
We follow West~\cite{west} for graph theory terminology and notation. 
When the graph is clear from the context, we denote the number of edges of the graph by \( m \) and the number of vertices by \( n \). 
For a graph \( G \), we denote the maximum degree of \( G \) by \( \Delta(G) \). 
For a subset \( S \) of the vertex set of \( G \), the \emph{subgraph of \( G \) induced by \( S \)} is denoted by \( G[S] \). 
An \emph{orientation} \( \vec{G} \) of a graph \( G \) is the directed graph obtained by assigning a direction on each edge of \( G \); that is, if \( uv \) is an edge in \( G \), then either \( (u,v) \) or \( (v,u) \) is an arc in \( \vec{G} \). 
An orientation is also called an \emph{oriented graph}. 
If \( (u,v) \) is an arc in an oriented graph~\( \vec{G} \), then \( u \) is an \emph{in-neighbour} of \( v \) in \( \vec{G} \). We denote the neighbourhood of a vertex \( v \) in a graph \( G \) by \( N_G(v) \), and the in-neighbourhood of a vertex \( v \) in an oriented graph \( \vec{G} \) by \( N_{\vec{G}}^-(v) \). 
The \emph{girth} of a graph with a cycle is the length of its shortest cycle. 
The treewidth of \( G \) is denoted as \( \text{tw}(G) \). 
A 3-regular graph is also called a \emph{cubic graph}, and a graph of maximum degree~3 is called a \emph{subcubic graph}. 
A graph \( G \) is \emph{\( 2 \)-degenerate} if there exists a left-to-right ordering of its vertices such that every vertex has at most two neighbours to its left.

If \( G \) is a graph and \( u,v \) are non-adjacent vertices in \( G \), the operation of \emph{identifying} vertices \( u \) and \( v \) in \( G \) involves (i)~introducing a new vertex \( w^* \), (ii)~joining \( w^* \) to each vertex in \( N_G(u)\cup N_G(v) \), and (iii)~removing vertices \( u \) and \( v \) \cite{west}. 
That is, in the resultant graph, say \( G^* \), the neighbourhood of \( w^* \) is \( N_{G^*}(w^*)=N_G(u)\cup N_G(v) \). 
Since vertices \( u \) and \( v \) are removed, the new vertex \( w^* \) in \( G^* \) can be unambiguously called as~\( u \) (or \( v \) for that matter). 
Observe that by definition, if vertices \( u \) and \( v \) have a common neighbour \( x \) in \( G \), then there is only one edge from \( w^* \) to \( x \) in \( G^* \) (i.e., no parallel edges in~\( G^* \)). 

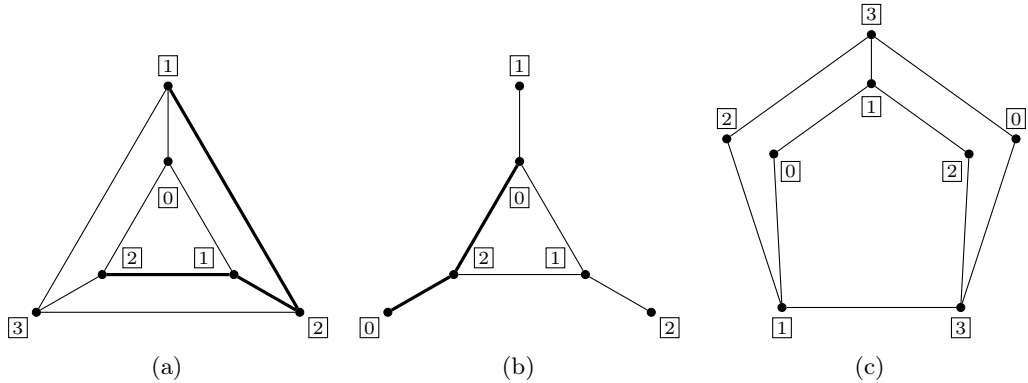
\begin{figure}[hbt]
\centering
\begin{subfigure}[b]{0.3\textwidth}
\centering
\begin{tikzpicture}\path
(  90:1) node (u0)[dot][label={[vcolour,label distance=8pt]-90:0}]{}
( -30:1) node (u1)[dot][label={[vcolour,label distance=7pt,yshift=-2pt]150:1}]{}
(-150:1) node (u2)[dot][label={[vcolour,label distance=7pt,yshift=-2pt]30:2}]{}
(  90:2) node (v0)[dot][label={[vcolour]90:1}]{}
( -30:2) node (v1)[dot][label={[vcolour]-30:2}]{}
(-150:2) node (v2)[dot][label={[vcolour]-150:3}]{};
\draw
(u0)--(v0)
(u1)--(v1)
(u2)--(v2);
\draw
(u0)--(u1)--(u2)--(u0)
(v0)--(v1)--(v2)--(v0);

\draw [very thick]
(u2)--(u1)--(v1)--(v0);
\end{tikzpicture}
\caption{}
\end{subfigure}\begin{subfigure}[b]{0.3\textwidth}
\centering
\begin{tikzpicture}\path
(  90:1) node (u0)[dot][label={[vcolour,label distance=8pt]-90:0}]{}
( -30:1) node (u1)[dot][label={[vcolour,label distance=7pt,yshift=-2pt]150:1}]{}
(-150:1) node (u2)[dot][label={[vcolour,label distance=7pt,yshift=-2pt]30:2}]{}
(  90:2) node (v0)[dot][label={[vcolour]90:1}]{}
( -30:2) node (v1)[dot][label={[vcolour]-30:2}]{}
(-150:2) node (v2)[dot][label={[vcolour]-150:0}]{};
\draw
(u0)--(v0)
(u1)--(v1)
(u2)--(v2);
\draw
(u0)--(u1)--(u2)--(u0);

\draw [very thick] 
(u0)--(u2)--(v2)
;
\end{tikzpicture}
\caption{}
\end{subfigure}\begin{subfigure}[b]{0.3\textwidth}
\centering
\begin{tikzpicture}
\coordinate (centre){};

\path (centre) +(18:2) coordinate[dot,label={[vcolour]\( 0 \)}](u5)
      (centre) +(90:2) coordinate[dot,label={[vcolour]above:\( 3 \)}](u4)
      (centre) +(162:2) coordinate[dot,label={[vcolour]\( 2 \)}](u3)
      (centre) +(-54:2) coordinate[dot,label={[vcolour]below:\( 3 \)}](u1)
      (centre) +(-126:2) coordinate[dot,label={[vcolour]below:\( 1 \)}](u2);
\draw (u1)--(u2)--(u3)--(u4)--(u5)--(u1);

\path (centre) +(18:1.35) coordinate[dot,label={[vcolour]below left:\( 2 \)}](u8)
      (centre) +(90:1.35) coordinate[dot,label={[vcolour,label distance=3pt]below:\( 1 \)}](u7)
      (centre) +(162:1.35) coordinate[dot,label={[vcolour]below right:\( 0 \)}](u6);
\draw (u2)--(u6)--(u7)--(u8)--(u1);

\draw (u4)--(u7);

\end{tikzpicture}
\caption{}
\end{subfigure}\caption{(a)~a 4-colouring of a graph, which is not a 4-star colouring (bicoloured \( P_4 \) highlighted), (b)~a 3-star colouring of a graph, which is not a 3-rs colouring (bicoloured \( P_3 \) with higher colour on middle highlighted), and (c)~a 4-rs colouring of a graph.}
\label{fig:star rs col eg}
\end{figure}

A \( k \)-colouring of a graph \( G \) is a function \( f \) from the vertex set of \( G \) to a set of \( k \) colours, say \( \{0,1,\dots,k-1\} \), such that \( f \) maps every pair of adjacent vertices to different colours. 
A \( k \)-star colouring of \( G \) is a \( k \)-colouring \( f \) of \( G \) such that there is no bicoloured \( P_4 \) in \( G \) (i.e., there is no path \( u,v,w,x \) in \( G \) with \( f(u)=f(w) \) and \( f(v)=f(x) \)). 
A \( k \)-rs colouring of \( G \) is a \( k \)-colouring \( f \) of \( G \) such that there is no bicoloured \( P_3 \) in \( G \) with higher colour on its middle vertex (i.e., there is no path \( u,v,w \) in \( G \) with \( f(v)>f(u)=f(w) \)). 
See Figure~\ref{fig:star rs col eg} for examples.

The \emph{star chromatic number} \( \chi_s(G) \) of a graph \( G \) is the least integer \( k \) such that \( G \) is \( k \)-star colourable. 
The \emph{rs chromatic number} \( \chi_{rs}(G) \) is defined similarly. 
The star chromatic number of line graph of \( G \) is called the \emph{star chromatic index} of \( G \). 
The problem \textsc{Star Colourability} takes a graph \( G \) and a positive integer \( k \) as input and asks whether \( G \) is \( k \)-star colourable. 
For \( k\in \mathbb{N} \), the decision problem \textsc{\( k \)-Colourability} takes a graph \( G \) as input and asks whether \( G \) is \( k \)-colourable. 
The problems \textsc{\( k \)\nobreakdash-Star Colourability} and \textsc{\( k \)\nobreakdash-RS Colourability} are defined analogously. 
To denote the restriction of a decision problem, we write the conditions in parenthesis. 
For instance, \textsc{\( 4 \)-Star Colourability}\( (\text{bipartite}, \Delta=5) \) denotes the problem \textsc{\( 4 \)-Star Colourability} restricted to the class of bipartite graphs \( G \) with \( \Delta(G)=5 \).

A \emph{homomorphism} from a graph \( G \) to a graph \( H \) is a function \( \psi \colon V(G)\to V(H) \) such that \( \psi(u)\psi(v) \) is an edge in \( H \) whenever \( uv \) is an edge in \( G \). 
A \emph{homomorphism} from an oriented graph \( \vec{G} \) to an oriented graph \( \vec{H} \) is a function \( \psi\colon V(\vec{G})\to V(\vec{H}) \) such that \( (\psi(u),\psi(v)) \) is an arc in \( \vec{H} \) whenever \( (u,v) \) is an arc in \( \vec{G} \). 
A homomorphism \( \psi \) from an oriented graph \( \vec{G} \) to an oriented graph \( \vec{H} \) is \emph{in-neighbourhood injective} if for every vertex \( v \) of \( \vec{G} \), the restriction of \( \psi \) to the in-neighbourhood \( N_{\vec{G}}^-(v) \) is an injective function from \( N_{\vec{G}}^-(v) \) to \( N_{\vec{H}}^-(\psi(v)) \) (i.e., \( \psi \) maphs distinct in-neighbours of \( v \) to distinct in-neighbours of \( \psi(v) \)). 
An \emph{automorphism} of a graph \( G \) is a bijective function \( \psi \colon V(G)\to V(G) \) such that \( \psi(u)\psi(v) \) is an edge in \( G \) if and only if \( uv \) is an edge in \( G \). 

}

\subsection{Hardness Transitions}\label{sec:intro hardness transitions}
Analysing the boundary between easy (i.e., polynomial-time solvable) and hard (e.g., NP-complete) problems is a common theme in complexity theory \cite{garey_johnson}. 
Studying the change in the complexity of a problem in response to a change in a single parameter falls in this category. 
Brause et al.~\cite{brause} studied the complexity of \textsc{3-Star Colourability} with the diameter of the graph as the parameter. 
For \( k\geq 3 \), we study the complexity of \( k \)-star colouring and \( k \)-rs colouring with the maximum degree of the graph as the parameter. 
Recall that we write the conditions in parenthesis to denote the restriction of a decision problem; 
e.g.: \textsc{\( 4 \)-Star Colourability}\( (\text{bipartite}, \Delta=5) \) denotes the problem \textsc{\( 4 \)-Star Colourability} restricted to the class of bipartite graphs \( G \) with \( \Delta(G)=5 \).
We assume \mbox{P \( \neq \) NP} throughout this paper; thus, NP is partitioned into three classes: P, NPC and NPI~\cite{ladner}.
We emphasise that our interest is in the classification of NP-problems with respect to the P vs.\ NPC vs.\ NPI trichotomy: that is, the complexity classes dealt with in this paper are only P, NPC and NPI.

A decision problem \( \Pi \) in NP has a \emph{hardness transition} with respect to a discrete parameter \( d \) at a point \( d=x \) if \( \Pi(d=x) \) and \( \Pi(d=x-1) \) belong to different complexity classes among P, NPC and NPI (e.g.: \( \Pi(d=x)\in \)\,NPC whereas \( \Pi(d=x-1)\in \)\,P; see \cite{mikero} for a discussion). 
For example, \textsc{3-Colourability} of a graph of maximum degree \( d \) is polynomial-time solvable for \( d=3 \) (due to Brook's theorem) and NP-complete for \( d=4 \) \cite{garey_johnson}. 
That is, \textsc{3-Colourability}\( (\Delta=3)\in \) P and \textsc{3-Colourability}\( (\Delta=4)\in \) NPC. 
Hence, \textsc{3-Colourability} has a hardness transition with respect to the maximum degree \( d \) at the point \( d=4 \). 
Note that each hardness transition presumably deals with the P vs.\ NPC boundary since no `natural' problem is known to be NP-intermediate~\cite{arora_barak}.

The number of hardness transitions depends on the problem as well as the parameter under consideration. 
Interestingly, a decision problem can have infinitely many hardness transitions. 
Cseh and Kavitha~\cite{cseh_kavitha} proved that the popular matching problem on complete graph \( K_n \) is in P for odd \( n \) whereas it is NP-complete for even \( n \). 
Therefore, the popular matching problem on complete graph with respect to the number of vertices \( n \) has infinitely many hardness transitions.

Let us consider the complexity of \( k \)-colouring in bounded degree graphs for fixed \( k\geq 3 \). 
Emden-Weinert et al.~\cite{emden-weinert} proved that \textsc{\( k \)-Colourability} is NP-complete for graphs of maximum degree \( k-1+\raisebox{1.5pt}{\big\lceil}\sqrt{k}\raisebox{1.5pt}{\big\rceil} \). 
Observe that if \textsc{\( k \)\nobreakdash-Colourability} is NP-complete for graphs of maximum degree \( d \), then it is NP-complete for graphs of maximum degree \( d+1 \) (to produce a reduction, it suffices to add a disjoint copy of \( K_{1,d+1} \)). 
This suggests the following problem. 
\begin{problem}\label{prob:what is Lk}
For \( k\geq 3 \), what is the least integer \( d \) such that \textsc{\( k \)-Colourability} is NP-complete for graphs of maximum degree \( d \)?
\end{problem}
\noindent Observe that Problem~\ref{prob:what is Lk} deals with locating a point of hardness transition. 
By the same argument, if \textsc{\( k \)\nobreakdash-Star Colourability} is NP-complete for graphs of maximum degree~\( d \), then \textsc{\( k \)-Star Colourability} is NP-complete for graphs of maximum degree \( d+1 \). 
The same is true of rs colouring. 
For \( k\geq 3 \), \textsc{\( k \)\nobreakdash-Star Colourability} and \textsc{\( k \)\nobreakdash-RS Colourability} are NP-complete for graphs of maximum degree~\( k \)~\cite{shalu_cyriac2,shalu_cyriac3}. 
Therefore, for each \( k\geq 3 \), there exists a unique integer \( d^* \) such that \textsc{\( k \)\nobreakdash-Colourability} (resp.\ \textsc{\( k \)\nobreakdash-Star Colourability} or \textsc{\( k \)\nobreakdash-RS Colourability}) in graphs of maximum degree \( d \) is NP-complete if and only if \( d\geq d^* \). 
Thus, one can ask the counterpart of Problem~\ref{prob:what is Lk} for star colouring and rs colouring. 
Let \( L^{(k)} \), \( L_s^{(k)} \) and \( L_{rs}^{(k)} \) denote the answers to Problem~\ref{prob:what is Lk} and its counterparts for star colouring and rs colouring; that is, \( L^{(k)} \) (resp.\  \( L_a^{(k)} \) or \( L_{rs}^{(k)} \)) is the least integer \( d \) such that \textsc{\( k \)\nobreakdash-Colourability} (resp.\ \textsc{\( k \)\nobreakdash-Star Colourability} or \textsc{\( k \)\nobreakdash-RS Colourability}) is NP-complete for graphs of maximum degree \( d \).

Due to Brook's theorem, \textsc{\( k \)-Colourability} is polynomial-time solvable for graphs of maximum degree \( k \), and thus \( L^{(k)}\geq k+1 \). 
For \( k\geq 3 \), \textsc{\( k \)\nobreakdash-Colourability} is NP-complete for graphs of maximum degree \( k-1+\raisebox{1.5pt}{\big\lceil}\sqrt{k}\raisebox{1.5pt}{\big\rceil} \) \cite{emden-weinert}, and thus \( k+1\leq L^{(k)}\leq k-1+\raisebox{1.5pt}{\big\lceil}\sqrt{k}\raisebox{1.5pt}{\big\rceil} \). 
Hence, \( L^{(3)}=4 \), \( L^{(4)}=5 \), \( 6\leq L^{(5)}\leq 7 \), and so on. 
For sufficiently large \( k \) and \( d<k-1+\raisebox{1.5pt}{\big\lceil}\sqrt{k}\raisebox{1.5pt}{\big\rceil} \), the problem \textsc{\( k \)\nobreakdash-Colourability} is in P for graphs of maximum degree \( d \) \cite[Theorem~43]{molloy_reed}. 
Therefore, \( L^{(k)}=k-1+\raisebox{1.5pt}{\big\lceil}\sqrt{k}\raisebox{1.5pt}{\big\rceil} \) for sufficiently large \( k \). 
Yet, the exact value of \( L^{(k)} \) is unknown for small values of \( k \) such as \( k=5 \) even though we know that \( L^{(5)}\in \{6,7\} \) (the complexity of \textsc{5\nobreakdash-Colourability} in graphs of maximum degree~6 is open \cite{paulusma}).

\subsection{Our Results}\label{sec:our results}
For \( k\geq 3 \), \textsc{\( k \)\nobreakdash-Star Colourability} and \textsc{\( k \)\nobreakdash-RS Colourability} are NP-complete for graphs of maximum degree~\( k \)~\cite{shalu_cyriac2,shalu_cyriac3}. 
For \( k\geq 4 \), we improve the maximum degree in these NP-completeness results to \( k-1 \), except for \textsc{\( k \)\nobreakdash-Star Colourability} with \( k\in \{4,6\} \).

We show that \textsc{\( 4 \)-RS Colourability} is NP-complete for planar 3-regular graphs of girth 5, and \textsc{\( k \)-RS Colourability} is NP-complete for triangle-free graphs of maximum degree \( k-1 \) for \( k\geq 4 \). 
We also prove that for \( k=5 \) and \( k\geq 7 \), \textsc{\( k \)\nobreakdash-Star Colourability} is NP-complete for graphs of maximum degree~\( k-1 \). 
In contrast, \textsc{\( k \)-Star Colourability} (resp.\ \textsc{\( k \)\nobreakdash-RS Colourability}) is polynomial-time solvable for graphs of maximum degree at most \( 0.33\, k^{\,2/3} \) (resp.\ \( \sqrt{k} \)). 
Hence, for \( k\geq 4 \), we have \( 0.33\, k^{\,2/3}<L_s^{(k)}\leq k \) and \( \sqrt{k}<L_{rs}^{(k)}\leq k-1 \).

The slight improvement of the maximum degree in the NP-completeness results of ~\cite{shalu_cyriac2,shalu_cyriac3} (see the first paragraph in this subsection) allows us to prove the following. 
\begin{itemize}
\item For \( k\geq 4 \) and \( d\leq k-1 \), \textsc{\( k \)\nobreakdash-Star Colourability} is NP-complete for \( d \)-regular graphs\\ if and only if \( d\geq L_s^{(k)} \). 
\item For \( k\geq 4 \), \textsc{\( k \)\nobreakdash-RS Colourability} is NP-complete for \( d \)-regular graphs\\ if and only if \( L_{rs}^{(k)}\leq d\leq k-1 \). 
\end{itemize}
It is not known whether the preceding result has a star colouring analogue 
(i.e. a result of the following form: for \( k\geq 4 \), there exist integers \( \ell_k \) and \( h_k \) such that \textsc{\( k \)\nobreakdash-Star Colourability} is NP-complete for \( d \)-regular graphs if and only if \( \ell_k\leq d\leq h_k \)).

\section{Star Colouring}\label{sec:star colouring}

\iftoggle{forThesis}
{\section{Hardness Transitions}\label{sec:star colouring hardness transitions}
In this section, we discuss the hardness transitions of star colouring with respect to the maximum degree. 
To be specific, we deal with the values \( \widetilde{L}_s^{(k)} \), \( L_s^{(k)} \) and \( H_s^{(k)} \). 
We show that \( \widetilde{L}_s^{(3)}=3 \) whereas \( L_s^{(3)} \) and \( H_s^{(3)} \) are undefined. 
In Section~\ref{sec:star colouring points of hardness transition}, we explain that (i)~\( \widetilde{L}_s^{(k)}\leq k \) for all \( k\geq 3 \) by Theorem~\ref{thm:3-star planar bip girth 8} and Theorem~\ref{thm:k-star max degree k}, (ii)~\( H_s^{(k)}\leq 2k-4 \) for all \( k\in \{4,5,7,8,\dots \} \) by Theorem~\ref{thm:lb chi_s}, and (iii)~\( H_s^{(4)}=4 \) by Corollary~\ref{cor:4-star colouring 4-regular NPC}. 
We prove that \( \widetilde{L}_s^{(k)}\leq k-1 \) and \( L_s^{(k)}=\widetilde{L}_s^{(k)} \) for \( k=5 \) and \( k\geq 7 \). 
To this end, we show that for \( k=5 \) and \( k\geq 7 \), (i)~\textsc{\( k \)-Star Colourability} is NP-complete for graphs of maximum degree \( k-1 \), and (ii)~for each \( d<k \), \textsc{\( k \)-Star Colourability} in graphs of maximum degree \( d \) is NP-complete if and only if \textsc{\( k \)-Star Colourability} in \( d \)-regular graphs is NP-complete. 
Let us discuss results (i) and (ii). 
The consequences of these results on the values of \( \widetilde{L}_s^{(k)} \), \( L_s^{(k)} \) and \( H_s^{(k)} \) are discussed later in Section~\ref{sec:star colouring points of hardness transition}.  
}{\subsection{Introduction and Literature Survey}\label{sec:star intro} Star colouring is studied in various graph classes such as planar graphs, bipartite graphs, regular graphs, sparse graphs and line graphs. 
For surveys on star colouring of planar graphs and line graphs, see \cite[Section~14]{borodin2013} and \cite{lei_shi} respectively.

Albertson et al.~\cite{albertson} proved that the maximum among star chromatic numbers of planar graphs is between 10 and 20. 
Kierstead et al.~\cite{timmons2009} proved that the maximum among star chromatic numbers of planar bipartite graphs is between 8 and 14. 
Fertin et al.~\cite{fertin2004} proved that \( \chi_s(G)\leq \binom{\text{tw}(G)+2}{2} \). 
They also proved that \( \chi_s(G)=O(d^{\frac{3}{2}}) \) where \( d=\Delta(G) \); and that there exist graphs \( G \) with \( \chi_s(G)=\Omega(d^\frac{3}{2}/(\log d)^\frac{1}{2}) \). The star chromatic number of the \( d \)-dimensional hypercube is at most \( d+1 \)~\cite{fertin2004}. 
Ne\v{s}et\v{r}il and Mendez~\cite{nesetril_mendez2003} related the star chromatic number of a graph to the chromatic numbers of its minors. 
Albertson et al.~\cite{albertson} and independently Ne\v{s}et\v{r}il and Mendez~\cite{nesetril_mendez2003} found that star colourings of a graph \( G \) are associated with orientations of \( G \). 
For every 3-regular graph \( G \), we have \( 4\leq \chi_s(G)\leq 6 \) \cite{xie,chen2013}.

Lyons~\cite{lyons} proved that \( \chi_s(G)=\text{tw}(G)+1 \) for every cograph \( G \). 
Linhares-Sales et al.~\cite{linhares-sales} designed a linear-time algorithm to compute the star chromatic number for two superclasses of cographs called \( P_4 \)-tidy graphs and \( (q,q-4) \)-graphs (for each fixed \( q \)). 
Omoomi et al.~\cite{omoomi} designed a polynomial-time algorithm to compute the star chromatic number of line graph of trees.

Coleman and More~\cite{coleman_more} proved that for \( k\geq 3 \), \textsc{\( k \)-Star Colourability} is NP-complete for (2\nobreakdash-degenerate) bipartite graphs. 
Albertson et al.~\cite{albertson} proved that \textsc{\( 3 \)-Star Colourability}is NP-complete for planar bipartite graphs. 
Lei et al.~\cite{lei} proved that \textsc{\( 3 \)-Star Colourability} is NP-complete for line graphs of subcubic graphs. 
Bok et al.~\cite{bokPreprint} provided complexity dichotomy results on \textsc{Star Colourability} and \textsc{\( k \)-Star Colourability} in \( H \)-free graphs except for one open case, namely \textsc{Star Colourability} in \( 2K_2 \)-free graphs. 
Bok et al.~\cite{bok} and independently Shalu and Antony~\cite{shalu_cyriac2} proved that \textsc{Star Colourability} is NP-complete for co-bipartite graphs. 
Brause et al.~\cite{brause} proved that \textsc{\( 3 \)-Star Colourability} in graphs of diameter at most \( d \) is polynomial-time solvable for \( d\leq 3 \), but NP-complete for \( d\geq 8 \). 
Bok et al.~\cite{bokPreprint} and independently Shalu and Antony~\cite{shalu_cyriac3} proved that \textsc{\( 3 \)-Star Colourability} is NP-complete for planar (bipartite) graphs of maximum degree 3 and arbitrarily large girth.

Gebremedhin et al.~\cite{gebremedhin2007} proved that for all \( \epsilon>0 \), it is NP-hard to approximate the star chromatic number of a (2-degenerate) bipartite graph within \( n^{\frac{1}{3}-\epsilon} \). 
In contrast, every 2-degenerate graph admits a star colouring with \( n^{\frac{1}{2}} \) colours \cite[Theorem~6.2]{karpas} (recall that every unique superior colouring is an rs colouring with colours reversed, and thus a star colouring). 
Hence, the star chromatic number of a 2-degenerate graph is approximable within \( n^{\frac{1}{2}} \).

For \( k\in \mathbb{N} \), \textsc{\( k \)-Star Colourability} can be expressed in Monadic Second-Order logic without edge set quantification (i.e., MSO\( _1 \)) \cite{shalu_cyriac3}, and thus admits FPT algorithms with parameter either treewidth or cliquewidth by Courcelle's theorem \cite{borie,courcelle}. 
It is easy to observe that the transformation from \textsc{\( k \)-Colourability} to \textsc{\( k \)-Star Colourability} in \cite{coleman_more} is a Polynomial Parameter Transformation (PPT)~\cite{fomin2019} when both problems are parameterized by treewidth (e.g.\ see \cite{shalu_cyriac3}). 
As a result, for \( k\geq 3 \), \textsc{\( k \)-Star Colourability} parameterized by treewidth does not admit a polynomial kernel unless \mbox{NP \( \subseteq \) coNP/poly}. 
Bhyravarapu and Reddy~\cite{bhyravarapu_reddy} proved that \textsc{Star Colourability} is fixed-parameter tractable when parameterized by (i)~neighbourhood diversity, (ii)~twin-cover, and (iii)~the combined parameters cliquewidth and the number of colours. 

\subsection{Hardness Transitions}\label{sec:star colouring hardness transitions}

}

We show that for \( k=5 \) and \( k\geq 7 \), \textsc{\( k \)-Star Colourability} is NP-complete for graphs of maximum degree \( k-1 \). 
First, we deal with \( k\geq 7 \) (smaller values of \( k \) are discussed later). 
We employ Construction~\ref{make:k-star colouring deg k-1} below to prove that for every \( k\geq 7 \), \textsc{\( k \)-Star Colourability} is NP-complete for graphs of maximum degree \( k-1 \). 
Fix an integer \( k\geq 7 \).

\begin{figure}[hbt]
\centering
\begin{subfigure}[b]{0.5\textwidth}
\centering
\begin{tikzpicture}
\path (0,0) node(w1) [dot][label=above left:\( w_1 \)]{} --++(1,0) node(w2) [dot][label=above left:\( w_2 \)]{} --++(1,0) node(w3) [dot][label=above left:\( w_3 \)]{} --++(1.5,0) node(wk-4) [dot]{} --++(1,0) node(wk-3) [dot][label=above:\( w_{k\,\text{-}\,3} \)]{} --++(1,0) node(wk-2) [dot][label=above:\( w_{k\,\text{-}\,2} \)]{};
\path (wk-4) node[above left=1pt and -3.75pt]{\( w_{k\,\text{-}\,4} \)};
\path (w3) -- node{\dots} (wk-4);
\path (w1) --++(0,-2) node(x1) [dot][label=below left:\( x_1 \)]{} --++(1,0) node(x2) [dot][label=below left:\( x_2 \)]{} --++(1,0) node(x3) [dot][label=below left:\( x_3 \)]{} --++(1.5,0) node(xk-4) [dot]{} --++(1,0) node(xk-3) [dot][label=below:\( x_{k\,\text{-}\,3} \)]{} --++(1,0) node(xk-2) [dot][label=below:\( x_{k\,\text{-}\,2} \)]{};
\path (xk-4) node[below left=1pt and -3.75pt]{\( x_{k\,\text{-}\,4} \)};
\path (x3) -- node{\dots} (xk-4);

\draw
(w1)--+(0,1) node(u1) [dot][label=\( u_1 \)]{} 
(w2)--+(0,1) node(u2) [dot][label=\( u_2 \)]{} 
(w3)--+(0,1) node(u3) [dot][label=\( u_3 \)]{} 
(wk-4)--+(0,1) node(uk-4) [dot][label=\( u_{k\,\text{-}\,4} \)]{};
\draw
(x1)--+(0,-1) node(y1) [dot][label=below:\( y_1 \)]{} 
(x2)--+(0,-1) node(y2) [dot][label=below:\( y_2 \)]{} 
(x3)--+(0,-1) node(y3) [dot][label=below:\( y_3 \)]{} 
(xk-4)--+(0,-1) node(yk-4) [dot][label=below:\( y_{k\,\text{-}\,4} \)]{}; 

\draw (wk-3)--(wk-2);
\draw (xk-3)--(xk-2);

\draw (w1)--(x1) (w1)--(x2) (w1)--(x3) (w1)--(xk-4) (w1)--(xk-3) (w1)--(xk-2);
\draw (w2)--(x1) (w2)--(x2) (w2)--(x3) (w2)--(xk-4) (w2)--(xk-3) (w2)--(xk-2);
\draw (w3)--(x1) (w3)--(x2) (w3)--(x3) (w3)--(xk-4) (w3)--(xk-3) (w3)--(xk-2);
\draw (wk-4)--(x1) (wk-4)--(x2) (wk-4)--(x3) (wk-4)--(xk-4) (wk-4)--(xk-3) (wk-4)--(xk-2);
\draw (wk-3)--(x1) (wk-3)--(x2) (wk-3)--(x3) (wk-3)--(xk-4) (wk-3)--(xk-3) (wk-3)--(xk-2);
\draw (wk-2)--(x1) (wk-2)--(x2) (wk-2)--(x3) (wk-2)--(xk-4) (wk-2)--(xk-3) (wk-2)--(xk-2);

\draw (y1)--(y2);
\end{tikzpicture}
\caption{}
\label{fig:gadget component k-star colouring deg k-1}
\end{subfigure}\begin{subfigure}[b]{0.5\textwidth}
\centering
\begin{tikzpicture}
\path (0,0) node(w1) [dot][label={[vcolour]above left:1}]{} --++(1,0) node(w2) [dot][label={[vcolour]above left:1}]{} --++(1,0) node(w3) [dot][label={[vcolour]above left:1}]{} --++(1.5,0) node(wk-4) [dot][label={[vcolour]above left:1}]{} --++(1,0) node(wk-3) [dot][label={[vcolour]above:0}]{} --++(1,0) node(wk-2) [dot][label={[vcolour]above:1}]{};
\path (w3) -- node{\dots} (wk-4);
\path (w1) --++(0,-2) node(x1) [dot][label={[vcolour]below left:2}]{} --++(1,0) node(x2) [dot][label={[vcolour]below left:3}]{} --++(1,0) node(x3) [dot][label={[vcolour]below left:4}]{} --++(1.5,0) node(xk-4) [dot][label={[vcolour,yshift=-1pt]below left:\( k\,\text{-}\,3 \)}]{} --++(1,0) node(xk-3) [dot][label={[vcolour]below:\( k\,\text{-}\,2 \)}]{} --++(1,0) node(xk-2) [dot][label={[vcolour]below:\( k\,\text{-}\,1 \)}]{};
\path (x3) -- node{\dots} (xk-4);

\draw
(w1)--+(0,1) node(u1) [dot][label={[vcolour]0}]{} 
(w2)--+(0,1) node(u2) [dot][label={[vcolour]0}]{} 
(w3)--+(0,1) node(u3) [dot][label={[vcolour]0}]{} 
(wk-4)--+(0,1) node(uk-4) [dot][label={[vcolour]0}]{};
\draw
(x1)--+(0,-1) node(y1) [dot][label={[vcolour]below:0}]{} 
(x2)--+(0,-1) node(y2) [dot][label={[vcolour]below:1}]{} 
(x3)--+(0,-1) node(y3) [dot][label={[vcolour]below:1}]{} 
(xk-4)--+(0,-1) node(yk-4) [dot][label={[vcolour]below:1}]{}; 

\draw (wk-3)--(wk-2);
\draw (xk-3)--(xk-2);

\draw (w1)--(x1) (w1)--(x2) (w1)--(x3) (w1)--(xk-4) (w1)--(xk-3) (w1)--(xk-2);
\draw (w2)--(x1) (w2)--(x2) (w2)--(x3) (w2)--(xk-4) (w2)--(xk-3) (w2)--(xk-2);
\draw (w3)--(x1) (w3)--(x2) (w3)--(x3) (w3)--(xk-4) (w3)--(xk-3) (w3)--(xk-2);
\draw (wk-4)--(x1) (wk-4)--(x2) (wk-4)--(x3) (wk-4)--(xk-4) (wk-4)--(xk-3) (wk-4)--(xk-2);
\draw (wk-3)--(x1) (wk-3)--(x2) (wk-3)--(x3) (wk-3)--(xk-4) (wk-3)--(xk-3) (wk-3)--(xk-2);
\draw (wk-2)--(x1) (wk-2)--(x2) (wk-2)--(x3) (wk-2)--(xk-4) (wk-2)--(xk-3) (wk-2)--(xk-2);

\draw (y1)--(y2);
\end{tikzpicture}
\caption{}
\label{fig:k-star colouring gadget component deg k-1}
\end{subfigure}\caption{(a) The gadget component, and (b) a \( k \)-star colouring of it.}
\end{figure}
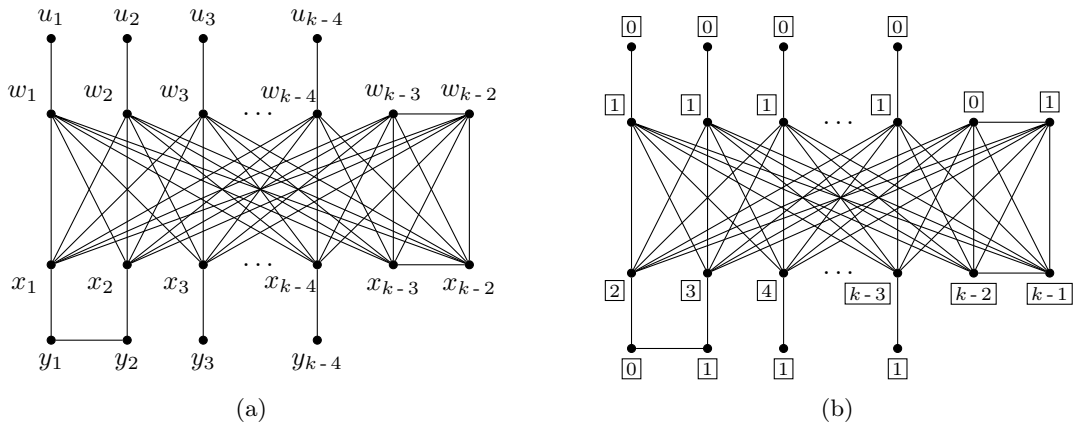

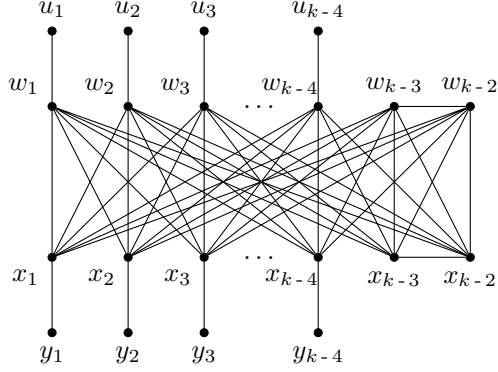
\begin{figure}[hbt]
\centering
\begin{tikzpicture}
\path (0,0) node(w1) [dot][label=above left:\( w_1 \)]{} --++(1,0) node(w2) [dot][label=above left:\( w_2 \)]{} --++(1,0) node(w3) [dot][label=above left:\( w_3 \)]{} --++(1.5,0) node(wk-4) [dot]{} --++(1,0) node(wk-3) [dot][label=above:\( w_{k\,\text{-}\,3} \)]{} --++(1,0) node(wk-2) [dot][label=above:\( w_{k\,\text{-}\,2} \)]{};
\path (wk-4) node[above left=1pt and -3.75pt]{\( w_{k\,\text{-}\,4} \)};
\path (w3) -- node{\dots} (wk-4);
\path (w1) --++(0,-2) node(x1) [dot][label=below left:\( x_1 \)]{} --++(1,0) node(x2) [dot][label=below left:\( x_2 \)]{} --++(1,0) node(x3) [dot][label=below left:\( x_3 \)]{} --++(1.5,0) node(xk-4) [dot]{} --++(1,0) node(xk-3) [dot][label=below:\( x_{k\,\text{-}\,3} \)]{} --++(1,0) node(xk-2) [dot][label=below:\( x_{k\,\text{-}\,2} \)]{};
\path (xk-4) node[below left=1pt and -3.75pt]{\( x_{k\,\text{-}\,4} \)};
\path (x3) -- node{\dots} (xk-4);

\draw
(w1)--+(0,1) node(u1) [dot][label=\( u_1 \)]{} 
(w2)--+(0,1) node(u2) [dot][label=\( u_2 \)]{} 
(w3)--+(0,1) node(u3) [dot][label=\( u_3 \)]{} 
(wk-4)--+(0,1) node(uk-4) [dot][label=\( u_{k\,\text{-}\,4} \)]{};
\draw
(x1)--+(0,-1) node(y1) [dot][label=below:\( y_1 \)]{} 
(x2)--+(0,-1) node(y2) [dot][label=below:\( y_2 \)]{} 
(x3)--+(0,-1) node(y3) [dot][label=below:\( y_3 \)]{} 
(xk-4)--+(0,-1) node(yk-4) [dot][label=below:\( y_{k\,\text{-}\,4} \)]{}; 

\draw (wk-3)--(wk-2);
\draw (xk-3)--(xk-2);

\draw (w1)--(x1) (w1)--(x2) (w1)--(x3) (w1)--(xk-4) (w1)--(xk-3) (w1)--(xk-2);
\draw (w2)--(x1) (w2)--(x2) (w2)--(x3) (w2)--(xk-4) (w2)--(xk-3) (w2)--(xk-2);
\draw (w3)--(x1) (w3)--(x2) (w3)--(x3) (w3)--(xk-4) (w3)--(xk-3) (w3)--(xk-2);
\draw (wk-4)--(x1) (wk-4)--(x2) (wk-4)--(x3) (wk-4)--(xk-4) (wk-4)--(xk-3) (wk-4)--(xk-2);
\draw (wk-3)--(x1) (wk-3)--(x2) (wk-3)--(x3) (wk-3)--(xk-4) (wk-3)--(xk-3) (wk-3)--(xk-2);
\draw (wk-2)--(x1) (wk-2)--(x2) (wk-2)--(x3) (wk-2)--(xk-4) (wk-2)--(xk-3) (wk-2)--(xk-2);

\end{tikzpicture}
\caption[A subgraph of the gadget component.]{A subgraph of the gadget component (only the edge \( y_1y_2 \) is missing).}
\label{fig:for gadget component k-star colouring deg k-1}
\end{figure}

The graph in Figure~\ref{fig:gadget component k-star colouring deg k-1}, called the gadget component, is used to build the main gadget in Construction~\ref{make:k-star colouring deg k-1}. 
In the gadget component, \( w_i \) is adjacent to \( x_j \) for all \( i,j\in\{1,2,\dots,k-2\} \). 
Consider the \( k \)-colouring of the gadget component exhibited in Figure~\ref{fig:k-star colouring gadget component deg k-1}. 
Under this colouring, (i)~every vertex with a binary colour (i.e., colour~0 or~1) has at most one neighbour with a binary colour (thus ruling out the possibility of a 4-vertex path coloured using only colours~0 and 1), and (ii)~for each non-binary colour~\( c \), exactly one vertex in the gadget has colour~\( c \) (thus ruling out the possibility of colour~\( c \) appearing in a bicoloured~\( P_4 \)). 
Hence, the \( k \)-colouring exhibited is a \( k \)-star colouring of the gadget component. 
Clearly, the graph in Figure~\ref{fig:for gadget component k-star colouring deg k-1} is a subgraph of the gadget component. 
Since the gadget component is \( k \)-star colourable, so is the graph in Figure~\ref{fig:for gadget component k-star colouring deg k-1}. 
Let \( U=\{u_1,u_2,\dots,u_{k-4}\} \), \( W=\{w_1,w_2,\dots,w_{k-2}\} \), \( X=\{x_1,x_2,\dots,x_{k-2}\} \), and \( Y=\{y_1,y_2,\dots,y_{k-4}\} \). 
The next lemma shows that \( U\cup W \) or \( X\cup Y \) is bicoloured by each \( k \)-star colouring of the graph in Figure~\ref{fig:for gadget component k-star colouring deg k-1}. 
This is later used to prove that \( U\cup W \) is bicoloured by each \( k \)-star colouring of the gadget component.

\begin{lemma}\label{lem:for gadget component k-star colouring deg k-1}
At least \( k \) colours \( (k\geq 7) \) are required to star colour the graph in Figure~\ref{fig:for gadget component k-star colouring deg k-1}. 
Moreover, for every \( k \)-star colouring \( f \) of the graph in Figure~\ref{fig:for gadget component k-star colouring deg k-1}, there exist distinct colours \( c_1 \) and \( c_2 \) such that 
either \( f(U\cup W)\subseteq \{c_1,c_2\} \) or \( f(X\cup Y)\subseteq \{c_1,c_2\} \). 
\end{lemma}
\begin{proof}
Let \( f \) be a star colouring of the graph in Figure~\ref{fig:for gadget component k-star colouring deg k-1}. 
Note that repetition of colours can occur in at most one of the sets \( W \) and \( X \) (if \( f(w_i)=f(w_j) \) for distinct \( w_i,w_j\in W \) and \( f(x_p)=f(x_q) \) for distinct \( x_p,x_q\in X \), then \( w_i,x_p,w_j,x_q \) is a bicoloured \( P_4 \)). 
Hence, vertices in \( W \) have pairwise distinct colours, or vertices in \( X \) have pairwise distinct colours. 
If vertices in \( X \) have pairwise distinct colours, then \( k-2 \) colours are required for vertices in \( X \) and two new colours are required for \( w_{k-3} \) and \(w_{k-2} \) since \( w_{k-3}w_{k-2} \) is an edge. 
Similarly, if vertices in \( W \) have pairwise distinct colours, then at least \( k \) colours are required to star colour the graph in Figure~\ref{fig:for gadget component k-star colouring deg k-1}. 
Hence, at least \( k \) colours are required to star colour the graph in Figure~\ref{fig:for gadget component k-star colouring deg k-1}.

Suppose that \( f \) is a \( k \)-star colouring of the graph in Figure~\ref{fig:for gadget component k-star colouring deg k-1}. 
We know that vertices in \( W \) have pairwise distinct colours, or vertices in \( X \) have pairwise distinct colours. 
Note that both cannot occur at the same time (if vertices in \( W \) have pairwise distinct colours (i.e, \( |f(W)|=k-2 \)), then only two colours are available for each vertex in \( X \) under~\( f \), and thus at least two vertices in \( X \) have the same colour by pigeonhole principle). 
Thus, either vertices in \( W \) have pairwise distinct colours, or vertices in \( X \) have pairwise distinct colours. 
We prove that (i)~if vertices in \( X \) have pairwise distinct colours, then \( U\cup W \) is bicoloured by \( f \), and (ii)~if vertices in \( W \) have pairwise distinct colours, then \( X\cup Y \) is bicoloured by \( f \). 
By symmetry, it suffices to prove~(i). 

Suppose that the vertices in \( X \) have pairwise distinct colours. 
That is, \( k-2 \) colours are used in set \( X \), and thus only two colours, say colour~0 and colour~1, are not used in \( X \). 
Without loss of generality, assume that \( f(x_j)=j+1 \) for all \( x_j\in X \). 
For every vertex \( w_i\in W \), only two colours are available, namely colour~0 and colour~1 (i.e., \( f(w_i)\in\{0,1\} \)).

Since \( w_{k-3}w_{k-2} \) is an edge and \( f(w_{k-3}),f(w_{k-2})\in\{0,1\} \), we can assume without loss of generality that \( f(w_{k-3})=0 \) and \( f(w_{k-2})=1 \).  
Consider the colour at vertices \( w_1 \) and \( u_1 \). 
If \( f(w_1)=0 \), then \( f(u_1)\neq f(x_j) \) for any \( x_j\in X \) (otherwise, \( u_1,w_1,x_j,w_{k-3} \) is a bicoloured \( P_4 \)). 
Similarly, if \( f(w_1)=1 \), then \( f(u_1)\neq f(x_j) \) for any \( x_j\in X \) (otherwise, \( u_1,w_1,x_j,w_{k-2} \) is a bicoloured \( P_4 \)). 
Thus, in both cases, \( f(u_1)\notin\{f(x_1),f(x_2),\dots,f(x_{k-2})\}=\{2,3,\dots,k-1\} \). 
That is, \( f(u_1)\in \{0,1\} \). 
Applying the same argument to vertices \( w_j \) and \( u_j \) reveals that \( f(u_j)\in\{0,1\} \) for \( 1\leq j\leq k-4 \). 
We know that \( f(w_i)\in\{0,1\} \) for \( 1\leq i\leq k-2 \), and thus \( f(U\cup W)\subseteq \{0,1\} \). 
Since the colours~0 and 1 are chosen arbitrarily, there exist distinct colours \( c_1 \) and \( c_2 \) such that \( f(U\cup W)\subseteq \{c_1,c_2\} \). 
\end{proof}

Thanks to Figure~\ref{fig:k-star colouring gadget component deg k-1}, the gadget component admits a \( k \)-star colouring. 
Let \( f \) be a \( k \)-star colouring of the gadget component. 
Then, \( f \) is a \( k \)-star colouring of its subgraph displayed in Figure~\ref{fig:for gadget component k-star colouring deg k-1} as well, and thus either \( U\cup W \) or \( X\cup Y \) is bicoloured by \( f \) by Lemma~\ref{lem:for gadget component k-star colouring deg k-1}. 
If \( X\cup Y \) is bicoloured by \( f \), then \( x_1,y_1,y_2,x_2 \) is a path in the gadget component bicoloured by \( f \), a contradiction. 
Hence, \( U\cup W \) is bicoloured by \( f \). 
Thus, we have the following lemma. 
\begin{lemma}\label{lem:gadget component k-star colouring deg k-1}
For every \( k \)-star colouring \( f \) of the gadget component (with \( k\geq 7 \)), there exist distinct colours \( c_1 \) and \( c_2 \) such that \( f(U\cup W)\subseteq \{c_1,c_2\} \). 
\qed 
\end{lemma}

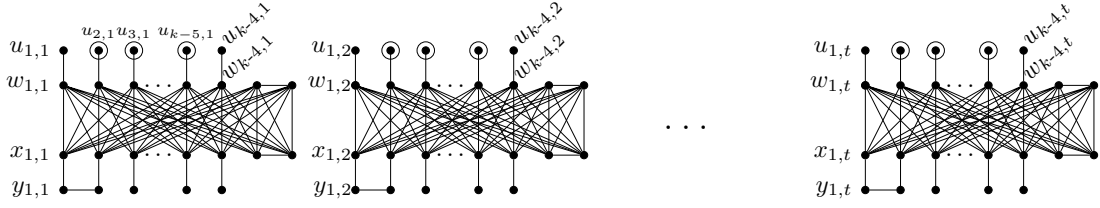
\begin{figure}[hbt]
\scalebox{0.925}{\begin{tikzpicture}[scale=0.5]
\path (0,0) node(w1) [dot][label=left:\( w_{1,1} \)]{} --++(1,0) node(w2) [dot]{} --++(1,0) node(w3) [dot]{} --++(1.5,0) node(wk-5)[dot]{} --++(1,0) node(wk-4) [dot][label={[rotate=45,yshift=-7pt,xshift=-2pt]above right:\( w_{k\text{-}4,1} \)}]{} node(wk-4 1) [dot]{} --++(1,0) node(wk-3) [dot]{} --++(1,0) node(wk-2) [dot]{};
\path (w3) -- node{\dots} (wk-5);
\path (w1) --++(0,-2) node(x1) [dot][label=left:\( x_{1,1} \)]{} --++(1,0) node(x2) [dot]{} --++(1,0) node(x3) [dot]{} --++(1.5,0) node(xk-5) [dot]{} --++(1,0) node(xk-4) [dot]{} --++(1,0) node(xk-3) [dot]{} --++(1,0) node(xk-2) [dot]{};
\path (x3) -- node{\dots} (xk-5);

\draw
(w1)--+(0,1) node(u1) [dot][label=left:\( u_{1,1} \)]{} 
(w2)--+(0,1) node(u2) [dot]{} node[terminal][label={[font=\scriptsize,label distance=-3pt]\( u_{2,1} \)}]{}
(w3)--+(0,1) node(u3) [dot]{} node[terminal][label={[font=\scriptsize,label distance=-3pt]\( u_{3,1} \)}]{} 
(wk-5)--+(0,1) node(uk-5) [dot]{} node[terminal][label={[font=\scriptsize,label distance=-3pt]\( u_{k-5,1} \)}]{} 
(wk-4)--+(0,1) node(uk-4) [dot][label={[rotate=45,yshift=-7pt,xshift=-2pt]above right:\( u_{k\text{-}4,1} \)}]{};
\draw
(x1)--+(0,-1) node(y1) [dot][label=left:\( y_{1,1} \)]{} 
(x2)--+(0,-1) node(y2) [dot]{} 
(x3)--+(0,-1) node(y3) [dot]{} 
(xk-5)--+(0,-1) node(yk-5) [dot]{}
(xk-4)--+(0,-1) node(yk-4) [dot]{}; 

\draw (wk-3)--(wk-2);
\draw (xk-3)--(xk-2);

\draw (w1)--(x1) (w1)--(x2) (w1)--(x3) (w1)--(xk-5) (w1)--(xk-4) (w1)--(xk-3) (w1)--(xk-2);
\draw (w2)--(x1) (w2)--(x2) (w2)--(x3) (w2)--(xk-5) (w2)--(xk-4) (w2)--(xk-3) (w2)--(xk-2);
\draw (w3)--(x1) (w3)--(x2) (w3)--(x3) (w3)--(xk-5) (w3)--(xk-4) (w3)--(xk-3) (w3)--(xk-2);
\draw (wk-5)--(x1) (wk-5)--(x2) (wk-5)--(x3) (wk-5)--(xk-5) (wk-5)--(xk-4) (wk-5)--(xk-3) (wk-5)--(xk-2);
\draw (wk-4)--(x1) (wk-4)--(x2) (wk-4)--(x3) (wk-4)--(xk-5) (wk-4)--(xk-4) (wk-4)--(xk-3) (wk-4)--(xk-2);
\draw (wk-3)--(x1) (wk-3)--(x2) (wk-3)--(x3) (wk-3)--(xk-5) (wk-3)--(xk-4) (wk-3)--(xk-3) (wk-3)--(xk-2);
\draw (wk-2)--(x1) (wk-2)--(x2) (wk-2)--(x3) (wk-2)--(xk-5) (wk-2)--(xk-4) (wk-2)--(xk-3) (wk-2)--(xk-2);

\draw (y1)--(y2);

\path (wk-2) --++(1.80,0) node(w1) [dot]{} node(w1 2) [dot][label={[label distance=-4pt]left:\( w_{1,2} \)}]{} --++(1,0) node(w2) [dot]{} --++(1,0) node(w3) [dot]{} --++(1.5,0) node(wk-5)[dot]{} --++(1,0) node(wk-4) [dot][label={[rotate=45,yshift=-7pt,xshift=-2pt]above right:\( w_{k\text{-}4,2} \)}]{} node(wk-4 2) [dot]{} --++(1,0) node(wk-3) [dot]{} --++(1,0) node(wk-2) [dot]{};
\path (w3) -- node{\dots} (wk-5);
\path (w1) --++(0,-2) node(x1) [dot][label={[label distance=-4pt]left:\( x_{1,2} \)}]{} --++(1,0) node(x2) [dot]{} --++(1,0) node(x3) [dot]{} --++(1.5,0) node(xk-5) [dot]{} --++(1,0) node(xk-4) [dot]{} --++(1,0) node(xk-3) [dot]{} --++(1,0) node(xk-2) [dot]{};
\path (x3) -- node{\dots} (xk-5);

\draw
(w1)--+(0,1) node(u1) [dot][label={[label distance=-4pt]left:\( u_{1,2} \)}]{} 
(w2)--+(0,1) node(u2) [dot]{} node[terminal]{}
(w3)--+(0,1) node(u3) [dot]{} node[terminal]{}
(wk-5)--+(0,1) node(uk-5) [dot]{} node[terminal]{} 
(wk-4)--+(0,1) node(uk-4) [dot][label={[rotate=45,yshift=-7pt,xshift=-2pt]above right:\( u_{k\text{-}4,2} \)}]{};
\draw
(x1)--+(0,-1) node(y1) [dot][label={[label distance=-4pt]left:\( y_{1,2} \)}]{} 
(x2)--+(0,-1) node(y2) [dot]{} 
(x3)--+(0,-1) node(y3) [dot]{} 
(xk-5)--+(0,-1) node(yk-5) [dot]{}
(xk-4)--+(0,-1) node(yk-4) [dot]{}; 

\draw (wk-3)--(wk-2);
\draw (xk-3)--(xk-2);

\draw (w1)--(x1) (w1)--(x2) (w1)--(x3) (w1)--(xk-5) (w1)--(xk-4) (w1)--(xk-3) (w1)--(xk-2);
\draw (w2)--(x1) (w2)--(x2) (w2)--(x3) (w2)--(xk-5) (w2)--(xk-4) (w2)--(xk-3) (w2)--(xk-2);
\draw (w3)--(x1) (w3)--(x2) (w3)--(x3) (w3)--(xk-5) (w3)--(xk-4) (w3)--(xk-3) (w3)--(xk-2);
\draw (wk-5)--(x1) (wk-5)--(x2) (wk-5)--(x3) (wk-5)--(xk-5) (wk-5)--(xk-4) (wk-5)--(xk-3) (wk-5)--(xk-2);
\draw (wk-4)--(x1) (wk-4)--(x2) (wk-4)--(x3) (wk-4)--(xk-5) (wk-4)--(xk-4) (wk-4)--(xk-3) (wk-4)--(xk-2);
\draw (wk-3)--(x1) (wk-3)--(x2) (wk-3)--(x3) (wk-3)--(xk-5) (wk-3)--(xk-4) (wk-3)--(xk-3) (wk-3)--(xk-2);
\draw (wk-2)--(x1) (wk-2)--(x2) (wk-2)--(x3) (wk-2)--(xk-5) (wk-2)--(xk-4) (wk-2)--(xk-3) (wk-2)--(xk-2);

\draw (y1)--(y2);

\path (wk-2) --++(2,0) coordinate(w1 3);

\path (wk-2) --++(8.00,0) node(w1) [dot][label=left:\( w_{1,t} \)]{} node(w1 t) [dot]{} --++(1,0) node(w2) [dot]{} --++(1,0) node(w3) [dot]{} --++(1.5,0) node(wk-5)[dot]{} --++(1,0) node(wk-4) [dot][label={[rotate=45,yshift=-7pt,xshift=-2pt]above right:\( w_{k\text{-}4,t} \)}]{} --++(1,0) node(wk-3) [dot]{} --++(1,0) node(wk-2) [dot]{};
\path (w3) -- node{\dots} (wk-5);
\path (w1) --++(0,-2) node(x1) [dot][label=left:\( x_{1,t} \)]{} --++(1,0) node(x2) [dot]{} --++(1,0) node(x3) [dot]{} --++(1.5,0) node(xk-5) [dot]{} --++(1,0) node(xk-4) [dot]{} --++(1,0) node(xk-3) [dot]{} --++(1,0) node(xk-2) [dot]{};
\path (x3) -- node{\dots} (xk-5);

\draw
(w1)--+(0,1) node(u1) [dot][label=left:\( u_{1,t} \)]{} 
(w2)--+(0,1) node(u2) [dot]{} node[terminal]{}
(w3)--+(0,1) node(u3) [dot]{} node[terminal]{}
(wk-5)--+(0,1) node(uk-5) [dot]{} node[terminal]{} 
(wk-4)--+(0,1) node(uk-4) [dot][label={[rotate=45,yshift=-7pt,xshift=-2pt]above right:\( u_{k\text{-}4,t} \)}]{};
\draw
(x1)--+(0,-1) node(y1) [dot][label=left:\( y_{1,t} \)]{} 
(x2)--+(0,-1) node(y2) [dot]{} 
(x3)--+(0,-1) node(y3) [dot]{} 
(xk-5)--+(0,-1) node(yk-5) [dot]{}
(xk-4)--+(0,-1) node(yk-4) [dot]{}; 

\draw (wk-3)--(wk-2);
\draw (xk-3)--(xk-2);

\draw (w1)--(x1) (w1)--(x2) (w1)--(x3) (w1)--(xk-5) (w1)--(xk-4) (w1)--(xk-3) (w1)--(xk-2);
\draw (w2)--(x1) (w2)--(x2) (w2)--(x3) (w2)--(xk-5) (w2)--(xk-4) (w2)--(xk-3) (w2)--(xk-2);
\draw (w3)--(x1) (w3)--(x2) (w3)--(x3) (w3)--(xk-5) (w3)--(xk-4) (w3)--(xk-3) (w3)--(xk-2);
\draw (wk-5)--(x1) (wk-5)--(x2) (wk-5)--(x3) (wk-5)--(xk-5) (wk-5)--(xk-4) (wk-5)--(xk-3) (wk-5)--(xk-2);
\draw (wk-4)--(x1) (wk-4)--(x2) (wk-4)--(x3) (wk-4)--(xk-5) (wk-4)--(xk-4) (wk-4)--(xk-3) (wk-4)--(xk-2);
\draw (wk-3)--(x1) (wk-3)--(x2) (wk-3)--(x3) (wk-3)--(xk-5) (wk-3)--(xk-4) (wk-3)--(xk-3) (wk-3)--(xk-2);
\draw (wk-2)--(x1) (wk-2)--(x2) (wk-2)--(x3) (wk-2)--(xk-5) (wk-2)--(xk-4) (wk-2)--(xk-3) (wk-2)--(xk-2);

\draw (y1)--(y2);

\path (w1) --++(-4,0) coordinate(u1 t);

\path (w1 3) -- node[yshift=-0.6cm,font=\LARGE]{\dots} (u1 t);
\end{tikzpicture}
}
\caption{\( t \) copies of the gadget component, where \( k\geq 7 \).}
\label{fig:prep chain gadget k-star colouring deg k-1}
\end{figure}

\begin{figure}[hbt]
\scalebox{0.925}{\begin{tikzpicture}[scale=0.5]
\path (0,0) node(w1) [dot][label=left:\( w_{1,1} \)]{} --++(1,0) node(w2) [dot]{} --++(1,0) node(w3) [dot]{} --++(1.5,0) node(wk-5)[dot]{} --++(1,0) node(wk-4) [dot]{} node(wk-4 1) [dot]{} --++(1,0) node(wk-3) [dot]{} --++(1,0) node(wk-2) [dot]{};
\path (w3) -- node{\dots} (wk-5);
\path (w1) --++(0,-2) node(x1) [dot][label=left:\( x_{1,1} \)]{} --++(1,0) node(x2) [dot]{} --++(1,0) node(x3) [dot]{} --++(1.5,0) node(xk-5) [dot]{} --++(1,0) node(xk-4) [dot]{} --++(1,0) node(xk-3) [dot]{} --++(1,0) node(xk-2) [dot]{};
\path (x3) -- node{\dots} (xk-5);

\draw
(w1)--+(0,1) node(u1) [dot][label=left:\( u_{1,1} \)]{} 
(w2)--+(0,1) node(u2) [dot]{} node[terminal][label={[font=\scriptsize,label distance=-3pt]\( u_{2,1} \)}]{}
(w3)--+(0,1) node(u3) [dot]{} node[terminal][label={[font=\scriptsize,label distance=-3pt]\( u_{3,1} \)}]{} 
(wk-5)--+(0,1) node(uk-5) [dot]{} node[terminal][label={[font=\scriptsize,label distance=-3pt]\( u_{k-5,1} \)}]{};
\draw
(x1)--+(0,-1) node(y1) [dot][label=left:\( y_{1,1} \)]{} 
(x2)--+(0,-1) node(y2) [dot]{} 
(x3)--+(0,-1) node(y3) [dot]{} 
(xk-5)--+(0,-1) node(yk-5) [dot]{}
(xk-4)--+(0,-1) node(yk-4) [dot]{}; 

\draw (wk-3)--(wk-2);
\draw (xk-3)--(xk-2);

\draw (w1)--(x1) (w1)--(x2) (w1)--(x3) (w1)--(xk-5) (w1)--(xk-4) (w1)--(xk-3) (w1)--(xk-2);
\draw (w2)--(x1) (w2)--(x2) (w2)--(x3) (w2)--(xk-5) (w2)--(xk-4) (w2)--(xk-3) (w2)--(xk-2);
\draw (w3)--(x1) (w3)--(x2) (w3)--(x3) (w3)--(xk-5) (w3)--(xk-4) (w3)--(xk-3) (w3)--(xk-2);
\draw (wk-5)--(x1) (wk-5)--(x2) (wk-5)--(x3) (wk-5)--(xk-5) (wk-5)--(xk-4) (wk-5)--(xk-3) (wk-5)--(xk-2);
\draw (wk-4)--(x1) (wk-4)--(x2) (wk-4)--(x3) (wk-4)--(xk-5) (wk-4)--(xk-4) (wk-4)--(xk-3) (wk-4)--(xk-2);
\draw (wk-3)--(x1) (wk-3)--(x2) (wk-3)--(x3) (wk-3)--(xk-5) (wk-3)--(xk-4) (wk-3)--(xk-3) (wk-3)--(xk-2);
\draw (wk-2)--(x1) (wk-2)--(x2) (wk-2)--(x3) (wk-2)--(xk-5) (wk-2)--(xk-4) (wk-2)--(xk-3) (wk-2)--(xk-2);

\draw (y1)--(y2);

\path (wk-2) --++(1.80,0) node(w1) [dot]{} node(w1 2) [dot]{} --++(1,0) node(w2) [dot]{} --++(1,0) node(w3) [dot]{} --++(1.5,0) node(wk-5)[dot]{} --++(1,0) node(wk-4) [dot]{} node(wk-4 2) [dot]{} --++(1,0) node(wk-3) [dot]{} --++(1,0) node(wk-2) [dot]{};
\path (w3) -- node{\dots} (wk-5);
\path (w1) --++(0,-2) node(x1) [dot][label={[label distance=-4pt]left:\( x_{1,2} \)}]{} --++(1,0) node(x2) [dot]{} --++(1,0) node(x3) [dot]{} --++(1.5,0) node(xk-5) [dot]{} --++(1,0) node(xk-4) [dot]{} --++(1,0) node(xk-3) [dot]{} --++(1,0) node(xk-2) [dot]{};
\path (x3) -- node{\dots} (xk-5);

\draw
(w2)--+(0,1) node(u2) [dot]{} node[terminal]{}
(w3)--+(0,1) node(u3) [dot]{} node[terminal]{}
(wk-5)--+(0,1) node(uk-5) [dot]{} node[terminal]{}; 
\draw
(x1)--+(0,-1) node(y1) [dot][label={[label distance=-4pt]left:\( y_{1,2} \)}]{} 
(x2)--+(0,-1) node(y2) [dot]{} 
(x3)--+(0,-1) node(y3) [dot]{} 
(xk-5)--+(0,-1) node(yk-5) [dot]{}
(xk-4)--+(0,-1) node(yk-4) [dot]{}; 

\draw (wk-3)--(wk-2);
\draw (xk-3)--(xk-2);

\draw (w1)--(x1) (w1)--(x2) (w1)--(x3) (w1)--(xk-5) (w1)--(xk-4) (w1)--(xk-3) (w1)--(xk-2);
\draw (w2)--(x1) (w2)--(x2) (w2)--(x3) (w2)--(xk-5) (w2)--(xk-4) (w2)--(xk-3) (w2)--(xk-2);
\draw (w3)--(x1) (w3)--(x2) (w3)--(x3) (w3)--(xk-5) (w3)--(xk-4) (w3)--(xk-3) (w3)--(xk-2);
\draw (wk-5)--(x1) (wk-5)--(x2) (wk-5)--(x3) (wk-5)--(xk-5) (wk-5)--(xk-4) (wk-5)--(xk-3) (wk-5)--(xk-2);
\draw (wk-4)--(x1) (wk-4)--(x2) (wk-4)--(x3) (wk-4)--(xk-5) (wk-4)--(xk-4) (wk-4)--(xk-3) (wk-4)--(xk-2);
\draw (wk-3)--(x1) (wk-3)--(x2) (wk-3)--(x3) (wk-3)--(xk-5) (wk-3)--(xk-4) (wk-3)--(xk-3) (wk-3)--(xk-2);
\draw (wk-2)--(x1) (wk-2)--(x2) (wk-2)--(x3) (wk-2)--(xk-5) (wk-2)--(xk-4) (wk-2)--(xk-3) (wk-2)--(xk-2);

\draw (y1)--(y2);

\path (wk-2) --++(2,0) node(w1 3) [dot]{};

\path (wk-2) --++(8.00,0) node(w1) [dot]{} node(w1 t) [dot]{} --++(1,0) node(w2) [dot]{} --++(1,0) node(w3) [dot]{} --++(1.5,0) node(wk-5)[dot]{} --++(1,0) node(wk-4) [dot][label={[rotate=45,yshift=-7pt,xshift=-2pt]above right:\( w_{k\text{-}4,t} \)}]{} --++(1,0) node(wk-3) [dot]{} --++(1,0) node(wk-2) [dot]{};
\path (w3) -- node{\dots} (wk-5);
\path (w1) --++(0,-2) node(x1) [dot][label=left:\( x_{1,t} \)]{} --++(1,0) node(x2) [dot]{} --++(1,0) node(x3) [dot]{} --++(1.5,0) node(xk-5) [dot]{} --++(1,0) node(xk-4) [dot]{} --++(1,0) node(xk-3) [dot]{} --++(1,0) node(xk-2) [dot]{};
\path (x3) -- node{\dots} (xk-5);

\draw
(w2)--+(0,1) node(u2) [dot]{} node[terminal]{}
(w3)--+(0,1) node(u3) [dot]{} node[terminal]{}
(wk-5)--+(0,1) node(uk-5) [dot]{} node[terminal]{} 
(wk-4)--+(0,1) node(uk-4) [dot][label={[rotate=45,yshift=-7pt,xshift=-2pt]above right:\( u_{k\text{-}4,t} \)}]{};
\draw
(x1)--+(0,-1) node(y1) [dot][label=left:\( y_{1,t} \)]{} 
(x2)--+(0,-1) node(y2) [dot]{} 
(x3)--+(0,-1) node(y3) [dot]{} 
(xk-5)--+(0,-1) node(yk-5) [dot]{}
(xk-4)--+(0,-1) node(yk-4) [dot]{}; 

\draw (wk-3)--(wk-2);
\draw (xk-3)--(xk-2);

\draw (w1)--(x1) (w1)--(x2) (w1)--(x3) (w1)--(xk-5) (w1)--(xk-4) (w1)--(xk-3) (w1)--(xk-2);
\draw (w2)--(x1) (w2)--(x2) (w2)--(x3) (w2)--(xk-5) (w2)--(xk-4) (w2)--(xk-3) (w2)--(xk-2);
\draw (w3)--(x1) (w3)--(x2) (w3)--(x3) (w3)--(xk-5) (w3)--(xk-4) (w3)--(xk-3) (w3)--(xk-2);
\draw (wk-5)--(x1) (wk-5)--(x2) (wk-5)--(x3) (wk-5)--(xk-5) (wk-5)--(xk-4) (wk-5)--(xk-3) (wk-5)--(xk-2);
\draw (wk-4)--(x1) (wk-4)--(x2) (wk-4)--(x3) (wk-4)--(xk-5) (wk-4)--(xk-4) (wk-4)--(xk-3) (wk-4)--(xk-2);
\draw (wk-3)--(x1) (wk-3)--(x2) (wk-3)--(x3) (wk-3)--(xk-5) (wk-3)--(xk-4) (wk-3)--(xk-3) (wk-3)--(xk-2);
\draw (wk-2)--(x1) (wk-2)--(x2) (wk-2)--(x3) (wk-2)--(xk-5) (wk-2)--(xk-4) (wk-2)--(xk-3) (wk-2)--(xk-2);

\draw (y1)--(y2);

\path (w1) --++(-4,0) node(u1 t) [dot]{};

\draw (wk-4 1) to[out=45,in=135] node[pos=0.30,above=20pt,rotate=70,font=\small]{\( {w_{k\text{-}4,1}=u_{1,2}} \)} node[pos=0.70,above=20pt,rotate=-70,font=\small]{\( {w_{1,2}=u_{k\text{-}4,1}} \)} (w1 2);

\draw (wk-4 2) to[out=45,in=135] node[pos=0.30,above=20pt,rotate=70,font=\small]{\( {w_{k\text{-}4,2}=u_{1,3}} \)} node[pos=0.70,above=20pt,rotate=-70,font=\small]{\( {w_{1,3}=u_{k\text{-}4,2}} \)} (w1 3);

\draw (u1 t) to[out=45,in=135] node[pos=0.30,above=20pt,rotate=70,font=\small]{\( {w_{k\text{-}4,t\text{-}1}=u_{1,t}} \)} node[pos=0.70,above=23pt,rotate=-70,font=\small]{\( {w_{1,t}=u_{k\text{-}4,t\text{-}1}} \)} (w1 t);

\path (w1 3) -- node[yshift=-0.6cm,font=\LARGE]{\dots} (u1 t);
\end{tikzpicture}
}
\caption[Chain gadget.]{Chain gadget (with \( t \) copies of the gadget component), where \( k\geq 7 \).}
\label{fig:chain gadget k-star colouring deg k-1}
\end{figure}
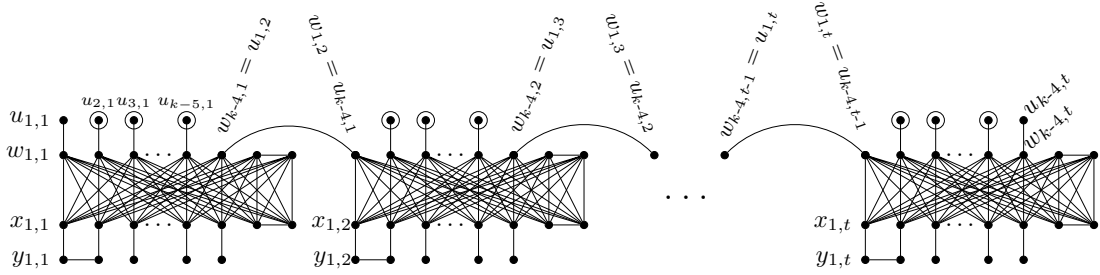

Using the gadget component, we construct a gadget called the chain gadget. 
To construct a chain gadget, first introduce \( t \) copies of the gadget component (for some \( t\in \mathbb{N} \)). 
Let us refer to the vertex \( w_1 \) (resp.\ \( u_1 \)) in the first copy of the gadget component as \( w_{1,1} \) (resp.\ \( u_{1,1} \)), the vertex \( w_1 \) in the second copy of the gadget component as \( w_{1,2} \), and so on (see Figure~\ref{fig:prep chain gadget k-star colouring deg k-1}). 
The vertices \( u_{2,1},\scalebox{0.5}{\dots},u_{k-5,1},\dots, u_{2,t},\scalebox{0.5}{\dots},u_{k-5,t} \) are marked as terminals; thus, we have \( k-6 \) (\( \geq 1 \)) terminals per gadget component. 
Let \( U_1=\{u_{1,1},u_{2,1},\dots,u_{k-4,1}\} \), \( W_1=\{w_{1,1},w_{2,1},\dots,w_{k-2,1}\} \), \( X_1=\{x_{1,1},x_{2,1},\dots,x_{k-2,1}\} \), \( Y_1=\{y_{1,1},y_{2,1},\dots,y_{k-4,1}\} \), and so on. 
Next, we perform a sequence of vertex identification operations (see Section~\ref{sec:def} for definition). 
Identify the vertex \( w_{k-4,1}  \) with \( u_{1,2}  \) and identify the vertex \( u_{k-4,1}  \) with \( w_{1,2} \). 
This operation in this context is the same as deleting vertices \( u_{k-4,1} \) and \( u_{1,2} \), and adding edge \( w_{k-4,1}w_{1,2} \); the small technical difference is that with vertex identification, the vertex \( w_{k-4,1} \) can also be referred to as \( u_{1,2} \) and the vertex \( w_{1,2} \) can also be referred to as \( u_{k-4,1} \). 
This small technical difference is the reason we prefer to present the operation as vertex identification. 
In general, for \( \ell\in\{1,2,\dots,t-1\} \), identify the vertex \( w_{k-4,\ell}  \) with \( u_{1,\ell+1}  \) and identify the vertex \( u_{k-4,\ell}  \) with \( w_{1,\ell+1} \) (compare Figure~\ref{fig:prep chain gadget k-star colouring deg k-1} with Figure~\ref{fig:chain gadget k-star colouring deg k-1}). 
Observe that in the chain gadget (see Figure~\ref{fig:chain gadget k-star colouring deg k-1}), the set \( U_1\cup W_1\cup X_1\cup Y_1 \) induces a copy of the gadget component, which we shall call as the first copy of the gadget component in the chain gadget. 
The fact that the vertex \( w_{k-4,1} (=u_{1,2}) \) also belongs to \( U_2 \) and the vertex \( u_{k-4,1} (=w_{1,2}) \) also belongs to \( W_2 \) does not cause us trouble. 
Similarly, for \( 1\leq \ell\leq t \), the subgraph of the chain gadget induced by \( U_\ell\cup W_\ell\cup X_\ell\cup Y_\ell \) is the \( \ell \)th copy of the gadget component in the chain gadget.

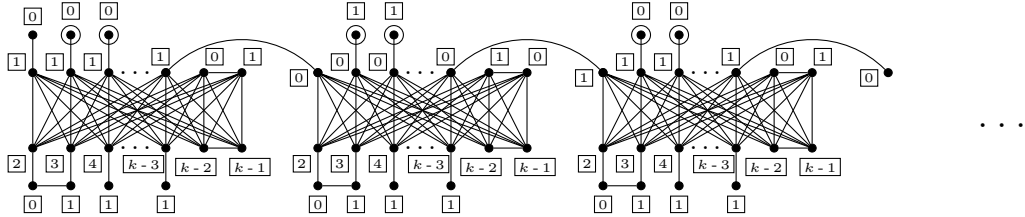
\begin{figure}[hbt]
\centering
\begin{tikzpicture}[scale=0.5]\path (0,0) node(w1) [dot][label={[vcolour,font=\tiny,xshift=-6pt,yshift=-3pt]1}]{} --++(1,0) node(w2) [dot][label={[vcolour,font=\tiny,xshift=-6pt,yshift=-3pt]1}]{} --++(1,0) node(w3) [dot][label={[vcolour,font=\tiny,xshift=-6pt,yshift=-3pt]1}]{} --++(1.5,0) node(wk-4) [dot][label={[vcolour,font=\tiny,xshift=-2pt,yshift=-1pt]1}]{} node(wk-4 1) [dot]{} --++(1,0) node(wk-3) [dot][label={[vcolour,font=\tiny,yshift=-1pt,xshift=4pt]0}]{} --++(1,0) node(wk-2) [dot][label={[vcolour,font=\tiny,yshift=-1pt,xshift=4pt]1}]{};
\path (w3) -- node{\dots} (wk-4);
\path (w1) --++(0,-2) node(x1) [dot][label={[vcolour,font=\tiny]below left:2}]{} --++(1,0) node(x2) [dot][label={[vcolour,font=\tiny]below left:3}]{} --++(1,0) node(x3) [dot][label={[vcolour,font=\tiny]below left:4}]{} --++(1.5,0) node(xk-4) [dot][label={[vcolour,font=\tiny,xshift=2pt]below left:\( k\,\text{-}\,3 \)}]{} --++(1,0) node(xk-3) [dot][label={[vcolour,font=\tiny,xshift=-3pt]below:\( k\,\text{-}\,2 \)}]{} --++(1,0) node(xk-2) [dot][label={[vcolour,font=\tiny,xshift=3pt]below:\( k\,\text{-}\,1 \)}]{};
\path (x3) -- node{\dots} (xk-4);

\draw
(w1)--+(0,1) node(u1) [dot][label={[vcolour,font=\tiny]0}]{} 
(w2)--+(0,1) node(u2) [dot]{} node[terminal][label={[vcolour,font=\tiny]0}]{}
(w3)--+(0,1) node(u3) [dot]{} node[terminal][label={[vcolour,font=\tiny]0}]{}; 
\draw
(x1)--+(0,-1) node(y1) [dot][label={[vcolour,font=\tiny]below:0}]{} 
(x2)--+(0,-1) node(y2) [dot][label={[vcolour,font=\tiny]below:1}]{} 
(x3)--+(0,-1) node(y3) [dot][label={[vcolour,font=\tiny]below:1}]{} 
(xk-4)--+(0,-1) node(yk-4) [dot][label={[vcolour,font=\tiny]below:1}]{}; 

\draw (wk-3)--(wk-2);
\draw (xk-3)--(xk-2);

\draw (w1)--(x1) (w1)--(x2) (w1)--(x3) (w1)--(xk-4) (w1)--(xk-3) (w1)--(xk-2);
\draw (w2)--(x1) (w2)--(x2) (w2)--(x3) (w2)--(xk-4) (w2)--(xk-3) (w2)--(xk-2);
\draw (w3)--(x1) (w3)--(x2) (w3)--(x3) (w3)--(xk-4) (w3)--(xk-3) (w3)--(xk-2);
\draw (wk-4)--(x1) (wk-4)--(x2) (wk-4)--(x3) (wk-4)--(xk-4) (wk-4)--(xk-3) (wk-4)--(xk-2);
\draw (wk-3)--(x1) (wk-3)--(x2) (wk-3)--(x3) (wk-3)--(xk-4) (wk-3)--(xk-3) (wk-3)--(xk-2);
\draw (wk-2)--(x1) (wk-2)--(x2) (wk-2)--(x3) (wk-2)--(xk-4) (wk-2)--(xk-3) (wk-2)--(xk-2);

\draw (y1)--(y2);

\path (wk-2) --++(2,0) node(w1) [dot][label={[vcolour,font=\tiny,yshift=-2pt]left:0}]{} node(w1 2) [dot]{} --++(1,0) node(w2) [dot][label={[vcolour,font=\tiny,xshift=-6pt,yshift=-3pt]0}]{} --++(1,0) node(w3) [dot][label={[vcolour,font=\tiny,xshift=-6pt,yshift=-3pt]0}]{} --++(1.5,0) node(wk-4) [dot][label={[vcolour,font=\tiny,xshift=-2pt,yshift=-1pt]0}]{} node(wk-4 2) [dot]{} --++(1,0) node(wk-3) [dot][label={[vcolour,font=\tiny,yshift=-1pt,xshift=4pt]1}]{} --++(1,0) node(wk-2) [dot][label={[vcolour,font=\tiny,yshift=-1pt,xshift=4pt]0}]{};
\path (w3) -- node{\dots} (wk-4);
\path (w1) --++(0,-2) node(x1) [dot][label={[vcolour,font=\tiny]below left:2}]{} --++(1,0) node(x2) [dot][label={[vcolour,font=\tiny]below left:3}]{} --++(1,0) node(x3) [dot][label={[vcolour,font=\tiny]below left:4}]{} --++(1.5,0) node(xk-4) [dot][label={[vcolour,font=\tiny,xshift=2pt]below left:\( k\,\text{-}\,3 \)}]{} --++(1,0) node(xk-3) [dot][label={[vcolour,font=\tiny,xshift=-3pt]below:\( k\,\text{-}\,2 \)}]{} --++(1,0) node(xk-2) [dot][label={[vcolour,font=\tiny,xshift=3pt]below:\( k\,\text{-}\,1 \)}]{};
\path (x3) -- node{\dots} (xk-4);

\draw
(w2)--+(0,1) node(u2) [dot]{} node[terminal][label={[vcolour,font=\tiny]1}]{}
(w3)--+(0,1) node(u3) [dot]{} node[terminal][label={[vcolour,font=\tiny]1}]{}; 
\draw
(x1)--+(0,-1) node(y1) [dot][label={[vcolour,font=\tiny]below:0}]{} 
(x2)--+(0,-1) node(y2) [dot][label={[vcolour,font=\tiny]below:1}]{} 
(x3)--+(0,-1) node(y3) [dot][label={[vcolour,font=\tiny]below:1}]{} 
(xk-4)--+(0,-1) node(yk-4) [dot][label={[vcolour,font=\tiny]below:1}]{}; 

\draw (wk-3)--(wk-2);
\draw (xk-3)--(xk-2);

\draw (w1)--(x1) (w1)--(x2) (w1)--(x3) (w1)--(xk-4) (w1)--(xk-3) (w1)--(xk-2);
\draw (w2)--(x1) (w2)--(x2) (w2)--(x3) (w2)--(xk-4) (w2)--(xk-3) (w2)--(xk-2);
\draw (w3)--(x1) (w3)--(x2) (w3)--(x3) (w3)--(xk-4) (w3)--(xk-3) (w3)--(xk-2);
\draw (wk-4)--(x1) (wk-4)--(x2) (wk-4)--(x3) (wk-4)--(xk-4) (wk-4)--(xk-3) (wk-4)--(xk-2);
\draw (wk-3)--(x1) (wk-3)--(x2) (wk-3)--(x3) (wk-3)--(xk-4) (wk-3)--(xk-3) (wk-3)--(xk-2);
\draw (wk-2)--(x1) (wk-2)--(x2) (wk-2)--(x3) (wk-2)--(xk-4) (wk-2)--(xk-3) (wk-2)--(xk-2);

\draw (y1)--(y2);

\path (wk-2) --++(2,0) node(w1) [dot][label={[vcolour,font=\tiny,yshift=-2pt]left:1}]{} node(w1 3) [dot]{} --++(1,0) node(w2) [dot][label={[vcolour,font=\tiny,xshift=-6pt,yshift=-3pt]1}]{} --++(1,0) node(w3) [dot][label={[vcolour,font=\tiny,xshift=-6pt,yshift=-3pt]1}]{} --++(1.5,0) node(wk-4) [dot][label={[vcolour,font=\tiny,xshift=-2pt,yshift=-1pt]1}]{} node(wk-4 3) [dot]{} --++(1,0) node(wk-3) [dot][label={[vcolour,font=\tiny,yshift=-1pt,xshift=4pt]0}]{} --++(1,0) node(wk-2) [dot][label={[vcolour,font=\tiny,yshift=-1pt,xshift=4pt]1}]{};
\path (w3) -- node{\dots} (wk-4);
\path (w1) --++(0,-2) node(x1) [dot][label={[vcolour,font=\tiny]below left:2}]{} --++(1,0) node(x2) [dot][label={[vcolour,font=\tiny]below left:3}]{} --++(1,0) node(x3) [dot][label={[vcolour,font=\tiny]below left:4}]{} --++(1.5,0) node(xk-4) [dot][label={[vcolour,font=\tiny,xshift=2pt]below left:\( k\,\text{-}\,3 \)}]{} --++(1,0) node(xk-3) [dot][label={[vcolour,font=\tiny,xshift=-3pt]below:\( k\,\text{-}\,2 \)}]{} --++(1,0) node(xk-2) [dot][label={[vcolour,font=\tiny,xshift=3pt]below:\( k\,\text{-}\,1 \)}]{};
\path (x3) -- node{\dots} (xk-4);

\draw
(w2)--+(0,1) node(u2) [dot]{} node[terminal][label={[vcolour,font=\tiny]0}]{}
(w3)--+(0,1) node(u3) [dot]{} node[terminal][label={[vcolour,font=\tiny]0}]{};
\draw
(x1)--+(0,-1) node(y1) [dot][label={[vcolour,font=\tiny]below:0}]{} 
(x2)--+(0,-1) node(y2) [dot][label={[vcolour,font=\tiny]below:1}]{} 
(x3)--+(0,-1) node(y3) [dot][label={[vcolour,font=\tiny]below:1}]{} 
(xk-4)--+(0,-1) node(yk-4) [dot][label={[vcolour,font=\tiny]below:1}]{}; 

\draw (wk-3)--(wk-2);
\draw (xk-3)--(xk-2);

\draw (w1)--(x1) (w1)--(x2) (w1)--(x3) (w1)--(xk-4) (w1)--(xk-3) (w1)--(xk-2);
\draw (w2)--(x1) (w2)--(x2) (w2)--(x3) (w2)--(xk-4) (w2)--(xk-3) (w2)--(xk-2);
\draw (w3)--(x1) (w3)--(x2) (w3)--(x3) (w3)--(xk-4) (w3)--(xk-3) (w3)--(xk-2);
\draw (wk-4)--(x1) (wk-4)--(x2) (wk-4)--(x3) (wk-4)--(xk-4) (wk-4)--(xk-3) (wk-4)--(xk-2);
\draw (wk-3)--(x1) (wk-3)--(x2) (wk-3)--(x3) (wk-3)--(xk-4) (wk-3)--(xk-3) (wk-3)--(xk-2);
\draw (wk-2)--(x1) (wk-2)--(x2) (wk-2)--(x3) (wk-2)--(xk-4) (wk-2)--(xk-3) (wk-2)--(xk-2);

\draw (y1)--(y2);

\path (wk-2) --++(2,0) node(u1 3) [dot][label={[vcolour,font=\tiny,yshift=-2pt]left:0}]{};

\draw (wk-4 1) to[out=45,in=135] (w1 2);

\draw (wk-4 2) to[out=45,in=135] (w1 3);

\draw (wk-4 3) to[out=45,in=135] (u1 3);

\node [below right=0.5cm and 1cm of u1 3,font=\LARGE]{\dots};
\end{tikzpicture}
\caption{A \( k \)-star colouring of the chain gadget, where \( k\geq 7 \).}
\label{fig:k-star colouring chain gadget deg k-1}
\end{figure}

A \( k \)-star colouring of the chain gadget is exhibited in Figure~\ref{fig:k-star colouring chain gadget deg k-1}.

\begin{lemma}\label{lem:chain gadget k-star colouring deg k-1}
For every \( k \)-star colouring of the chain gadget (with \( k\geq 7 \)), there exist distinct colours \( c_1 \) and \( c_2 \) such that the terminals of the gadget and their neighbours within the gadget are coloured either \( c_1 \) or \( c_2 \). 
\end{lemma}
\begin{proof}
Let \( f \) be a \( k \)-star colouring of the chain gadget. 
Applying Lemma~\ref{lem:gadget component k-star colouring deg k-1} to the first (resp.\ second) copy of the gadget component in the chain gadget reveals that \( U_1\cup W_1 \) (resp.\ \( U_2\cup W_2 \)) is bicoloured by \( f \). 
Suppose that \( U_1\cup W_1 \) is bicoloured by \( f \) using two colours \( c_1 \) and \( c_2 \) (i.e., \( f(U_1\cup W_1)\subseteq \{c_1,c_2\} \)). 
Since \( u_{k-4,1}w_{k-4,1} \) is an edge in the chain gadget, \( \{f(w_{k-4,1}),f(u_{k-4,1})\}=\{c_1,c_2\} \). 
Since \( w_{k-4,1}=u_{1,2} \) and \( u_{k-4,1}=w_{1,2} \), we have \( \{f(w_{1,2}),f(u_{1,2})\}=\{c_1,c_2\} \). 
Hence, \( U_2\cup W_2 \) is bicoloured by \( f \) using colours \( c_1 \) and \( c_2 \). 
By repeating the same argument, we can show that \( U_\ell\cup W_\ell \) is bicoloured by \( f \) using colours \( c_1 \) and \( c_2 \) for \( 1\leq \ell\leq t \). 
Since every terminal of the chain gadget is a vertex of the form \( u_{i,\ell}\in U_\ell \) and the neighbour of a terminal is of the form \( w_{i,\ell}\in W_\ell \), the lemma is proved. 
\end{proof}

The next construction is employed to show that \textsc{\( k \)-Star Colourability} is NP-complete for graphs of maximum degree \( k-1 \) (where \( k\geq 7 \)).

\begin{construct}\label{make:k-star colouring deg k-1}
\emph{Parameter:} An integer \( k\geq 7 \).\\
\emph{Input:} A \( (k-2) \)-regular graph \( G \).\\
\emph{Output:} A graph \( G' \) of maximum degree \( k-1 \).\\
\emph{Guarantee:} \( G \) is \( (k-2) \)-edge colourable if and only if \( G' \) is \( k \)-star colourable.\\
\emph{Steps:}\\
Let \( v_1,v_2,\dots,v_n \) be the vertices and \( e_1,e_2,\dots,e_m \) be the edges in \( G \). 
Introduce a chain gadget \( H \) (see Figure~\ref{fig:chain gadget k-star colouring deg k-1}) with \( q \) copies of the gadget component (i.e., use \( t=q \)), where \( q=\ceil{\frac{3n}{k-6}} \). 
We know that there are exactly \( k-6 \) terminals in each gadget component of the chain gadget. 
Thus, the choice of \( q \) ensures that the chain gadget has at least \( 3n \) terminals. 
For each vertex \( v_i \) of \( G \), choose three terminals of the chain gadget \( H \) which are not already chosen, and label them \( v_{i,1} \), \( v_{i,2} \) and \( v_{i,3} \), respectively.
For each edge \( e_\ell=v_iv_j \) of \( G \), introduce a new vertex \( e_\ell \) in \( G' \) and join it to the vertices \( v_{i,1},v_{i,2},v_{i,3},v_{j,1},v_{j,2} \) and \( v_{j,3} \). 
To clarify, \( V(G')=V(H)\cup E(G) \) and \( E(G')=E(H)\cup \{v_i e_\ell \colon v_i\in V(G),\allowbreak e_\ell\in E(G) \text{, and } v_i \text{ is incident on } e_\ell \text{ in } G \} \). 
Moreover, the subgraph of \( G' \) induced by \( \{v_{i,j}\colon 1\leq i\leq n,\, 1\leq j\leq 3\}\bigcup\,\{e_\ell\colon 1\leq \ell\leq m\}  \) is a bipartite graph with degree \( \deg_G(v_i) \) for each vertex \( v_{i,j} \) and degree 6 for each vertex \( e_\ell \). 
\end{construct}
\begin{proof}[Proof of guarantee]
Suppose that \( G \) admits a \( (k-2) \)-edge colouring \( f\colon E(G)\to\{2,3,\dots,k-1\} \). 
Note that colours~0 and 1 are not used by \( f \). 
We use \( f \) to obtain a \( k \)-colouring \( f' \) of \( G' \). 
Consider the function \( f'\colon V(G')\to \{0,1,\dots,k-1\} \) obtained by employing the colouring scheme in Figure~\ref{fig:k-star colouring chain gadget deg k-1} on the chain gadget \( H \), and by assigning \( f'(e_\ell)=f(e_\ell) \) for each \( e_\ell\in E(G) \). 
We know that \( f'(e_\ell)=f(e_\ell)\geq 2 \) for \( 1\leq \ell\leq m \), whereas each terminal in \( G' \) is coloured~0 or~1 by \( f' \) (see Figure~\ref{fig:k-star colouring chain gadget deg k-1}). 
Since vertices of the form \( e_\ell \) are adjacent only to terminals in \( G' \), \( f' \) is a \( k \)-colouring of \( G' \). 
\setcounter{claim}{0}
\begin{claim}\label{clm:k-2 edge colouring to k-star colourring}
\( f' \) is a \( k \)-star colouring of \( G' \). 
\end{claim}
\noindent 
Contrary to the claim, assume that there is a 4-vertex path \( Q \) in \( G' \) bicoloured by \( f' \). 
Observe that \( f' \) employs a \( k \)-star colouring scheme on the chain gadget \( H \) (see Figure~\ref{fig:k-star colouring chain gadget deg k-1}). 
Moreover, the restriction of \( f' \) to \( E(G) \) is a star colouring of \( G'[E(G)] \) since \( E(G) \) is an independent set in \( G' \). 
Hence, the bicoloured 4-vertex path \( Q \) contains an edge of the form \( v_{i,j}e_\ell \), where \( v_{i,j}\in V(H) \) and \( e_\ell\in E(G) \). 
We have two cases: either (i)~\( Q \) does not contain any edge from the chain gadget \( H \) (i.e., \( E(Q)\cap E(H)=\emptyset \)), or (ii)~\( Q \) contains an edge from \( H \). 
In Case~(i), \( Q \) is of the form \( e_s,v_{i,j},e_t,v_{p,q} \) where \( f'(e_s)=f'(e_t) \) and \( f'(v_{i,j})=f'(v_{p,q}) \). 
By the definition of \( G' \), \( e_s \) and \( e_t \) are edges of \( G \) incident on the vertex \( v_i \) of \( G \), and thus \( f(e_s)\neq f(e_t) \) (because \( f \) is an edge colouring of \( G \)). 
Since \( f(e_s)=f'(e_s)=f'(e_t)=f(e_t) \), we have a contradiction. 
This rules out Case~(i). 
Consider Case~(ii); that is, \( Q \) contains an edge from the chain gadget~\( H \). 
Since the path \( Q \) contains an edge from \( H \) as well as an edge  of the form \( v_{i,j}e_\ell \), the path \( Q \) contains a 3-vertex path segment of the from \( w,v_{i,j},e_\ell \), where \( w \) is the neighbour of the terminal \( v_{i,j} \) within the chain gadget \( H \). 
By the colouring scheme employed on the chain gadget (namely Figure~\ref{fig:k-star colouring chain gadget deg k-1}), \( f'(v_{i,j}),f'(w)\in \{0,1\} \). 
Since \( v_{i,j}w \) is an edge in~\( G' \), there is a binary colour \( b\in \{0,1\} \) such that \( f'(v_{i,j})=b \) and \( f'(w)=1-b \). 
Since \( f'(e_\ell)=f(e_\ell)\geq 2 \), the segment \( w,v_{i,j},e_\ell \) of the path \( Q \) is tricoloured by~\( f' \). 
Hence, \( f' \) uses at least three colours on the path \( Q \), a contradiction. 
This rules out Case~(ii). 
Since both Case~(i) and Case~(ii) are ruled out, there is no 4-vertex path in \( G' \) bicoloured by \( f' \). 
That is, \( f' \) is indeed a \( k \)-star colouring of \( G' \). 
This proves Claim~\ref{clm:k-2 edge colouring to k-star colourring}.\\ 

Conversely, suppose that \( G' \) admits a \( k \)-star colouring \( f':V(G\bm{'})\to\{0,1,\dots,k-1\} \). 
By Lemma~\ref{lem:chain gadget k-star colouring deg k-1}, there exist distinct colours \( c_1 \) and \( c_2 \) such that the terminals of the chain gadget and their neighbours within the chain gadget are coloured either \( c_1 \) or \( c_2 \). 
Without loss of generality, assume that \( c_1=0 \) and \( c_2=1 \). 
Thus, we have the following claim.
\begin{claim}\label{clm:k-star colouring binary terminals}
All terminals of the chain gadget and their neighbours within the gadget have binary colours (i.e., colour~0 or colour~1). 
\end{claim}
\begin{claim}\label{clm:k-star colouring non-binary el}
For each \( e_\ell\in E(G) \), the vertex \( e_\ell \) of \( G' \) has a non-binary colour under \( f' \) (i.e., \( f'(e_\ell)\geq 2 \)). 
\end{claim}
\noindent On the contrary, assume that \( f'(e_\ell)=b \) for some \( b\in \{0,1\} \), where \( e_\ell\in E(G) \). 
Let \( v_i \) be a vertex incident on the edge \( e_\ell \) in \( G \). 
In \( G' \), the vertex \( e_\ell \) is adjacent to \( v_{i,1} \), \( v_{i,2} \) and \( v_{i,3} \), and thus \( f'(v_{i,1})=f'(v_{i,2})=f'(v_{i,3})=1-b \) (because \( f'(v_{i,j})\in\{0,1\} \) by Claim~\ref{clm:k-star colouring binary terminals}). 
For \( 1\leq j\leq 3 \), the neighbour of the terminal \( v_{i,j} \) in the chain gadget has a binary colour by Claim~\ref{clm:k-star colouring binary terminals}. 
As shown in Figure~\ref{fig:k-star colouring non-binary el}, this signals a 4-vertex path in \( G' \) bicoloured by \( f' \).
This contradiction proves Claim~\ref{clm:k-star colouring non-binary el}. 

\begin{figure}[hbt]
\centering
\begin{tikzpicture}[label distance=-2pt]
\node (chain) [draw,densely dotted,ellipse,minimum width=2.5cm,minimum height=3.5cm][label={[align=center]left:chain\\gadget}]{};
\path (chain.center) --++(0.5,0) node(vi2)[dot]{} node[terminal][label={[yshift=-3pt]\( v_{i,2} \)}][label={[vcolour,label distance=-2pt,xshift=-3pt]below left: 1-\( b \)}]{};
\path (vi2) --+(0,0.85)  node(vi1)[dot]{} node[terminal][label={[yshift=-3pt]\( v_{i,1} \)}][label={[vcolour,label distance=-2pt,xshift=-3pt]below left: 1-\( b \)}]{};
\path (vi2) --+(0,-0.85) node(vi3)[dot]{} node[terminal][label={[yshift=-3pt]\( v_{i,3} \)}][label={[vcolour,label distance=-2pt,xshift=-3pt]below left: 1-\( b \)}]{};
\draw (vi1) --+(-1.0,0) node(l1)[dot][label={[vcolour]left:\( b \)}]{};
\draw (vi2) --+(-1.0,0) node(l2)[dot][label={[vcolour]left:\( b \)}]{};
\draw (vi3) --+(-1.0,0) node(l3)[dot][label={[vcolour]left:\( b \)}]{};
\node (el) [right =1.5cm of chain][dot][label=right:\( e_\ell \)][label={[vcolour]\( b \)}]{};
\draw (el)--(vi1)  (el)--(vi2)  (el)--(vi3);
\draw [very thick] (l1)--(vi1)--(el)--(vi2);
\end{tikzpicture}
\caption{A binary colour at \( e_\ell \) implies a bicoloured \( P_4 \).}
\label{fig:k-star colouring non-binary el}
\end{figure}
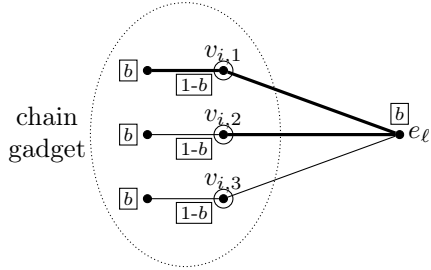

Let \( f \) be the restriction of \( f' \) to \( E(G) \). 
Due to Claim~\ref{clm:k-star colouring non-binary el}, \( f \) uses only colours \( 2,3,\dots,k-1 \). 
Hence, let us view \( f \) as a function from \( E(G) \) to \( \{2,3,\dots,k-1\} \). 

\begin{claim}\label{clm:k-star colourring to k-2 edge colouring}
\( f \) is a \( (k-2) \)-edge colouring of \( G \). 
\end{claim}
\noindent On the contrary, assume that \( f(e_s)=f(e_t) \) for two edges \( e_s \) and \( e_t \) of \( G \) incident on a common vertex \( v_i \) in \( G \). 
By the definition of \( G' \), both vertices \( e_s \) and \( e_t \) of \( G' \) are adjacent to vertices \( v_{i,1} \), \( v_{i,2} \) and \( v_{i,3} \) in \( G' \). 
Recall that \( f'(v_{i,1}),f'(v_{i,2}),f'(v_{i,3})\in \{0,1\} \) by Claim~\ref{clm:k-star colouring binary terminals}. 
Hence, by pigeonhole principle, at least two of these vertices have the same colour, say \( f'(v_{i,1})=f'(v_{i,2}) \). 
As a result, \( v_{i,1},e_s,v_{i,2},e_t \) is a 4-vertex path in \( G' \) bicoloured by \( f' \); a contradiction. 
Therefore, \( f \) is indeed a \( (k-2) \)-edge colouring of \( G \). 
This proves Claim~\ref{clm:k-star colourring to k-2 edge colouring}. 
\end{proof}

Note that the chain gadget has \( (4k-14)q+2 \) vertices and \( (k^2-2k-2)q+1 \) edges. 
Hence, \( G' \) has \( (4k-14)q+2+m=O(m+n) \) vertices and \( (k^2-2k-2)q+1+6m=O(m+n) \) edges, where \( m=|E(G)| \) and \( n=|V(G)| \) (because \( q=O(n) \)).
Thus, Construction~\ref{make:k-star colouring deg k-1} requires only time polynomial in \( m+n \). 
Leven and Galil~\cite{leven_galil} proved that for all \( k\geq 3 \), \textsc{Edge \( k \)-Colourability} is NP-complete for \( k \)-regular graphs. 
Thus, we have the following theorem by Construction~\ref{make:k-star colouring deg k-1}. 
\begin{theorem}\label{thm:k-star colouring deg k-1 k>=7}
For \( k\geq 7 \), \textsc{\( k \)-Star Colourability} is NP-complete for graphs of maximum degree \( k-1 \). 
\qed 
\end{theorem}

Next, let us deal with smaller values of \( k \). 
For \( k\leq 3 \), \textsc{\( k \)-Star Colourability} in graphs of maximum degree \( k-1 \) is polynomial-time solvable. 
The status is open for \( k=4 \). Using Construction~\ref{make:4-star girth 5} below, we show that \textsc{4-Star Colourability} is NP-complete for graphs of maximum degree~4. 
Interestingly, the graph used as the gadget component in Construction~\ref{make:4-star girth 5}, which is Petersen graph minus one vertex, has maximum degree~3. 
We suspect that \textsc{4-Star Colourability} is NP-complete for graphs of maximum degree~3, and Petersen graph minus one vertex might be useful in producing an NP-completeness reduction.

We use Petersen graph minus one vertex as the gadget component to build gadgets in Construction~\ref{make:4-star girth 5}. 
See Figure~\ref{fig:cases inner C5} for a diagram of the gadget component. 
Clearly, the gadget component has girth five. 
The following lemma explains why it is interesting for 4-star colouring.
\begin{lemma}\label{lem:petersen-1}
Every 4-star colouring of Petersen graph minus one vertex must assign the same colour on all three degree-2 vertices of the graph \( ( \)namely, \( w_1,w_4 \) and \( v_5 \)\( ) \).
\end{lemma}
\begin{proof}
\begin{figure}[hbt]
\centering
\begin{subfigure}[b]{0.33\textwidth}
\centering
\begin{tikzpicture}[scale=0.8]
\coordinate (base);
\path (base)--+(-162:1) node[dot](v1)[label=below:\( v_1 \)][label={[vcolour]above:2}]{}
      (base)--+(126:1) node[dot](v2)[label=left:\( v_2 \)][label={[vcolour]above:1}]{}
      (base)--+(54:1) node[dot](v3)[label=right:\( v_3 \)][label={[vcolour]above:1}]{}
      (base)--+(-18:1) node[dot](v4)[label=below:\( v_4 \)][label={[vcolour]above:3}]{}
      (base)--+(-90:1) node[dot](v5)[label=below:\( v_5 \)][label={[vcolour]right:0}]{};
\draw (v1)--(v3)--(v5)--(v2)--(v4)--(v1);
\draw (v1)--+(-162:1) node[dot](w1)[label=below:\( w_1 \)]{}
      (v2)--+(126:1) node[dot](w2)[label=above left:\( w_2 \)]{}
      (v3)--+(54:1) node[dot](w3)[label=above right:\( w_3 \)]{}
      (v4)--+(-18:1) node[dot](w4)[label=below:\( w_4 \)]{};
\draw (w1)--(w2)--(w3)--(w4);
\end{tikzpicture}
\caption{Case 1}
\end{subfigure}\begin{subfigure}[b]{0.33\textwidth}
\centering
\begin{tikzpicture}[scale=0.8]
\coordinate (base);
\path (base)--+(-162:1) node[dot](v1)[label=below:\( v_1 \)][label={[vcolour]above:0}]{}
      (base)--+(126:1) node[dot](v2)[label=left:\( v_2 \)][label={[vcolour]above:2}]{}
      (base)--+(54:1) node[dot](v3)[label=right:\( v_3 \)][label={[vcolour]above:1}]{}
      (base)--+(-18:1) node[dot](v4)[label=below:\( v_4 \)][label={[vcolour]above:1}]{}
      (base)--+(-90:1) node[dot](v5)[label=below:\( v_5 \)][label={[vcolour]right:3}]{};
\draw (v1)--(v3)--(v5)--(v2)--(v4)--(v1);
\draw (v1)--+(-162:1) node[dot](w1)[label=below:\( w_1 \)]{}
      (v2)--+(126:1) node[dot](w2)[label=above left:\( w_2 \)]{}
      (v3)--+(54:1) node[dot](w3)[label=above right:\( w_3 \)]{}
      (v4)--+(-18:1) node[dot](w4)[label=below:\( w_4 \)]{};
\draw (w1)--(w2)--(w3)--(w4);
\end{tikzpicture}
\caption{Case 2}
\end{subfigure}\begin{subfigure}[b]{0.33\textwidth}
\centering
\begin{tikzpicture}[scale=0.8]
\coordinate (base);
\path (base)--+(-162:1) node[dot](v1)[label=below:\( v_1 \)][label={[vcolour]above:1}]{}
      (base)--+(126:1) node[dot](v2)[label=left:\( v_2 \)][label={[vcolour]above:2}]{}
      (base)--+(54:1) node[dot](v3)[label=right:\( v_3 \)][label={[vcolour]above:0}]{}
      (base)--+(-18:1) node[dot](v4)[label=below:\( v_4 \)][label={[vcolour]above:3}]{}
      (base)--+(-90:1) node[dot](v5)[label=below:\( v_5 \)][label={[vcolour]right:1}]{};
\draw (v1)--(v3)--(v5)--(v2)--(v4)--(v1);
\draw (v1)--+(-162:1) node[dot](w1)[label=below:\( w_1 \)]{}
      (v2)--+(126:1) node[dot](w2)[label=above left:\( w_2 \)]{}
      (v3)--+(54:1) node[dot](w3)[label=above right:\( w_3 \)]{}
      (v4)--+(-18:1) node[dot](w4)[label=below:\( w_4 \)]{};
\draw (w1)--(w2)--(w3)--(w4);
\end{tikzpicture}
\caption{Case 3}
\end{subfigure}\caption{Possible ways of 4-star colouring inner \( C_5 \).}
\label{fig:cases inner C5}
\end{figure}
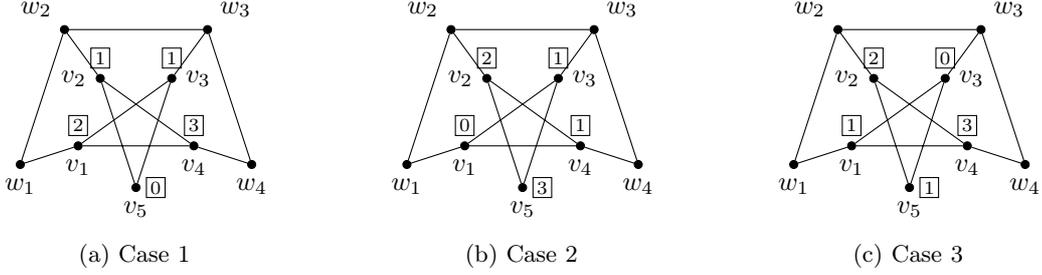

We fix a drawing of the Petersen graph, and assume that the vertex removed is from the outer \( C_5 \) (see Figure~\ref{fig:cases inner C5}). 
To star colour \( C_5 \), four colours are needed. 
Moreover, in every 4-star colouring of \( C_5 \), exactly one colour should repeat. 
Hence, without loss of generality, we assume that the inner \( C_5 \) is coloured in the pattern 1,0,1,2,3. 
So, exactly one of the following holds: (i)~\( f(v_1)=0 \), (ii)~\( f(v_2)=0 \), (iii)~\( f(v_3)=0 \), (iv)~\( f(v_4)=0 \), or (v)~\( f(v_5)=0 \). 
Up to symmetry and swapping of colours~2 and 3, we have only the three cases displayed in Figure~\ref{fig:cases inner C5} (note that \( f(v_4)=0 \) is symmetric to Case~2, and \( f(v_2)=0 \) is symmetric to Case~3). 
Let \( f \) be a 4-star colouring of Petersen graph minus one vertex. 
We need to prove that \( f(w_1)=f(w_4)=f(v_5) \).

\noindent Case~1:\\
If \( f(w_1)=1 \), then the bicoloured path \( w_1,w_2,v_2 \) will be part of a bicoloured~\( P_4 \) irrespective of the colour at \( w_2 \) (see Figure~\ref{fig:case 1 contradiction}). 
Hence, \( f(w_1)\neq 1 \). 
Similarly, \( f(w_4)\neq 1 \). 
We show that \( f(w_1)=0 \). 
On the contrary, assume that \( f(w_1)\neq 0 \). 
Note that \( f(w_2)\neq 0 \) (if not, path \( w_2,v_2,v_5,v_3 \) is a bicoloured \( P_4 \)). 
So, \( f(w_1),f(w_2)\in \{2,3\} \) (because colours~0 and 1 are ruled out). 
Therefore, path \( w_2,w_1,v_1,v_4 \) is a \( P_4 \) coloured with only two colours 2 and 3, a contradiction. 
This proves that \( f(w_1)=0 \). 
By symmetry, \( f(w_4)=0 \) as well. 
So, \( f(w_1)=f(w_4)=f(v_5)=0 \) in Case~1.\\

\noindent Case~2:\\
If \( f(w_3)=0 \), then \( w_3,v_3,v_1,v_4 \) is a bicoloured \( P_4 \). 
So, \( f(w_3)\in\{2,3\} \). 
We show that \( f(w_3)=2 \). 
On the contrary, assume that \( f(w_3)=3 \). 
Then, \( f(w_2)\in\{0,1\} \). 
If \( f(w_2)=1 \), then \( w_2,w_3,v_3,v_5 \) is a bicoloured \( P_4 \). 
So, \( f(w_2)=0 \). 
This leads to a contradiction as the path \( w_2,w_1,v_1 \) will be part of a bicoloured \( P_4 \) irrespective of the colour at \( w_1 \) (see Figure~\ref{fig:case 2 contradiction}). 
Thus, by contradiction, \( f(w_3)=2 \). 
So, \( f(w_4)\in\{0,3\} \). 
If \( f(w_4)=0 \), then \( w_4,v_4,v_1,v_3 \) is a bicoloured \( P_4 \). 
Hence, \( f(w_4)=3 \) . 
As a result, \( f(w_2)\neq\{1,3\} \) (if not, either path \( w_3,w_2,v_2,v_4 \) or path \( w_4,w_3,w_2,v_2 \) is a bicoloured \( P_4 \)). 
So, \( f(w_2)=0 \). 
This in turn forces \( f(w_1)=3 \) (if \( f(w_1)\in\{1,2\} \), then either \( w_2,w_1,v_1,v_4 \) or \( w_3,w_2,w_1,v_1 \) a bicoloured \( P_4 \)). 
So, \( f(w_1)=f(w_4)=f(v_5)=3 \) in Case~2.

\begin{figure}[hbt]
\centering
\begin{minipage}[b]{0.45\textwidth}
\centering
\begin{tikzpicture}
\coordinate (base);
\path (base)--+(-162:1) node[dot](v1)[label=below:\( v_1 \)][label={[vcolour]above:2}]{}
      (base)--+(126:1) node[dot](v2)[label=left:\( v_2 \)][label={[vcolour]above:1}]{}
      (base)--+(54:1) node[dot](v3)[label=right:\( v_3 \)][label={[vcolour]above:1}]{}
      (base)--+(-18:1) node[dot](v4)[label=below:\( v_4 \)][label={[vcolour]above:3}]{}
      (base)--+(-90:1) node[dot](v5)[label=below:\( v_5 \)][label={[vcolour]right:0}]{};
\draw (v1)--(v3)--(v5)--(v2)--(v4)--(v1);
\draw (v1)--+(-162:1) node[dot](w1)[label=below:\( w_1 \)][label={[vcolour]left:1}]{}
      (v2)--+(126:1) node[dot](w2)[label=above left:\( w_2 \)][label={[vcolour]above right:\textbf{?}}]{}
      (v3)--+(54:1) node[dot](w3)[label=above right:\( w_3 \)]{}
      (v4)--+(-18:1) node[dot](w4)[label=below:\( w_4 \)]{};
\draw (w1)--(w2)--(w3)--(w4);
\end{tikzpicture}
\caption{In Case~1, \( f(w_1)=1 \) leads to a contradiction.}
\label{fig:case 1 contradiction}
\end{minipage}\hfill
\begin{minipage}[b]{0.45\textwidth}
\centering
\begin{tikzpicture}
\coordinate (base);
\path (base)--+(-162:1) node[dot](v1)[label=below:\( v_1 \)][label={[vcolour]above:0}]{}
      (base)--+(126:1) node[dot](v2)[label=left:\( v_2 \)][label={[vcolour]above:2}]{}
      (base)--+(54:1) node[dot](v3)[label=right:\( v_3 \)][label={[vcolour]above:1}]{}
      (base)--+(-18:1) node[dot](v4)[label=below:\( v_4 \)][label={[vcolour]above:1}]{}
      (base)--+(-90:1) node[dot](v5)[label=below:\( v_5 \)][label={[vcolour]right:3}]{};
\draw (v1)--(v3)--(v5)--(v2)--(v4)--(v1);
\draw (v1)--+(-162:1) node[dot](w1)[label=below:\( w_1 \)][label={[vcolour]left:\textbf{?}}]{}
      (v2)--+(126:1) node[dot](w2)[label=above left:\( w_2 \)][label={[vcolour]above right:0}]{}
      (v3)--+(54:1) node[dot](w3)[label=above right:\( w_3 \)][label={[vcolour]above left:3}]{}
      (v4)--+(-18:1) node[dot](w4)[label=below:\( w_4 \)]{};
\draw (w1)--(w2)--(w3)--(w4);
\end{tikzpicture}
\caption{In Case~2, \( f(w_3)=3 \) leads to a contradiction.}
\label{fig:case 2 contradiction}
\end{minipage}\end{figure}

\noindent Case~3:\\
If \( f(w_3)=2 \), then the path \( w_3,w_2,v_2 \) will be part of a bicoloured \( P_4 \) irrespective of the colour at \( w_2 \) (see Figure~\ref{fig:case 3 contradiction1}). 
If \( f(w_3)=3 \), then the path \( w_3,w_4,v_4 \) will be part of a bicoloured \( P_4 \) irrespective of the colour at \( w_4 \) (see Figure~\ref{fig:case 3 contradiction2}). 
So, \( f(w_3)=1 \). 
Clearly, \( f(w_2)\in\{0,3\} \) and \( f(w_4)\in\{0,2\} \). 
Observe that \( f(w_2)\neq 0 \) and \( f(w_4)\neq 0 \) (if not, either \( v_1,v_3,w_3,w_2 \) or \( v_1,v_3,w_3,w_4 \) is a bicoloured \( P_4 \)). 
So, \( f(w_2)=3 \) and \( f(w_4)=2 \). 
But, then path \( w_2,v_2,v_4,w_4 \) is a bicoloured \( P_4 \). 
This contradiction rules out Case~3.

\begin{figure}[hbt]
\centering
\begin{subfigure}[b]{0.45\textwidth}
\centering
\begin{tikzpicture}
\coordinate (base);
\path (base)--+(-162:1) node[dot](v1)[label=below:\( v_1 \)][label={[vcolour]above:1}]{}
      (base)--+(126:1) node[dot](v2)[label=left:\( v_2 \)][label={[vcolour]above:2}]{}
      (base)--+(54:1) node[dot](v3)[label=right:\( v_3 \)][label={[vcolour]above:0}]{}
      (base)--+(-18:1) node[dot](v4)[label=below:\( v_4 \)][label={[vcolour]above:3}]{}
      (base)--+(-90:1) node[dot](v5)[label=below:\( v_5 \)][label={[vcolour]right:1}]{};
\draw (v1)--(v3)--(v5)--(v2)--(v4)--(v1);
\draw (v1)--+(-162:1) node[dot](w1)[label=below:\( w_1 \)]{}
      (v2)--+(126:1) node[dot](w2)[label=above left:\( w_2 \)][label={[vcolour]above right:\textbf{?}}]{}
      (v3)--+(54:1) node[dot](w3)[label=above right:\( w_3 \)][label={[vcolour]above left:2}]{}
      (v4)--+(-18:1) node[dot](w4)[label=below:\( w_4 \)]{};
\draw (w1)--(w2)--(w3)--(w4);
\end{tikzpicture}
\caption{In Case~3, \( f(w_3)=2 \) leads to a contradiction.}
\label{fig:case 3 contradiction1}
\end{subfigure}\hfill
\begin{subfigure}[b]{0.45\textwidth}
\centering
\begin{tikzpicture}
\coordinate (base);
\path (base)--+(-162:1) node[dot](v1)[label=below:\( v_1 \)][label={[vcolour]above:1}]{}
      (base)--+(126:1) node[dot](v2)[label=left:\( v_2 \)][label={[vcolour]above:2}]{}
      (base)--+(54:1) node[dot](v3)[label=right:\( v_3 \)][label={[vcolour]above:0}]{}
      (base)--+(-18:1) node[dot](v4)[label=below:\( v_4 \)][label={[vcolour]above:3}]{}
      (base)--+(-90:1) node[dot](v5)[label=below:\( v_5 \)][label={[vcolour]right:1}]{};
\draw (v1)--(v3)--(v5)--(v2)--(v4)--(v1);
\draw (v1)--+(-162:1) node[dot](w1)[label=below:\( w_1 \)]{}
      (v2)--+(126:1) node[dot](w2)[label=above left:\( w_2 \)]{}
      (v3)--+(54:1) node[dot](w3)[label=above right:\( w_3 \)][label={[vcolour]above left:3}]{}
      (v4)--+(-18:1) node[dot](w4)[label=below:\( w_4 \)][label={[vcolour]right:\textbf{?}}]{};
\draw (w1)--(w2)--(w3)--(w4);
\end{tikzpicture}
\caption{In Case~3, \( f(w_3)=3 \) leads to a contradiction.}
\label{fig:case 3 contradiction2}
\end{subfigure}
\caption{\( f(w_3)=1 \) in Case 3.}
\end{figure}
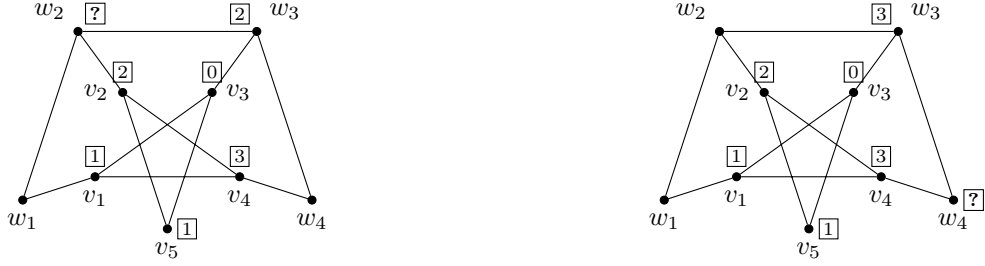

Since Case~3 is ruled out by contradiction, we have \( f(w_1)=f(w_4)=f(v_5) \) by Cases~1 and 2.
This completes the proof. 
\end{proof}

\begin{construct}\label{make:4-star girth 5}
\emph{Input:} A 4-regular graph \( G \).\\
\emph{Output:} A graph \( G' \) of maximum degree four and girth five.\\
\emph{Guarantee:} \( G \) is 3-colourable if and only if \( G' \) is 4-star colourable.\\
\emph{Steps:}\\
Let \( v_1,v_2,\dots,v_n \) be the vertices in \( G \). 
First, replace each vertex of \( G \) by a vertex gadget as shown in Figure~\ref{fig:repeat 4-star vertex replacement}. 
The vertex gadget for \( v_i \) has five terminals, and the terminals \( v_{i,1},v_{i,2},v_{i,3},v_{i,4} \) accommodate the four edges incident on \( v_i \) in \( G \) in a one-to-one fashion (order does not matter). 
So, corresponding to each edge \( v_iv_j \) in \( G \), there is an edge \( v_{i,k}v_{j,\ell} \) in \( G' \) for some \( k,\ell\in\{1,2,3,4\} \). 
Finally, introduce the chain gadget displayed in Figure~\ref{fig:repeat 4-star girth 5 chain gadget}, and join \( v_{i,0} \) to \( v_i^* \) for \( i=1,2,\dots,n \).

\begin{figure}[hbt]
\centering
\begin{tikzpicture}[scale=0.6]
\node[bigdot] (vi)[label=above:\( v_i \)]{};
\draw[dashed] (vi)--+(-36:1) coordinate (rightNbr)
              (vi)--+(-72:1)
              (vi)--+(-108:1)
              (vi)--+(-144:1) coordinate (leftNbr);
\draw [decorate,decoration={brace,amplitude=6pt,raise=10pt,mirror}] (leftNbr)++(-0.2,0)--node[yshift=-25pt]{\( 4 \) edges} ($(rightNbr)+(0.2,0)$);

\path (vi) ++(2,0) coordinate (from) --++(0.75,0) coordinate(to);
\draw [-stealth,draw=black,line width=3pt] (from)--(to);

\path (to) ++(3,0) coordinate (vi0base);
\path (vi0base)--+(-162:1) node[dot](v1){}
      (vi0base)--+(126:1) node[dot](v2){}
      (vi0base)--+(54:1) node[dot](v3){}
      (vi0base)--+(-18:1) node[dot](v4){}
      (vi0base)--+(-90:1) node[dot](v5){} node[terminal](vi1)[label=right:\( v_{i,1} \)]{};
\draw (v1)--(v3)--(v5)--(v2)--(v4)--(v1);
\draw (v1)--+(-162:1) node[dot](w1){} node[terminal][label=below:\( v_{i,0} \)]{}
      (v2)--+(126:1) node[dot](w2){}
      (v3)--+(54:1) node[dot](w3){}
      (v4)--+(-18:1) node[dot](w4){};
\draw (w1)--(w2)--(w3)--(w4);
\draw[dashed] (vi1)--+(0,-1);

\path (vi0base)--+(3.8,0) coordinate (vi1base);
\path (vi1base)--+(-162:1) node[dot](v1){}
      (vi1base)--+(126:1) node[dot](v2){}
      (vi1base)--+(54:1) node[dot](v3){}
      (vi1base)--+(-18:1) node[dot](v4){}
      (vi1base)--+(-90:1) node[dot](v5){} node[terminal](vi2)[label=right:\( v_{i,2} \)]{};
\draw (v1)--(v3)--(v5)--(v2)--(v4)--(v1);
\draw (v1)--+(-162:1) node[dot](w1){}
      (v2)--+(126:1) node[dot](w2){}
      (v3)--+(54:1) node[dot](w3){}
      (v4)--+(-18:1) node[dot](w4){};
\draw (w1)--(w2)--(w3)--(w4);
\draw[dashed] (vi2)--+(0,-1);

\path (vi1base)--+(3.8,0) coordinate (vi2base);
\path (vi2base)--+(-162:1) node[dot](v1){}
      (vi2base)--+(126:1) node[dot](v2){}
      (vi2base)--+(54:1) node[dot](v3){}
      (vi2base)--+(-18:1) node[dot](v4){}
      (vi2base)--+(-90:1) node[dot](v5){} node[terminal](vi3)[label=right:\( v_{i,3} \)]{};
\draw (v1)--(v3)--(v5)--(v2)--(v4)--(v1);
\draw (v1)--+(-162:1) node[dot](w1){}
      (v2)--+(126:1) node[dot](w2){}
      (v3)--+(54:1) node[dot](w3){}
      (v4)--+(-18:1) node[dot](w4){};
\draw (w1)--(w2)--(w3)--(w4);
\draw[dashed] (vi3)--+(0,-1);

\path (vi2base)--+(3.8,0) coordinate (vi3base);
\path (vi3base)--+(-162:1) node[dot](v1){}
      (vi3base)--+(126:1) node[dot](v2){}
      (vi3base)--+(54:1) node[dot](v3){}
      (vi3base)--+(-18:1) node[dot](v4){}
      (vi3base)--+(-90:1) node[dot](v5){} node[terminal](vi4)[label=right:\( v_{i,4} \)]{};
\draw (v1)--(v3)--(v5)--(v2)--(v4)--(v1);
\draw (v1)--+(-162:1) node[dot](w1){}
      (v2)--+(126:1) node[dot](w2){}
      (v3)--+(54:1) node[dot](w3){}
      (v4)--+(-18:1) node[dot](w4){};
\draw (w1)--(w2)--(w3)--(w4);
\draw[dashed] (vi4)--+(0,-1);
\end{tikzpicture}
\caption{Replacement of vertex by vertex gadget.}
\label{fig:repeat 4-star vertex replacement}
\end{figure}
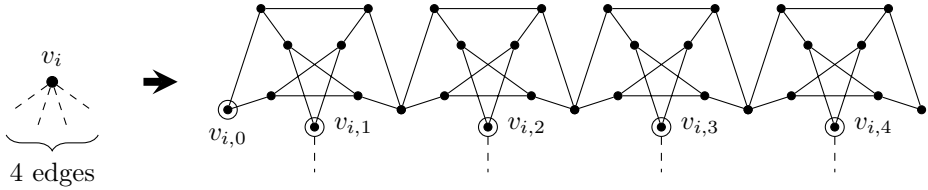
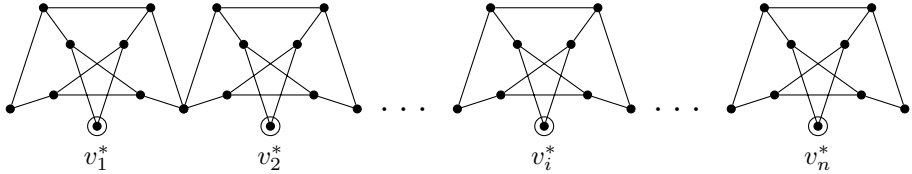
\begin{figure}[hbt]
\centering
\begin{tikzpicture}[scale=0.6]
\coordinate (vi0base);
\path (vi0base)--+(-162:1) node[dot](v1){}
      (vi0base)--+(126:1) node[dot](v2){}
      (vi0base)--+(54:1) node[dot](v3){}
      (vi0base)--+(-18:1) node[dot](v4){}
      (vi0base)--+(-90:1) node[dot](v5){} node[terminal][label=below:\( v^*_1 \)]{};
\draw (v1)--(v3)--(v5)--(v2)--(v4)--(v1);
\draw (v1)--+(-162:1) node[dot](w1){}
      (v2)--+(126:1) node[dot](w2){}
      (v3)--+(54:1) node[dot](w3){}
      (v4)--+(-18:1) node[dot](w4){};
\draw (w1)--(w2)--(w3)--(w4);

\path (vi0base)--+(3.8,0) coordinate (vi1base);
\path (vi1base)--+(-162:1) node[dot](v1){}
      (vi1base)--+(126:1) node[dot](v2){}
      (vi1base)--+(54:1) node[dot](v3){}
      (vi1base)--+(-18:1) node[dot](v4){}
      (vi1base)--+(-90:1) node[dot](v5){} node[terminal][label=below:\( v^*_2 \)]{};
\draw (v1)--(v3)--(v5)--(v2)--(v4)--(v1);
\draw (v1)--+(-162:1) node[dot](w1){}
      (v2)--+(126:1) node[dot](w2){}
      (v3)--+(54:1) node[dot](w3){}
      (v4)--+(-18:1) node[dot](w4){};
\draw (w1)--(w2)--(w3)--(w4);

\path (vi1base)--+(6,0) coordinate (vi2base);
\path (vi2base)--+(-162:1) node[dot](v1){}
      (vi2base)--+(126:1) node[dot](v2){}
      (vi2base)--+(54:1) node[dot](v3){}
      (vi2base)--+(-18:1) node[dot](v4){}
      (vi2base)--+(-90:1) node[dot](v5){} node[terminal][label=below:\( v^*_i \)]{};
\draw (v1)--(v3)--(v5)--(v2)--(v4)--(v1);
\draw (v1)--+(-162:1) node[dot](w1){};
\path (w4)--node[font=\LARGE]{\dots} (w1); \draw (v2)--+(126:1) node[dot](w2){}
      (v3)--+(54:1) node[dot](w3){}
      (v4)--+(-18:1) node[dot](w4){};
\draw (w1)--(w2)--(w3)--(w4);

\path (vi2base)--+(6,0) coordinate (vi3base);
\path (vi3base)--+(-162:1) node[dot](v1){}
      (vi3base)--+(126:1) node[dot](v2){}
      (vi3base)--+(54:1) node[dot](v3){}
      (vi3base)--+(-18:1) node[dot](v4){}
      (vi3base)--+(-90:1) node[dot](v5){} node[terminal][label=below:\( v^*_n \)]{};
\draw (v1)--(v3)--(v5)--(v2)--(v4)--(v1);
\draw (v1)--+(-162:1) node[dot](w1){};
\path (w4)--node[font=\LARGE]{\dots} (w1); \draw (v2)--+(126:1) node[dot](w2){}
      (v3)--+(54:1) node[dot](w3){}
      (v4)--+(-18:1) node[dot](w4){};
\draw (w1)--(w2)--(w3)--(w4);
\end{tikzpicture}
\caption{Chain gadget in Construction~\ref{make:4-star girth 5}.}
\label{fig:repeat 4-star girth 5 chain gadget}
\end{figure}
\end{construct}

\begin{proof}[Proof of Guarantee]
Suppose that \( G \) admits a 3-colouring \( f\colon V(G)\to\{1,2,3\} \). 
A 4-star colouring \( f' \) of \( G' \) is constructed as follows. 
Assign \( f'(v_{i,j})=f(v_i) \) for \( 1\leq i\leq n \) and \( 0\leq j\leq 4 \). 
This partial colouring can be extended into a 4-star colouring of each vertex gadget by the scheme in Figure~\ref{fig:4-star colouring vertex gadget} (if terminals of the gadget are coloured~\mbox{\( c\in\{2,3\} \)}, swap colour~1 with colour~\( c \)). 
Also, assign \( f'(v_i^*)=0 \) for \( 1\leq i\leq n \). 
This can be extended into a 4-star colouring of the chain gadget; for instance, use a scheme similar to the one in Figure~\ref{fig:4-star colouring vertex gadget} (it does not matter which 4-star colouring extension is used). 

\begin{figure}[hbt]
\centering
\begin{tikzpicture}[scale=0.75]
\coordinate (vi0base);
\path (vi0base)--+(-162:1) node[dot](v1)[label={[vcolour]below:2}]{}
      (vi0base)--+(126:1) node[dot](v2)[label={[vcolour]left:0}]{}
      (vi0base)--+(54:1) node[dot](v3)[label={[vcolour]right:0}]{}
      (vi0base)--+(-18:1) node[dot](v4)[label={[vcolour]below:3}]{}
      (vi0base)--+(-90:1) node[dot](v5){} node[terminal][label=below:\( v_{i,1} \)](vi1)[label={[vcolour]right:1}]{};
\draw (v1)--(v3)--(v5)--(v2)--(v4)--(v1);
\draw (v1)--+(-162:1) node[dot](w1){} node[terminal][label=below:\( v_{i,0} \)][label={[vcolour]left:1}]{}
      (v2)--+(126:1) node[dot](w2)[label={[vcolour]above:2}]{}
      (v3)--+(54:1) node[dot](w3)[label={[vcolour]above:3}]{}
      (v4)--+(-18:1) node[dot](w4)[label={[vcolour]below:1}]{};
\draw (w1)--(w2)--(w3)--(w4);

\path (vi0base)--+(3.8,0) coordinate (vi1base);
\path (vi1base)--+(-162:1) node[dot](v1)[label={[vcolour]below:2}]{}
      (vi1base)--+(126:1) node[dot](v2)[label={[vcolour]left:0}]{}
      (vi1base)--+(54:1) node[dot](v3)[label={[vcolour]right:0}]{}
      (vi1base)--+(-18:1) node[dot](v4)[label={[vcolour]below:3}]{}
      (vi1base)--+(-90:1) node[dot](v5){} node[terminal][label=below:\( v_{i,2} \)](vi2)[label={[vcolour]right:1}]{};
\draw (v1)--(v3)--(v5)--(v2)--(v4)--(v1);
\draw (v1)--+(-162:1) node[dot](w1){}
      (v2)--+(126:1) node[dot](w2)[label={[vcolour]above:2}]{}
      (v3)--+(54:1) node[dot](w3)[label={[vcolour]above:3}]{}
      (v4)--+(-18:1) node[dot](w4)[label={[vcolour]below:1}]{};
\draw (w1)--(w2)--(w3)--(w4);

\path (vi1base)--+(3.8,0) coordinate (vi2base);
\path (vi2base)--+(-162:1) node[dot](v1)[label={[vcolour]below:2}]{}
      (vi2base)--+(126:1) node[dot](v2)[label={[vcolour]left:0}]{}
      (vi2base)--+(54:1) node[dot](v3)[label={[vcolour]right:0}]{}
      (vi2base)--+(-18:1) node[dot](v4)[label={[vcolour]below:3}]{}
      (vi2base)--+(-90:1) node[dot](v5){} node[terminal][label=below:\( v_{i,3} \)](vi3)[label={[vcolour]right:1}]{};
\draw (v1)--(v3)--(v5)--(v2)--(v4)--(v1);
\draw (v1)--+(-162:1) node[dot](w1){}
      (v2)--+(126:1) node[dot](w2)[label={[vcolour]above:2}]{}
      (v3)--+(54:1) node[dot](w3)[label={[vcolour]above:3}]{}
      (v4)--+(-18:1) node[dot](w4)[label={[vcolour]below:1}]{};
\draw (w1)--(w2)--(w3)--(w4);

\path (vi2base)--+(3.8,0) coordinate (vi3base);
\path (vi3base)--+(-162:1) node[dot](v1)[label={[vcolour]below:2}]{}
      (vi3base)--+(126:1) node[dot](v2)[label={[vcolour]left:0}]{}
      (vi3base)--+(54:1) node[dot](v3)[label={[vcolour]right:0}]{}
      (vi3base)--+(-18:1) node[dot](v4)[label={[vcolour]below:3}]{}
      (vi3base)--+(-90:1) node[dot](v5){} node[terminal][label=below:\( v_{i,4} \)](vi4)[label={[vcolour]right:1}]{};
\draw (v1)--(v3)--(v5)--(v2)--(v4)--(v1);
\draw (v1)--+(-162:1) node[dot](w1){}
      (v2)--+(126:1) node[dot](w2)[label={[vcolour]above:2}]{}
      (v3)--+(54:1) node[dot](w3)[label={[vcolour]above:3}]{}
      (v4)--+(-18:1) node[dot](w4)[label={[vcolour]right:1}]{};
\draw (w1)--(w2)--(w3)--(w4);

\end{tikzpicture}
\caption{A 4-star colouring of the vertex gadget with colour 1 at terminals.}
\label{fig:4-star colouring vertex gadget}
\end{figure}
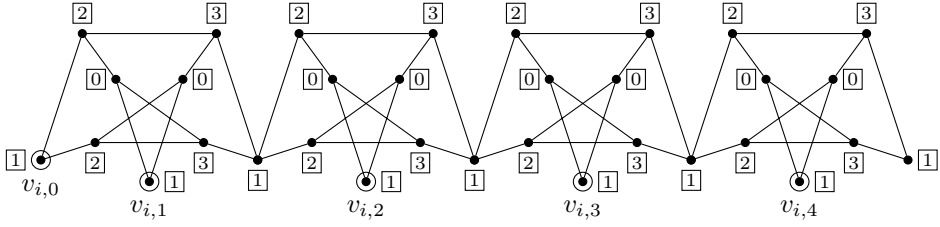

Note that for each 3-vertex path \( Q \) in a vertex/chain gadget with a terminal of the gadget as an endpoint, \( Q \) is not bicoloured by \( f' \). 
Hence, there is no bicoloured \( P_4 \) in \( G' \) with three vertices from one gadget and one vertex from another. 
To prove that \( f' \) is a 4-star colouring, it suffices to show that there is no bicoloured \( P_4 \) in \( G' \) with two vertices from one gadget and two vertices from another gadget. 
Observe that for \( 1\leq i\leq n \) and \( 1\leq k\leq 4 \), neighbours of \( v_{i,k} \) within the vertex gadget for \( v_i \) are coloured~0. 
So, there is no bicoloured \( P_4 \) in \( G' \) containing an edge of the form \( v_{i,k}v_{j,\ell} \) as its middle vertex. 
Moreover, for \( 1\leq i\leq n \), neighbours of \( v_{i,0} \) within the vertex gadget are not coloured 0. 
Hence, there is no bicoloured \( P_4 \) in \( G' \) containing an edge of the form \( v_{i,0}v_i^* \) as its middle vertex (recall that \( f'(v_i^*)=0 \)). 
Therefore, there is no bicoloured \( P_4 \) in \( G' \) containing two vertices from one gadget and two vertices from another gadget. 
This proves that \( f' \) is a 4-star colouring of \( G' \).

Conversely, suppose that \( G' \) admits a 4-star colouring \( f' \). 
Thanks to Lemma~\ref{lem:petersen-1}, terminals of a vertex/chain gadget should get the same colour. 
That is, \( f'(v_{i,0})=f'(v_{i,1})=f'(v_{i,2})=f'(v_{i,3})=f'(v_{i,4}) \) for all \( v_i\in V(G) \), and \( f'(v^*_1)=f'(v^*_2)=\dots=f'(v^*_n) \). 
Without loss of generality, assume that \( f'(v_i^*)=0 \) for \( i=1,2,\dots,n \). 
Since \( v_{i,0}v_i^* \) is an edge for \( 1\leq i\leq n \), the chain gadget forbids colour 0 at terminals \( v_{i,j} \) for \( 1\leq i\leq n \) and \( 0\leq j\leq 4 \). 
Consider the function \( f\colon V(G)\to\{1,2,3\} \) defined as \( f(v_i)=f'(v_{i,0}) \) for \( 1\leq i\leq n \). 
For each edge \( v_iv_j \) of \( G \), there exists an edge in \( G' \) between terminals \( v_{i,k} \) and \( v_{j,\ell} \) for some \( k,\ell\in\{1,2,3,4\} \), and thus \( f'(v_{i,k})\neq f'(v_{j,\ell}) \). So, \( f(v_i)\neq f(v_j) \) for each edge \( v_iv_j \) of \( G \) (due to Lemma~\ref{lem:petersen-1}, \( f(v_i)=f'(v_{i,0})=f'(v_{i,k}) \) and \( f(v_j)=f'(v_{j,0})=f'(v_{j,\ell}) \)). 
Therefore, \( f \) is a 3-colouring of \( G \). 
This proves the converse part. 
\end{proof}

Construction~\ref{make:4-star girth 5} establishes a reduction from \textsc{3-Colourability}(4-regular) to \textsc{4-Star Colourability}(\( \Delta=4 \), \( \text{girth}=5 \)). 
Note that Construction~\ref{make:4-star girth 5} requires only time polynomial in \( m+n \) because \( |E(G')|=61n+m \) and \( |V(G')|=41n+1 \) (where \( m=|E(G)| \) and \( n=|V(G)| \)). 
Thus, we have the following theorem.

\begin{theorem}\label{thm:4-star colouring npc max-deg 4}
\textsc{4-Star Colourability} is NP-complete for graphs of maximum degree four and girth five.
\qed
\end{theorem}

Next, we show that \textsc{5-Star Colourability} is NP-complete for graphs of maximum degree~4.
Construction~\ref{make:5-star max-deg 4} below is employed to establish a reduction from \textsc{3\nobreakdash-Colourability}(4\nobreakdash-regular) to \textsc{5\nobreakdash-Star Colourability}(triangle-free, 4\nobreakdash-regular). 
Construction~\ref{make:5-star max-deg 4} is similar to Construction~\ref{make:4-star girth 5}, albeit a bit more complicated. 
For instance, we will need two chain gadgets this time because two colours should be forbidden. 
The gadgets used in the construction are made of two gadgets called 2-in-2-out gadget and not-equal gadget. 
These are in turn made of one fixed graph, namely Gr\"{o}tzsch graph minus one vertex; we call it the gadget component (in Construction~\ref{make:5-star max-deg 4}) for obvious reason. 
The gadget component is displayed in Figure~\ref{fig:5-star gadget compnent}. 
The following lemma explains why it is interesting for 5-star colouring. 
\begin{lemma}\label{lem:miecelski-1}
Under every 5-star colouring of the gadget component, the degree-2 vertices of the graph should get pairwise distinct colours. 
Moreover, every 5-star colouring of the gadget component must be of the form displayed in Figure~\ref{fig:5-star colouring of gadget component-a} or Figure~\ref{fig:5-star colouring of gadget component-b} up to colour swaps.
\end{lemma}
\begin{proof}
Let \( f \) be a 5-star colouring of the gadget component that uses colours 0,1,2,3 and~4.\\[5pt]
\noindent \hypertarget{lnk:all 5 colours on C5}{\textbf{Claim~1:}} \( f \) must use all five colours on the 5-vertex cycle \( (v_1,v_2,v_3,v_4,v_5) \).\\[5pt]
On the contrary, assume that two vertices of the 5-vertex cycle are assigned the same colour by \( f \). 
Without loss of generality, assume that \( f(v_1)=f(v_3)=1 \) and \mbox{\(  f(v_2)=0 \)}. 
Since \( v_1,v_2,v_3,v_4 \) and \( v_5,v_1,v_2,v_3 \) are \( P_4 \)'s, \( f(v_4)\neq 0 \) and \( f(v_5)\neq 0 \); hence, new colours must be assigned at \( v_4 \) and \( v_5 \). 
Without loss of generality, assume that \( f(v_4)=2 \) and \( f(v_5)=3 \). 
If \( f(w_2)\in\{0,2,3\} \), then one of the three paths (i)~\( v_1,w_2,v_3,v_2 \), (ii)~\( v_1,w_2,v_3,v_4 \), or (iii)~\( v_5,v_1,w_2,v_3 \) is a bicoloured~\( P_4 \). 
So, \( f(w_2)=4 \). 
If \( f(w_4)=0 \) or \( f(w_4)=4 \), then either \( v_1,v_2,v_3,w_4 \) or \( v_1,w_2,v_3,w_4 \) is a bicoloured~\( P_4 \). 
Hence, \( f(w_4)\notin\{0,4\} \) and thus \( f(w_4)=2 \). 
Similarly, \( f(w_5)\notin\{0,4\} \) and thus \( f(w_5)=3 \) (if \( f(w_5)\in\{0,4\} \), then either \( w_5,v_1,v_2,v_3 \) or \( w_5,v_1,w_2,v_3 \) is a bicoloured \( P_4 \)). 
But, then path \( w_4,v_5,v_4,w_5 \) is a bicoloured \( P_4 \). 
This contradiction proves Claim~1.\\

Thanks to Claim~1, we assume without loss of generality that \( f(v_i)=i-1 \) for \( 1\leq i\leq 5 \).\\[5pt] 
\noindent \textbf{Claim~2:} For \( 1\leq i\leq 5 \), \( f(w_i)\neq f(v_i) \).\\[5pt]
Assume the contrary, say for \( i=1 \), i.e., \( f(w_1)=f(v_1)=0 \). 
If \( f(w_2)=1 \) or 4, then either \( w_1,v_2,v_1,w_2 \) or \( w_1,v_5,v_1,w_2 \) is a bicoloured~\( P_4 \). 
So, \( f(w_2)=3 \). 
If \( f(w_5)=1 \) or 4, then either \( w_1,v_2,v_1,w_5 \) or \( w_1,v_5,v_1,w_5 \) is a bicoloured~\( P_4 \). 
So, \( f(w_5)=2 \). 
But, then path \( w_2,v_3,v_4,w_5 \) is a bicoloured \( P_4 \). 
This contradiction proves Claim~2.\\

Due to Claim~2, \( f(w_1)\in\{2,3\} \). 
Similarly, \( f(w_2)\in\{3,4\} \), \( f(w_3)\in\{4,0\} \), \( f(w_4)\in\{0,1\} \) and \( f(w_5)\in\{1,2\} \).\\

\noindent Case 1: \( f(w_1)=2 \).\\
This forces colour~0 at \( w_4 \) (if \( f(w_4)=1 \), then \( w_1,v_2,v_3,w_4 \) is a bicoloured~\( P_4 \)). 
This in turn forces colour~3 at \( w_2 \) similarly. 
By repeating this argument, we can show that \( f \) must be of the form displayed in Figure~\ref{fig:5-star colouring of gadget component-a} up to colour swaps.\\

\noindent Case 2: \( f(w_1)=3 \).\\
This forces colour~0 at \( w_3 \) (if \( f(w_3)=4 \), then \( w_1,v_5,v_4,w_3 \) is a bicoloured~\( P_4 \)). 
This in turn forces colour 2 at \( w_5 \) similarly. 
By repeating this argument, we can show that \( f \) must be of the form displayed in Figure~\ref{fig:5-star colouring of gadget component-b} up to colour swaps.
\end{proof}

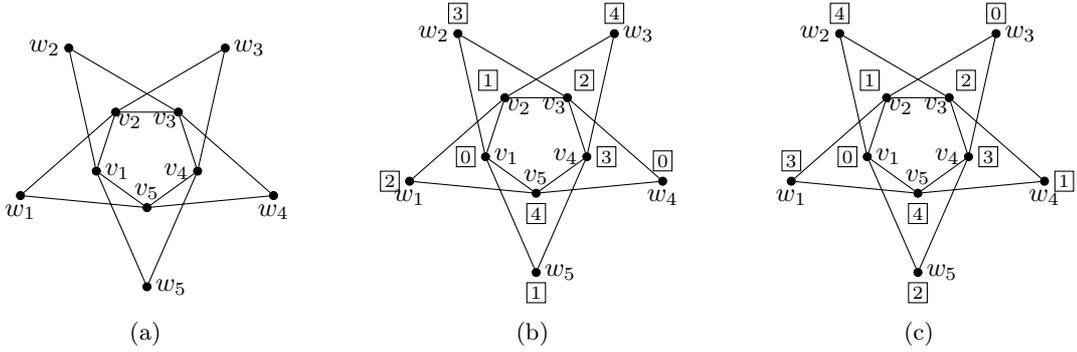
\begin{figure}[hbt]
\centering
\begin{subfigure}[b]{0.33\textwidth}
\centering
\begin{tikzpicture}[scale=0.70]
\tikzstyle every label=[label distance=-2pt]
\coordinate (base);
\path (base)--+(-162:1) node[dot](v1)[label=right:\( v_1 \)]{}
      (base)--+(126:1) node[dot](v2)[label={[label distance=-6pt]below right:\( v_2 \)}]{}
      (base)--+(54:1) node[dot](v3)[label={[label distance=-6.5pt]below left:\( v_3 \)}]{}
      (base)--+(-18:1) node[dot](v4)[label=left:\( v_4 \)]{}
      (base)--+(-90:1) node[dot](v5)[label=above:\( v_5 \)]{};
\draw (v1)--(v2)--(v3)--(v4)--(v5)--(v1);

\path (v1)--+(-162:1.5) node[dot](w1)[label=below:\( w_1 \)]{}
      (v2)--+(126:1.5) node[dot](w2)[label=left:\( w_2 \)]{}
      (v3)--+(54:1.5) node[dot](w3)[label=right:\( w_3 \)]{}
      (v4)--+(-18:1.5) node[dot](w4)[label=below:\( w_4 \)]{}
      (v5)--+(-90:1.5) node[dot](w5)[label=right:\( w_5 \)]{};
\draw (v1)--(w2)--(v3)
      (v2)--(w3)--(v4)
      (v3)--(w4)--(v5)
      (v4)--(w5)--(v1)
      (v5)--(w1)--(v2);
\end{tikzpicture}
\caption{}\label{fig:5-star gadget compnent}
\end{subfigure}\begin{subfigure}[b]{0.33\textwidth}
\centering
\begin{tikzpicture}[scale=0.70]
\tikzstyle every label=[label distance=-2pt]
\coordinate (base);
\path (base)--+(-162:1) node[dot](v1)[label=right:\( v_1 \)][label={[vcolour]left:0}]{}
      (base)--+(126:1) node[dot](v2)[label={[label distance=-6pt]below right:\( v_2 \)}][label={[vcolour]above left:1}]{}
      (base)--+(54:1) node[dot](v3)[label={[label distance=-6.5pt]below left:\( v_3 \)}][label={[vcolour]above right:2}]{}
      (base)--+(-18:1) node[dot](v4)[label=left:\( v_4 \)][label={[vcolour]right:3}]{}
      (base)--+(-90:1) node[dot](v5)[label=above:\( v_5 \)][label={[vcolour]below:4}]{};
\draw (v1)--(v2)--(v3)--(v4)--(v5)--(v1);

\path (v1)--+(-162:1.5) node[dot](w1)[label=below:\( w_1 \)][label={[vcolour]left:2}]{}
      (v2)--+(126:1.5) node[dot](w2)[label={[vcolour]above:3}][label=left:\( w_2 \)]{}
      (v3)--+(54:1.5) node[dot](w3)[label={[vcolour]above:4}][label=right:\( w_3 \)]{}
      (v4)--+(-18:1.5) node[dot](w4)[label={[vcolour]above:0}][label=below:\( w_4 \)]{}
      (v5)--+(-90:1.5) node[dot](w5)[label={[vcolour]below:1}][label=right:\( w_5 \)]{};
\draw (v1)--(w2)--(v3)
      (v2)--(w3)--(v4)
      (v3)--(w4)--(v5)
      (v4)--(w5)--(v1)
      (v5)--(w1)--(v2);
\end{tikzpicture}
\caption{}
\label{fig:5-star colouring of gadget component-a}
\end{subfigure}\begin{subfigure}[b]{0.33\textwidth}
\centering
\begin{tikzpicture}[scale=0.70]
\tikzstyle every label=[label distance=-2pt]
\coordinate (base);
\path (base)--+(-162:1) node[dot](v1)[label=right:\( v_1 \)][label={[vcolour]left:0}]{}
      (base)--+(126:1) node[dot](v2)[label={[label distance=-6pt]below right:\( v_2 \)}][label={[vcolour]above left:1}]{}
      (base)--+(54:1) node[dot](v3)[label={[label distance=-6.5pt]below left:\( v_3 \)}][label={[vcolour]above right:2}]{}
      (base)--+(-18:1) node[dot](v4)[label=left:\( v_4 \)][label={[vcolour]right:3}]{}
      (base)--+(-90:1) node[dot](v5)[label=above:\( v_5 \)][label={[vcolour]below:4}]{};
\draw (v1)--(v2)--(v3)--(v4)--(v5)--(v1);

\path (v1)--+(-162:1.5) node[dot](w1)[label=below:\( w_1 \)][label={[vcolour]above:3}]{}
      (v2)--+(126:1.5) node[dot](w2)[label={[vcolour]above:4}][label=left:\( w_2 \)]{}
      (v3)--+(54:1.5) node[dot](w3)[label={[vcolour]above:0}][label=right:\( w_3 \)]{}
      (v4)--+(-18:1.5) node[dot](w4)[label={[vcolour]right:1}][label=below:\( w_4 \)]{}
      (v5)--+(-90:1.5) node[dot](w5)[label={[vcolour]below:2}][label=right:\( w_5 \)]{};
\draw (v1)--(w2)--(v3)
      (v2)--(w3)--(v4)
      (v3)--(w4)--(v5)
      (v4)--(w5)--(v1)
      (v5)--(w1)--(v2);
\end{tikzpicture}
\caption{}
\label{fig:5-star colouring of gadget component-b}
\end{subfigure}\caption[Gadget component and general form of 5-star colouring of it.]{(a) Gadget component, (b,c) General form of 5-star colouring of it.}
\label{fig:miecelski-1}
\end{figure}

The 2-in-2-out gadget is displayed in Figure~\ref{fig:2-in-2-out gadget}. Observe that two copies of the gadget component are part of this gadget. The following lemma shows why 5-star colouring of this gadget is interesting.

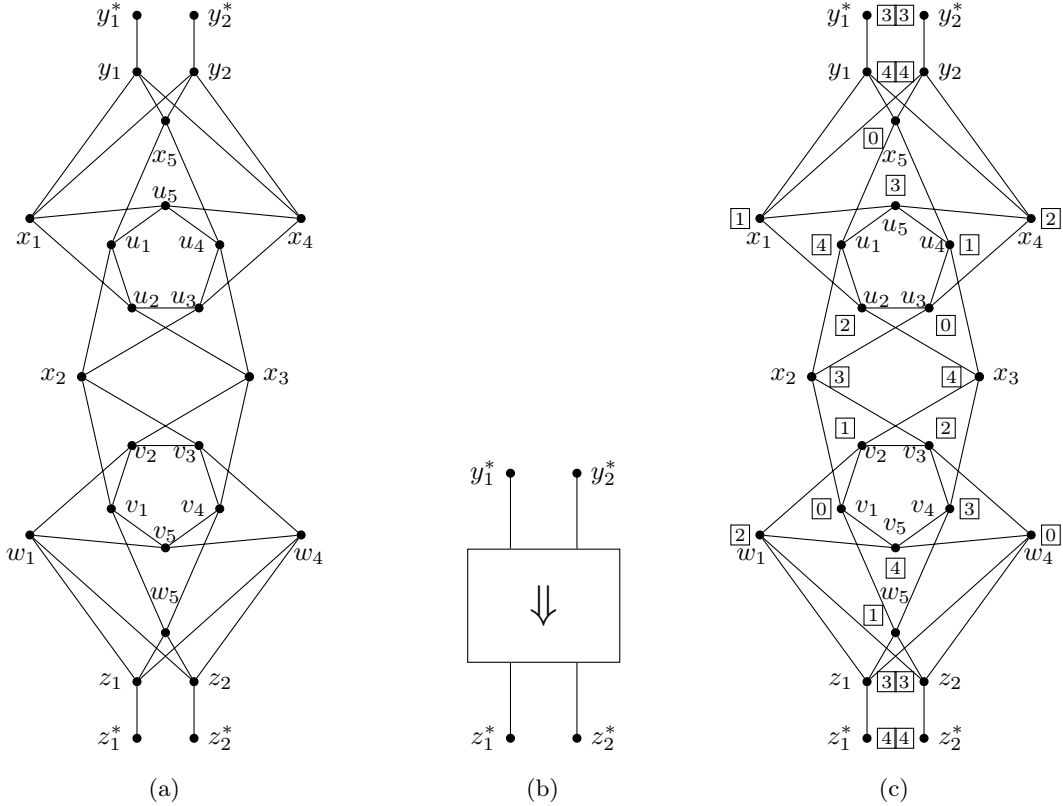
\begin{figure}[hbt]
\begin{subfigure}[b]{0.3\textwidth}
\centering
\begin{tikzpicture}[scale=0.75]
\coordinate (base1);
\path (base1)--+(162:1) node[dot](u1)[label=right:\( u_1 \)]{}
      (base1)--+(-126:1) node[dot](u2)[label={[label distance=-6pt]above right:\( u_2 \)}]{}
      (base1)--+(-54:1) node[dot](u3)[label={[label distance=-6pt]above left:\( u_3 \)}]{}
      (base1)--+(18:1) node[dot](u4)[label=left:\( u_4 \)]{}
      (base1)--+(90:1) node[dot](u5)[label={[label distance=-4pt]above:\( u_5 \)}]{};
\draw (u1)--(u2)--(u3)--(u4)--(u5)--(u1);

\path (u1)--+(162:1.5) node[dot](x1)[label=below:\( x_1 \)]{}
      (u2)--+(-126:1.5) node[dot](x2)[label=left:\( x_2 \)]{}
      (u3)--+(-54:1.5) node[dot](x3)[label=right:\( x_3 \)]{}
      (u4)--+(18:1.5) node[dot](x4)[label=below:\( x_4 \)]{}
      (u5)--+(90:1.5) node[dot](x5)[label={[label distance=6pt]below:\( x_5 \)}]{};
\draw (u1)--(x2)--(u3)
      (u2)--(x3)--(u4)
      (u3)--(x4)--(u5)
      (u4)--(x5)--(u1)
      (u5)--(x1)--(u2);

\path (base1)++(0,-4.05) coordinate (base2);
\path (base2)--+(-162:1) node[dot](v1)[label=right:\( v_1 \)]{}
      (base2)--+(126:1) node[dot](v2)[label={[label distance=-6pt]below right:\( v_2 \)}]{}
      (base2)--+(54:1) node[dot](v3)[label={[label distance=-6pt]below left:\( v_3 \)}]{}
      (base2)--+(-18:1) node[dot](v4)[label=left:\( v_4 \)]{}
      (base2)--+(-90:1) node[dot](v5)[label={[label distance=-3pt]above:\( v_5 \)}]{};
\draw (v1)--(v2)--(v3)--(v4)--(v5)--(v1);

\path (v1)--+(-162:1.5) node[dot](w1)[label={[xshift=-3pt]below:\( w_1 \)}]{}
      (v2)--+(126:1.5) node[dot](w2){}
      (v3)--+(54:1.5) node[dot](w3){}
      (v4)--+(-18:1.5) node[dot](w4)[label={[xshift=3pt]below:\( w_4 \)}]{}
      (v5)--+(-90:1.5) node[dot](w5)[label={[label distance=6pt]above:\( w_5 \)}]{};
\draw (v1)--(w2)--(v3)
      (v2)--(w3)--(v4)
      (v3)--(w4)--(v5)
      (v4)--(w5)--(v1)
      (v5)--(w1)--(v2);

\draw (x5)--+(120:1) node[dot] (y1)[label=left:\( y_1 \)]{}
      (x5)--+(60:1) node[dot] (y2)[label=right:\( y_2 \)]{};
\draw (y1)--(x1)
      (y1)--(x4);
\draw (y2)--(x4)
      (y2)--(x1);
\draw (y1)--+(0,1) node[dot] (y1*)[label=left:\( y_1^* \)]{};
\draw (y2)--+(0,1) node[dot] (y2*)[label=right:\( y_2^* \)]{};

\draw (w5)--+(-120:1) node[dot] (z1)[label=left:\( z_1 \)]{}
      (w5)--+(-60:1) node[dot] (z2)[label=right:\( z_2 \)]{};
\draw (z1)--(w1)
      (z1)--(w4);
\draw (z2)--(w4)
      (z2)--(w1);
\draw (z1)--+(0,-1) node[dot] (z1*)[label=left:\( z_1^* \)]{};
\draw (z2)--+(0,-1) node[dot] (z2*)[label=right:\( z_2^* \)]{};
\end{tikzpicture}
\caption{}
\end{subfigure}\hspace*{10pt}\begin{subfigure}[b]{0.3\textwidth}
\centering
\begin{tikzpicture}
\node (box) [22box]{};
\draw (box.120) --+(0,1) node[dot] (y1*)[label=left:\( y_1^* \)]{}
      (box.60) --+(0,1) node[dot] (y2*)[label=right:\( y_2^* \)]{}
      (box.-120) --+(0,-1) node[dot] (z1*)[label=left:\( z_1^* \)]{}
      (box.-60) --+(0,-1) node[dot] (z2*)[label=right:\( z_2^* \)]{};
\end{tikzpicture}
\caption{}
\label{fig:2-in-2-out gadget symbol}
\end{subfigure}\begin{subfigure}[b]{0.3\textwidth}
\centering
\begin{tikzpicture}[scale=0.75]
\coordinate (base1);
\path (base1)--+(162:1) node[dot](u1)[label=right:\( u_1 \)][label={[vcolour]left:4}]{}
      (base1)--+(-126:1) node[dot](u2)[label={[label distance=-6pt]above right:\( u_2 \)}][label={[vcolour]below left:2}]{}
      (base1)--+(-54:1) node[dot](u3)[label={[label distance=-6pt]above left:\( u_3 \)}][label={[vcolour]below right:0}]{}
      (base1)--+(18:1) node[dot](u4)[label={[label distance=-4pt]left:\( u_4 \)}][label={[vcolour]right:1}]{}
      (base1)--+(90:1) node[dot](u5)[label=below:\( u_5 \)][label={[vcolour]above:3}]{};
\draw (u1)--(u2)--(u3)--(u4)--(u5)--(u1);

\path (u1)--+(162:1.5) node[dot](x1)[label=below:\( x_1 \)][label={[vcolour]left:1}]{}
      (u2)--+(-126:1.5) node[dot](x2)[label=left:\( x_2 \)]{}
      (u3)--+(-54:1.5) node[dot](x3)[label=right:\( x_3 \)]{}
      (u4)--+(18:1.5) node[dot](x4)[label=below:\( x_4 \)][label={[vcolour]right:2}]{}
      (u5)--+(90:1.5) node[dot](x5)[label={[label distance=6pt]below:\( x_5 \)}][label={[vcolour,xshift=-2pt]below left:0}]{};
\draw (u1)--(x2)--(u3)
      (u2)--(x3)--(u4)
      (u3)--(x4)--(u5)
      (u4)--(x5)--(u1)
      (u5)--(x1)--(u2);

\path (base1)++(0,-4.05) coordinate (base2);
\path (base2)--+(-162:1) node[dot](v1)[label=right:\( v_1 \)][label={[vcolour]left:0}]{}
      (base2)--+(126:1) node[dot](v2)[label={[label distance=-6pt]below right:\( v_2 \)}][label={[vcolour]above left:1}]{}
      (base2)--+(54:1) node[dot](v3)[label={[label distance=-6pt]below left:\( v_3 \)}][label={[vcolour]above right:2}]{}
      (base2)--+(-18:1) node[dot](v4)[label=left:\( v_4 \)][label={[vcolour]right:3}]{}
      (base2)--+(-90:1) node[dot](v5)[label=above:\( v_5 \)][label={[vcolour]below:4}]{};
\draw (v1)--(v2)--(v3)--(v4)--(v5)--(v1);

\path (v1)--+(-162:1.5) node[dot](w1)[label={[xshift=-3pt]below:\( w_1 \)}][label={[vcolour]left:2}]{}
      (v2)--+(126:1.5) node[dot](w2)[label={[vcolour,label distance=5pt]right:3}]{}
      (v3)--+(54:1.5) node[dot](w3)[label={[vcolour,label distance=5pt]left:4}]{}
      (v4)--+(-18:1.5) node[dot](w4)[label={[xshift=3pt]below:\( w_4 \)}][label={[vcolour]right:0}]{}
      (v5)--+(-90:1.5) node[dot](w5)[label={[label distance=6pt]above:\( w_5 \)}][label={[vcolour,xshift=-2pt]above left:1}]{};
\draw (v1)--(w2)--(v3)
      (v2)--(w3)--(v4)
      (v3)--(w4)--(v5)
      (v4)--(w5)--(v1)
      (v5)--(w1)--(v2);

\draw (x5)--+(120:1) node[dot] (y1)[label=left:\( y_1 \)][label={[vcolour]right:4}]{}
      (x5)--+(60:1) node[dot] (y2)[label=right:\( y_2 \)][label={[vcolour]left:4}]{};
\draw (y1)--(x1)
      (y1)--(x4);
\draw (y2)--(x4)
      (y2)--(x1);
\draw (y1)--+(0,1) node[dot] (y1*)[label=left:\( y_1^* \)][label={[vcolour,solid]right:3}]{};
\draw (y2)--+(0,1) node[dot] (y2*)[label=right:\( y_2^* \)][label={[vcolour,solid]left:3}]{};

\draw (w5)--+(-120:1) node[dot] (z1)[label=left:\( z_1 \)][label={[vcolour]right:3}]{}
      (w5)--+(-60:1) node[dot] (z2)[label=right:\( z_2 \)][label={[vcolour]left:3}]{};
\draw (z1)--(w1)
      (z1)--(w4);
\draw (z2)--(w4)
      (z2)--(w1);
\draw (z1)--+(0,-1) node[dot] (z1*)[label=left:\( z_1^* \)][label={[vcolour,solid]right:4}]{};
\draw (z2)--+(0,-1) node[dot] (z2*)[label=right:\( z_2^* \)][label={[vcolour,solid]left:4}]{};
\end{tikzpicture}
\caption{}
\label{fig:5-star colouring 2-in-2-out gadget}
\end{subfigure}\caption{(a)~The 2-in-2-out gadget, (b)~its symbolic representation, and (c)~a 5-star colouring of the gadget.}
\label{fig:2-in-2-out gadget}
\end{figure}

\begin{lemma}\label{lem:2-in-2-out gadget}
For every 5-star colouring \( f \) of the 2-in-2-out gadget, there exist two distinct colours \( c_1 \) and \( c_2 \) such that \( f(y_1)=f(y_2)=f(z_1^*)=f(z_2^*)=c_1 \) and \( f(y_1^*)=f(y_2^*)=f(z_1)=f(z_2)=c_2 \). 
Moreover, every 3-vertex path containing one of the pendant edges of the gadget is tricoloured by \( f \).
\end{lemma}
\begin{proof}
The 2-in-2-out gadget is displayed in Figure~\ref{fig:2-in-2-out gadget}. 
Let \( f \) be a 5-star colouring of the 2-in-2-out gadget that uses colours 0,1,2,3 and 4. 
We prove Lemma~\ref{lem:2-in-2-out gadget} for the case when the bottom copy of the gadget component in the 2-in-2-out gadget is coloured by the scheme in Figure~\ref{fig:5-star colouring of gadget component-a}; the proof is similar when the scheme in Figure~\ref{fig:5-star colouring of gadget component-b} is used instead. 
That is, vertices \( v_1,v_2,v_3,v_4,v_5 \) and vertices \( w_4,w_5,w_1,x_2,x_3 \) are coloured \( 0,1,2,3,4 \). Clearly, \( f(z_1),f(z_2)\in\{3,4\} \). 
Note that \( f(z_1)\neq 4 \)  (if not, \( v_1,v_5,w_4,z_1 \) is a bicoloured \( P_4 \)). Hence, \( f(z_1)=3 \). Similarly, \( f(z_2)=3 \). Since \( z_1 \) and \( z_2 \) have common neighbours coloured 0,1 and 2, both \( z_1^* \) and \( z_2^* \) must be coloured 4 by \( f \) (e.g.: if \( f(z_1^*)=0 \), then \( z_1^*,z_1,w_4,z_2 \) is a bicoloured \( P_4 \)).

Observe that \( f(u_2)\neq 3 \) (if not, \( v_5,v_4,x_3,u_2 \) is a bicoloured \( P_4 \)). 
Similarly, \mbox{\( f(u_4)\neq 3 \)}. 
By Claim~\hyperlink{lnk:all 5 colours on C5}{1} in the proof of Lemma~\ref{lem:miecelski-1}, all five colours must be used by \( f \) on the inner \( C_5 \) of a gadget component. 
So, colour~3 must be used on the cycle \( (u_1,u_2,u_3,u_4,u_5) \), and thus \( f(u_5)=3 \) (see Figure~\ref{fig:partial colouring of 2-in-2-out gadget}).

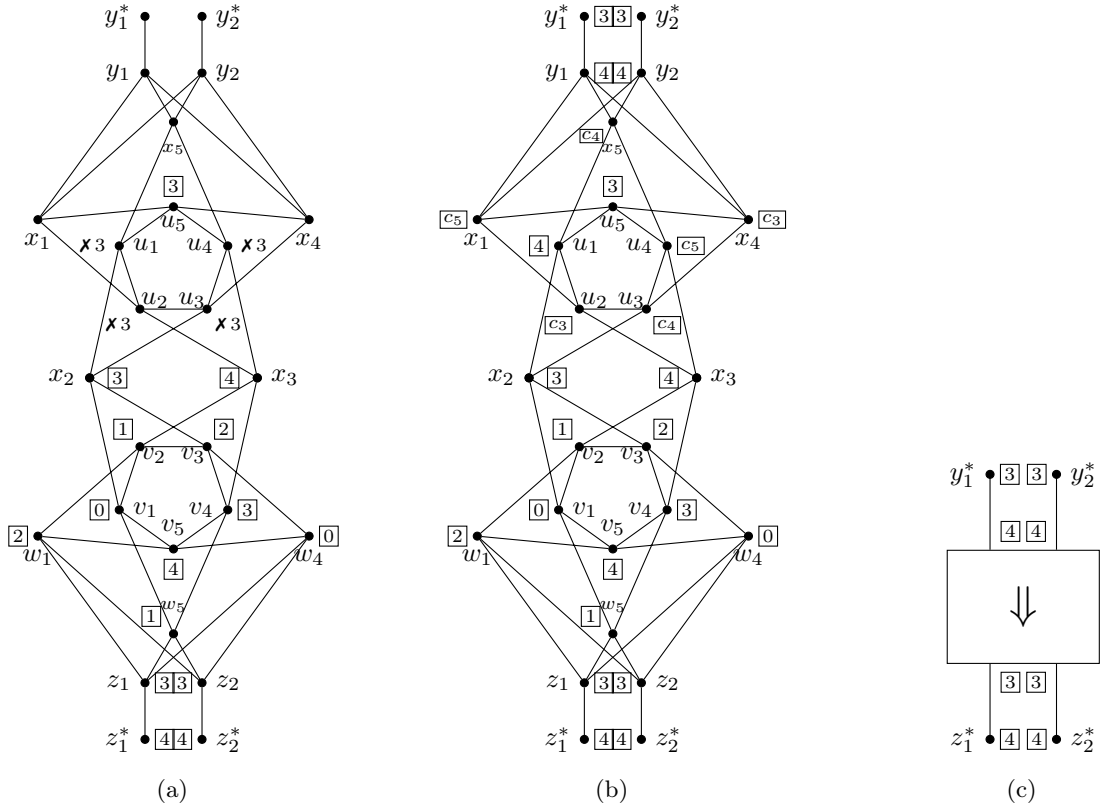
\begin{figure}[hbt]
\centering
\begin{subfigure}[b]{0.3\textwidth}
\centering
\begin{tikzpicture}[scale=0.75]
\coordinate (base1);
\path (base1)--+(162:1) node[dot](u1)[label=right:\( u_1 \)][label={[font=\scriptsize]left:\( \xmark \,3 \)}]{}
      (base1)--+(-126:1) node[dot](u2)[label={[label distance=-6pt]above right:\( u_2 \)}][label={[font=\scriptsize,label distance=-2pt]below left:\( \xmark \,3 \)}]{}
      (base1)--+(-54:1) node[dot](u3)[label={[label distance=-6pt]above left:\( u_3 \)}][label={[font=\scriptsize,label distance=-2pt]below right:\( \xmark \,3 \)}]{}
      (base1)--+(18:1) node[dot](u4)[label=left:\( u_4 \)][label={[font=\scriptsize]right:\( \xmark \,3 \)}]{}
      (base1)--+(90:1) node[dot](u5)[label={[label distance=-2pt]below:\( u_5 \)}][label={[vcolour]above:3}]{};
\draw (u1)--(u2)--(u3)--(u4)--(u5)--(u1);

\path (u1)--+(162:1.5) node[dot](x1)[label=below:\( x_1 \)]{}
      (u2)--+(-126:1.5) node[dot](x2)[label=left:\( x_2 \)]{}
      (u3)--+(-54:1.5) node[dot](x3)[label=right:\( x_3 \)]{}
      (u4)--+(18:1.5) node[dot](x4)[label=below:\( x_4 \)]{}
      (u5)--+(90:1.5) node[dot](x5)[label={[label distance=3pt,font=\scriptsize]below:\( x_5 \)}]{};
\draw (u1)--(x2)--(u3)
      (u2)--(x3)--(u4)
      (u3)--(x4)--(u5)
      (u4)--(x5)--(u1)
      (u5)--(x1)--(u2);

\path (base1)++(0,-4.05) coordinate (base2);
\path (base2)--+(-162:1) node[dot](v1)[label=right:\( v_1 \)][label={[vcolour]left:0}]{}
      (base2)--+(126:1) node[dot](v2)[label={[label distance=-6pt]below right:\( v_2 \)}][label={[vcolour]above left:1}]{}
      (base2)--+(54:1) node[dot](v3)[label={[label distance=-6pt]below left:\( v_3 \)}][label={[vcolour]above right:2}]{}
      (base2)--+(-18:1) node[dot](v4)[label=left:\( v_4 \)][label={[vcolour]right:3}]{}
      (base2)--+(-90:1) node[dot](v5)[label=above:\( v_5 \)][label={[vcolour]below:4}]{};
\draw (v1)--(v2)--(v3)--(v4)--(v5)--(v1);

\path (v1)--+(-162:1.5) node[dot](w1)[label=below:\( w_1 \)][label={[vcolour]left:2}]{}
      (v2)--+(126:1.5) node[dot](w2)[label={[vcolour,label distance=5pt]right:3}]{}
      (v3)--+(54:1.5) node[dot](w3)[label={[vcolour,label distance=5pt]left:4}]{}
      (v4)--+(-18:1.5) node[dot](w4)[label=below:\( w_4 \)][label={[vcolour]right:0}]{}
      (v5)--+(-90:1.5) node[dot](w5)[label={[label distance=3pt,font=\scriptsize]above:\( w_5 \)}][label={[vcolour,xshift=-2pt]above left:1}]{};
\draw (v1)--(w2)--(v3)
      (v2)--(w3)--(v4)
      (v3)--(w4)--(v5)
      (v4)--(w5)--(v1)
      (v5)--(w1)--(v2);

\draw (x5)--+(120:1) node[dot] (y1)[label=left:\( y_1 \)]{}
      (x5)--+(60:1) node[dot] (y2)[label=right:\( y_2 \)]{};
\draw (y1)--(x1)
      (y1)--(x4);
\draw (y2)--(x4)
      (y2)--(x1);
\draw (y1)--+(0,1) node[dot] (y1*)[label=left:\( y_1^* \)]{};
\draw (y2)--+(0,1) node[dot] (y2*)[label=right:\( y_2^* \)]{};

\draw (w5)--+(-120:1) node[dot] (z1)[label=left:\( z_1 \)][label={[vcolour]right:3}]{}
      (w5)--+(-60:1) node[dot] (z2)[label=right:\( z_2 \)][label={[vcolour]left:3}]{};
\draw (z1)--(w1)
      (z1)--(w4);
\draw (z2)--(w4)
      (z2)--(w1);
\draw (z1)--+(0,-1) node[dot] (z1*)[label=left:\( z_1^* \)][label={[vcolour]right:4}]{};
\draw (z2)--+(0,-1) node[dot] (z2*)[label=right:\( z_2^* \)][label={[vcolour]left:4}]{};
\end{tikzpicture}
\caption{}
\label{fig:partial colouring of 2-in-2-out gadget}
\end{subfigure}\hfill
\begin{subfigure}[b]{0.3\textwidth}
\centering
\begin{tikzpicture}[scale=0.75]
\coordinate (base1);
\path (base1)--+(162:1) node[dot](u1)[label=right:\( u_1 \)][label={[vcolour]left:4}]{}
      (base1)--+(-126:1) node[dot](u2)[label={[label distance=-6pt]above right:\( u_2 \)}][label={[vcolour,inner sep=1.25pt]below left:\( c_3 \)}]{}
      (base1)--+(-54:1) node[dot](u3)[label={[label distance=-6pt]above left:\( u_3 \)}][label={[vcolour,inner sep=1.25pt]below right:\( c_4 \)}]{}
      (base1)--+(18:1) node[dot](u4)[label=left:\( u_4 \)][label={[vcolour,inner sep=1.25pt]right:\( c_5 \)}]{}
      (base1)--+(90:1) node[dot](u5)[label={[label distance=-2pt]below:\( u_5 \)}][label={[vcolour]above:3}]{};
\draw (u1)--(u2)--(u3)--(u4)--(u5)--(u1);

\path (u1)--+(162:1.5) node[dot](x1)[label=below:\( x_1 \)][label={[vcolour,inner sep=1.25pt]left:\( c_5 \)}]{}
      (u2)--+(-126:1.5) node[dot](x2)[label=left:\( x_2 \)]{}
      (u3)--+(-54:1.5) node[dot](x3)[label=right:\( x_3 \)]{}
      (u4)--+(18:1.5) node[dot](x4)[label=below:\( x_4 \)][label={[vcolour,inner sep=1.25pt]right:\( c_3 \)}]{}
      (u5)--+(90:1.5) node[dot](x5)[label={[label distance=3pt,font=\scriptsize]below:\( x_5 \)}][label={[vcolour,xshift=-1pt,inner sep=0.5pt]below left:\( c_4 \)}]{};
\draw (u1)--(x2)--(u3)
      (u2)--(x3)--(u4)
      (u3)--(x4)--(u5)
      (u4)--(x5)--(u1)
      (u5)--(x1)--(u2);

\path (base1)++(0,-4.05) coordinate (base2);
\path (base2)--+(-162:1) node[dot](v1)[label=right:\( v_1 \)][label={[vcolour]left:0}]{}
      (base2)--+(126:1) node[dot](v2)[label={[label distance=-6pt]below right:\( v_2 \)}][label={[vcolour]above left:1}]{}
      (base2)--+(54:1) node[dot](v3)[label={[label distance=-6pt]below left:\( v_3 \)}][label={[vcolour]above right:2}]{}
      (base2)--+(-18:1) node[dot](v4)[label=left:\( v_4 \)][label={[vcolour]right:3}]{}
      (base2)--+(-90:1) node[dot](v5)[label=above:\( v_5 \)][label={[vcolour]below:4}]{};
\draw (v1)--(v2)--(v3)--(v4)--(v5)--(v1);

\path (v1)--+(-162:1.5) node[dot](w1)[label=below:\( w_1 \)][label={[vcolour]left:2}]{}
      (v2)--+(126:1.5) node[dot](w2)[label={[vcolour,label distance=5pt]right:3}]{}
      (v3)--+(54:1.5) node[dot](w3)[label={[vcolour,label distance=5pt]left:4}]{}
      (v4)--+(-18:1.5) node[dot](w4)[label=below:\( w_4 \)][label={[vcolour]right:0}]{}
      (v5)--+(-90:1.5) node[dot](w5)[label={[label distance=3pt,font=\scriptsize]above:\( w_5 \)}][label={[vcolour,xshift=-2pt]above left:1}]{};
\draw (v1)--(w2)--(v3)
      (v2)--(w3)--(v4)
      (v3)--(w4)--(v5)
      (v4)--(w5)--(v1)
      (v5)--(w1)--(v2);

\draw (x5)--+(120:1) node[dot] (y1)[label=left:\( y_1 \)][label={[vcolour]right:4}]{}
      (x5)--+(60:1) node[dot] (y2)[label=right:\( y_2 \)][label={[vcolour]left:4}]{};
\draw (y1)--(x1)
      (y1)--(x4);
\draw (y2)--(x4)
      (y2)--(x1);
\draw (y1)--+(0,1) node[dot] (y1*)[label=left:\( y_1^* \)][label={[vcolour,solid]right:3}]{};
\draw (y2)--+(0,1) node[dot] (y2*)[label=right:\( y_2^* \)][label={[vcolour,solid]left:3}]{};

\draw (w5)--+(-120:1) node[dot] (z1)[label=left:\( z_1 \)][label={[vcolour]right:3}]{}
      (w5)--+(-60:1) node[dot] (z2)[label=right:\( z_2 \)][label={[vcolour]left:3}]{};
\draw (z1)--(w1)
      (z1)--(w4);
\draw (z2)--(w4)
      (z2)--(w1);
\draw (z1)--+(0,-1) node[dot] (z1*)[label=left:\( z_1^* \)][label={[vcolour,solid]right:4}]{};
\draw (z2)--+(0,-1) node[dot] (z2*)[label=right:\( z_2^* \)][label={[vcolour,solid]left:4}]{};
\end{tikzpicture}
\caption{}
\label{fig:repeat 5-star colouring 2-in-2-out gadget}
\end{subfigure}\hfill
\begin{subfigure}[b]{0.25\textwidth}
\centering
\begin{tikzpicture}
\node (box) [22box]{};
\draw (box.120) node[label={[vcolour,solid,yshift=-2pt,xshift=-1pt]above right:4}]{}--+(0,1) node[dot] (y1*)[label=left:\( y_1^* \)][label={[vcolour,solid]right:3}]{}
      (box.60) node[label={[vcolour,solid,yshift=-2pt,xshift=1pt]above left:4}]{} --+(0,1) node[dot] (y2*)[label=right:\( y_2^* \)][label={[vcolour,solid]left:3}]{}
      (box.-120) node[label={[vcolour,solid,yshift=2pt,xshift=-1pt]below right:3}]{} --+(0,-1) node[dot] (z1*)[label=left:\( z_1^* \)][label={[vcolour,solid]right:4}]{}
      (box.-60) node[label={[vcolour,solid,yshift=2pt,xshift=1pt]below left:3}]{}--+(0,-1) node[dot] (z2*)[label=right:\( z_2^* \)][label={[vcolour,solid]left:4}]{};
\end{tikzpicture}
\caption{}
\end{subfigure}\caption[Partial 5-star colourings of the 2-in-2-out gadget.]{(a)~A partial 5-star colouring of the 2-in-2-out gadget, (b)~a 5-star colouring of the 2-in-2-out gadget where \( \{c_3,c_4,c_5\} \) is a permutation of \( \{0,1,2\} \), and (c)~the symbolic representation of the colouring in (b).}
\end{figure}

By Lemma~\ref{lem:miecelski-1}, vertices \( x_1 \), \( x_2 \), \( x_3 \), \( x_4 \), \( x_5 \) and vertices \( u_1 \), \( u_2 \), \( u_3 \), \( u_4 \), \( u_5 \) must be coloured by the same cyclic order of colours (see Figure~\ref{fig:miecelski-1}). 
So, \mbox{\( f(u_1)=4 \)}. 
For the same reason, \( f(x_4)=f(u_2) \), \( f(x_5)=f(u_3) \) and \( f(x_1)=f(u_4) \). 
Since all five colours must be used on the cycle \( (u_1,u_2,u_3,u_4,u_5) \), we have \( \{f(x_4),f(x_5),f(x_1)\}\allowbreak =\{f(u_2),f(u_3),f(u_4)\}=\{c_3,c_4,c_5\} \) where \( \{c_3,c_4,c_5\} \) is a permutation of \( \{0,1,2\} \). 
So, \mbox{\( f(y_1),f(y_2)\in\{3,4\} \)}. 
If \( f(y_1)=3 \), then \( u_4,u_5,x_1,y_1 \) is a bicoloured \( P_4 \). 
Hence, \( f(y_1)=4 \). 
Similarly, \( f(y_2)=4 \). 
Since \( y_1 \) and \( y_2 \) have common neighbours coloured 0,1 and 2, both \( y_1^* \) and \( y_2^* \) must be coloured 3 by \( f \). 
That is, the colouring \( f \) is as shown in Figure~\ref{fig:repeat 5-star colouring 2-in-2-out gadget}. 
Clearly, there exist distinct colours \( c_1 \) and \( c_2 \) such that \( f(y_1)=f(y_2)=f(z_1^*)=f(z_2^*)=c_1 \) and \( f(y_1^*)=f(y_2^*)=f(z_1)=f(z_2)=c_2 \) (here, \( c_1=4 \) and \( c_2=3 \)). 
Also, every 3-vertex path containing a pendant edge of the gadget is tricoloured by \( f \). 
This completes the proof. 
\end{proof}

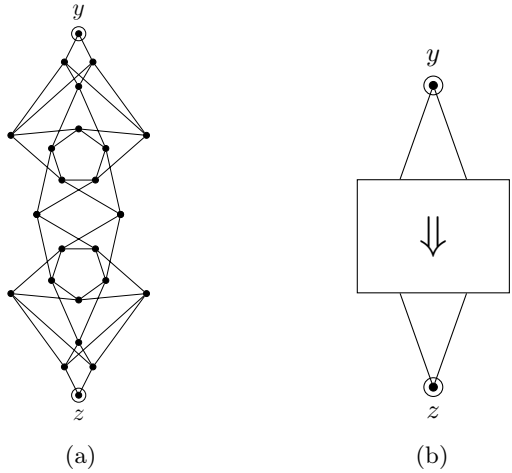
\begin{figure}[hbt]
\centering
\begin{subfigure}[b]{0.3\textwidth}
\centering
\scalebox{0.75}{\begin{tikzpicture}[scale=0.5]
\coordinate (base1);
\path (base1)--+(162:1) node[dot](u1){}
      (base1)--+(-126:1) node[dot](u2){}
      (base1)--+(-54:1) node[dot](u3){}
      (base1)--+(18:1) node[dot](u4){}
      (base1)--+(90:1) node[dot](u5){};
\draw (u1)--(u2)--(u3)--(u4)--(u5)--(u1);

\path (u1)--+(162:1.5) node[dot](x1){}
      (u2)--+(-126:1.5) node[dot](x2){}
      (u3)--+(-54:1.5) node[dot](x3){}
      (u4)--+(18:1.5) node[dot](x4){}
      (u5)--+(90:1.5) node[dot](x5){};
\draw (u1)--(x2)--(u3)
      (u2)--(x3)--(u4)
      (u3)--(x4)--(u5)
      (u4)--(x5)--(u1)
      (u5)--(x1)--(u2);

\path (base1)++(0,-4.05) coordinate (base2);
\path (base2)--+(-162:1) node[dot](v1){}
      (base2)--+(126:1) node[dot](v2){}
      (base2)--+(54:1) node[dot](v3){}
      (base2)--+(-18:1) node[dot](v4){}
      (base2)--+(-90:1) node[dot](v5){};
\draw (v1)--(v2)--(v3)--(v4)--(v5)--(v1);

\path (v1)--+(-162:1.5) node[dot](w1){}
      (v2)--+(126:1.5) node[dot](w2){}
      (v3)--+(54:1.5) node[dot](w3){}
      (v4)--+(-18:1.5) node[dot](w4){}
      (v5)--+(-90:1.5) node[dot](w5){};
\draw (v1)--(w2)--(v3)
      (v2)--(w3)--(v4)
      (v3)--(w4)--(v5)
      (v4)--(w5)--(v1)
      (v5)--(w1)--(v2);

\draw (x5)--+(120:1) node[dot] (y1){}
      (x5)--+(60:1) node[dot] (y2){};
\draw (y1)--(x1)
      (y1)--(x4);
\draw (y2)--(x4)
      (y2)--(x1);
\draw (y1)--+(0.5, 1) node[dot] (y){} node[terminal][label={[font=\large]above:\( y \)}]{};
\draw (y2)--(y);

\draw (w5)--+(-120:1) node[dot] (z1){}
      (w5)--+(-60:1) node[dot] (z2){};
\draw (z1)--(w1)
      (z1)--(w4);
\draw (z2)--(w4)
      (z2)--(w1);
\draw (z1)--+(0.5,-1) node[dot] (z){} node[terminal][label={[font=\large]below:\( z \)}]{};
\draw (z2)--(z);
\end{tikzpicture}
}\caption{}
\end{subfigure}\begin{subfigure}[b]{0.3\textwidth}
\centering
\begin{tikzpicture}
\node (box) [22box]{};
\path (box) --+(0,2) node(u)[dot]{} node[terminal][label=above:\( y \)]{};
\path (box) --+(0,-2) node(v)[dot]{} node[terminal][label=below:\( z \)]{};
\draw (box.120) --(u) --(box.60)
      (box.-120) --(v) --(box.-60);
\end{tikzpicture}
\caption{}
\end{subfigure}\caption[A not-equal gadget between terminals \( y \) and \( z \), and its symbolic representation.]{(a)~A not-equal gadget between terminals \( y \) and \( z \) (it is made of one 2-in-2-out gadget), and (b)~its symbolic representation.}
\label{fig:not-equal gadget}
\end{figure}

The not-equal gadget is the graph displayed in Figure~\ref{fig:not-equal gadget}. 
The not-equal gadget is made from one 2-in-2-out gadget by identifying vertex \( y_1^* \) of the 2-in-2-out gadget with vertex \( y_2^* \) and identifying vertex \( z_1^* \) with vertex \( z_2^* \).
Hence, the next lemma follows from Lemma~\ref{lem:2-in-2-out gadget} (note that \( c_1\neq c_2 \) in Lemma~\ref{lem:2-in-2-out gadget}).

\begin{lemma}\label{lem:not-equal gadget}
The terminals of the not-equal gadget should get different colours under each 5-star colouring \( f \). 
Moreover, every 3-vertex path within the gadget with a terminal as one endpoint is tricoloured by \( f \).
\qed
\end{lemma}

We are now ready to present the construction.
\begin{construct}\label{make:5-star max-deg 4}
\emph{Input:} A 4-regular graph \( G \).\\
\emph{Output:} A triangle-free graph \( G' \) of maximum degree four.\\
\emph{Guarantee:} \( G \) is 3-colourable if and only if \( G' \) is 5-star colourable.\\
\emph{Steps:}\\
Let \( v_1,v_2,\dots,v_n \) be the vertices in \( G \). 
First, replace each vertex \( v_i \) of \( G \) by a vertex gadget as shown in Figure~\ref{fig:5-star vertex replacement}. 
The vertex gadget for \( v_i \) has six terminals namely \( v_{i,0},v_{i,1},v_{i,2},v_{i,3},v_{i,4} \) and \( v_{i,5} \). 
The terminals \( v_{i,1},v_{i,2},v_{i,3},v_{i,4} \) accommodate the edges incident on \( v_i \) in \( G \). 
The replacement of vertices by vertex gadgets converts each edge \( v_iv_j \) of \( G \) to an edge between terminals \( v_{i,k} \) and \( v_{j,\ell} \) for some \( k,\ell\in\{1,2,3,4\} \).

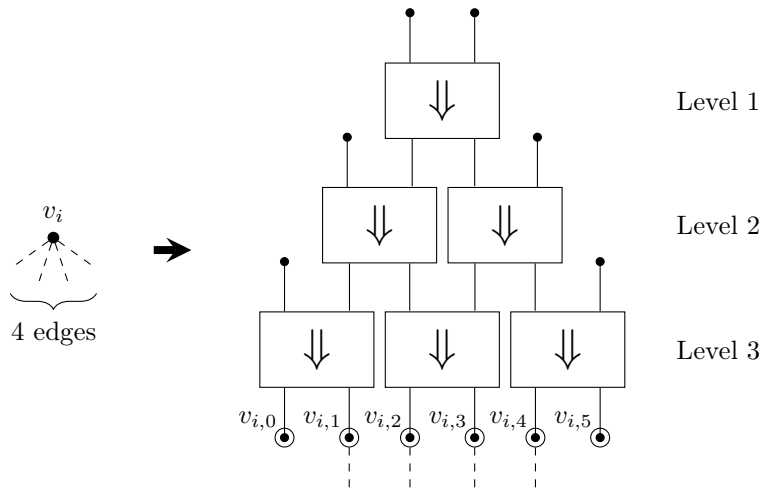
\begin{figure}[hbt]
\centering
\begin{tikzpicture}[scale=0.33]
\node (vi)[bigdot][label=above:\( v_i \)]{};
\draw[dashed] (vi)--+(-36:2) coordinate (rightNbr)
              (vi)--+(-72:2)
              (vi)--+(-108:2)
              (vi)--+(-144:2) coordinate (leftNbr);
\draw [decorate,decoration={brace,amplitude=6pt,raise=10pt,mirror}] (leftNbr)++(-0.1,0)--node[yshift=-25pt]{4 edges} ($(rightNbr)+(0.1,0)$);

\path (vi)++(4,-0.5) coordinate(from) ++(1.5,0) coordinate(to);
\draw [-stealth,draw=black,line width=3pt] (from)--(to);

\path (to)--+(10,6) node(l1)[22boxsmall] {};
\path (l1)--+ (-2.5,-5) node(l21)[22boxsmall] {}
      (l1)--+ (2.5,-5) node(l22)[22boxsmall] {};
\path (l21)--+ (-2.5,-5) node(l31)[22boxsmall] {}
      (l21)--+ (2.5,-5) node(l32)[22boxsmall] {};
\path (l22)--+ (2.5,-5) node(l33)[22boxsmall] {};

\path (l33)--++(6,0) node{Level \( 3 \)}--++(0,5) node{Level \( 2 \)}--++(0,5) node{Level \( 1 \)};

\draw (l1.-50) --+(0,-1.9)
      (l1.-128)--+(0,-1.9);
\draw (l21.-52) --+(0,-1.9) 
      (l21.-128)--+(0,-1.9);
\draw (l22.-52) --+(0,-1.9) 
      (l22.-128)--+(0,-1.9);

\draw ( l1.130)--+(0,2) node[dot]{};
\draw (l21.130)--+(0,2) node[dot]{};
\draw (l31.130)--+(0,2) node[dot]{};
\draw ( l1.50)--+(0,2) node[dot]{};
\draw (l22.50)--+(0,2) node[dot]{};
\draw (l33.50)--+(0,2) node[dot]{};
\draw (l31.-130)--++(0,-2) node[dot]{} node(vi0)[terminal][label={[label distance=-5pt]above left:\( v_{i,0} \)}]{};
\draw (l31. -50)--++(0,-2) node[dot]{} node(vi1)[terminal][label={[label distance=-5pt]above left:\( v_{i,1} \)}]{};
\draw (l32.-130)--++(0,-2) node[dot]{} node(vi2)[terminal][label={[label distance=-5pt]above left:\( v_{i,2} \)}]{};
\draw (l32. -50)--++(0,-2) node[dot]{} node(vi3)[terminal][label={[label distance=-5pt]above left:\( v_{i,3} \)}]{};
\draw (l33.-130)--++(0,-2) node[dot]{} node(vi4)[terminal][label={[label distance=-5pt]above left:\( v_{i,4} \)}]{};
\draw (l33. -50)--++(0,-2) node[dot]{} node(vi5)[terminal][label={[label distance=-5pt]above left:\( v_{i,5} \)}]{};
\draw [dashed]
(vi1)--+(0,-2)
(vi2)--+(0,-2)
(vi3)--+(0,-2)
(vi4)--+(0,-2);
\end{tikzpicture}
\caption{Replacement of vertex by vertex gadget.}
\label{fig:5-star vertex replacement}
\end{figure}

Next, replace each edge \( v_{i,k}\,v_{j,\ell} \) between terminals by a not-equal gadget between \( v_{i,k} \) and \( v_{j,\ell} \) (that is, introduce a not-equal gadget, identify one terminal of the gadget with vertex \( v_{i,k} \) and identify the other terminal with the vertex \( v_{j,\ell} \)). 
Next, introduce two chain gadgets. The chain gadget is displayed in Figure~\ref{fig:5-star chain gadget}. 

\begin{figure}[hbt]
\centering
\begin{tikzpicture}[scale=0.33]
\path (0,0)--+(10,6) node(l1)[22boxsmall] {};
\path (l1)--+ (-2.5,-5) node(l21)[22boxsmall] {}
      (l1)--+ (2.5,-5) node(l22)[22boxsmall] {};
\path (l21)--+ (-2.5,-5) node(l31)[22boxsmall,transparent] {}
      (l21)--+ (2.5,-5) node(l32)[22boxsmall,transparent] {} node[font=\Large]{\vdots};
\path (l22)--+ (2.5,-5) node(l33)[22boxsmall,transparent] {};
\path (l31)--+ (-2.5,-5) node(l41)[22boxsmall] {}
      (l31)--++ (2.5,-5) node(l42)[22boxsmall,transparent] {} node{\dots};
\path (l33)--+ (-2.5,-5) node(l43)[22boxsmall] {}
      (l33)--+ (2.5,-5) node(l44)[22boxsmall] {};

\path (l44)--++(7,0) node{Level \( \lceil\frac{n+1}{2}\rceil \)}--++(0,10) node{Level \( 2 \)}--++(0,5) node{Level \( 1 \)};

\draw (l1.-50) --+(0,-1.9)
      (l1.-128)--+(0,-1.9);
\draw (l21.-52) --+(0,-1.9) 
      (l21.-128)--+(0,-1.9);
\draw (l22.-52) --+(0,-1.9) 
      (l22.-128)--+(0,-1.9);
\draw (l31.-128)--+(0,-1.9);
\draw (l32.-52) --+(0,-1.9);
\draw (l33.-52) --+(0,-1.9) 
      (l33.-128)--+(0,-1.9);

\draw ( l1.130)--+(0,2) node[dot]{};
\draw (l21.130)--+(0,2) node[dot]{};
\draw (l41.130)--+(0,2) node[dot]{};
\draw ( l1.50)--+(0,2) node[dot]{};
\draw (l22.50)--+(0,2) node[dot]{};
\draw (l44.50)--+(0,2) node[dot]{};
\draw (l41.-130)--++(0,-2) node[dot]{} node(vi0)[terminal][label=below:\( v_{1,t}^* \)]{};
\draw (l41. -50)--++(0,-2) node[dot]{} node(vi1)[terminal][label=below:\( v_{2,t}^* \)]{};
\draw (l43.-130)--++(0,-2) node[dot]{} node(vi4)[terminal][label=below:\( v_{n-1,t}^* \)]{};
\draw (l43. -50)--++(0,-2) node[dot]{} node(vi5)[terminal][label=below:\( v_{n,t}^* \)]{};
\draw (l44.-130)--++(0,-2) node[dot]{} node(vi4)[terminal][label=below:\( x_{1,t} \)]{};
\draw (l44. -50)--++(0,-2) node[dot]{} node(vi5)[terminal][label=below:\( x_{2,t} \)]{};
\end{tikzpicture}
\caption[\( t \)-th chain gadget in Construction~\ref{make:5-star max-deg 4}.]{\( t \)-th chain gadget in Construction~\ref{make:5-star max-deg 4} if \( n \) is even, where \( t=1 \text{ or } 2 \). 
If \( n \) is odd, the \( t \)-th chain gadget is the same except that it has only \( n+1 \) terminals \( v_{1,t}^*,v_{2,t}^*,\dots,v_{n,t}^* \) and \( x_{1,t} \). 
A chain gadget is similar to a vertex gadget; the only difference is that it has more levels and terminals.}
\label{fig:5-star chain gadget}
\end{figure}
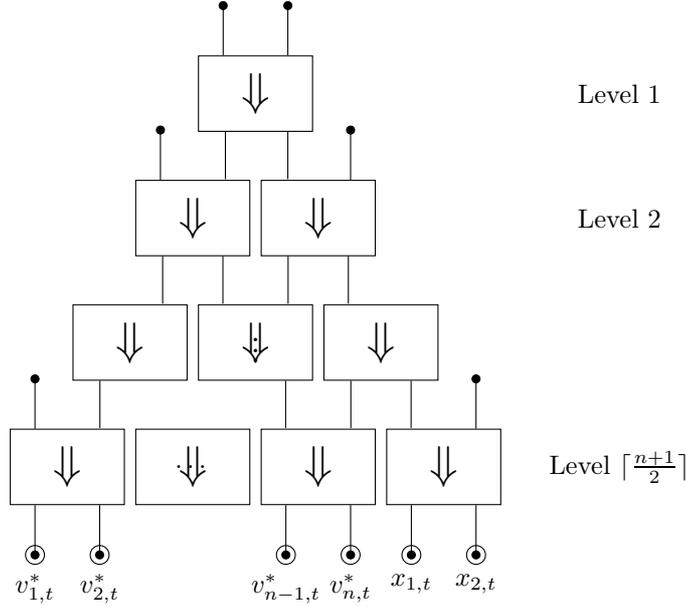

Next, add a not-equal gadget between \( v_{i,0} \) and \( v_{i,1}^* \) for \( 1\leq i\leq n \). 
Similarly, introduce a not-equal gadget between \( v_{i,5} \) and \( v_{i,2}^* \) for \( 1\leq i\leq n \). 
Finally, add a not-equal gadget between \( x_{1,1} \) and \( x_{1,2} \). 
\end{construct}
\begin{proof}[Proof of Guarantee]
For convenience, let us call the edges \( y_1y_1^*, y_2y_2^* \) of a 2-in-2-out gadget (see Figure~\ref{fig:2-in-2-out gadget}) as in-edges of the 2-in-2-out gadget, edges \( z_1z_1^*, z_2z_2^* \) as out-edges of the 2-in-2-out-gadget, vertices \( y_1^*,y_2^* \) as in-vertices of the 2-in-2-out gadget, and vertices \( z_1^*,z_2^* \) as out-vertices of the 2-in-2-out gadget. 

The next claim follows from Lemma~\ref{lem:2-in-2-out gadget}. \setcounter{claim}{0}
\begin{claim}\label{cl:chain of 2-in-2-out gadgets}
If an in-edge of a 2-in-2-out gadget is an out-edge of another 2-in-2-out gadget, the colour of the out-vertices of both gadgets must be the same. \end{claim}
\noindent Next, we point out a property of the vertex gadget and the chain gadget. \begin{claim}\label{cl:5-star gadget terminals}
All terminals of a vertex gadget (resp.\ chain gadget) should get the same colour under a 5-star colouring. \end{claim}

By Claim~\ref{cl:chain of 2-in-2-out gadgets}, if an in-edge of a 2-in-2-out gadget is an out-edge of another 2-in-2-out gadget, the colour of out-vertices of both gadgets must be the same. 
Repeated application of this idea proves Claim~\ref{cl:5-star gadget terminals} (see supplement for a detailed proof).

We are now ready to prove the guarantee. 
Suppose that \( G \) admits a 3-colouring \( f\colon V(G)\to\{2,3,4\} \). 
A 5-colouring \( f':V(G')\to\{0,1,2,3,4\} \) of \( G' \) is constructed as follows. 
First, assign \( f'(v_{i,j})=f(v_i) \) for \( 1\leq i\leq n \) and \( 0\leq j\leq 5 \). 
Extend this into a 5-star colouring of the vertex gadget by using the scheme in Figure~\ref{fig:5-star colouring 2-in-2-out gadget} on each 2-in-2-out-gadget within the vertex gadget (use the scheme in Figure~\ref{fig:5-star colouring 2-in-2-out gadget} if \( f'(v_{i,j})=4 \); suitably swap  colours in other cases). 
To colour the first chain gadget, colour each 2-in-2-out gadget within this chain gadget using the scheme obtained from Figure~\ref{fig:5-star colouring 2-in-2-out gadget} by swapping colour~\( 4 \) with colour~\( 0 \). 
Similarly, for the second chain gadget, colour each 2-in-2-out gadget within the chain gadget using the scheme obtained from Figure~\ref{fig:5-star colouring 2-in-2-out gadget} by swapping colour~\( 4 \) with colour~\( 1 \). 
To complete the colouring, it suffices to extend the partial colouring to not-equal gadgets. 
For each not-equal gadget between two terminals, say terminal \( y \) and terminal \( z \), colour the 2-in-2-out gadget within the not-equal gadget using the scheme obtained from Figure~\ref{fig:5-star colouring 2-in-2-out gadget} by swapping colour~\( 3 \) with colour~\( f'(y) \) and swapping colour~\( 4 \) with colour~\( f'(z) \).

By Lemma~\ref{lem:2-in-2-out gadget} and Lemma~\ref{lem:not-equal gadget} (see the second statements in both lemmas), every 3-vertex path in any gadget in \( G' \) containing a terminal of the gadget as an endpoint is tricoloured by \( f' \). 
In addition, the construction of the graph \( G' \) is merely glueing together terminals of different gadgets. 
Therefore, there is no \( P_4 \) in \( G' \) bicoloured by \( f' \); that is, \( f' \) is a 5-star colouring of \( G' \).

Conversely, suppose that \( G' \) admits a 5-star colouring \( f'\colon V(G')\to\{0,1,2,3,4\} \). 
By Claim~\ref{cl:5-star gadget terminals}, all terminals of a vertex/chain gadget should have the same colour under \( f' \). 
As there is a not-equal gadget between \( x_{1,1} \) and \( x_{1,2} \), \( f'(x_{1,1})\neq f'(x_{1,2}) \) (by Lemma~\ref{lem:not-equal gadget}). 
Without loss of generality, assume that \( f'(x_{1,1})=0 \) and \( f'(x_{1,2})=1 \). 
By Claim~\ref{cl:5-star gadget terminals}, all terminals of the first chain gadget have colour~0; that is, \( f'(x_{1,1})=f(x_{2,1})=f'(v_{i,1}^*)=0 \) for \( 1\leq i\leq n \). 
Similarly, all terminals of the second chain gadget have colour~1; that is, \( f'(x_{1,2})=f'(x_{2,2})=f'(v_{i,2}^*)=1 \) for \( 1\leq i\leq n \). 
By Claim~\ref{cl:5-star gadget terminals}, all terminals of the vertex gadget for \( v_1 \) have the same colour under \( f' \), say colour~\( c \). 
Since there is a not-equal gadget between \( v_{1,0} \) and \( v_{1,1}^* \), we have \( c=f'(v_{1,0})\neq f'(v_{1,1}^*)=0 \).
Since there is a not-equal gadget between \( v_{1,5} \) and \( v_{1,2}^* \), we have \( c=f'(v_{1,5})\neq f'(v_{1,2}^*)=1 \).
So, \( c\in\{2,3,4\} \). 
Hence, for \( 0\leq j\leq 5 \), \( f'(v_{1,j})\in\{2,3,4\} \). 
Similarly, for \( 1\leq i\leq n \) and \( 0\leq j\leq 5 \), \( f'(v_{i,j})\in\{2,3,4\} \). 
Moreover, whenever \( v_iv_j \) is an edge in \( G \), there is a not-equal gadget between terminals \( v_{i,k} \) and \( v_{j,\ell} \) in \( G' \) for some \( k,\ell\in\{1,2,3,4\} \) and hence \( f'(v_{i,k})\neq f'(v_{j,\ell}) \). 
Therefore, the function \( f\colon V(G)\to\{2,3,4\} \) defined as \( f(v_i)=f'(v_{i,0}) \) is indeed a 3-colouring of \( G \). 
This proves the converse part and thus the guarantee. 
\end{proof}

\begin{theorem}\label{thm:5-star colouring npc max-deg 4}
\textsc{5-Star Colourability} is NP-complete for triangle-free graphs of maximum degree four. 
\end{theorem}
\begin{proof}
We employ Construction~\ref{make:5-star max-deg 4} to establish a reduction from \textsc{3\nobreakdash-Colourability}\allowbreak(4\nobreakdash-regular) to \textsc{5\nobreakdash-Star Colourability}(triangle-free, \( \Delta=4 \)). 
Let \( G \) be an instance of \textsc{3-Colourability}(4-regular). 
From \( G \), construct an instance \( G' \) of \textsc{5-Star Colourability}(triangle-free, \( \Delta=4 \)) by Construction~\ref{make:5-star max-deg 4}.

Let \( m=|E(G)| \) and \( n=|V(G)| \). 
In \( G' \), there are at most \( 6n+m+2(1+2+\dots+\ceil{(n+1)/2}+n)+1\leq \frac{1}{4}(n^2+46n+12) \) 2-in-2-out gadgets and in addition at most \( 16n+8 \) vertices and \( 32n+12 \) edges. 
So, \( G' \) can be constructed in time polynomial in~\( n \). 
By the guarantee in Construction~\ref{make:5-star max-deg 4}, \( G \) is 3-colourable if and only if \( G' \) is 5-star colourable.
\end{proof}

We have the following theorem by combining Theorem~\ref{thm:k-star colouring deg k-1 k>=7} and Theorem~\ref{thm:5-star colouring npc max-deg 4}. 
\begin{theorem}\label{thm:k-star colouring deg k-1}
For \( k=5 \) and \( k\geq 7 \), \textsc{\( k \)-Star Colourability} is NP-complete for graphs of maximum degree \( k-1 \). 
\qed 
\end{theorem}
The complexity status of \textsc{\( k \)-Star Colourability} in graphs of maximum degree \( k-1 \) is open for \( k=4 \) and \( k=6 \). 

Next, let us shift our attention to regular graphs. 
We prove that for all \( k\geq 3 \) and \( d<k \), the complexity of \textsc{\( k \)-Star Colourability} is the same for graphs of maximum degree \( d \) and \( d \)-regular graphs. That is, for all \( k\geq 3 \) and \( d<k \), \textsc{\( k \)-Star Colourability} restricted to graphs of maximum degree \( d \) is in P (resp.\ NP-complete) if and only if \textsc{\( k \)-Star Colourability} restricted to \( d \)-regular graphs is in P (resp.\ NP-complete). 
First, we show that for all \( k\geq 3 \), the complexity of \textsc{\( k \)-Star Colourability} is the same for graphs of maximum degree \( k-1 \) and \( (k-1) \)-regular graphs.

\begin{construct}\label{make:star colouring degree k-1 vs k-1 regular}
\emph{Parameter:} An integer \( k\geq 3 \).\\ \emph{Input:} A graph \( G \) of maximum degree \( k-1 \).\\ 
\emph{Output:} A \( (k-1) \)-regular graph \( G' \).\\
\emph{Guarantee~1:} \( G \) is \( k \)-star colourable if and only if \( G' \) is \( k \)-star colourable.\\
\emph{Guarantee~2:} If \( G \) is triangle-free (resp.\ bipartite), then \( G' \) is triangle-free (resp.\ bipartite).\\
\emph{Steps:}\\
Introduce two copies of \( G \). 
For each vertex \( v \) of \( G \), introduce \( (k-1)-deg_G(v) \) filler gadgets (see Figure~\ref{fig:filler gadget d=k-1}) between the two copies of \( v \).

\begin{figure}[hbt]
\centering
\begin{tikzpicture}[scale=0.85]
\path (0,0) node(1st)[dot]{};

\path (1st) ++(1,1.5) node(x1)[dot]{} ++(0,-1) node(x2)[dot]{} ++(0,-1.5) node(xk-2)[dot]{};
\path (x2)--node[sloped,font=\large]{\( \dots \)} (xk-2);
\path (x1) ++(2,0) node(y1)[dot]{} ++(0,-1) node(y2)[dot]{} ++(0,-1.5) node(yk-2)[dot]{};
\path (y2)--node[sloped,font=\large]{\( \dots \)} (yk-2);
\path (y1) ++(1,-1.5) node(1stEnd)[dot]{};

\draw (1st)--(x1)  (1st)--(x2)  (1st)--(xk-2);
\draw (x1)--(y1)  (x1)--(y2)  (x1)--(yk-2);
\draw (x2)--(y1)  (x2)--(y2)  (x2)--(yk-2);
\draw (xk-2)--(y1)  (xk-2)--(y2)  (xk-2)--(yk-2);
\draw (1stEnd)--(y1)  (1stEnd)--(y2)  (1stEnd)--(yk-2);

\path (x1)--+(0,0.15) coordinate(brace1Start);
\path (xk-2)--+(0,-0.15) coordinate(brace1End);
\draw [opacity=0.4,decorate,decoration={brace,aspect=0.75,amplitude=16pt,raise=3pt,mirror}] (brace1Start)--  node[left=7mm,yshift=-6mm,opacity=1]{\( k-2 \)} (brace1End);

\draw 
(1st)--+(-1.5,0) node[dot]{} node[terminal][label=left:\( v \)]{}
(1stEnd)--+(1.5,0) node[dot]{} node[terminal][label=right:\( v \)]{};
\end{tikzpicture}
\caption{A filler gadget for \( v\in V(G) \).}
\label{fig:filler gadget d=k-1}
\end{figure}
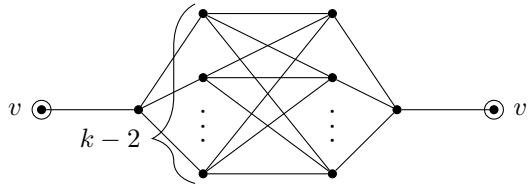
\end{construct}

\begin{proof}[Proof of Guarantee~1]
If \( G' \) is \( k \)-star colourable, then \( G \) is \( k \)-star colourable because \( G \) is a subgraph of \( G' \). 
Conversely, suppose that \( G \) admits a \( k \)-star colouring \( f\colon V(G)\to\{0,1,\dots,k-1\} \). 
We produce a \( k \)-colouring \( f' \) of \( G' \) as follows. 
Colour both copies of \( G \) by \( f \) (i.e., for each \( v\in V(G) \), assign the colour \( f(v) \) to both copies of \( v \)). 
For each vertex \( v \) of \( G \), consider each filler gadget for \( v \) one by one and do the following for each filler gadget under consideration: (i)~choose a colour \( c \) not yet used in the closed neighbourhood of \( v \) in \( G' \) (the scheme we employ ensures that colours already used in the closed neighbourhood are exactly the same for the first copy of \( v \) and the second copy of \( v \)), 
and (ii)~colour the filler gadget by the \( k \)-star colouring scheme obtained from Figure~\ref{fig:scheme for filler gadget d=k-1} by swapping colours in filler gadget suitably so that copies of \( v \) get colour~\( f(v) \) and their neighbours in the filler gadget get colour~\( c \). 
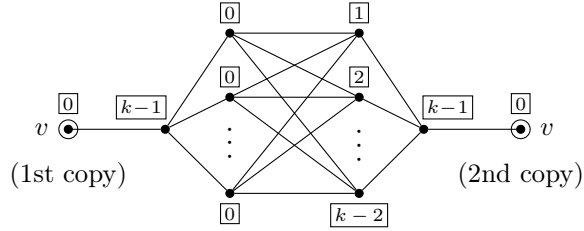
\begin{figure}[hbt]
\centering
\begin{tikzpicture}[scale=0.85]
\path (0,0) node(1st)[dot][label={[vcolour,xshift=-9pt]above:\( k\!-\!1 \)}]{}; 

\path (1st) ++(1,1.5) node(x1)[dot][label={[vcolour]above:\( 0 \)}]{} ++(0,-1) node(x2)[dot][label={[vcolour]above:\( 0 \)}]{} ++(0,-1.5) node(xk-2)[dot][label={[vcolour]below:\( 0 \)}]{};
\path (x2) -- node[sloped,font=\large]{\dots} (xk-2);
\path (x1) ++(2,0) node(y1)[dot][label={[vcolour]above:\( 1 \)}]{} ++(0,-1) node(y2)[dot][label={[vcolour]above:\( 2 \)}]{} ++(0,-1.5) node(yk-2)[dot][label={[vcolour]below:\( k-2 \)}]{};
\path (y2)--node[sloped,font=\large]{\( \dots \)} (yk-2);
\path (y1) ++(1,-1.5) node(1stEnd)[dot][label={[vcolour,xshift=9pt]above:\( k\!-\!1 \)}]{};

\draw (1st)--(x1)  (1st)--(x2)  (1st)--(xk-2);
\draw (x1)--(y1)  (x1)--(y2)  (x1)--(yk-2);
\draw (x2)--(y1)  (x2)--(y2)  (x2)--(yk-2);
\draw (xk-2)--(y1)  (xk-2)--(y2)  (xk-2)--(yk-2);
\draw (1stEnd)--(y1)  (1stEnd)--(y2)  (1stEnd)--(yk-2);

\draw 
(1st)--+(-1.5,0) node[dot]{} node[terminal][label=left:\( v \)][label={[vcolour,solid]above:\( 0 \)}][label={[label distance=5pt]below:(1st copy)}]{}
(1stEnd)--+(1.5,0) node[dot]{} node[terminal][label=right:\( v \)][label={[vcolour,solid]above:\( 0 \)}][label={[label distance=5pt]below:(2nd copy)}]{};
\end{tikzpicture}
\caption[A \( k \)-star colouring scheme for the filler gadget.]{A \( k \)-star colouring scheme for the filler gadget for \( v \) (also, swap colour~0 with \( f(v) \) and colour~\( k-1 \) with the chosen colour~\( c \)).}
\label{fig:scheme for filler gadget d=k-1}
\end{figure}
\begin{figure}[hbt]
\centering
\begin{subfigure}[b]{0.2\textwidth}
\centering
\begin{tikzpicture}[scale=0.85]
\draw (0,0) node(u1)[dot][label={[vcolour]above:\( 1 \)}]{} --++(1,-1) node(w)[dot][label={[vcolour]right:\( 2 \)}]{} --++(-1,-1) node(y)[dot][label={[vcolour,xshift=-2pt]left:\( 1 \)}]{} --++(-1,1) node(x)[dot][label={[vcolour]left:\( 0 \)}]{} -- (u1);
\draw (w) -- (x);
\draw (y) --+(0,-1) node(v1)[dot][label={[vcolour]below:\( 3 \)}]{};
\end{tikzpicture}
\caption{\( G \)}
\end{subfigure}\begin{subfigure}[b]{0.8\textwidth}
\centering
\begin{tikzpicture}[scale=0.85]
\draw (0,0) node(u1)[dot]{} node[terminal][label={[vcolour]above:\( 1 \)}]{} --++(1,-1) node(w)[dot][label={[vcolour]right:\( 2 \)}]{} --++(-1,-1) node(y)[dot][label={[vcolour,xshift=-2pt]left:\( 1 \)}]{} --++(-1,1) node(x)[dot][label={[vcolour]left:\( 0 \)}]{} -- (u1);
\draw (w) -- (x);
\draw (y) --+(0,-1) node(v1)[dot]{} node[terminal][label={[vcolour]below:\( 3 \)}]{};

\draw (6,0) node(u2)[dot]{} node[terminal][label={[vcolour]above:\( 1 \)}]{} --++(1,-1) node(w)[dot][label={[vcolour]right:\( 2 \)}]{} --++(-1,-1) node(y)[dot][label={[vcolour,xshift=2pt]right:\( 1 \)}]{} --++(-1,1) node(x)[dot][label={[vcolour]left:\( 0 \)}]{} -- (u2);
\draw (w) -- (x);
\draw (y) --+(0,-1) node(v2)[dot]{} node[terminal][label={[vcolour]below:\( 3 \)}]{};

\draw (u1) --++(1.5,0) node(1st)[dot][label={[vcolour]above:\( 3 \)}]{}; 

\path (1st) ++(1,0.5) node(x1)[dot][label={[vcolour]above:\( 1 \)}]{} ++(0,-1) node(x2)[dot][label={[vcolour]above:\( 1 \)}]{};
\path (x1) ++(1,0) node(y1)[dot][label={[vcolour]above:\( 0 \)}]{} ++(0,-1) node(y2)[dot][label={[vcolour]above:\( 2 \)}]{};
\path (y1) ++(1,-0.5) node(1stEnd)[dot][label={[vcolour]above:\( 3 \)}]{};

\draw (1st)--(x1)    (1st)--(x2);
\draw (x1)--(y1)     (x1)--(y2);
\draw (x2)--(y1)     (x2)--(y2);
\draw (1stEnd)--(y1) (1stEnd)--(y2);

\draw (1stEnd) -- (u2);

\draw (v1) --++(1.5,1) node(1st)[dot][label={[vcolour]above:\( 0 \)}]{}; 

\path (1st) ++(1,0.5) node(x1)[dot][label={[vcolour]above:\( 3 \)}]{} ++(0,-1) node(x2)[dot][label={[vcolour]above:\( 3 \)}]{};
\path (x1) ++(1,0) node(y1)[dot][label={[vcolour]above:\( 1 \)}]{} ++(0,-1) node(y2)[dot][label={[vcolour]above:\( 2 \)}]{};
\path (y1) ++(1,-0.5) node(1stEnd)[dot][label={[vcolour]above:\( 0 \)}]{};

\draw (1st)--(x1)    (1st)--(x2);
\draw (x1)--(y1)     (x1)--(y2);
\draw (x2)--(y1)     (x2)--(y2);
\draw (1stEnd)--(y1) (1stEnd)--(y2);

\draw (1stEnd)--(v2);

\draw (v1) --++(1.5,-1) node(1st)[dot][label={[vcolour]above:\( 2 \)}]{}; 

\path (1st) ++(1,0.5) node(x1)[dot][label={[vcolour]above:\( 3 \)}]{} ++(0,-1) node(x2)[dot][label={[vcolour]above:\( 3 \)}]{};
\path (x1) ++(1,0) node(y1)[dot][label={[vcolour]above:\( 0 \)}]{} ++(0,-1) node(y2)[dot][label={[vcolour]above:\( 1 \)}]{};
\path (y1) ++(1,-0.5) node(1stEnd)[dot][label={[vcolour]above:\( 2 \)}]{};

\draw (1st)--(x1)    (1st)--(x2);
\draw (x1)--(y1)     (x1)--(y2);
\draw (x2)--(y1)     (x2)--(y2);
\draw (1stEnd)--(y1) (1stEnd)--(y2);
  
\draw (1stEnd) -- (v2);
\end{tikzpicture}
\caption{\( G' \)}
\end{subfigure}\caption{(a)~A \( 4 \)-star colouring \( f \) of \( G \), and (b)~the corresponding \( 4 \)-star colouring \( f' \) of \( G' \).}
\label{fig:eg k-star colouring of G'}
\end{figure}
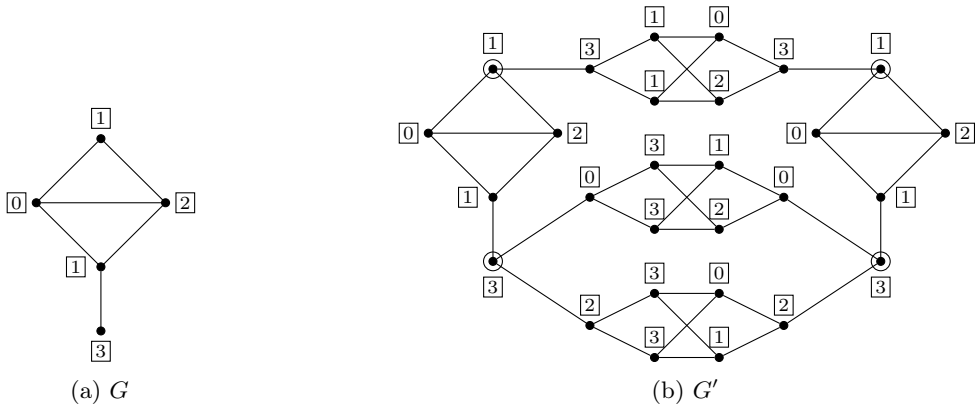

See Figure~\ref{fig:eg k-star colouring of G'} for an example. 
Clearly, \( f' \) is a \( k \)-colouring of \( G' \). 
\setcounter{claim}{0}
\begin{claim}\label{cl:f' is a star colouring}
\( f' \) is a \( k \)-star colouring of \( G' \).
\end{claim}
\noindent Assume that there is a 4-vertex path \( u,v,w,x \) in \( G' \) bicoloured by \( f' \) (i.e., \( f'(u)=f'(w) \) and \( f'(v)=f'(x) \)). 
We know that \( f' \) employs a \( k \)-star colouring scheme on both copies of \( G \) and each filler gadget. 
So, path \( u,v,w,x \) must contain vertices from one filler gadget as well as vertices from a copy of \( G \) or another filler gadget. 
In both cases, one of the two middle vertices in path \( u,v,w,x \) must be a terminal of a filler gadget. 
Suppose that \( v \) is a terminal of a filler gadget \( G_1 \), \( w \) is a vertex in \( G_1 \), and \( u \) is either in a copy of \( G \) or in another filler gadget \( G_2 \). 
If \( u \) is in another filler gadget \( G_2 \), we may assume without loss of generality that the filler gadget \( G_1 \) is coloured after the filler gadget \( G_2 \) is coloured. 
When the filler gadget \( G_1 \) was coloured, a colour not yet used in the closed neighbourhood of \( v \) in \( G' \) was chosen as the colour of \( w \); this is a contradiction to \( f(w)=f(u) \) (because \( u \) was already coloured and thus \( f(u) \) was already present in \( N_G[v] \)). 
This proves the claim by contradiction. 
This completes the proof of the converse part.
\end{proof}
\begin{proof}[Proof of Guarantee~2]
Note that the filler gadget is a bipartite graph. 
Suppose that \( G \) is triangle-free (resp.\ bipartite). 
Then, the graph with two disjoint copies of \( G \) (i.e., \( 2G \)) is also triangle-free (resp.\ bipartite). 
Moreover, for each \( v\in V(G) \), the operation of adding a filler gadget between the two copies of \( v \) preserves triangle-free property (resp.\ bipartiteness). 
\end{proof}

For \( k\geq 3 \), Construction~\ref{make:star colouring degree k-1 vs k-1 regular} establishes a reduction from \textsc{\( k \)-Star Colourability}(\( \Delta=k-1 \)) to \textsc{\( k \)-Star Colourability}(\( (k-1) \)-regular). 
Hence, for \( k\geq 3 \), if \textsc{\( k \)\nobreakdash-Star Colourability} is NP-complete for graphs of maximum degree \( k-1 \), then \textsc{\( k \)-Star Colourability} is NP-complete for \( (k-1) \)-regular graphs. 
Clearly, if \textsc{\( k \)-Star Colourability} is NP-complete for \( (k-1) \)-regular graphs, then \textsc{\( k \)-Star Colourability} is NP-complete for graphs of maximum degree \( k-1 \). 
Thus, we have the following theorem. 
\begin{theorem}\label{thm:star colouring degree k-1 vs k-1 regular}
For all \( k\geq 3 \), \textsc{\( k \)-Star Colourability} is NP-complete for graphs of maximum degree \( k-1 \) if and only if \textsc{\( k \)-Star Colourability} is NP-complete for \( (k-1) \)-regular graphs. 
In addition, for \( k\geq 3 \), \textsc{\( k \)-Star Colourability} is NP-complete for triangle-free \( ( \)resp.\ bipartite\( ) \) graphs of maximum degree \( k-1 \) if and only if \textsc{\( k \)-Star Colourability} is NP-complete for triangle-free \( ( \)resp.\ bipartite\( ) \) \( (k-1) \)-regular graphs. 
\qed
\end{theorem}
Therefore, we have the following by Theorem~\ref{thm:k-star colouring deg k-1 k>=7} and Theorem~\ref{thm:5-star colouring npc max-deg 4}. 
\begin{theorem}\label{thm:k-star colouring npc (k-1)-regular}
For \( k=5 \) and \( k\geq 7 \), \textsc{\( k \)-Star Colourability} is NP-complete for \( (k-1) \)-regular graphs. 
Moreover, \textsc{5-Star Colourability} is NP-complete for triangle-free 4-regular graphs. 
\qed 
\end{theorem}

\begin{construct}\label{make:star colouring bdd degree vs regular}
\emph{Parameters:} Integers \( k\geq 3 \) and \( d\leq k-1 \).\\
\emph{Input:} A graph \( G \) of maximum degree \( d \).\\ 
\emph{Output:} A \( d \)-regular graph \( G^* \).\\
\emph{Guarantee:} \( G\) is \( k \)-star colourable if and only if \( G^* \) is \( k \)-star colourable.\\
\emph{Steps:}\\
Introduce two copies of \( G \). 
For each vertex \( v \) of \( G \), introduce \( d-deg_G(v) \) filler gadgets (see Figure~\ref{fig:filler gadget d<k}) between the two copies of \( v \). 

\begin{figure}[hbt]
\centering
\begin{tikzpicture}[scale=0.85]
\path (0,0) node(1st)[dot]{};

\path (1st) ++(1,1.5) node(x1)[dot]{} ++(0,-1) node(x2)[dot]{} ++(0,-1.5) node(xk-2)[dot]{};
\path (x2)--node[sloped,font=\large]{\( \dots \)} (xk-2);
\path (x1) ++(2,0) node(y1)[dot]{} ++(0,-1) node(y2)[dot]{} ++(0,-1.5) node(yk-2)[dot]{};
\path (y2)--node[sloped,font=\large]{\( \dots \)} (yk-2);
\path (y1) ++(1,-1.5) node(1stEnd)[dot]{};

\draw (1st)--(x1)  (1st)--(x2)  (1st)--(xk-2);
\draw (x1)--(y1)  (x1)--(y2)  (x1)--(yk-2);
\draw (x2)--(y1)  (x2)--(y2)  (x2)--(yk-2);
\draw (xk-2)--(y1)  (xk-2)--(y2)  (xk-2)--(yk-2);
\draw (1stEnd)--(y1)  (1stEnd)--(y2)  (1stEnd)--(yk-2);

\path (x1)--+(0,0.15) coordinate(brace1Start);
\path (xk-2)--+(0,-0.15) coordinate(brace1End);
\draw [opacity=0.4,decorate,decoration={brace,aspect=0.75,amplitude=16pt,raise=3pt,mirror}] (brace1Start)--  node[left=7mm,yshift=-6mm,opacity=1]{\( d-1 \)} (brace1End);

\draw 
(1st)--+(-1,0) node[dot]{} node[terminal][label=left:\( v \)]{}
(1stEnd)--+(1,0) node[dot]{} node[terminal][label=right:\( v \)]{};
\end{tikzpicture}
\caption{A filler gadget for \( v\in V(G) \) in Construction~\ref{make:star colouring bdd degree vs regular}.}
\label{fig:filler gadget d<k}
\end{figure}
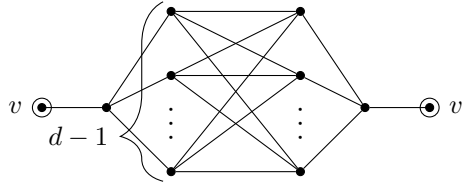
\end{construct}
To prove the guarantee, observe that \( G \) is a subgraph of \( G^* \) and \( G^* \) is a subgraph of \( G' \) (the output graph in Construction~\ref{make:star colouring degree k-1 vs k-1 regular}). \\
Thanks to Construction~\ref{make:star colouring bdd degree vs regular}, we have the following theorem. 
\begin{theorem}\label{thm:star colouring degree d vs d-regular}
For all \( k\geq 3 \) and \( d\leq k-1 \), \textsc{\( k \)-Star Colourability} is NP-complete for \( ( \)triangle-free/bipartite\( ) \) graphs of maximum degree \( d \) if and only if \textsc{\( k \)-Star Colourability} is NP-complete for \( ( \)triangle-free/bipartite\( ) \) \( d \)-regular graphs. 
\qed
\end{theorem}
We have the following corollary since \( L_s^{(k)} \) is the least integer \( d \) such that \textsc{\( k \)-Star Colourability} is NP-complete for graphs of maximum degree \( d \) (where \( k\geq 3 \)). 
\begin{corollary}
For \( k\geq 4 \) and \( d\leq k-1 \), \textsc{\( k \)\nobreakdash-Star Colourability} is NP-complete for \( d \)-regular graphs\\ if and only if \( d\geq L_s^{(k)} \). 
\qed
\end{corollary}

\subsection[On Values of \( L_s^{(k)} \) and Two Similar Parameters]{\boldmath On Values of \( L_s^{(k)} \) and Two Similar Parameters}\label{sec:star colouring points of hardness transition}
Recall that for \( k\geq 3 \), \( L_s^{(k)} \) is the least integer \( d \) such that \textsc{\( k \)-Star Colourability} in graphs of maximum degree \( d \) is NP-complete. 
Bear in mind that we assume P \( \neq \) NP throughout this paper; thus, NP is partitioned into three classes: P, NPC and NPI~\cite{paschos}. 
If a problem in NP is not NP-complete (i.e., not in NPC), then it is either in P or in NPI. 
By the definition of \( L_s^{(k)} \), \textsc{\( k \)-Star Colourability}(\( \Delta=d \)) is not NP-complete for \( d<L_s^{(k)} \), which means that the problem is either in P or in NPI (we do not know which is the case).

Clearly, the star chromatic number of a graph of maximum degree \( d \) can be computed in polynomial time if \( d\leq 2 \). 
Hence, \( L_s^{(k)}\geq 3 \) for \( k\geq 3 \). 
For \( k\geq 3 \), \textsc{\( k \)-Star Colourability} is NP-complete for graphs of maximum degree \( k \)~\cite[Theorems~10 and 16]{shalu_cyriac3}, and thus \( L_s^{(k)}\leq k \). 
Next, we show that \( L_s^{(k)}=\Omega(k^{2/3}) \) for all \( k\geq 3 \). 
\begin{observation}\label{obs:Ls_k LB}
For \( d\leq 0.33\,k^{2/3} \), \textsc{\( k \)-Star Colourability} is polynomial-time solvable for graphs of maximum degree \( d \). 
Hence, \( L_s^{(k)}>0.33\,k^{2/3} \) for all \( k\geq 3 \). 
\end{observation}
\begin{proof}
The observation is trivially true for \( d\leq 2 \).
It suffices to prove the observation for \( d\geq 3 \).
Suppose that \( d\geq 3 \).
Ndreca et al.~\cite{ndreca} proved that \( \chi_s(G)<4.34\,d^{\,3/2}+1.5\,d \) for every graph \( G \) of maximum degree \( d \).
Since \( d\geq 3 \), we have \( d^{\,1/2}\geq 3^{1/2}>1/0.58 \), and thus \( d<0.58\, d^{\,3/2} \).
Thus, \( \chi_s(G)<(4.34+1.5\times 0.58)d^{\,3/2}=5.21\,d^{\,3/2} \) for every graph \( G \) of maximum degree \( d \).
Hence, when \( k\geq 5.21\,d^{\,3/2} \), every graph of maximum degree \( d \) is \( k \)-star colourable.
In other words, if \( d\leq (5.21)^{-2/3}k^{\,2/3} \), then every graph of maximum degree \( d \) is \( k \)-star colourable.
Note that \( 0.33<(5.21)^{-2/3} \).
Hence, if \( d\leq 0.33\,k^{2/3} \), then \( d\leq (5.21)^{-2/3}k^{2/3} \).
Therefore, for \( d\leq 0.33\,k^{2/3} \), every graph of maximum degree \( d \) is \( k \)-star colourable, and thus \textsc{\( k \)-Star Colourability} is polynomial-time solvable for graphs of maximum degree \( d \).
As a result, \( L_s^{(k)}>0.33\,k^{2/3} \) for \( k\geq 3 \).
\end{proof}
Theorem~\ref{thm:k-star colouring deg k-1} proved that for \( k=5 \) and \( k\geq 7 \), \textsc{\( k \)-Star Colourability} is NP-complete for graphs of maximum degree \( k-1 \), and thus \( L_s^{(k)}\leq k-1 \).\\

Next, let us consider regular graphs. 
By Theorem~\ref{thm:k-star colouring npc (k-1)-regular}, \textsc{\( k \)-Star Colourability} is NP-complete for \( (k-1) \)-regular graphs for \( k=5 \) and \( k\geq 7 \). 
Also, \textsc{\( 4 \)-Star Colourability} is NP-complete for 4-regular graphs~\cite[Corollary~5.1]{cyriac}.

For \( d\geq 3 \), at least \mbox{\( \lceil (d+4)/2 \rceil \)} colours are required to star colour a \( d \)-regular graph \cite{shalu_cyriac3}. 
If \( k\geq 3 \) and \( d\geq 2k-3 \), then \mbox{\( \lceil (d+4)/2 \rceil>k \)}, and thus no \( d \)-regular graph is \( k \)-star colourable. 
Therefore, for \( k\geq 3 \), \textsc{\( k \)-Star Colourability} in \( d \)-regular graphs is polynomial-time solvable for each \( d\geq 2k-3 \) (because the answer is always `no'). By Observation~\ref{obs:Ls_k LB}, for \( k\geq 3 \), \textsc{\( k \)-Star Colourability} in \( d \)-regular graphs is polynomial-time solvable for \( d\leq \max \{2,0.33 k^{2/3}\} \). 
In particular, \textsc{3-Star Colourability} in \( d \)-regular graphs is polynomial-time solvable for all \( d\in \mathbb{N} \). 
In contrast, for \( k\in \{4,5,7,8,\dots \} \), there exists an integer \( d \) such that \textsc{\( k \)-Star Colourability} in \( d \)-regular graphs is NP-complete (see the last paragraph). 
Hence, for \( k\in \{4,5,7,8,\dots \} \), we are interested in the least (resp.\ highest) integer \( d \) such that \textsc{\( k \)-Star Colourability} in \( d \)-regular graphs is NP-complete, and we denote it by \( \widetilde{L}_s^{(k)} \) (resp.\ \( \widetilde{H}_s^{(k)} \)). 
By the definitions, \( L_s^{(k)}\leq \widetilde{L}_s^{(k)} \leq \widetilde{H}_s^{(k)} \) for \( k\in \{4,5,7,8,\dots \} \). 
We have \( \widetilde{L}_s^{(4)}\leq 4 \) since \textsc{\( 4 \)-Star Colourability} is NP-complete for \( 4 \)-regular graphs~\cite{cyriac}. 
Similarly, for \( k=5 \) and \( k\geq 7 \), \( \widetilde{L}_s^{(k)}\leq k-1 \) since \textsc{\( k \)-Star Colourability} is NP-complete for \( (k-1) \)-regular graphs (see Theorem~\ref{thm:k-star colouring npc (k-1)-regular}).

Theorem~\ref{thm:star colouring degree d vs d-regular} proved that for \( k\geq 3 \) and \( d\leq k-1 \), \textsc{\( k \)-Star Colourability} in graphs of maximum degree \( d \) is NP-complete if and only if \textsc{\( k \)-Star Colourability} in \( d \)-regular graphs is NP-complete. 
By the definition of \( L_s^{(k)} \), for \( k\geq 3 \), \textsc{\( k \)-Star Colourability} in graphs of maximum degree \( d \) is NP-complete for \( d=L_s^{(k)} \), and not NP-complete for \( d<L_s^{(k)} \). 
Hence, for \( k=5 \) and \( k\geq 7 \), \textsc{\( k \)-Star Colourability} in \( d \)-regular graphs is NP-complete for \( d=L_s^{(k)} \), and not NP-complete for \( d<L_s^{(k)} \) by Theorem~\ref{thm:star colouring degree d vs d-regular} (applicable because \( d\leq L_s^{(k)}\leq k-1 \)). 
This proves that for \( k=5 \) and \( k\geq 7 \), \( L_s^{(k)} \) is the least integer \( d \) such that \textsc{\( k \)-Star Colourability} in \( d \)-regular graphs is NP-complete; that is, \( \widetilde{L}_s^{(k)}=L_s^{(k)} \). 
\begin{theorem}
For \( k=5 \) and \( k\geq 7 \), we have \( \widetilde{L}_s^{(k)}=L_s^{(k)} \). 
\qed 
\end{theorem}

As mentioned above, for \( k\geq 3 \), \textsc{\( k \)-Star Colourability} in \( d \)-regular graphs is polynomial-time solvable for each \( d\geq 2k-3 \). 
Hence, for \( k\in \{4,5,7,8,\dots \} \), we have \( \widetilde{H}_s^{(k)}\leq 2k-4 \), and the same bound holds whenever \( \widetilde{H}_s^{(k)} \) can be defined (i.e, \( \exists d\in \mathbb{N} \), \textsc{\( k \)-Star Colourability}(\( d \)-regular) \( \in \) NPC). 
For \( k=5 \) and \( k\geq 7 \), \textsc{\( k \)-Star Colourability} is NP-complete for \( (k-1) \)-regular graphs by Theorem~\ref{thm:k-star colouring npc (k-1)-regular}, and thus \( L_s^{(k)}=\widetilde{L}_s^{(k)} \leq k-1\leq \widetilde{H}_s^{(k)}\leq 2k-4 \). 

See the concluding section (Section~\ref{sec:conclusion}) for a discussion of the open problems.

\section{Restricted Star Colouring}\label{sec:rs colouring}

\providetoggle{forThesis} 

\providetoggle{extended} 

\iftoggle{forThesis}
{ \iftoggle{extended}
  { \subsection{Reducing the Maximum Degree in the Hardness Result}
  } { \subsection{Hardness Transitions}\label{sec:rs colouring hardness transitions}
In this section, we discuss hardness transitions of rs colouring with respect to the maximum degree. 
To be specific, we deal with the values \( \widetilde{L}_{rs}^{(k)} \), \( L_{rs}^{(k)} \) and \( H_{rs}^{(k)} \). 
We show that \( \widetilde{L}_{rs}^{(3)}=3 \) whereas \( L_{rs}^{(3)} \) and \( H_{rs}^{(3)} \) are undefined (see Section~\ref{sec:rs colouring points of hardness transition} for details). 
For \( k\geq 4 \), we prove that \( \widetilde{L}_{rs}^{(k)}\leq k-1 \), \( L_{rs}^{(k)}=\widetilde{L}_{rs}^{(k)} \) and \( H_{rs}^{(k)}=k-1 \). 
To this end, we show that for \( k\geq 4 \), (i)~\textsc{\( k \)-RS Colourability} is NP-complete for graphs of maximum degree \( k-1 \), and (ii)~for each \( d<k \), \textsc{\( k \)-RS Colourability} in graphs of maximum degree \( d \) is NP-complete if and only if \textsc{\( k \)-RS Colourability} in \( d \)-regular graphs is NP-complete. 
Let us discuss results (i) and (ii). 
Consequences of these results to the values of \( \widetilde{L}_{rs}^{(k)} \), \( L_{rs}^{(k)} \) and \( H_{rs}^{(k)} \) are discussed later in Section~\ref{sec:rs colouring points of hardness transition}.

First, we prove that for all \( k\geq 4 \), \textsc{\( k \)-RS Colourability} is NP-complete for graphs of maximum degree \( k-1 \). 
}\iftoggle{extended}
{ We show that for all \( k\geq 4\), \textsc{\( k \)-RS Colourability} is NP-complete for triangle-free graphs of maximum degree \( k-1 \). 
} { }First, we deal with \( k=4 \). 
In fact, we show that \textsc{4-RS Colourability} is NP-complete for planar 3-regular graphs of girth~5. 
Construction~\ref{make:4-rs colouring planar cubic} below is employed to this end. 
Construction~\ref{make:4-rs colouring planar cubic} makes use of the following observation and Construction~\ref{make:3-rs colouring planar}. 
Construction~\ref{make:3-rs colouring planar} is the construction used in Theorem~\ref{thm:k-rsc planar bipartite girth g} to show that \textsc{3-RS Colourability} is NP-complete for planar graphs of maximum degree 3.

\begin{observation}\label{obs:no colour 0 at distance 2}
Let \( f \) be an rs colouring of a graph \( G \), and let \( u \) and \( v \) be two vertices in \( G \) at distance two in \( G \). 
Then, \( f(u)\neq 0 \) or \( f(v)\neq 0 \) \( ( \)or both\( ) \). 
\end{observation}
\begin{proof}
To produce a contradiction, assume the contrary: \( f(u)=f(v)=0 \).
Let \( w \) be a common neighbour of \( u \) and \( v \).
Then, irrespective of the colour at \( w \), the path \( u,w,v \) is a bicoloured \( P_3 \) with a higher colour on its middle vertex; a contradiction.
\end{proof}

\iftoggle{extended}
{ \textcolor{red}{Write as ``Construction 1.1 (Re-stated)'' or write ``the construction used in Theorem~\ref{thm:3-rsc planar bipartite girth 6}. \{newline\} Construction 1.1''} 
} { } 

} { 

\iftoggle{forThesis}
{ } { \subsection{Introduction and Literature Survey}\label{sec:rs intro}} Restricted star colouring is a variant of star colouring as well as a generalisation of vertex ranking. 
Therefore, the restricted star chromatic number \( \chi_{rs}(G) \) of a graph \( G \) is bounded from below by the star chromatic number and bounded from above by the ranking number, better known as the treedepth~\cite{nesetril_mendez2012}. 
The treedepth is in turn bounded from above by vertex cover number plus one~\cite{gima}. 
For complete \( r \)-partite graphs and split graphs, the rs chromatic number is equal to vertex cover number plus one~\cite{fertin2004,shalu_cyriac2}.

It is easy to observe that for \( k\in \mathbb{N} \), a \( k \)-rs colourable graph is \( (k-1) \)-degenerate~\cite{almeter}, and hence no \( d \)-regular graph is \( d \)-rs colourable. 
Almeter et al.~\cite{almeter} proved that \( \chi_{rs}(G)\leq 7 \) for every subcubic graph \( G \). 
They also proved that the rs chromatic number of the hypercube \( Q_d \) is exactly \( d+1 \).
For every \( d \), there exists a graph \( G \) with maximum degree \( d \) such that \( \chi_{rs}(G)\geq \Omega(d^2/\log d) \)~\cite{almeter}. 
Karpas et al.~\cite{karpas} proved that (i)~\( \chi_{rs}(T)=O(\log n/\log \log n) \) for every tree \( T \), and this bound is tight, and (ii)~\( \chi_{rs}(G)=O(r\sqrt{n}) \) for every \( r \)-degenerate graph \( G \). 
For every \( n \), there exists a 2-degenerate 3-regular graph \( G \) with \( \chi_{rs}(G)>n^{1/3} \)~\cite{karpas}. 
Also, \( \chi_{rs}(G)=O(\log n) \) for every planar graph \( G \), and this result holds for every graph class excluding a fixed minor \cite{karpas}. 
Shalu and Sandhya \cite{shalu_sandhya} proved that \( \chi_{rs}(G)\leq 4\alpha(G) \) for every graph \( G \) of girth at least~5.

For \( k\geq 3 \), \textsc{\( k \)-RS Colourability} is NP-complete for (2-degenerate) planar bipartite graphs of maximum degree \( k \) and arbitrarily large girth~\cite{shalu_cyriac2}. 
In addition, it is NP-complete to test whether a 3-star colourable graph admits a 3-rs colouring~\cite{shalu_cyriac2}.
The optimization version of rs colouring is NP-hard to approximate within \( n^{\frac{1}{3}-\epsilon} \) for all \( \epsilon>0 \) in the class of 2-degenerate bipartite graphs~\cite{shalu_cyriac2}; in contrast, every 2-degenerate graph admits an rs colouring with \( n^{\frac{1}{2}} \) colours \cite[Theorem~6.2]{karpas}, and thus the optimization version of rs colouring is approximable within \( n^{\frac{1}{2}} \) for 2-degenerate graphs.

On the positive side, for \textsc{3-RS Colourability}, there is a linear-time algorithm for the class of trees and a polynomial-time algorithm for the class of chordal graphs in~\cite{shalu_cyriac2}. 
The complexity of \textsc{\( k \)-RS Colourability} in chordal graphs is open for \( k\geq 4 \). 
For each \( k\in \mathbb{N} \), \textsc{\( k \)-RS Colourability} can be expressed in  MSO\(_1 \) \cite{shalu_cyriac2},  and thus admits FPT algorithms with parameter either treewidth or cliquewidth by Courcelle's theorem \cite{borie,courcelle}. 
Thanks to Observation~\ref{obs:rs in terms of OIH}, \textsc{\( k \)-RS Colourability} can be expressed in the Locally Checkable Vertex Subset and Partitioning problems (LC-VSP) framework of Telle and Proskurowski~\cite{telle_proskurowski} (see supplement for details).  
This implies the existence of practically fast FPT algorithms for the problem~\cite{bui-xuan2010,bui-xuan2013}.

\iftoggle{forThesis}
{ } { \subsection{RS Colouring in Terms of Homomorphisms}\label{sec:rs in terms of OIH}
Let \( \vec{K_q} \) denote the tournament with vertex set \( \mathbb{Z}_q \) and edge set \( \{(i,j)\colon i,j\in \mathbb{Z}_q \text{ and } i<j\} \). 
Observe that a homomorphism \( \psi \) from an oriented graph \( \vec{H} \) to \( \vec{K_q} \) is in-neighbourhood injective if and only if no vertex \( v \) of \( \vec{H} \) has two in-neighbours \( u \) and \( w \) with \( \psi(v)>\psi(u)=\psi(w) \). 
Hence, an in-neighbourhood injective homomorphism from an orientation of a graph \( G \) to \( \vec{K_q} \) is a \( q \)-rs colouring of \( G \). 
Moreover, if \( f \) is a \( q \)-rs colouring of \( G \), then orienting each edge of \( G \) as an arc from the lower-coloured vertex to the higher-coloured vertex gives an (acyclic) orientation \( \vec{G} \) of \( G \) such that \( f \) is an in-neighbourhood injective homomorphism from \( \vec{G} \) to \( \vec{K_q} \). 
In short, a \( q \)-rs colouring of a graph \( G \) is precisely an in-neighbourhood injective homomorphism from an orientation of \( G \) to \( \vec{K_q} \). 
Thus, we have the following (since every transitive tournament on \( q \) vertices is isomorphic to \( \vec{K_q} \) as a digraph). 
\begin{observation}\label{obs:rs in terms of OIH}
A graph \( G \) admits a \( q \)-rs colouring if and only if \( G \) has an orientation that admits an in-neighbourhood injective homomorphism to a transitive tournament on \( q \) vertices. 
\qed
\end{observation}
To study minor-closed classes, Ne\v{s}et\v{r}il and Mendez~\cite{nesetril_mendez2006} introduced a generalisation of in-neighbourhood injective homomorphism, called folding. 
The complexity of in-neighbourhood injective homomorphisms to (reflexive) tournaments is studied by MacGillivray and Swarts~\cite{macgillivray}. 
Given an orientation \( \vec{G} \) of a graph \( G \), one can test in polynomial time whether \( \vec{G} \) admits an in-neighbourhood injective homomorphism to \( \vec{K_3} \)~\cite{macgillivray}. 
On the other hand, it is NP-complete to test whether an input graph \( G \) has an orientation that admits an in-neighbourhood injective homomorphism to \( \vec{K_3} \) (by Observation~\ref{obs:rs in terms of OIH} and \cite[Theorem~1]{shalu_cyriac2}).

\subsection{Hardness Transitions} \label{sec:rs colouring hardness transitions}
} For all \( k\geq 3 \), \( k \)-\textsc{RS Colourability} is NP-complete for graphs of maximum degree \( k \)~\cite[Theorem~3]{shalu_cyriac2}. 
In this section, we lower the maximum degree in this hardness result from \( k \) to \( k-1 \) except for \( k=3 \) (for \( k=3 \), the problem is polynomial-time solvable in graphs of maximum degree \( k-1 \)). 
We show that for all \( k\geq 4\), \textsc{\( k \)-RS Colourability} is NP-complete for triangle-free graphs of maximum degree~\( k-1 \). 
First, we prove this for \( k=4 \). 
In fact, we show that \textsc{4-RS Colourability} is NP-complete for planar 3-regular graphs of girth~5. 
Construction~\ref{make:4-rs colouring planar cubic} below is employed to this end. 
Construction~\ref{make:4-rs colouring planar cubic} makes use of the following observation and Construction~\ref{make:3-rs colouring planar}. 
Construction~\ref{make:3-rs colouring planar} was used in Theorem~1 of \cite{shalu_cyriac2} to show that \textsc{3-RS Colourability} is NP-complete for planar graphs of maximum degree 3.

\begin{observation}\label{obs:no colour 0 at distance 2}
Let \( f \) be an rs colouring of a graph \( G \). 
If \( u \) and \( v \) be two vertices in \( G \) that are within distance two in \( G \), then \( f(u)\neq 0 \) or \( f(v)\neq 0 \) \( ( \)or both\( ) \). 
\end{observation}
} 

\iftoggle{forThesis}
{  \begin{construct}\label{make:3-rs colouring planar}
} { \begin{construct}[\cite{shalu_cyriac2}]\label{make:3-rs colouring planar}
} \emph{Input:} A positive boolean formula \( B=(X,C) \) such that the graph of \( B \) is a planar 3-regular graph.\\
\emph{Output:} A planar graph \( G \) of maximum degree 3 and girth 6.\\
\emph{Guarantee~\cite{shalu_cyriac2}:} \( B \) has a 1-in-3 satisfying truth assignment if and only if \( G \) is 3-rs colourable.\\
\emph{Steps:}\\
Let \( X=\{x_1,x_2,\dots,x_m\} \) and \( C=\{c_1,c_2,\dots,c_m\} \) (note that \( |X|=|C| \) since the graph of \( B \) is 3-regular). 
Since \( B \) is a positive formula, each clause \( c_j \) is a 3-element subset of \( X \). 
Recall that the graph of \( B \), denoted by \( G_B \), is the graph with vertex set \( X\cup C \) and edge set \( \{ x_ic_j\ :\ x_i \in c_j\} \). 
To construct \( G \) from \( G_B \), first replace each vertex \( c_j \) of \( G_B \) by a triangle \( (c_{j1},c_{j2},c_{j3}) \), and then subdivide every edge of the resultant graph exactly once (see Figure~\ref{fig:3-rs colouring planar} for an example). 
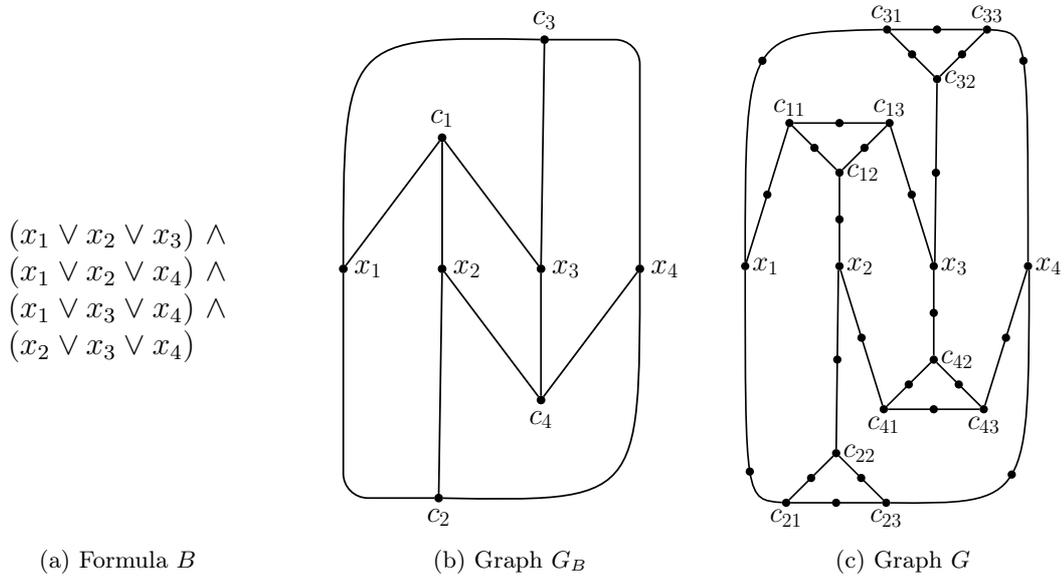
\begin{figure}[hbt]
\centering
\begin{subfigure}[b]{0.33\textwidth}
\centering
\large
\( (x_1\vee x_2\vee x_3) \) \( \wedge \)\\
\( (x_1\vee x_2\vee x_4) \) \( \wedge \)\\
\( (x_1\vee x_3 \vee x_4) \) \( \wedge \)\\
\( (x_2\vee x_3 \vee x_4) \)\phantom{ \( \wedge \)}\\[2.25cm]
\caption{Formula \( B \)}
\end{subfigure}\hfill
\begin{subfigure}[b]{0.33\textwidth}
\centering
\scalebox{0.65}{
  \begin{tikzpicture}[node distance=2cm,line width=1pt]
\tikzset{
dot/.style={draw,fill,circle,inner sep = 0pt,minimum size = 3pt},
vcolour/.style={draw,inner sep=1.5pt,font=\scriptsize,label distance=2pt},
subgraph/.style={draw,ellipse,minimum width=1.75cm,minimum height=2cm},
subgraphHoriz/.style={draw,ellipse,minimum width=1.5cm,minimum height=1.25cm},
}
\tikzstyle every label=[font=\LARGE]
  \tikzstyle bigDot=[dot,minimum size=4pt]

  \node [bigDot] (x1)[label=right:\( x_1 \)]{};
  \node [bigDot] (x2)[right of=x1,node distance=2cm,label=right:\( x_2 \)]{};
  \node [bigDot] (x3)[right of=x2,node distance=2cm,label=right:\( x_3 \)]{};
  \node [bigDot] (x4)[right of=x3,node distance=2cm,label=right:\( x_4 \)]{};

\node [bigDot] (C1) [above = 2.5 of x2][label={above:\( c_1 \)}]{};
  \node [bigDot] (C2) [below = 4.5 of x2,xshift=-2pt][label={below:\( c_2 \)}]{};
  \node [bigDot] (C3) [above = 4.5 of x3,xshift=2pt][label={above:\( c_3 \)}]{};
  \node [bigDot] (C4) [below = 2.5 of x3][label={below:\( c_4 \)}]{};
  \draw [rounded corners=0.5cm]
  (x1)--(C1)
(x2)--(C1)
  (x3)--(C1)
  (x3)--(C3)
  (x4)|-(C3)
  (x1)|-(C2)
  (x2)--(C2)
  (x2)--(C4)
  (x3)--(C4)
  (x4)--(C4);
\draw
  (x1) .. controls +(0,4.75).. (C3)
  (x4) .. controls +(0,-4.75).. (C2);
  \end{tikzpicture}
}
\caption{Graph \( G_B \)}
\end{subfigure}\hfill
\begin{subfigure}[b]{0.33\textwidth} \centering
\scalebox{0.62}{
  \begin{tikzpicture}[node distance=1.5cm,label distance=-0.7mm,line width=1pt]
\tikzset{
dot/.style={draw,fill,circle,inner sep = 0pt,minimum size = 3pt},
vcolour/.style={draw,inner sep=1.5pt,font=\scriptsize,label distance=2pt},
subgraph/.style={draw,ellipse,minimum width=1.75cm,minimum height=2cm},
subgraphHoriz/.style={draw,ellipse,minimum width=1.5cm,minimum height=1.25cm},
}
  \tikzstyle every label=[font=\LARGE]
  \tikzstyle bigDot=[dot,minimum size=4pt]

  \node [bigDot] (x1)[label=right:\( x_1 \)]{};
  \node [bigDot] (x2)[right of=x1,node distance=2cm,label=right:\( x_2 \)]{};
  \node [bigDot] (x3)[right of=x2,node distance=2cm,label=right:\( x_3 \)]{};
  \node [bigDot] (x4)[right of=x3,node distance=2cm,label=right:\( x_4 \)]{};
  
  \node [bigDot] (c12)[above of=x2,label=right:\( c_{12} \),node distance=2cm]{};
  \node [bigDot] (c11)[above left of=c12,label=above:\( c_{11} \)]{};
  \node [bigDot] (c13)[above right of=c12,label=above:\( c_{13} \)]{};
  
  \node [bigDot] (c32)[above of=x3,xshift=2pt,label=right:\( c_{32} \),node distance=4cm]{};
  \node [bigDot] (c31)[above left of=c32,label=above:\( c_{31} \)]{};
  \node [bigDot] (c33)[above right of=c32,label=above:\( c_{33} \)]{};
  
  \node [bigDot] (c42)[below of=x3,label=right:\( c_{42} \),node distance=2cm]{};
  \node [bigDot] (c41)[below left of=c42,label=below:\( c_{41} \)]{};
  \node [bigDot] (c43)[below right of=c42,label=below:\( c_{43} \)]{};
  
  \node [bigDot] (c22)[below of=x2,xshift=-2pt,label=right:\( c_{22} \),node distance=4cm]{};
  \node [bigDot] (c21)[below left of=c22,label=below:\( c_{21} \)]{};
  \node [bigDot] (c23)[below right of=c22,label=below:\( c_{23} \)]{};
  
  \draw [rounded corners=0.5cm]
  (x1) -- node[bigDot](y11){} (c11)
  (x1) .. controls +(0,5).. node[bigDot](y13){} (c31)
  (x2)--node[bigDot](y21){} (c12)
  (x3.105)--node[bigDot](y31){} (c13)
  (x3.75)--node[bigDot](y33){} (c32)
  (x4) .. controls +(0,5).. node[bigDot](y43){} (c33)
  (x1) .. controls +(0,-5).. node[bigDot](y12){} (c21)
  (x2.-105)--node[bigDot](y22){} (c22)
  (x2.-75)--node[bigDot](y24){} (c41)
  (x3)--node[bigDot](y34){} (c42)
  (x4.-105)--node[bigDot](y44){} (c43)
  (x4.-75) .. controls +(0,-5).. node[bigDot](y42){} (c23);
  
  \draw
  (c11)--node[bigDot](b11)[midway]{} (c12)--node[bigDot](b12)[midway]{} (c13)--node[bigDot](b13)[midway]{} (c11)
  (c21)--node[bigDot](b21)[midway]{} (c22)--node[bigDot](b22)[midway]{} (c23)--node[bigDot](b23)[midway]{} (c21)
  (c31)--node[bigDot](b31)[midway]{} (c32)--node[bigDot](b32)[midway]{} (c33)--node[bigDot](b33)[midway]{} (c31)
  (c41)--node[bigDot](b41)[midway]{} (c42)--node[bigDot](b42)[midway]{} (c43)--node[bigDot](b43)[midway]{} (c41);
  
  \end{tikzpicture}
}
\caption{Graph \( G \)}
\end{subfigure}
\caption{Example of Construction~\ref{make:3-rs colouring planar}.}
\label{fig:3-rs colouring planar}
\end{figure}
\end{construct}

We employ Construction~\ref{make:4-rs colouring planar cubic} below to prove that \textsc{4-RS Colourability} is NP-complete. 
A gadget called \emph{colour forcing gadget} is employed in the construction. 
The graph displayed in Figure~\ref{fig:gadget component 4-rs colouring planar} is the main component of the colour forcing gadget; let us call it the gadget component.
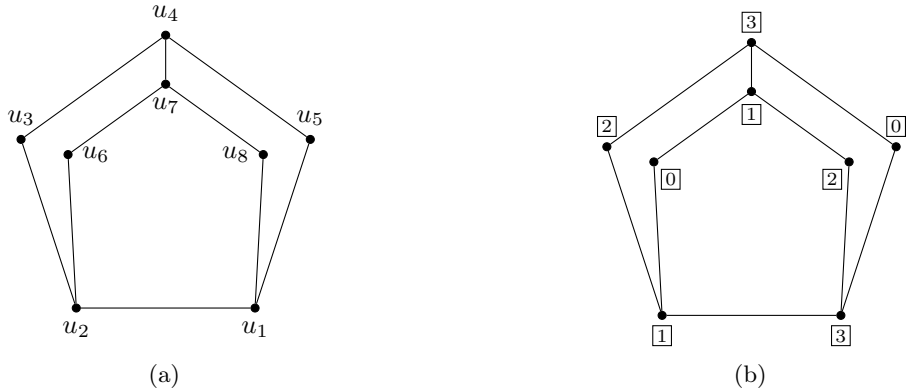
\begin{figure}[hbt] \centering
\begin{subfigure}[b]{0.5\textwidth}
\centering
\begin{tikzpicture}
\coordinate (centre){};

\path (centre) +(18:2) coordinate[dot,label=\( u_5 \)](u5)
      (centre) +(90:2) coordinate[dot,label=above:\( u_4 \)](u4)
      (centre) +(162:2) coordinate[dot,label=\( u_3 \)](u3)
      (centre) +(-54:2) coordinate[dot,label=below:\( u_1 \)](u1)
      (centre) +(-126:2) coordinate[dot,label=below:\( u_2 \)](u2);
\draw (u1)--(u2)--(u3)--(u4)--(u5)--(u1);

\path (centre) +(18:1.35) coordinate[dot,label=left:\( u_8 \)](u8)
      (centre) +(90:1.35) coordinate[dot,label=below:\( u_7 \)](u7)
      (centre) +(162:1.35) coordinate[dot,label=right:\( u_6 \)](u6);
\draw (u2)--(u6)--(u7)--(u8)--(u1);

\draw (u4)--(u7);

\end{tikzpicture}
\caption{}
\label{fig:gadget component 4-rs colouring planar}
\end{subfigure}\begin{subfigure}[b]{0.5\textwidth}
\centering
\begin{tikzpicture}
\coordinate (centre){};

\path (centre) +(18:2) coordinate[dot,label={[vcolour]\( 0 \)}](u5)
      (centre) +(90:2) coordinate[dot,label={[vcolour]above:\( 3 \)}](u4)
      (centre) +(162:2) coordinate[dot,label={[vcolour]\( 2 \)}](u3)
      (centre) +(-54:2) coordinate[dot,label={[vcolour]below:\( 3 \)}](u1)
      (centre) +(-126:2) coordinate[dot,label={[vcolour]below:\( 1 \)}](u2);
\draw (u1)--(u2)--(u3)--(u4)--(u5)--(u1);

\path (centre) +(18:1.35) coordinate[dot,label={[vcolour]below left:\( 2 \)}](u8)
      (centre) +(90:1.35) coordinate[dot,label={[vcolour,label distance=3pt]below:\( 1 \)}](u7)
      (centre) +(162:1.35) coordinate[dot,label={[vcolour]below right:\( 0 \)}](u6);
\draw (u2)--(u6)--(u7)--(u8)--(u1);

\draw (u4)--(u7);

\end{tikzpicture}
\caption{}
\label{fig:4-rs colouring of gadget component planar}
\end{subfigure}\caption{Gadget component in Construction~\ref{make:4-rs colouring planar cubic}, and a 4-rs colouring of it.}
\end{figure}
\begin{lemma}\label{lem:gadget component 4-rs colouring}
The gadget component has rs chromatic number 4. 
Besides, for every 4-rs colouring \( f \) of the gadget component, \( f(u_3)=0 \) or \( f(u_5)=0 \).
\end{lemma}
\begin{proof}
Observe that at least 4 colours are needed to rs colour a 5-vertex cycle (in fact, at least 4 colours are needed to star colour a 5-vertex cycle \cite{fertin2004}). 
Hence, an rs colouring of the gadget component requires at least four colours. 
In addition, each 4-rs colouring \( f \) of the gadget component must use all four colours on each 5-vertex cycle in it. 
Thus, colour~0 has to occur on all four 5-vertex cycles in the gadget component. 
Since no vertex of the gadget component is in all of those four 5-vertex cycles, colour~0 has to occur at least twice in the gadget component.

Thanks to Observation~\ref{obs:no colour 0 at distance 2}, no two vertices within distance two can both get colour~0 under~\( f \). 
Except for \( \{u_3,u_8\} \) and \( \{u_5,u_6\} \), all pairs of vertices from the gadget component are within distance two in the gadget component. 
Since colour~0 has to occur at least twice in the gadget component (see the last paragraph), the 0-th colour class \( f^{-1}(0) \) is either \( \{u_3,u_8\} \) or \( \{u_5,u_6\} \). 
Hence, \( f(u_3)=0 \) or \( f(u_5)=0 \). 
\end{proof}

\iftoggle{forThesis}
{ } { For every construction in this paper, only selected vertices within each gadget are allowed to have neighbours outside the gadget. 
We call such vertices as \emph{the terminals} of the gadget, and highlight them in diagrams by drawing a circle around them.
} 

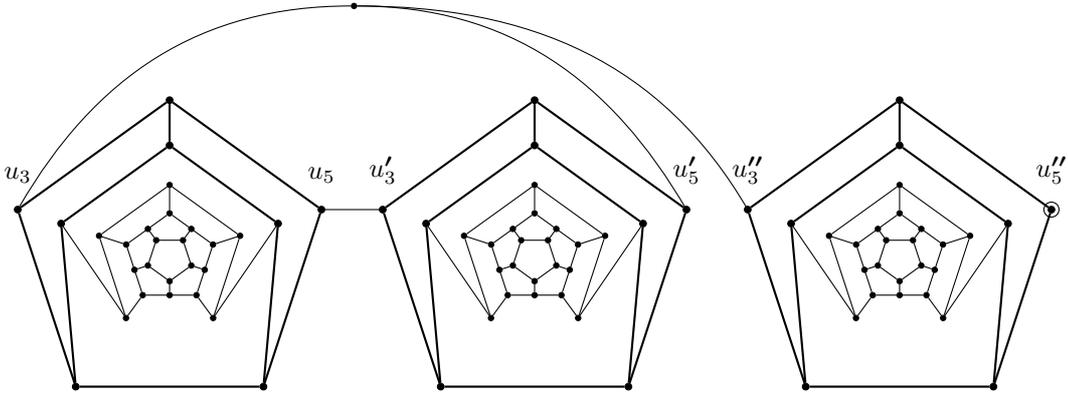
\begin{figure}[hbt] \centering
\begin{tikzpicture}[scale=0.3]
\tikzstyle tinyvcolour=[vcolour,font=\tiny,inner sep=0.8pt]
\tikzstyle smallDot=[dot,minimum size=2pt]
\tikzstyle medDot=[dot,minimum size=2.5pt]

\coordinate (centreinCopy1){};

\path (centreinCopy1) +(18:7) coordinate[medDot,label={[label distance=5pt]above:\( u_5 \)}](u5inCopy1)
      (centreinCopy1) +(90:7) coordinate[medDot](u4inCopy1)
      (centreinCopy1) +(162:7) coordinate[medDot,label={[label distance=5pt]above:\( u_3 \)}](u3inCopy1)
      (centreinCopy1) +(-54:7) coordinate[medDot](u1inCopy1)
      (centreinCopy1) +(-126:7) coordinate[medDot](u2inCopy1);
\draw[thick] (u1inCopy1)--(u2inCopy1)--(u3inCopy1)--(u4inCopy1)--(u5inCopy1)--(u1inCopy1);

\path (centreinCopy1) +(18:5) coordinate[medDot](u8inCopy1)
      (centreinCopy1) +(90:5) coordinate[medDot](u7inCopy1)
      (centreinCopy1) +(162:5) coordinate[medDot](u6inCopy1);
\draw[thick] (u2inCopy1)--(u6inCopy1)--(u7inCopy1)--(u8inCopy1)--(u1inCopy1);

\draw[thick] (u4inCopy1)--(u7inCopy1);

\path (centreinCopy1) +(18:3.25) coordinate[smallDot](35inCopy1)
      (centreinCopy1) +(90:3.25) coordinate[smallDot](34inCopy1)
      (centreinCopy1) +(162:3.25) coordinate[smallDot](33inCopy1)
      (centreinCopy1) +(-54:3.25) coordinate[smallDot](31inCopy1)
      (centreinCopy1) +(-126:3.25) coordinate[smallDot](32inCopy1);
\draw (32inCopy1)--(33inCopy1)--(34inCopy1)--(35inCopy1)--(31inCopy1);

\draw (31inCopy1)--(u8inCopy1)
      (32inCopy1)--(u6inCopy1);

\path (centreinCopy1) +(18:2) coordinate[smallDot](25inCopy1)
      (centreinCopy1) +(90:2) coordinate[smallDot](24inCopy1)
      (centreinCopy1) +(162:2) coordinate[smallDot](23inCopy1)
      (centreinCopy1) +(-54:2) coordinate[smallDot](21inCopy1)
      (centreinCopy1) +(-126:2) coordinate[smallDot](22inCopy1);
\draw (21inCopy1)--coordinate[smallDot](21pt5inCopy1) (22inCopy1)--coordinate[smallDot](22pt5inCopy1) (23inCopy1)--coordinate[smallDot](23pt5inCopy1) (24inCopy1)--coordinate[smallDot](24pt5inCopy1) (25inCopy1)--coordinate[smallDot](25pt5inCopy1) (21inCopy1);

\draw (31inCopy1)--(21inCopy1)
      (32inCopy1)--(22inCopy1)
      (33inCopy1)--(23inCopy1)
      (34inCopy1)--(24inCopy1)
      (35inCopy1)--(25inCopy1);

\path (centreinCopy1) +(-18:1) coordinate[smallDot](15inCopy1)
      (centreinCopy1) +(-90:1) coordinate[smallDot](14inCopy1)
      (centreinCopy1) +(-162:1) coordinate[smallDot](13inCopy1)
      (centreinCopy1) +(54:1) coordinate[smallDot](11inCopy1)
      (centreinCopy1) +(126:1) coordinate[smallDot](12inCopy1);
\draw (11inCopy1)--(12inCopy1)--(13inCopy1)--(14inCopy1)--(15inCopy1)--(11inCopy1);

\draw (14inCopy1)--(21pt5inCopy1)
      (13inCopy1)--(22pt5inCopy1)
      (12inCopy1)--(23pt5inCopy1)
      (11inCopy1)--(24pt5inCopy1)
      (15inCopy1)--(25pt5inCopy1);

\path (centreinCopy1) +(16,0) coordinate(centreinCopy2);

\path (centreinCopy2) +(18:7) coordinate[medDot,label={[label distance=5pt]above:\( u_5\bm{'} \)}](u5inCopy2) coordinate(u5'inCopy2)
      (centreinCopy2) +(90:7) coordinate[medDot](u4inCopy2)
      (centreinCopy2) +(162:7) coordinate[medDot,label={[label distance=5pt]above:\( u_3\bm{'} \)}](u3inCopy2) coordinate(u3'inCopy2)
      (centreinCopy2) +(-54:7) coordinate[medDot](u1inCopy2)
      (centreinCopy2) +(-126:7) coordinate[medDot](u2inCopy2);
\draw[thick] (u1inCopy2)--(u2inCopy2)--(u3inCopy2)--(u4inCopy2)--(u5inCopy2)--(u1inCopy2);

\path (centreinCopy2) +(18:5) coordinate[medDot](u8inCopy2)
      (centreinCopy2) +(90:5) coordinate[medDot](u7inCopy2)
      (centreinCopy2) +(162:5) coordinate[medDot](u6inCopy2);
\draw[thick] (u2inCopy2)--(u6inCopy2)--(u7inCopy2)--(u8inCopy2)--(u1inCopy2);

\draw[thick] (u4inCopy2)--(u7inCopy2);

\path (centreinCopy2) +(18:3.25) coordinate[smallDot](35inCopy2)
      (centreinCopy2) +(90:3.25) coordinate[smallDot](34inCopy2)
      (centreinCopy2) +(162:3.25) coordinate[smallDot](33inCopy2)
      (centreinCopy2) +(-54:3.25) coordinate[smallDot](31inCopy2)
      (centreinCopy2) +(-126:3.25) coordinate[smallDot](32inCopy2);
\draw (32inCopy2)--(33inCopy2)--(34inCopy2)--(35inCopy2)--(31inCopy2);

\draw (31inCopy2)--(u8inCopy2)
      (32inCopy2)--(u6inCopy2);

\path (centreinCopy2) +(18:2) coordinate[smallDot](25inCopy2)
      (centreinCopy2) +(90:2) coordinate[smallDot](24inCopy2)
      (centreinCopy2) +(162:2) coordinate[smallDot](23inCopy2)
      (centreinCopy2) +(-54:2) coordinate[smallDot](21inCopy2)
      (centreinCopy2) +(-126:2) coordinate[smallDot](22inCopy2);
\draw (21inCopy2)--coordinate[smallDot](21pt5inCopy2) (22inCopy2)--coordinate[smallDot](22pt5inCopy2) (23inCopy2)--coordinate[smallDot](23pt5inCopy2) (24inCopy2)--coordinate[smallDot](24pt5inCopy2) (25inCopy2)--coordinate[smallDot](25pt5inCopy2) (21inCopy2);

\draw (31inCopy2)--(21inCopy2)
      (32inCopy2)--(22inCopy2)
      (33inCopy2)--(23inCopy2)
      (34inCopy2)--(24inCopy2)
      (35inCopy2)--(25inCopy2);

\path (centreinCopy2) +(-18:1) coordinate[smallDot](15inCopy2)
      (centreinCopy2) +(-90:1) coordinate[smallDot](14inCopy2)
      (centreinCopy2) +(-162:1) coordinate[smallDot](13inCopy2)
      (centreinCopy2) +(54:1) coordinate[smallDot](11inCopy2)
      (centreinCopy2) +(126:1) coordinate[smallDot](12inCopy2);
\draw (11inCopy2)--(12inCopy2)--(13inCopy2)--(14inCopy2)--(15inCopy2)--(11inCopy2);

\draw (14inCopy2)--(21pt5inCopy2)
      (13inCopy2)--(22pt5inCopy2)
      (12inCopy2)--(23pt5inCopy2)
      (11inCopy2)--(24pt5inCopy2)
      (15inCopy2)--(25pt5inCopy2);

\draw (u5inCopy1)--(u3'inCopy2);

\path (u3inCopy1) --coordinate(midu3u5') (u5'inCopy2);
\path (midu3u5') +(0,9) coordinate[smallDot](joiner);
\draw (u3inCopy1) to[out=60,in=-180] (joiner);
\draw (u5'inCopy2) to[out=120,in=0] (joiner);

\path (centreinCopy2) +(16,0) coordinate(centreinCopy3);

\path (centreinCopy3) +(18:7) coordinate[medDot,label={[label distance=5pt]above:\( u_5\bm{''} \)}](u5inCopy3) node[draw,circle, inner sep=2.0pt]{}
      (centreinCopy3) +(90:7) coordinate[medDot](u4inCopy3)
      (centreinCopy3) +(162:7) coordinate[medDot,label={[label distance=5pt]above:\( u_3\bm{''} \)}](u3inCopy3) coordinate(u3''inCopy3)
      (centreinCopy3) +(-54:7) coordinate[medDot](u1inCopy3)
      (centreinCopy3) +(-126:7) coordinate[medDot](u2inCopy3);
\draw[thick] (u1inCopy3)--(u2inCopy3)--(u3inCopy3)--(u4inCopy3)--(u5inCopy3)--(u1inCopy3);

\path (centreinCopy3) +(18:5) coordinate[medDot](u8inCopy3)
      (centreinCopy3) +(90:5) coordinate[medDot](u7inCopy3)
      (centreinCopy3) +(162:5) coordinate[medDot](u6inCopy3);
\draw[thick] (u2inCopy3)--(u6inCopy3)--(u7inCopy3)--(u8inCopy3)--(u1inCopy3);

\draw[thick] (u4inCopy3)--(u7inCopy3);

\path (centreinCopy3) +(18:3.25) coordinate[smallDot](35inCopy3)
      (centreinCopy3) +(90:3.25) coordinate[smallDot](34inCopy3)
      (centreinCopy3) +(162:3.25) coordinate[smallDot](33inCopy3)
      (centreinCopy3) +(-54:3.25) coordinate[smallDot](31inCopy3)
      (centreinCopy3) +(-126:3.25) coordinate[smallDot](32inCopy3);
\draw (32inCopy3)--(33inCopy3)--(34inCopy3)--(35inCopy3)--(31inCopy3);

\draw (31inCopy3)--(u8inCopy3)
      (32inCopy3)--(u6inCopy3);

\path (centreinCopy3) +(18:2) coordinate[smallDot](25inCopy3)
      (centreinCopy3) +(90:2) coordinate[smallDot](24inCopy3)
      (centreinCopy3) +(162:2) coordinate[smallDot](23inCopy3)
      (centreinCopy3) +(-54:2) coordinate[smallDot](21inCopy3)
      (centreinCopy3) +(-126:2) coordinate[smallDot](22inCopy3);
\draw (21inCopy3)--coordinate[smallDot](21pt5inCopy3) (22inCopy3)--coordinate[smallDot](22pt5inCopy3) (23inCopy3)--coordinate[smallDot](23pt5inCopy3) (24inCopy3)--coordinate[smallDot](24pt5inCopy3) (25inCopy3)--coordinate[smallDot](25pt5inCopy3) (21inCopy3);

\draw (31inCopy3)--(21inCopy3)
      (32inCopy3)--(22inCopy3)
      (33inCopy3)--(23inCopy3)
      (34inCopy3)--(24inCopy3)
      (35inCopy3)--(25inCopy3);

\path (centreinCopy3) +(-18:1) coordinate[smallDot](15inCopy3)
      (centreinCopy3) +(-90:1) coordinate[smallDot](14inCopy3)
      (centreinCopy3) +(-162:1) coordinate[smallDot](13inCopy3)
      (centreinCopy3) +(54:1) coordinate[smallDot](11inCopy3)
      (centreinCopy3) +(126:1) coordinate[smallDot](12inCopy3);
\draw (11inCopy3)--(12inCopy3)--(13inCopy3)--(14inCopy3)--(15inCopy3)--(11inCopy3);

\draw (14inCopy3)--(21pt5inCopy3)
      (13inCopy3)--(22pt5inCopy3)
      (12inCopy3)--(23pt5inCopy3)
      (11inCopy3)--(24pt5inCopy3)
      (15inCopy3)--(25pt5inCopy3);

\draw (joiner) to[out=0,in=120] (u3''inCopy3);
\end{tikzpicture}
\caption{The colour forcing gadget.}
\label{fig:new colour forcing gadget}
\end{figure}

\begin{figure}[hbt] \centering
\begin{tikzpicture}[scale=0.3]
\tikzstyle tinyvcolour=[vcolour,font=\tiny,inner sep=0.8pt]
\tikzstyle smallDot=[dot,minimum size=2pt]
\tikzstyle medDot=[dot,minimum size=2.5pt]

\coordinate (centreinCopy1){};

\path (centreinCopy1) +(18:7) coordinate[medDot,label={[label distance=5pt]above:\( u_5 \)},label={[tinyvcolour]below right:0}](u5inCopy1)
      (centreinCopy1) +(90:7) coordinate[medDot,label={[tinyvcolour]above:3}](u4inCopy1)
      (centreinCopy1) +(162:7) coordinate[medDot,label={[label distance=5pt]above:\( u_3 \)},label={[tinyvcolour]left:2}](u3inCopy1)
      (centreinCopy1) +(-54:7) coordinate[medDot,label={[tinyvcolour]below:3}](u1inCopy1)
      (centreinCopy1) +(-126:7) coordinate[medDot,label={[tinyvcolour]below:1}](u2inCopy1);
\draw[thick] (u1inCopy1)--(u2inCopy1)--(u3inCopy1)--(u4inCopy1)--(u5inCopy1)--(u1inCopy1);

\path (centreinCopy1) +(18:5) coordinate[medDot,label={[tinyvcolour]right:2}](u8inCopy1)
      (centreinCopy1) +(90:5) coordinate[medDot,label={[tinyvcolour,yshift=1pt]right:1}](u7inCopy1)
      (centreinCopy1) +(162:5) coordinate[medDot,label={[tinyvcolour]left:0}](u6inCopy1);
\draw[thick] (u2inCopy1)--(u6inCopy1)--(u7inCopy1)--(u8inCopy1)--(u1inCopy1);

\draw[thick] (u4inCopy1)--(u7inCopy1);

\path (centreinCopy1) +(18:3.25) coordinate[smallDot,label={[tinyvcolour]right:2}](35inCopy1)
      (centreinCopy1) +(90:3.25) coordinate[smallDot,label={[tinyvcolour]above:1}](34inCopy1)
      (centreinCopy1) +(162:3.25) coordinate[smallDot,label={[tinyvcolour]left:3}](33inCopy1)
      (centreinCopy1) +(-54:3.25) coordinate[smallDot,label={[tinyvcolour]below:0}](31inCopy1)
      (centreinCopy1) +(-126:3.25) coordinate[smallDot,label={[tinyvcolour]below:2}](32inCopy1);
\draw (32inCopy1)--(33inCopy1)--(34inCopy1)--(35inCopy1)--(31inCopy1);

\draw (31inCopy1)--(u8inCopy1)
      (32inCopy1)--(u6inCopy1);

\path (centreinCopy1) +(18:2) coordinate[smallDot,label={[tinyvcolour,xshift=1pt,label distance=1pt]above:3}](25inCopy1)
      (centreinCopy1) +(90:2) coordinate[smallDot,label={[tinyvcolour,xshift=-1pt,yshift=2pt]right:2}](24inCopy1)
      (centreinCopy1) +(162:2) coordinate[smallDot,label={[tinyvcolour,xshift=-1pt,label distance=1pt]above:0}](23inCopy1)
      (centreinCopy1) +(-54:2) coordinate[smallDot,label={[tinyvcolour,label distance=1pt,xshift=2pt]below left:2}](21inCopy1)
      (centreinCopy1) +(-126:2) coordinate[smallDot,label={[tinyvcolour,label distance=1pt,xshift=-2pt]below right:1}](22inCopy1);
\draw (21inCopy1)--coordinate[smallDot,label={[tinyvcolour,label distance=0.5pt]below:3}](21pt5inCopy1) (22inCopy1)--coordinate[smallDot,label={[tinyvcolour,label distance=1pt]left:3}](22pt5inCopy1) (23inCopy1)--coordinate[smallDot,label={[tinyvcolour,label distance=1pt,xshift=2pt]above left:3}](23pt5inCopy1) (24inCopy1)--coordinate[smallDot,label={[tinyvcolour,label distance=1pt,xshift=-2pt]above right:0}](24pt5inCopy1) (25inCopy1)--coordinate[smallDot,label={[tinyvcolour,label distance=1pt]right:1}](25pt5inCopy1) (21inCopy1);

\draw (31inCopy1)--(21inCopy1)
      (32inCopy1)--(22inCopy1)
      (33inCopy1)--(23inCopy1)
      (34inCopy1)--(24inCopy1)
      (35inCopy1)--(25inCopy1);

\path (centreinCopy1) +(-18:1) coordinate[smallDot,label={[tinyvcolour,xshift=-1pt,label distance=1pt]above right:3}](15inCopy1)
      (centreinCopy1) +(-90:1) coordinate[smallDot,label={[tinyvcolour,label distance=1.5pt]above:0}](14inCopy1)
      (centreinCopy1) +(-162:1) coordinate[smallDot,label={[tinyvcolour,xshift=1pt,label distance=1pt]above left:2}](13inCopy1)
      (centreinCopy1) +(54:1) coordinate[smallDot,label={[tinyvcolour,xshift=-2pt,label distance=0.5pt]above:2}](11inCopy1)
      (centreinCopy1) +(126:1) coordinate[smallDot,label={[tinyvcolour,xshift=2pt,label distance=0.5pt]above:1}](12inCopy1);
\draw (11inCopy1)--(12inCopy1)--(13inCopy1)--(14inCopy1)--(15inCopy1)--(11inCopy1);

\draw (14inCopy1)--(21pt5inCopy1)
      (13inCopy1)--(22pt5inCopy1)
      (12inCopy1)--(23pt5inCopy1)
      (11inCopy1)--(24pt5inCopy1)
      (15inCopy1)--(25pt5inCopy1);

\path (centreinCopy1) +(16,0) coordinate(centreinCopy2);

\path (centreinCopy2) +(18:7) coordinate[medDot,label={[label distance=5pt]above:\( u_5\bm{'} \)},label={[tinyvcolour]right:0}](u5inCopy2) coordinate(u5'inCopy2)
      (centreinCopy2) +(90:7) coordinate[medDot,label={[tinyvcolour]above:3}](u4inCopy2)
      (centreinCopy2) +(162:7) coordinate[medDot,label={[label distance=5pt]above:\( u_3\bm{'} \)},label={[tinyvcolour]below left:1}](u3inCopy2) coordinate(u3'inCopy2)
      (centreinCopy2) +(-54:7) coordinate[medDot,label={[tinyvcolour]below:3}](u1inCopy2)
      (centreinCopy2) +(-126:7) coordinate[medDot,label={[tinyvcolour]below:2}](u2inCopy2);
\draw[thick] (u1inCopy2)--(u2inCopy2)--(u3inCopy2)--(u4inCopy2)--(u5inCopy2)--(u1inCopy2);

\path (centreinCopy2) +(18:5) coordinate[medDot,label={[tinyvcolour]right:1}](u8inCopy2)
      (centreinCopy2) +(90:5) coordinate[medDot,label={[tinyvcolour]above right:2}](u7inCopy2)
      (centreinCopy2) +(162:5) coordinate[medDot,label={[tinyvcolour]left:0}](u6inCopy2);
\draw[thick] (u2inCopy2)--(u6inCopy2)--(u7inCopy2)--(u8inCopy2)--(u1inCopy2);

\draw[thick] (u4inCopy2)--(u7inCopy2);

\path (centreinCopy2) +(18:3.25) coordinate[smallDot,label={[tinyvcolour]right:2}](35inCopy2)
      (centreinCopy2) +(90:3.25) coordinate[smallDot,label={[tinyvcolour]above:1}](34inCopy2)
      (centreinCopy2) +(162:3.25) coordinate[smallDot,label={[tinyvcolour]left:3}](33inCopy2)
      (centreinCopy2) +(-54:3.25) coordinate[smallDot,label={[tinyvcolour]below:0}](31inCopy2)
      (centreinCopy2) +(-126:3.25) coordinate[smallDot,label={[tinyvcolour]below:2}](32inCopy2);
\draw (32inCopy2)--(33inCopy2)--(34inCopy2)--(35inCopy2)--(31inCopy2);

\draw (31inCopy2)--(u8inCopy2)
      (32inCopy2)--(u6inCopy2);

\path (centreinCopy2) +(18:2) coordinate[smallDot,label={[tinyvcolour,xshift=1pt,label distance=1pt]above:3}](25inCopy2)
      (centreinCopy2) +(90:2) coordinate[smallDot,label={[tinyvcolour,xshift=-1pt,yshift=2pt]right:2}](24inCopy2)
      (centreinCopy2) +(162:2) coordinate[smallDot,label={[tinyvcolour,xshift=-1pt,label distance=1pt]above:0}](23inCopy2)
      (centreinCopy2) +(-54:2) coordinate[smallDot,label={[tinyvcolour,label distance=1pt,xshift=2pt]below left:2}](21inCopy2)
      (centreinCopy2) +(-126:2) coordinate[smallDot,label={[tinyvcolour,label distance=1pt,xshift=-2pt]below right:1}](22inCopy2);
\draw (21inCopy2)--coordinate[smallDot,label={[tinyvcolour,label distance=0.5pt]below:3}](21pt5inCopy2) (22inCopy2)--coordinate[smallDot,label={[tinyvcolour,label distance=1pt]left:3}](22pt5inCopy2) (23inCopy2)--coordinate[smallDot,label={[tinyvcolour,label distance=1pt,xshift=2pt]above left:3}](23pt5inCopy2) (24inCopy2)--coordinate[smallDot,label={[tinyvcolour,label distance=1pt,xshift=-2pt]above right:0}](24pt5inCopy2) (25inCopy2)--coordinate[smallDot,label={[tinyvcolour,label distance=1pt]right:1}](25pt5inCopy2) (21inCopy2);

\draw (31inCopy2)--(21inCopy2)
      (32inCopy2)--(22inCopy2)
      (33inCopy2)--(23inCopy2)
      (34inCopy2)--(24inCopy2)
      (35inCopy2)--(25inCopy2);

\path (centreinCopy2) +(-18:1) coordinate[smallDot,label={[tinyvcolour,xshift=-1pt,label distance=1pt]above right:3}](15inCopy2)
      (centreinCopy2) +(-90:1) coordinate[smallDot,label={[tinyvcolour,label distance=1.5pt]above:0}](14inCopy2)
      (centreinCopy2) +(-162:1) coordinate[smallDot,label={[tinyvcolour,xshift=1pt,label distance=1pt]above left:2}](13inCopy2)
      (centreinCopy2) +(54:1) coordinate[smallDot,label={[tinyvcolour,xshift=-2pt,label distance=0.5pt]above:2}](11inCopy2)
      (centreinCopy2) +(126:1) coordinate[smallDot,label={[tinyvcolour,xshift=2pt,label distance=0.5pt]above:1}](12inCopy2);
\draw (11inCopy2)--(12inCopy2)--(13inCopy2)--(14inCopy2)--(15inCopy2)--(11inCopy2);

\draw (14inCopy2)--(21pt5inCopy2)
      (13inCopy2)--(22pt5inCopy2)
      (12inCopy2)--(23pt5inCopy2)
      (11inCopy2)--(24pt5inCopy2)
      (15inCopy2)--(25pt5inCopy2);

\draw (u5inCopy1)--(u3'inCopy2);

\path (u3inCopy1) --coordinate(midu3u5') (u5'inCopy2);
\path (midu3u5') +(0,9) coordinate[smallDot,label={[tinyvcolour]below:3}](joiner);
\draw (u3inCopy1) to[out=60,in=-180] (joiner);
\draw (u5'inCopy2) to[out=120,in=0] (joiner);

\path (centreinCopy2) +(16,0) coordinate(centreinCopy3);

\path (centreinCopy3) +(18:7) coordinate[medDot,label={[label distance=5pt]above:\( u_5\bm{''} \)}](u5inCopy3) node[draw,circle, inner sep=2.0pt][label={[tinyvcolour]right:0}]{}
      (centreinCopy3) +(90:7) coordinate[medDot,label={[tinyvcolour]above:3}](u4inCopy3)
      (centreinCopy3) +(162:7) coordinate[medDot,label={[label distance=5pt]above:\( u_3\bm{''} \)},label={[tinyvcolour]left:1}](u3inCopy3) coordinate(u3''inCopy3)
      (centreinCopy3) +(-54:7) coordinate[medDot,label={[tinyvcolour]below:3}](u1inCopy3)
      (centreinCopy3) +(-126:7) coordinate[medDot,label={[tinyvcolour]below:2}](u2inCopy3);
\draw[thick] (u1inCopy3)--(u2inCopy3)--(u3inCopy3)--(u4inCopy3)--(u5inCopy3)--(u1inCopy3);

\path (centreinCopy3) +(18:5) coordinate[medDot,label={[tinyvcolour]right:1}](u8inCopy3)
      (centreinCopy3) +(90:5) coordinate[medDot,label={[tinyvcolour]above right:2}](u7inCopy3)
      (centreinCopy3) +(162:5) coordinate[medDot,label={[tinyvcolour]left:0}](u6inCopy3);
\draw[thick] (u2inCopy3)--(u6inCopy3)--(u7inCopy3)--(u8inCopy3)--(u1inCopy3);

\draw[thick] (u4inCopy3)--(u7inCopy3);

\path (centreinCopy3) +(18:3.25) coordinate[smallDot,label={[tinyvcolour]right:2}](35inCopy3)
      (centreinCopy3) +(90:3.25) coordinate[smallDot,label={[tinyvcolour]above:1}](34inCopy3)
      (centreinCopy3) +(162:3.25) coordinate[smallDot,label={[tinyvcolour]left:3}](33inCopy3)
      (centreinCopy3) +(-54:3.25) coordinate[smallDot,label={[tinyvcolour]below:0}](31inCopy3)
      (centreinCopy3) +(-126:3.25) coordinate[smallDot,label={[tinyvcolour]below:2}](32inCopy3);
\draw (32inCopy3)--(33inCopy3)--(34inCopy3)--(35inCopy3)--(31inCopy3);

\draw (31inCopy3)--(u8inCopy3)
      (32inCopy3)--(u6inCopy3);

\path (centreinCopy3) +(18:2) coordinate[smallDot,label={[tinyvcolour,xshift=1pt,label distance=1pt]above:3}](25inCopy3)
      (centreinCopy3) +(90:2) coordinate[smallDot,label={[tinyvcolour,xshift=-1pt,yshift=2pt]right:2}](24inCopy3)
      (centreinCopy3) +(162:2) coordinate[smallDot,label={[tinyvcolour,xshift=-1pt,label distance=1pt]above:0}](23inCopy3)
      (centreinCopy3) +(-54:2) coordinate[smallDot,label={[tinyvcolour,label distance=1pt,xshift=2pt]below left:2}](21inCopy3)
      (centreinCopy3) +(-126:2) coordinate[smallDot,label={[tinyvcolour,label distance=1pt,xshift=-2pt]below right:1}](22inCopy3);
\draw (21inCopy3)--coordinate[smallDot,label={[tinyvcolour,label distance=0.5pt]below:3}](21pt5inCopy3) (22inCopy3)--coordinate[smallDot,label={[tinyvcolour,label distance=1pt]left:3}](22pt5inCopy3) (23inCopy3)--coordinate[smallDot,label={[tinyvcolour,label distance=1pt,xshift=2pt]above left:3}](23pt5inCopy3) (24inCopy3)--coordinate[smallDot,label={[tinyvcolour,label distance=1pt,xshift=-2pt]above right:0}](24pt5inCopy3) (25inCopy3)--coordinate[smallDot,label={[tinyvcolour,label distance=1pt]right:1}](25pt5inCopy3) (21inCopy3);

\draw (31inCopy3)--(21inCopy3)
      (32inCopy3)--(22inCopy3)
      (33inCopy3)--(23inCopy3)
      (34inCopy3)--(24inCopy3)
      (35inCopy3)--(25inCopy3);

\path (centreinCopy3) +(-18:1) coordinate[smallDot,label={[tinyvcolour,xshift=-1pt,label distance=1pt]above right:3}](15inCopy3)
      (centreinCopy3) +(-90:1) coordinate[smallDot,label={[tinyvcolour,label distance=1.5pt]above:0}](14inCopy3)
      (centreinCopy3) +(-162:1) coordinate[smallDot,label={[tinyvcolour,xshift=1pt,label distance=1pt]above left:2}](13inCopy3)
      (centreinCopy3) +(54:1) coordinate[smallDot,label={[tinyvcolour,xshift=-2pt,label distance=0.5pt]above:2}](11inCopy3)
      (centreinCopy3) +(126:1) coordinate[smallDot,label={[tinyvcolour,xshift=2pt,label distance=0.5pt]above:1}](12inCopy3);
\draw (11inCopy3)--(12inCopy3)--(13inCopy3)--(14inCopy3)--(15inCopy3)--(11inCopy3);

\draw (14inCopy3)--(21pt5inCopy3)
      (13inCopy3)--(22pt5inCopy3)
      (12inCopy3)--(23pt5inCopy3)
      (11inCopy3)--(24pt5inCopy3)
      (15inCopy3)--(25pt5inCopy3);

\draw (joiner) to[out=0,in=120] (u3''inCopy3);
\end{tikzpicture}
\caption{A 4-rs colouring of the colour forcing gadget.}
\label{fig:4-rs colouring new colour forcing gadget}
\end{figure}
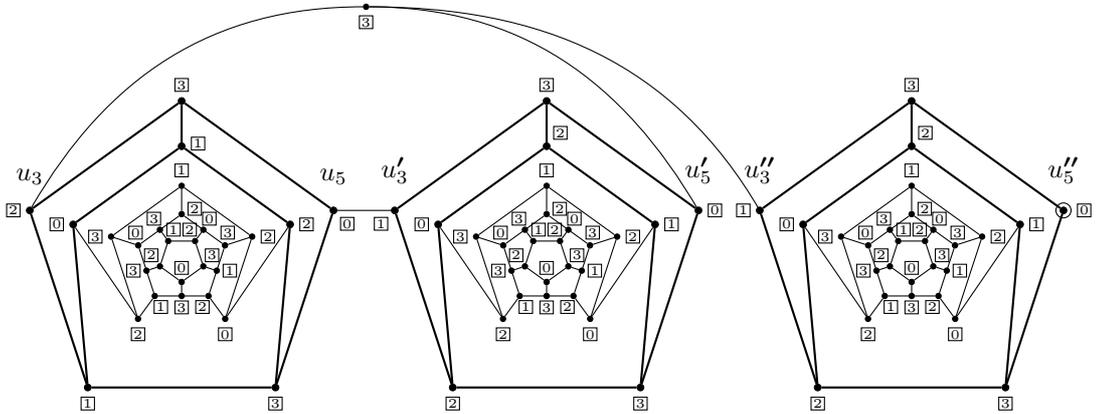

The colour forcing gadget is displayed in Figure~\ref{fig:new colour forcing gadget}. 
Consider the 4-colouring of the gadget displayed in Figure~\ref{fig:4-rs colouring new colour forcing gadget}. 
It is a 4-rs colouring of the gadget because (i)~no vertex coloured~1 has two neighbours coloured~0, (ii)~except for the vertex \( u_3 \) coloured~2 and its two neighbours coloured~3, no vertex coloured~2 has two neighbours of the same colour, and (iii)~no vertex coloured~3 has two neighbours of the same colour.

\begin{figure}[hbt]\centering
\begin{tikzpicture}[scale=0.3,label distance=5pt]

\tikzstyle smallDot=[dot,minimum size=2pt]
\tikzstyle medDot=[dot,minimum size=2.5pt]

\coordinate (centreinCopy1){};

\path (centreinCopy1) +(18:7) coordinate[dot,label=\( u_5 \)](u5inCopy1)
      (centreinCopy1) +(90:7) coordinate[dot](u4inCopy1)
      (centreinCopy1) +(162:7) coordinate[dot,label=\( u_3 \)](u3inCopy1) coordinate(u3FirstinCopy1)
      (centreinCopy1) +(-54:7) coordinate[dot](u1inCopy1)
      (centreinCopy1) +(-126:7) coordinate[dot](u2inCopy1);
\draw (u1inCopy1)--(u2inCopy1)--(u3inCopy1)--(u4inCopy1)--(u5inCopy1)--(u1inCopy1);

\path (centreinCopy1) +(18:5) coordinate[dot](u8inCopy1)
      (centreinCopy1) +(90:5) coordinate[dot](u7inCopy1)
      (centreinCopy1) +(162:5) coordinate[dot](u6inCopy1);
\draw (u2inCopy1)--(u6inCopy1)--(u7inCopy1)--(u8inCopy1)--(u1inCopy1);

\draw (u4inCopy1)--(u7inCopy1);

\path (centreinCopy1) +(17,0) coordinate(centreinCopy2);

\path (centreinCopy2) +(18:7) coordinate[dot,label=\( u_5\bm{'} \)](u5inCopy2)

      (centreinCopy2) +(90:7) coordinate[dot](u4inCopy2)
      (centreinCopy2) +(162:7) coordinate[dot,label=\( u_3\bm{'} \)](u3inCopy2)
      (centreinCopy2) +(-54:7) coordinate[dot](u1inCopy2)
      (centreinCopy2) +(-126:7) coordinate[dot](u2inCopy2);
\draw (u1inCopy2)--(u2inCopy2)--(u3inCopy2)--(u4inCopy2)--(u5inCopy2)--(u1inCopy2);

\path (centreinCopy2) +(18:5) coordinate[dot](u8inCopy2)
      (centreinCopy2) +(90:5) coordinate[dot](u7inCopy2)
      (centreinCopy2) +(162:5) coordinate[dot](u6inCopy2);
\draw (u2inCopy2)--(u6inCopy2)--(u7inCopy2)--(u8inCopy2)--(u1inCopy2);

\draw (u4inCopy2)--(u7inCopy2);

\draw (u5inCopy1)--(u3inCopy2);

\path (u3inCopy1) --coordinate(midu3u5') (u5inCopy2);
\path (midu3u5') +(0,9) coordinate[dot](joiner);
\draw (u3inCopy1) to[out=60,in=-180] (joiner);
\draw (u5inCopy2) to[out=120,in=0] (joiner);

\path (centreinCopy2) +(17,0) coordinate(centreinCopy3);

\path (centreinCopy3) +(18:7) coordinate[dot,label=\( u_5\bm{''} \)](u5inCopy3)
      (centreinCopy3) +(90:7) coordinate[dot](u4inCopy3)
      (centreinCopy3) +(162:7) coordinate[dot,label=\( u_3\bm{''} \)](u3inCopy3)
      (centreinCopy3) +(-54:7) coordinate[dot](u1inCopy3)
      (centreinCopy3) +(-126:7) coordinate[dot](u2inCopy3);
\draw (u1inCopy3)--(u2inCopy3)--(u3inCopy3)--(u4inCopy3)--(u5inCopy3)--(u1inCopy3);

\path (centreinCopy3) +(18:5) coordinate[dot](u8inCopy3)
      (centreinCopy3) +(90:5) coordinate[dot](u7inCopy3)
      (centreinCopy3) +(162:5) coordinate[dot](u6inCopy3);
\draw (u2inCopy3)--(u6inCopy3)--(u7inCopy3)--(u8inCopy3)--(u1inCopy3);

\draw (u4inCopy3)--(u7inCopy3);

\draw (joiner) to[out=0,in=120] (u3inCopy3);

\end{tikzpicture}
\caption{The subgraph \( H \) of the colour forcing gadget.}
\label{fig:colour forcing gadget 4-rs colouring}
\end{figure}
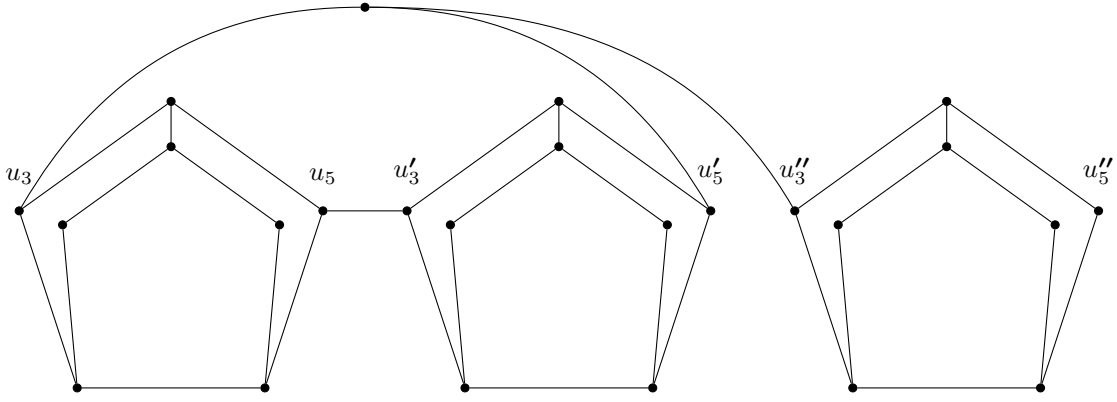

The colour forcing gadget is named so because the terminal of the gadget must be coloured~0 under each 4-rs colouring of the gadget (see Lemma~\ref{lem:colour forcing gadget 4-rs colouring} below). 
Since the colour forcing gadget is 4-rs colourable, its subgraph \( H \) shown in Figure~\ref{fig:colour forcing gadget 4-rs colouring} is 4-rs colourable as well.

\begin{lemma}\label{lem:old colour forcing gadget 4-rs colouring}
\( f(u\bm{''}_5)=0 \) for every 4-rs colouring \( f \) of the graph \( H \) in Figure~\ref{fig:colour forcing gadget 4-rs colouring}. 
\end{lemma}
\begin{proof}
Let \( f \) be a 4-rs colouring of \( H \). 
Note that there are three copies of the gadget component in \( H \). 
By Lemma~\ref{lem:gadget component 4-rs colouring}, \( f(u_3)=0 \) or \( f(u_5)=0 \). 
For the same reason, \( f(u_3\bm{'})=0 \) or \( f(u_5\bm{'})=0 \). 
Similarly, \( f(u_3\bm{''})=0 \) or \( f(u_5\bm{''})=0 \). 
We claim that \( f(u_3\bm{''})\neq 0 \). 
On the contrary, assume that \( f(u_3\bm{''})=0 \). 
This implies that \( f(u_3)\neq 0 \) and \( f(u_5\bm{'})\neq 0 \) by Observation~\ref{obs:no colour 0 at distance 2}. 
Since, \( u_3 \) or \( u_5 \) must be coloured 0, we have \( f(u_5)=0 \). 
Similarly, \( u_3\bm{'} \) or \( u_5\bm{'} \) must be coloured~0 and thus \( f(u_3\bm{'})=0 \). 
We have a contradiction since \( f(u_5)=0=f(u_3\bm{'}) \) and \( u_5u_3\bm{'} \) is an edge. 
Therefore, \( f(u_3\bm{''})\neq 0 \) by contradiction. 
Hence, \( f(u_5\bm{''})=0 \) because \( u_3\bm{''} \) or \( u_5\bm{''} \) must be coloured 0. 
\end{proof}

Since \( H \) is a subgraph of the colour forcing gadget, every 4-rs colouring \( f \) of the gadget is a 4-rs colouring of its subgraph \( H \), and thus \( f(u\bm{''}_5)=0 \) by Lemma~\ref{lem:old colour forcing gadget 4-rs colouring}. 
Hence, we have the following. 
\begin{lemma}\label{lem:colour forcing gadget 4-rs colouring}
Every 4-rs colouring \( f \) of the colour forcing gadget must assign colour~0 on its terminal \( ( \)that is, \( f(u\bm{''}_5)=0 \)\( ) \). 
\qed
\end{lemma}
We employ the graph in Figure~\ref{fig:new colour forcing gadget} as the colour forcing gadget rather than the graph in Figure~\ref{fig:colour forcing gadget 4-rs colouring} to ensure that the output graph is 3-regular. 
With the help of the next construction, we prove that \textsc{\( 4 \)-RS Colourability} is NP-complete for planar 3-regular graphs. 
\begin{construct}\label{make:4-rs colouring planar cubic}
\emph{Input:} A positive boolean formula \( B=(X,C) \) such that the graph of \( B \) is a planar 3-regular graph.\\
\emph{Output:} A planar 3-regular graph \( G' \) of girth five.\\
\emph{Guarantee:} \( B \) has a 1-in-3 satisfying truth assignment if and only if \( G' \) is 4-rs colourable.\\
\emph{Steps:}\\
First, construct a graph \( G \) from formula \( B \) by Construction~\ref{make:3-rs colouring planar} (see page~\pageref{make:3-rs colouring planar}). 
Then, for every degree-2 vertex \( v \) of \( G \), introduce a colour forcing gadget (see Figure~\ref{fig:new colour forcing gadget}) and join the terminal of the gadget to \( v \) by an edge. 
\end{construct}
Clearly, \( G' \) is a planar 3-regular graph. 
Since \( G \) has girth~6 (see Construction~\ref{make:3-rs colouring planar}) and the colour forcing gadget has girth~5, the graph \( G' \) has girth~5. 
\begin{proof}[Proof of guarantee]
Suppose that the formula \( B \) has a 1-in-3 satisfying truth assignment. 
By the guarantee in Construction~\ref{make:3-rs colouring planar}, \( G \) admits a 3-rs colouring \( f\colon V(G)\to \{1,2,3\} \). 
Observe that \( f(v)>0 \) for all \( v\in V(G) \). 
Extend \( f \) into a 4-colouring \( f' \) of \( G' \) by applying the 4-rs colouring scheme in Figure~\ref{fig:4-rs colouring new colour forcing gadget} on each colour forcing gadget.  
\setcounter{claim}{0}
\begin{claim}\label{clm:4-rs colouring cubic planar}
\( f' \) is a 4-rs colouring of \( G' \). 
\end{claim}
Assume the contrary. 
That is, there is a path \( Q=x,y,z \) in \( G' \) with \( f'(y)>f'(x)=f'(z) \). 
Since the copy of \( G \) in \( G' \) and the colour forcing gadgets are coloured by rs-colouring schemes (namely, \( f \) and Figure~\ref{fig:4-rs colouring new colour forcing gadget}), \( Q \) contains an edge \( u_5\bm{''}v \), where \( u_5\bm{''} \) is the terminal of a colour forcing gadget and \( v\in V(G) \). 
By symmetry, we assume without loss of generality that \( xy \) is the edge \( u_5\bm{''}v \). 
Hence, either (i)~\( x=u_5\bm{''} \), \( y=v \) (and \( z\in V(G) \)); or (ii)~\( y=u_5\bm{''} \), \( x=v \) (and \( z \) is in a colour forcing gadget). 
Note that \( f'(u_5\bm{''})=0 \) (see Figure~\ref{fig:4-rs colouring new colour forcing gadget}) and \( f'(v)=f(v)>0 \). 
Since \( f'(y)>f'(x) \), we have \( x=u_5\bm{''} \) and \( y=v \) (i.e., Case~(i) occurs). 
Thus, \( x \) is the terminal of the colour forcing gadget attached at \( y \)~\( (=v) \) and \( f'(x)=f'(u_5\bm{''})=0 \). 
Since Case~(i) occurs, \( z\in V(G) \). 
Hence, \( f'(z)=f(z)>0 \). 
This is a contradiction since \( f'(z)=f'(x)=0 \). 
This proves Claim~\ref{clm:4-rs colouring cubic planar}. 
Therefore, \( G' \) is 4-rs colourable. 

Conversely, suppose that \( G' \) admits a 4-rs colouring \( f' \). 
By Lemma~\ref{lem:colour forcing gadget 4-rs colouring}, terminals of all colour forcing gadgets must be coloured~0 by \( f' \). 
Note that every vertex in \( G \) is either a degree-2 vertex or adjacent to a degree-2 vertex. 
Since a colour forcing gadget is attached to each degree-2 vertex of \( G \), every vertex \( v\in V(G) \) is within distance two from a terminal in \( G' \). 
Thanks to Observation~\ref{obs:no colour 0 at distance 2}, this means that no vertex \( v\in V(G) \) is coloured~0 by \( f' \). 
Since \( f' \) restricted to \( V(G) \) uses only colours 1,2 and 3, the restriction is indeed a 3-rs colouring of \( G \). 
By the guarantee in Construction~\ref{make:3-rs colouring planar}, this implies that \( B \) has a 1-in-3 satisfying truth assignment.
\end{proof}

We know that the construction of graph \( G \) (i.e., Construction~\ref{make:3-rs colouring planar}) requires only time polynomial in the input size. 
Construction~\ref{make:4-rs colouring planar cubic} requires only time polynomial in the input size because (i)~the colour forcing gadget is a fixed graph, and (ii)~at most \( |V(G)| \) colour forcing gadgets are introduced in Construction~\ref{make:4-rs colouring planar cubic}. 
Given a positive boolean formula \( B = (X,C) \) such that \( G_B \) is a planar 3-regular graph, it is NP-complete to test whether \( B \) has a 1-in-3 satisfying truth assignment \cite{moore_robson}. 
Thus, Construction~\ref{make:4-rs colouring planar cubic} gives the following result. 
\begin{theorem}
\textsc{4-RS Colourability} is NP-complete for planar 3-regular graphs of girth 5. 
\qed
\end{theorem}
\begin{corollary}\label{cor:4-rs colouring planar}
\textsc{4-RS Colourability} is NP-complete for triangle-free graphs \( G \) of maximum degree~3. 
\qed
\end{corollary}

Next, we generalise Corollary~\ref{cor:4-rs colouring planar} as follows: for \( k\geq 4 \), \textsc{\( k \)-RS Colourability} is NP-complete for triangle-free graphs of maximum degree \( k-1 \).
For \( k\geq 5 \), we employ Construction~\ref{make:k-rs colouring} below to establish a reduction from \textsc{\( (k-2) \)-RS Colourability} of graphs of maximum degree \( k-2 \) to \textsc{\( k \)-RS Colourability} of graphs of maximum degree \( k-1 \). 

\begin{figure}[hbtp]
\centering
\begin{tikzpicture}
\path (0,0) node(x1)[dot][label=left:\( x_1 \)]{} ++(0,-1) node(x2)[dot][label=left:\( x_2 \)]{} ++(0,-2) node(xk-2)[dot][label={[xshift=3pt]left:\( x_{k-2} \)}]{} ++(0,-1) node(xk-1)[dot][label={[xshift=3pt,yshift=-3pt]left:\( x_{k-1} \)}]{};
\path (x2)--node[sloped,font=\large]{\( \dots \)} (xk-2);
\path (x1) ++(3,0) node(y1)[dot][label=right:\( y_1 \)]{} ++(0,-1) node(y2)[dot][label=right:\( y_2 \)]{} ++(0,-2) node(yk-2)[dot][label=right:\( y_{k-2} \)]{} ++(0,-1) node(yk-1)[dot][label=right:\( y_{k-1} \)]{};
\path (y2)--node[sloped,font=\large]{\( \dots \)} (yk-2);

\path (x1)--node(u1)[dot][yshift=10pt][label=above:\( u_1 \)]{} (y1);
\draw (x1)--(u1)--(y1)
      (x1)--(y2)
      (x1)--(yk-2)
      (x1)--(yk-1);
\draw (x2)--(y1)
      (x2)--(yk-2)
      (x2)--(yk-1);
\draw (xk-2)--(y1)
      (xk-2)--(y2)
      (xk-2)--(yk-1);
\path (xk-1)--node(v1)[dot][yshift=-10pt][label=below:\( v_1 \)]{} (yk-1);
\draw (xk-1)--(y1)
      (xk-1)--(y2)
      (xk-1)--(yk-2)
      (xk-1)--(v1)--(yk-1);

\draw (u1)--++(3.5,1) node(u2)[dot][label=above:\( u_2 \)]{};
\draw (y2)--(u2)
      (yk-2)--(u2);
\draw (u2)--++(1.5,0) node(u3)[dot]{} node[terminal][label=above:\( u_3 \)]{};

\draw (v1)--++(-3.5,-1) node(v2)[dot][label=below:\( v_2 \)]{};
\draw (x2)--(v2)
      (xk-2)--(v2);
\draw (v2)--++(-1.5,0) node(v3)[dot][label=below:\( v_3 \)]{}; \end{tikzpicture}
\caption{The colour blocking gadget}
\label{fig:colour blocking gadget}
\end{figure}
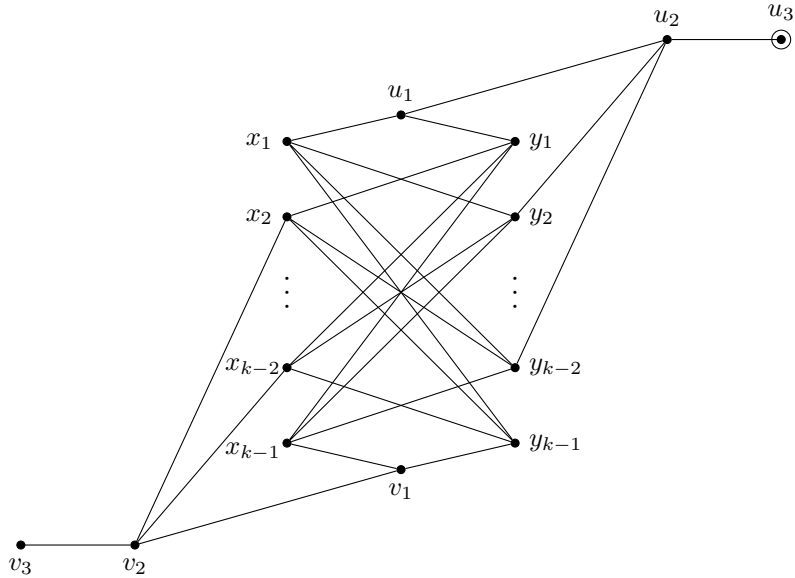
\begin{figure}[hbtp]
\centering
\begin{tikzpicture}
\path (0,0) node(x1)[dot][label=\( x_1 \)][label={[vcolour]left:\( k\text{-1} \)}]{} ++(0,-1) node(x2)[dot][label=\( x_2 \)][label={[vcolour]left:\( 2 \)}]{} ++(0,-1) node(x3)[dot][label={[xshift=2pt]below:\( x_3 \)}][label={[vcolour]above:\( 3 \)}]{} ++(0,-2) node(xk-2)[dot][label={[xshift=-7pt]\( x_{k-2} \)}][label={[vcolour]left:\( k\text{-2} \)}]{} ++(0,-1) node(xk-1)[dot][label=below:\( x_{k-1} \)][label={[vcolour]left:\( 1 \)}]{};
\path (x3)--node[sloped,font=\large]{\( \dots \)} (xk-2);
\path (x1) ++(3,0) node(y1)[dot][label=\( y_1 \)][label={[vcolour]right:\( k\text{-1} \)}]{} ++(0,-1) node(y2)[dot][label=\( y_2 \)][label={[vcolour]right:\( 2 \)}]{} ++(0,-1) node(y3)[dot][label=below:\( y_3 \)][label={[vcolour,xshift=-2pt,label distance=3pt]above:\( 3 \)}]{} ++(0,-2) node(yk-2)[dot][label={[xshift=2pt]below:\( y_{k-2} \)}][label={[vcolour]right:\( k\text{-2} \)}]{} ++(0,-1) node(yk-1)[dot][label=below:\( y_{k-1} \)][label={[vcolour]right:\( 1 \)}]{};
\path (y3)--node[sloped,font=\large]{\( \dots \)} (yk-2);

\path (x1)--node(u1)[dot][yshift=10pt][label=below:\( u_1 \)][label={[vcolour]above:\( 0 \)}]{} (y1);
\draw (x1)--(u1)--(y1)
      (x1)--(y2)
      (x1)--(y3)
      (x1)--(yk-2)
      (x1)--(yk-1);
\draw (x2)--(y1)
      (x2)--(y3)
      (x2)--(yk-2)
      (x2)--(yk-1);
\draw (x3)--(y1)
      (x3)--(y2)
      (x3)--(yk-2)
      (x3)--(yk-1);
\draw (xk-2)--(y1)
      (xk-2)--(y2)
      (xk-2)--(y3)
      (xk-2)--(yk-1);
\path (xk-1)--node(v1)[dot][yshift=-10pt][label=\( v_1 \)][label={[vcolour]below:\( 0 \)}]{} (yk-1);
\draw (xk-1)--(y1)
      (xk-1)--(y2)
      (xk-1)--(y3)
      (xk-1)--(yk-2)
      (xk-1)--(v1)--(yk-1);

\draw (u1)--++(3.5,1) node(u2)[dot][label=below right:\( u_2 \)][label={[vcolour]above:\( k\text{-1} \)}]{};
\draw (y2)--(u2)
      (y3)--(u2)
      (yk-2)--(u2);
\draw (u2)--++(1.5,0) node(u3)[dot][label=below:\( u_3 \)]{} node[terminal][label={[vcolour]above:\( 1 \)}]{};

\draw (v1)--++(-3.5,-1) node(v2)[dot][label=above left:\( v_2 \)][label={[vcolour]below:\( k\text{-1} \)}]{};
\draw (x2)--(v2)
      (x3)--(v2)
      (xk-2)--(v2);
\draw (v2)--++(-1.5,0) node(v3)[dot][label=\( v_3 \)][label={[vcolour]below:\( 1 \)}]{}; \end{tikzpicture}
\caption{A \( k \)-rs colouring of the colour blocking gadget.}
\label{fig:eg rs colouring colour blocking gadget}
\end{figure}
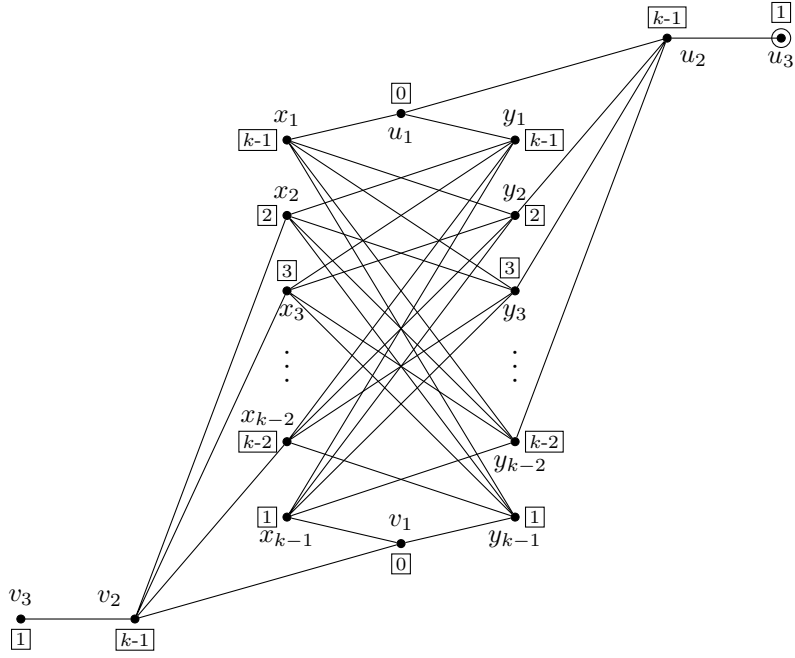

\iftoggle{forThesis}
{ The important gadget in Construction~\ref{make:k-rs colouring} is the \emph{colour blocking gadget} shown in Figure~\ref{fig:colour blocking gadget}. Let \( X=\{x_1,x_2,\dots,x_{k-1}\} \) and \( Y=\{y_1,y_2,\dots,y_{k-1}\} \). 
For \( i,j\in\{1,2,\dots,k-1\} \), \( x_i \) is adjacent to \( y_j \) except for \( j=i \) (i.e., the subgraph of the gadget induced by \( X\cup Y \) is a complete bipartite graph minus a perfect matching). 
The vertex \( u_2 \) is adjacent to vertices \( u_1,u_3,y_2,y_3,\dots,y_{k-2} \), and vertex \( v_2 \) is adjacent to vertices \( v_1,v_3,x_2,x_3,\dots,x_{k-2} \). 
Other edges in the gadget are \( u_1x_1, u_1y_1, v_1x_{k-1} \) and \( v_1y_{k-1} \). 
Consider the \( k \)-colouring of the gadget shown in Figure~\ref{fig:eg rs colouring colour blocking gadget}. 
Observe that under this colouring, the bicoloured 3-vertex paths in the gadget are \pagetarget{lnk:list of bicol P3}{\( x_1,u_1,y_1 \); \( x_1,u_1,u_2 \); \( y_1,u_1,u_2 \); \( v_2,x_2,y_1 \); \( v_2,x_3,y_1 \); \dots,\, \( v_2,x_{k-2},y_1 \); \( u_2,y_2,x_1 \); \( u_2,y_3,x_1 \); \dots,\, \( u_2,y_{k-2},x_1 \) and \( x_{k-1},v_1,y_{k-1} \).} 
Among these, the path \( x_{k-1},v_1,y_{k-1} \) have the lower colour (i.e., colour~0) in the middle, not the higher colour. 
Besides, all other bicoloured 3-vertex paths have the highest colour (i.e., colour \( k-1 \)) on their endvertices. 
Thus, none of the paths listed above is a bicoloured \( P_3 \) with the higher colour on its middle vertex. 
Hence, this \( k \)-colouring is a \( k \)-rs colouring of the colour blocking gadget. 
} { The important gadget in this construction is the \emph{colour blocking gadget} shown in Figure~\ref{fig:colour blocking gadget} (note that vertex \( u_2 \) is adjacent to vertices \( u_1,u_3,y_2,y_3,\dots,y_{k-2} \), vertex \( v_2 \) is adjacent to vertices \( v_1,v_3,x_2,x_3,\dots,x_{k-2} \), and for each \( i,j\in\{1,2,\dots,k-1\} \), \( x_i \) is adjacent to \( y_j \) except for \( j=i \)). 
} Observe that the colour blocking gadget has maximum degree~\( k-1 \). 
Lemma~\ref{lem:colour blocking gadget} attests that the name of the gadget is meaningful.
\begin{lemma}\label{lem:colour blocking gadget}
Let \( k\geq 5 \). 
Let \( f \) be a \( k \)-rs colouring of the colour blocking gadget (displayed in Figure~\ref{fig:colour blocking gadget}). 
Then, \( f(u_2)\neq 0 \), \( f(u_3)\neq 0 \), \( f(v_2)\neq 0 \), and \( f(v_3)\neq 0 \). 
In particular, \( f \) must assign a non-zero colour on the terminal of the gadget and the neighbour of the terminal in the gadget. 
\end{lemma}
\begin{proof}
To prove the lemma, let us discuss an observation. 
By the definition of rs colouring, two vertices coloured~\( c \) cannot have a common neighbour of higher colour. 
Besides, two vertices \( w_1 \) and \( w_2 \) both coloured~\( c \) cannot have two common neighbours \( w_1' \) and \( w_2' \) both coloured~\( c' \) (otherwise, path \( w_1,w_1',w_2,w_2' \) is a bicoloured \( P_4 \), and hence \( f \) is not even a star colouring, let alone a restricted star colouring). 
Therefore, we have the following. 
\setcounter{claim}{0}
\begin{claim}\label{clm:rs colouring pairwise distinct}
If two vertices \( x_i \) and \( x_j \) have colour~\( c \), then their common neighbours get pairwise distinct colours less than \( c \). 
\end{claim}

Recall that \( X=\{x_1,x_2,\dots,x_{k-1}\} \) and \( Y=\{y_1,y_2,\dots,y_{k-1}\} \). 
For convenience, we call sets \( X \) and \( Y \) as two sides. 
Observe that the sets \( X \) and \( Y \) are symmetric because `rotating' the gadget by \( 180^\circ \) gives an automorphism \( \psi \) of the gadget that maps \( X \) to \( Y \) and vice versa (define \( \psi \) as \( \psi(x_p)=y_{k-p} \) and \( \psi(y_p)=x_{k-p} \) for \( 1\leq p\leq k-1 \), and \( \psi(u_q)=v_q \) and \( \psi(v_q)=u_q \) for \( 1\leq q\leq 3 \)). 
We consider two cases: 
Case~1 when a colour repeats on side \( X \), and Case~2 when not (i.e, no colour repeats on side \( X \)); see page~\pageref{lnk:case 2 in lemma k rs col} for Case~2.\\[5pt] 
\textit{Case~1:} A colour repeats on side \( X \) (i,e., \( f(x_i)=f(x_j)=c \), where \( 1\leq i<j\leq k-1 \)).\\[5pt]
Each \( y\in Y \) is adjacent to \( x_i \) or \( x_j \) (or both), and thus \( f(y)\neq c \). 
We consider various subcases depending on the values of \( i \) and \( j \).\\[5pt] 
Subcase~1.1 (\( i=1 \))\,: \( f(x_1)=f(x_j)=c \) where \( j>1 \).\\[5pt]
Since \( u_1 \) is adjacent to \( x_1 \), we have \( f(u_1)\neq c \). 
Hence, no vertex in \( \{u_1\}\cup Y \) is coloured~\( c \) (i.e., \( c\notin f(\{u_1\}\cup Y) \)). 
Vertices in \( Y\setminus \{y_1,y_j\} \) are common neighbours of \( x_1 \) and \( x_j \). 
Thus, by Claim~\ref{clm:rs colouring pairwise distinct}, vertices in \( Y\setminus \{y_1,y_j\} \) have pairwise distinct colours less than \( c \). 
Hence, for each vertex \( y_p\in Y\setminus \{y_1,y_j\} \), we have \( f(y_p)<c \) and thus \( f(y_1)\neq f(y_p) \) (if not, the bicoloured path \( y_1,x_j,y_p \) has a higher colour on its middle vertex). 
That is, \( y_1 \) cannot get a colour used in \( Y\setminus \{y_1,y_j\} \) (i.e., \( f(y_1)\notin f(Y\setminus \{y_1,y_j\}) \)). 
Similarly, if \( f(u_1)=f(y_p) \) (resp.\ \( f(y_j)=f(y_p) \)) for some \( y_p\in Y\setminus \{y_1,y_j\} \), then the bicoloured path \( u_1,x_1,y_p \) (resp.\ \( y_j,x_1,y_p \)) has the higher colour on its middle vertex; a contradiction. 
Hence, \( u_1 \) (resp.\ \( y_j \)) cannot get a colour used in \( Y\setminus \{y_1,y_j\} \). 
That is,  \( f(u_1)\notin f(Y\setminus \{y_1,y_j\}) \) and \( f(y_j)\notin f(Y\setminus \{y_1,y_j\}) \). 
Hence, we have the following. 
\begin{claim}[of Subcase~1.1]\label{clm:rs colouring k-1 u1 y1 yj}
\( f(y_1)\notin f(Y\setminus \{y_1,y_j\}) \), \( f(u_1)\notin f(Y\setminus \{y_1,y_j\}) \) and \( f(y_j)\notin f(Y\setminus \{y_1,y_j\}) \). 
\end{claim}
Since \( u_1y_1 \) is an edge, \( f(u_1)\neq f(y_1) \). 
We also know that \( f(y_1)\notin f(Y\setminus \{y_1,y_j\}) \), \( f(u_1)\notin f(Y\setminus \{y_1,y_j\}) \), and vertices in \( Y\setminus \{y_1,y_j\} \) have pairwise distinct colours. 
Hence, vertices in \( \{u_1\}\cup Y\setminus \{y_j\} \) have pairwise distinct colours. 
Besides, we know that \( |\{u_1\}\cup Y\setminus \{y_j\}|=k-1 \) and \( c\notin f(\{u_1\}\cup Y) \). 
Thus, we have the following claim. 
\begin{claim}[of Subcase~1.1]\label{clm:rs colouring k-1 u1 Y no yj}
Vertices in \( \{u_1\}\cup Y\setminus \{y_j\} \) get a permutation of the \( k-1 \) colours \( 0,\dots,c-1,c+1,\dots,k-1 \). 
\end{claim}
Since \( f(y_j)\neq c \),\; \( f(y_j)\in \{0,\dots,c-1,c+1,\dots,k-1\}=f(\{u_1\}\cup Y\setminus \{y_j\}) \) (see Claim~\ref{clm:rs colouring k-1 u1 Y no yj}). 
Since \( f(y_j)\notin f(Y\setminus \{y_1,y_j\}) \) (see Claim~\ref{clm:rs colouring k-1 u1 y1 yj}) and \( f(y_j)\in f(\{u_1\}\cup Y\setminus \{y_j\}) \), we have the following claim. 
\begin{claim}[of Subcase~1.1]\label{clm:rs colouring yj from u1 or y1}
\( f(y_j)\in \{f(u_1),f(y_1)\} \). 
\end{claim}
Let us consider the colour of an arbitrary vertex \( x_p\in X\setminus \{x_1,x_j\} \). 
By Claim~\ref{clm:rs colouring k-1 u1 Y no yj}, all \( k \) colours except \( c \) are used in \( \{u_1\}\cup Y\setminus \{y_j\} \). 
In particular, all \( k \) colours except \( c \) are used in \( \{u_1\}\cup Y \), and thus \( f(x_p)\in \{c,f(u_1)\}\cup f(Y) \). 
Since \( x_p \) is adjacent to every vertex in \( Y\setminus \{y_p\} \), we have \( f(x_p)\notin f(Y\setminus \{y_p\}) \). Therefore, \( f(x_p)\in \{c,f(u_1),f(y_p)\} \).

We show that \( f(x_p)=f(u_1) \) leads to a contradiction. 
Suppose that \( f(x_p)=f(u_1) \). 
Then, \( f(u_1)=f(x_p)\neq f(y_j) \) (because \( x_py_j \) is an edge). 
Since \( f(y_j)\neq f(u_1) \), \( f(y_j)=f(y_1) \) by Claim~\ref{clm:rs colouring yj from u1 or y1}. 
Hence, path \( u_1,y_1,x_p,y_j \) is a bicoloured \( P_4 \), a contradiction. 
Thus, by contradiction, \( f(x_p)\neq f(u_1) \).

Next, we show that \( f(x_p)=c \) leads to a contradiction. 
Suppose that \( f(x_p)=c \). 
That is, \( f(x_p)=c=f(x_1)=f(x_j) \); in particular, \( f(x_p)=f(x_1) \) and \( f(x_p)=f(x_j) \). 
By Claim~\ref{clm:rs colouring yj from u1 or y1}, either \( f(y_j)=f(u_1) \) or \( f(y_j)=f(y_1) \). 
As a result, either path \( u_1,x_1,y_j,x_p \) or path \( x_j,y_1,x_p,y_j \) is a bicoloured \( P_4 \).  
Thus, by contradiction, \( f(x_p)\neq c \).

Therefore, the only possibility is \( f(x_p)=f(y_p) \). 
Since \( x_p \) is arbitrary, we have the following claim. 
\begin{claim}[of Subcase~1.1]\label{clm:rs colouring xp same as yp}
\( f(x_p)=f(y_p) \) for each \( p\in\{1,2,\dots,k-1\}\setminus\{1,j\} \). 
\end{claim} 
We consider two subcases based on the value of \( j \).\\ 
Subcase~1.1.1: \( j\neq k-1 \).\\[5pt]
By Claim~\ref{clm:rs colouring xp same as yp}, \( f(x_{k-1})=f(y_{k-1}) \) (note that \( j\neq k-1 \)). 

Consider the colour at \( v_1 \). 
Since all \( k \) colours except \( c \) are used in \( \{u_1\}\cup Y \) (see Claim~\ref{clm:rs colouring k-1 u1 Y no yj}), \( f(v_1)\in \{c,f(u_1)\}\cup f(Y) \). 
Since \( v_1y_{k-1} \) is an edge, \( f(v_1)\neq f(y_{k-1}) \). 
For \( 1\leq p\leq k-2 \),\; \( f(v_1)\neq f(y_p) \) (otherwise, path \( y_p,x_{k-1},v_1,y_{k-1} \) is a bicoloured \( P_4 \)). 
Thus, \( f(v_1)\notin f(Y) \). 
Hence, \( f(v_1)\in \{c,f(u_1)\} \).

We show by contradiction that \( f(v_1)\neq c \). 
Suppose that \( f(v_1)=c \). 
As a result, \( f(v_1)=f(x_1) \). 
Since \( f(x_{k-1})=f(y_{k-1}) \), path \( x_{k-1},v_1,y_{k-1},x_1 \) is a bicoloured \( P_4 \), a contradiction. 
Thus, by contradiction, \( f(v_1)\neq c \). 
Since \( f(v_1)\in \{c,f(u_1)\} \), it follows that \( f(v_1)=f(u_1) \).

Observe that \( f(y_j)\neq f(v_1) \) (otherwise, path \( y_j,x_{k-1},v_1,y_{k-1} \) is a bicoloured \( P_4 \)). 
Since \( f(v_1)=f(u_1) \), this implies that \( f(y_j)\neq f(u_1) \). 
Since \( f(y_j)\in \{f(u_1),f(y_1)\} \) (see Claim~\ref{clm:rs colouring yj from u1 or y1}) and \( f(y_j)\neq f(u_1) \), we have \( f(y_j)=f(y_1) \). 
Since vertices \( y_1 \) and \( y_j \) have the same colour and they are at distance two from each other, \( f(y_1)=f(y_j)\neq 0 \) by Observation~\ref{obs:no colour 0 at distance 2}. 
Similarly, \( f(x_{k-1})=f(y_{k-1})\neq 0 \). 
By Claim~\ref{clm:rs colouring k-1 u1 Y no yj}, vertices in \( \{u_1\}\cup Y\setminus \{y_j\} \) get a permutation of colours \( 0,\dots,c-1,c+1,\dots,k-1 \). 
In particular, one of the vertices in \( \{u_1\}\cup Y\setminus \{y_j\} \) is coloured~0. 
Since \( f(y_1)\neq 0 \) and \( f(y_{k-1})\neq 0 \), one of the vertices in \( \{u_1\}\cup Y\setminus \{y_1,y_j,y_{k-1}\}=\{u_1,y_2,\dots,y_{j-1},y_{j+1},\dots,y_{k-2}\} \) is coloured~0. 
Since \( u_2 \) and \( u_3 \) are within distance two from each of these vertices, \( f(u_2)\neq 0 \) and \( f(u_3)\neq 0 \) (by Observation~\ref{obs:no colour 0 at distance 2}). 
Since \( f(v_1)=f(u_1) \) (see previous paragraph) and \( f(x_p)=f(y_p) \) for each \( p\in\{1,2,\dots,k-1\}\setminus\{1,j\} \) (see Claim~\ref{clm:rs colouring xp same as yp}), one of the vertices \( v_1,x_2,\dots,x_{j-1},x_{j+1},\dots,x_{k-2} \) is coloured~0. 
Since \( v_2 \) and \( v_3 \) are within distance two from each of these vertices, \( f(v_2)\neq 0 \) and \( f(v_3)\neq 0 \) (by Observation~\ref{obs:no colour 0 at distance 2}). 
Therefore, \( f(u_2)\neq 0 \), \( f(u_3)\neq 0 \), \( f(v_2)\neq 0 \) and \( f(v_3)\neq 0 \).\\

\noindent Subcase~1.1.2: \( j=k-1 \).\\[5pt]
Due to Claim~\ref{clm:rs colouring k-1 u1 Y no yj}, all \( k \) colours except \( c \) are used in \( \{u_1\}\cup Y \).
Consider the colour at~\( v_1 \). 
Clearly, \( f(v_1)\in f(\{u_1\}\cup Y) \) and \( f(v_1)\neq f(y_{k-1}) \). 
For each \( y_p\in Y\setminus \{y_1,y_{k-1}\} \), we know that \( f(y_p)<c \) (due to Claim~\ref{clm:rs colouring pairwise distinct}), and hence \( f(v_1)\neq f(y_p) \) (otherwise, the bicoloured path \( v_1,x_{k-1},y_p \) has a higher colour on its middle vertex). 
Hence, \( f(v_1)\in\{f(u_1),f(y_1)\} \). 
By Claim~\ref{clm:rs colouring yj from u1 or y1}, \( f(y_{k-1})=f(y_j)\in\{f(u_1),f(y_1)\} \). 
Thus, \( f(v_1), f(y_{k-1}) \in \{f(u_1),f(y_1)\} \). 
Since \( v_1y_{k-1} \) is an edge, we have the following. 
\begin{claim}[of Subcase~1.1.2]\label{clm:v1 yk-1 same as u1 y1}
\( \{f(v_1),f(y_{k-1})\}=\{f(u_1),f(y_1)\} \). 
\end{claim}
Since \( f(y_{k-1})\in\{f(u_1),f(y_1)\} \) and \( u_1 \) (resp.\ \( y_1 \)) is at distance two from \( y_{k-1} \), we have \( f(y_{k-1})\neq 0 \) by Observation~\ref{obs:no colour 0 at distance 2}. 
Similarly, \( y_1 \) is at distance two from  both \( v_1 \) and \( y_{k-1} \), and \( f(y_1)\in\{f(v_1),f(y_{k-1})\} \); hence, \( f(y_1)\neq 0 \). 
Since \( f(y_1)\neq 0 \) and \( f(y_{k-1})\neq 0 \), one of the vertices \( u_1,y_2,y_3,\dots,y_{k-2} \) is coloured~0 (thanks to Claim~\ref{clm:rs colouring k-1 u1 Y no yj}). 
Since \( u_2 \) and \( u_3 \) are within distance two from each of these vertices, \( f(u_2)\neq 0 \) and \( f(u_3)\neq 0 \) (by Observation~\ref{obs:no colour 0 at distance 2}). 
We need to prove that \( f(v_2)\neq 0 \) and \( f(v_3)\neq 0 \) as well.

Since \( f(x_p)=f(y_p) \) for each \( p\notin\{1,k-1\} \) (see Claim~\ref{clm:rs colouring xp same as yp}), one of the vertices \( u_1,x_2,x_3,\dots,x_{k-2} \) is coloured~0. 
Clearly, \( v_2 \) and \( v_3 \) are within distance two from each of the vertices \( x_2,x_3,\dots,x_{k-2} \). 
If \( f(u_1)\neq 0 \), then one of the vertices \( x_2,x_3,\dots,x_{k-2} \) is coloured~0, and thus \( f(v_2)\neq 0 \) and \( f(v_3)\neq 0 \) by Observation~\ref{obs:no colour 0 at distance 2}. 
Hence, it suffices to prove that \( f(v_2)\neq 0 \) and \( f(v_3)\neq 0 \) when \( f(u_1)=0 \).

Suppose that \( f(u_1)=0 \). 
By Claim~\ref{clm:v1 yk-1 same as u1 y1}, \( \{f(v_1),f(y_{k-1})\}=\{f(u_1),f(y_1)\} \) and in particular, \( f(y_{k-1})\in \{f(u_1),f(y_1)\} \). 
Since \( f(u_1)=0 \), \( f(y_{k-1})\neq f(u_1) \) by Observation~\ref{obs:no colour 0 at distance 2}, and thus \( f(y_{k-1})=f(y_1) \). 
Since \( \{f(v_1),f(y_{k-1})\}=\{f(u_1),f(y_1)\} \) and \( f(y_{k-1})=f(y_1) \), we have \( f(v_1)=f(u_1) \). 
Therefore, \( f(v_1)=f(u_1)=0 \).  
Since \( v_2 \) and \( v_3 \) are within distance two from \( v_1 \),  we have \( f(v_2)\neq 0 \) and \( f(v_3)\neq 0 \) by Observation~\ref{obs:no colour 0 at distance 2}.
This proves that \( f(v_2)\neq 0 \) and \( f(v_3)\neq 0 \) when \( f(u_1)=0 \). 
Therefore, \( f(u_2)\neq 0 \), \( f(u_3)\neq 0 \), \( f(v_2)\neq 0 \) and \( f(v_3)\neq 0 \).\\
 
\noindent Subcase~1.2: \( f(x_i)=f(x_j)=c \) where \( 2\leq i<j\leq k-2 \).\\[5pt]
Vertices in \( \{v_2\}\cup Y\setminus \{y_i,y_j\} \) are common neighbours of \( x_i \) and \( x_j \). 
Thus, by Claim~\ref{clm:rs colouring pairwise distinct}, vertices in \( \{v_2\}\cup Y\setminus \{y_i,y_j\} \) get pairwise distinct colours less than~\( c \). 
For each \( w\in\{v_2\}\cup Y\setminus \{y_i,y_j\} \), we have \( f(y_i)\neq f(w) \) (if not, path \( x_i,w,x_j,y_i \) is a bicoloured \( P_4 \)). 
We also know that \( f(y_i)\neq c \). 
Therefore, vertices in \( \{v_2\}\cup Y\setminus \{y_j\} \) get a permutation of the \( k-1 \) colours \( 0,\dots,c-1,c+1,\dots,k-1 \). 
Similarly, \( f(y_j)\notin \{c,f(w)\} \) for each \( w\in\{v_2\}\cup Y\setminus \{y_i,y_j\} \) and hence vertices in \( \{v_2\}\cup Y\setminus \{y_i\} \) get a permutation of colours \( 0,\dots,c-1,c+1,\dots,k-1 \). 
Since \( f(\{v_2\}\cup Y\setminus \{y_j\}) \) and \( f(\{v_2\}\cup Y\setminus \{y_i\}) \) are both permutations of \( \{0,\dots,c-1,c+1,\dots,k-1\} \), we have \( f(y_i)=f(y_j) \). 
That is, \( f(x_i)=f(x_j)=c \) and \( f(y_i)=f(y_j)\neq c \).

By symmetry of sides \( X \) and \( Y \), we assume without loss of generality that \( c<f(y_i) \). 
Since \( c<f(y_i)\leq k-1 \), we have \( c\leq k-2 \). 
Since \( x_i \) and \( x_j \) have \( k-2 \) common neighbours (namely, vertices in \( \{v_2\}\cup Y\setminus \{y_i,y_j\} \)) and they require pairwise disjoint colours less than \( c \) (see Claim~\ref{clm:rs colouring pairwise distinct}), \( c=k-2 \) and vertices in \( \{v_2\}\cup Y\setminus \{y_i,y_j\} \) get a permutation of colours \( 0,1,\dots,k-3 \). 
Since \( f(y_i)=f(y_j)>c=k-2 \), we have \( f(y_i)=f(y_j)=k-1 \). 
\begin{claim}[of Subcase~1.2]\label{clm:xi xj yi yj colours}
\( f(x_i)=f(x_j)=c=k-2 \) and \( f(y_i)=f(y_j)=k-1 \). 
\end{claim}
Since vertices in \( \{v_2\}\cup Y\setminus \{y_i,y_j\} \) get a permutation of colours \( 0,1,\dots,k-3 \) (see previous paragraph) and \( f(y_i)=f(y_j)=k-1 \), we have the following claim. 
\begin{claim}[of Subcase~1.2]\label{clm:subcase 1.2 permutation of col}
Colours \( 0,1,\dots,k-3,k-1 \) are used in \( \{v_2\}\cup Y \).
\end{claim}
Consider the colour at \( x_1 \). 
Due to Claim~\ref{clm:subcase 1.2 permutation of col}, \( f(x_1)\in \{k-2\}\cup f(\{v_2\}\cup Y) \). 
Since \( f(x_1)\in \{k-2\}\cup f(\{v_2\}\cup Y) \) and \( x_1 \) is adjacent to every vertex in \( Y \) except \( y_1 \), we have \( f(x_1)\in \{k-2,f(v_2),f(y_1)\} \). 
Observe that \( f(x_1)\neq k-2 \) (if not, the bicoloured path \( x_1,y_j,x_i \) has the higher colour on its middle vertex). 
Hence, \( f(x_1)\in\{f(v_2),f(y_1)\} \). 
Similarly, \( f(x_{k-1})\in\{f(v_2),f(y_{k-1})\} \). 
\begin{claim}[of Subcase~1.2]\label{clm:subcase 1.2 rs col of x1 xk-1}
\( f(x_1)\in\{f(v_2),f(y_1)\} \) and \( f(x_{k-1})\in\{f(v_2),f(y_{k-1})\} \). 
\end{claim}
We show that \( f(x_{k-1})\neq f(v_2) \). 
We prove this by considering two scenarios: (i)~\( f(x_1)\neq f(y_1) \), and (ii)~\( f(x_1)=f(y_1) \).

Suppose that \( f(x_1)\neq f(y_1) \). 
Since \( f(x_1)\in\{f(v_2),f(y_1)\} \), this implies that \( f(x_1)=f(v_2) \). 
Hence, \( f(x_{k-1})\neq f(v_2) \) (if \( f(x_{k-1})=f(v_2)=f(x_1) \), then path \( x_1,y_i,x_{k-1},y_j \) is a bicoloured~\( P_4 \), and thus \( f \) is not even a star colouring). 
Thus, by contradiction, \( f(x_1)\neq f(y_1) \) implies that \( f(x_{k-1})\neq f(v_2) \).

Suppose that \( f(x_1)=f(y_1) \). 
Consider the colour at \( u_1 \). 
By Claim~\ref{clm:subcase 1.2 permutation of col}, colours \( 0,1,\dots,k-3,k-1 \) are used in \( \{v_2\}\cup Y \). 
Hence, \( f(u_1)\in \{k-2\}\cup f(\{v_2\}\cup Y) \). 
Observe that \( f(u_1)\neq k-2 \) (otherwise, \( x_1,u_1,y_1,x_i \) is a bicoloured \( P_4 \)), and thus \( f(u_1)\in f(\{v_2\}\cup Y) \). 
Since \( u_1y_1 \) is an edge, \( f(u_1)\neq f(y_1) \). 
For \( 2\leq p\leq k-1 \),\; \( f(u_1)\neq f(y_p) \) (otherwise, \( y_p,x_1,u_1,y_1 \) is a bicoloured \( P_4 \)). 
Thus, \( f(u_1)=f(v_2) \). 
As a result, \( f(x_{k-1})\neq f(v_2) \) (if \( f(x_{k-1})=f(v_2)=f(u_1) \), then path \( x_1,u_1,y_1,x_{k-1} \) is a bicoloured \( P_4 \)). 
Thus, by contradiction, \( f(x_1)=f(y_1) \) implies that \( f(x_{k-1})\neq f(v_2) \).

Hence, whether \( f(x_1)=f(y_1) \) or not, \( f(x_{k-1})\neq f(v_2) \). 
Since \( f(x_{k-1})\neq f(v_2) \) and \( f(x_{k-1})\in\{f(v_2),f(y_{k-1})\} \) (see Claim~\ref{clm:subcase 1.2 rs col of x1 xk-1}), we have \( f(x_{k-1})=f(y_{k-1}) \). 
Consider the colour at \( v_1 \). 
Due to Claim~\ref{clm:subcase 1.2 permutation of col}, \( f(v_1)\in \{k-2\}\cup f(\{v_2\}\cup Y) \). 
Observe that \( f(v_1)\neq k-2 \) (otherwise, \( x_{k-1},v_1,y_{k-1},x_i \) is a bicoloured \( P_4 \)), and thus \( f(v_1)\in f(\{v_2\}\cup Y) \). 
Clearly, \( f(v_1)\neq f(v_2) \) and \( f(v_1)\neq f(y_{k-1}) \). 
For \( 1\leq p\leq k-2 \),\; \( f(v_1)\neq f(y_p) \) (otherwise, \( y_p,x_{k-1},v_1,y_{k-1} \) is a bicoloured \( P_4 \)). 
Therefore, no colour is available for vertex \( v_1 \), and thus Subcase~1.2 leads to a contradiction.\\

\noindent Subcase~1.3: \( f(x_i)=f(x_{k-1}) \) where \( 2\leq i\leq k-2 \).\\[5pt]
We know that rotating the gadget by \( 180^\circ \) gives an automorphism \( \psi \) of the gadget such that \( \psi(x_p)=y_{k-p} \) and \( \psi(y_p)=x_{k-p} \) for \( 1\leq p\leq k-1 \). 
Hence, it suffices to consider the case \( f(y_1)=f(y_{k-i})=c \) where \( 2\leq i\leq k-2 \). 
In other words, it suffices to consider the case \( f(y_1)=f(y_j)=c \) where \( 2\leq j\leq k-2 \). 
We can use arguments similar to that in Subcase~1.1 to prove that \( f(u_2)\neq 0 \), \( f(u_3)\neq 0 \), \( f(v_2)\neq 0 \) and \( f(v_3)\neq 0 \); an alternate argument is given below for completeness.

Suppose that \( f(y_1)=f(y_j)=c \) where \( 2\leq j\leq k-2 \). 
Clearly, no vertex in \( X \) is coloured \( c \). 
By Claim~\ref{clm:rs colouring pairwise distinct}, the common neighbours of \( y_1 \) and \( y_j \), namely vertices in \( X\setminus \{x_1,x_j\} \), get pairwise distinct colours less than \( c \). 
Moreover, for each \mbox{\( x_p\in X\setminus \{x_1,x_j\} \)}, we have \( f(x_p)\neq f(x_1) \), \( f(x_p)\neq f(x_j) \) and \( f(x_p)\neq f(u_1) \) (otherwise, path \( y_1,x_p,y_j,x_1 \), path \( x_j,y_1,x_p,y_j \), or path \( u_1,y_1,x_p,y_j \) respectively is a bicoloured~\( P_4 \)). 
If \( f(x_1)=f(x_j) \), then \( f(u_2)\neq 0 \), \( f(u_3)\neq 0 \), \( f(v_2)\neq 0 \) and \( f(v_3)\neq 0 \) by Subcase~1.1. 
So, it suffices to consider the scenario \( f(x_1)\neq f(x_j) \). 
Since \( f(x_1)\neq f(x_j) \) and vertices in \( X\setminus \{x_1,x_j\} \) have pairwise disjoint colours different from \( f(x_1) \) and~\( f(x_j) \), vertices \( x_1,x_2,\dots,x_{k-1} \) get a permutation of colours \( 0,\dots,c-1,c+1,\dots,k-1 \) (i.e., \( f(X)=\{0,\dots,c-1,c+1,\dots,k-1\} \)). 
Hence, \( f(u_1)\in \{c\}\cup f(X) \). 
Since \( f(u_1)\neq f(y_1)=c \) and \( f(u_1)\neq f(x_1) \),\; \( f(u_1)\in f(X\setminus \{x_1\}) \). 
Since \( f(u_1)\neq f(x_p) \) for each \( x_p\in X\setminus \{x_1,x_j\} \), we have \( f(u_1)=f(x_j) \). 
, 
For each \( y_p\in Y\setminus \{y_1,y_j\} \), we have \( f(y_p)\in \{c\}\cup f(X) \) (because \( f(X)=\{0,\dots,c-1,c+1,\dots,k-1\} \)) and \( f(y_p)\neq c \) (if not, path \( u_1,y_1,x_j,y_p \) is a bicoloured \( P_4 \)); that is, \( f(y_p)\in f(X) \). 
Since \( y_p \) is adjacent all vertices in \( X\setminus \{x_p\} \), we have \( f(y_p)=f(x_p) \) for all \( p\notin\{1,j\} \). 
In particular, since \( j\neq k-1 \), we have \( f(x_{k-1})=f(y_{k-1}) \). 
Clearly, \( f(v_1)\in \{c\}\cup f(X) \) and \( f(v_1)\neq f(x_{k-1}) \). 
If \( f(v_1)=f(x_p) \) for \( p<k-1 \) (resp.\ \( f(v_1)=c \)), then path \( x_{k-1},v_1,y_{k-1},x_p \) (resp.\ path \( y_1,x_{k-1},v_1,y_{k-1} \)) is a bicoloured \( P_4 \). 
Consequently, no colour is available for vertex \( v_1 \), a contradiction. 
~\\

\phantomsection
\noindent \textit{Case 2:} No colour repeats on side \( X \). \label{lnk:case 2 in lemma k rs col}\\[5pt]
By Symmetry of sides \( X \) and \( Y \), we may assume that no colour repeats on side \( Y \) either. 
Clearly, there exists a colour \( c \) such that vertices \( x_1,x_2,\dots,x_{k-1} \) get a permutation of colours \( 0,\dots,c-1,c+1,\dots,k-1 \) (i.e., \( f(X)=\{0,\dots,c-1,c+1,\dots,k-1\} \)). 
For \( 1\leq p\leq k-1 \), since \( f(y_p)\in \{c\}\cup f(X) \) and \( y_p \) is adjacent to every vertex in \( X\setminus \{x_p\} \), we have \( f(y_p)\in\{c,f(x_p)\} \). 
\begin{claim}[of Case~2]\label{clm:case 2 rs col y1 yk-1}
\( f(y_p)\in\{c,f(x_p)\} \) for \( 1\leq p\leq k-1 \). 
In particular, we have \( f(y_1)\in\{c,f(x_1)\} \) and \( f(y_{k-1})\in\{c,f(x_{k-1})\} \). 
\end{claim}
We show that if \( f(y_1)=f(x_1) \), then \( f(u_1)=c \). 
Suppose that \( f(y_1)=f(x_1) \). 
For \( 2\leq p\leq k-1 \), we have \( f(u_1)\neq f(x_p) \) (if not, \( x_1,u_1,y_1,x_p \) is a bicoloured \( P_4 \)). 
Hence, \( f(u_1)\notin f(X)=\{0,\dots,c-1,c+1,\dots,k-1\} \); i.e., \( f(u_1)=c \). 
This proves that if \( f(y_1)=f(x_1) \), then \( f(u_1)=c \). 
Similarly, if \( f(y_{k-1})=f(x_{k-1}) \), then \( f(v_1)=c \). 
\begin{claim}[of Case~2]\label{clm:case 2 rs col imlication chain 1}
If \( f(y_1)=f(x_1) \), then \( f(u_1)=c \). 
If \( f(y_{k-1})=f(x_{k-1}) \), then \( f(v_1)=c \). 
\end{claim}

Next, we show that \( f(y_1)=c \) leads to a contradiction. 
Suppose that \( f(y_1)=c \). 
Since no colour repeats on side \( Y \), \( f(y_{k-1})\neq c \), and thus \( f(y_{k-1})=f(x_{k-1}) \) by Claim~\ref{clm:case 2 rs col y1 yk-1}.
By Claim~\ref{clm:case 2 rs col imlication chain 1}, this implies that \( f(v_1)=c \). 
Thus, the path \( y_1,x_{k-1},v_1,y_{k-1} \) is a bicoloured \( P_4 \), a contradiction.

Since \( f(y_1)=c \) leads to a contradiction, \( f(y_1)=f(x_1) \) by Claim~\ref{clm:case 2 rs col y1 yk-1}. 
Similarly, \( f(y_{k-1})=c \) leads to contradiction, and thus \( f(y_{k-1})=f(x_{k-1}) \). 
Since \( f(y_1)=f(x_1) \) and \( f(y_{k-1})=f(x_{k-1}) \), we have \( f(u_1)=f(v_1)=c \) by Claim~\ref{clm:case 2 rs col imlication chain 1}. 
Let \( c_1=f(y_1) \) and \( c_2=f(y_{k-1}) \). 
Clearly, \( f(x_1)=f(y_1)=c_1 \) and \( f(x_{k-1})=f(y_{k-1})=c_2 \). 
We also know that \( f(X)=\{0,\dots,c-1,c+1,\dots,k-1\} \) and \( f(u_1)=f(v_1)=c \). 
Hence, vertices in \( \{v_1\}\cup X\setminus \{x_1,x_{k-1}\} \) get a permutation of colours \( \{0,1,\dots,k-1\}\setminus \{c_1,c_2\} \). 
For \( 2\leq p\leq k-2 \), we have \( f(y_p)\neq c \) (if not, path \( y_p,x_1,u_1,y_1 \) is a bicoloured \( P_4 \)) and thus \( f(y_p)=f(x_p) \) by Claim~\ref{clm:case 2 rs col y1 yk-1}. 
Hence, vertices in \( \{u_1\}\cup Y\setminus \{y_1,y_{k-1}\} \) get a permutation of colours \( \{0,1,\dots,k-1\}\setminus \{c_1,c_2\} \). 
Applying Observation~\ref{obs:no colour 0 at distance 2} on the path \( x_1,u_1,y_1 \) reveals that \( c_1\neq 0 \). 
Similarly, \( c_2\neq 0 \) (consider path \( x_{k-1},v_1,y_{k-1} \)). 
As a result, colour~0 is assigned to some vertex in \( \{v_1\}\cup X\setminus \{x_1,x_{k-1}\} \). 
Since \( v_2 \) and \( v_3 \) are within distance two from each vertex in \( \{v_1\}\cup X\setminus \{x_1,x_{k-1}\} \), \( f(v_2)\neq 0 \) and \( f(v_3)\neq 0 \) by Observation~\ref{obs:no colour 0 at distance 2}. 
Similarly, colour~0 is assigned to some vertex in \( \{u_1\}\cup Y\setminus \{y_1,y_{k-1}\} \), and thus \( f(u_2)\neq 0 \) and \( f(u_3)\neq 0 \) by Observation~\ref{obs:no colour 0 at distance 2}.
Thus, \( f(u_2)\neq 0 \), \( f(u_3)\neq 0 \), \( f(v_2)\neq 0 \) and \( f(v_3)\neq 0 \) in Case~2 as well.

Therefore, in both cases, \( f(u_2)\neq 0 \), \( f(u_3)\neq 0 \), \( f(v_2)\neq 0 \) and \( f(v_3)\neq 0 \). 
\end{proof}

\begin{construct}\label{make:k-rs colouring}
\emph{Parameter:} An integer \( k\geq 5 \).\\
\emph{Input:} A triangle-free graph \( G \) of maximum degree \( k-2 \).\\
\emph{Output:} A triangle-free graph \( G' \) of maximum degree \( k-1 \).\\
\emph{Guarantee:} \( G \) is \( (k-2) \)-rs colourable if and only if \( G' \) is \( k \)-rs colourable.\\
\emph{Steps:}\\
Introduce a copy of \( G \). 
For each vertex \( w \) in the copy of \( G \), attach \( k-1-\deg_G(w) \) colour blocking gadgets one by one at \( w \) (a colour blocking gadget is attached at \( w \) by identifying the terminal of the gadget with~\( w \); see Section~\ref{sec:def} for the definition of vertex identification). 
\end{construct}

Observe that \( \deg_{G'}(w)=k-1 \) for each \( w\in V(G) \) because \( w \) has exactly \( \deg_G(w) \) neighbours within the copy of \( G \), and \( w \) has exactly one neighbour in each of the \( k-1-\deg_G(w) \) colour blocking gadgets attached at \( w \) in \( G' \). 
Moreover, each non-terminal vertex in a colour blocking gadget has degree at most \( k-1 \) (in \( G' \)). 
Hence, \( G' \) has maximum degree \( k-1 \). 
It is easy to observe that the colour blocking gadget is triangle-free. 
Since \( G \) is triangle-free and \( G' \) is obtained from \( G \) by attaching copies of colour blocking gadgets, \( G' \) is triangle-free as well. 
\begin{figure}[hbt]
\centering
\begin{tikzpicture}
\path (0,0) node(x1)[dot][label=\( x_1 \)][label={[vcolour]left:\( k\text{-1} \)}]{} ++(0,-1) node(x2)[dot][label=\( x_2 \)][label={[vcolour,xshift=1pt,label distance=9pt]below:1}]{} ++(0,-2) node(xmid-1)[dot][label={[vcolour,font=\scriptsize,xshift=-2pt,label distance=9pt]above:\( c\,\text{-1} \)}]{} ++(0,-1) node(xmid+1)[dot][label={[vcolour,font=\scriptsize,xshift=3pt,label distance=10pt]below:\( c\text{+1} \)}]{} ++(0,-2.25) node(xk-2)[dot][label={[label distance=-2pt]left:\( x_{k-2} \)}][label={[vcolour,font=\scriptsize,xshift=-2pt,label distance=12pt]above:\( k\,\text{-2} \)}]{} ++(0,-1) node(xk-1)[dot][label=below:\( x_{k-1} \)][label={[vcolour]left:\( c \)}]{};
\path (x2)--node[sloped,font=\large]{\( \dots \)} (xmid-1);
\path (xmid+1)--node[sloped,font=\large]{\( \dots \)} (xk-2);
\path (x1) ++(3,0) node(y1)[dot][label=\( y_1 \)][label={[vcolour]right:\( k\text{-1} \)}]{} ++(0,-1) node(y2)[dot][label=\( y_2 \)][label={[vcolour,label distance=7pt]below:1}]{} ++(0,-2)  node(ymid-1)[dot][label={[vcolour,font=\scriptsize,xshift=-2pt,label distance=9pt]above:\( c\,\text{-1} \)}]{} ++(0,-1) node(ymid+1)[dot][label={[vcolour,font=\scriptsize,label distance=6pt]below:\( c\text{+1} \)}]{} ++(0,-2.25) node(yk-2)[dot][label={[yshift=-2pt]right:\( y_{\mystrut k-2} \)}][label={[vcolour,font=\scriptsize,label distance=12pt]above:\( k\text{-2} \)}]{} ++(0,-1) node(yk-1)[dot][label=below:\( y_{k-1} \)][label={[vcolour]right:\( c \)}]{};
\path (y2)--node[sloped,font=\large]{\( \dots \)} (ymid-1);
\path (ymid+1)--node[sloped,font=\large]{\( \dots \)} (yk-2);

\path (x1)--node(u1)[dot][yshift=10pt][label=above:\( u_1 \)][label={[vcolour]below:0}]{} (y1);
\draw (x1)--(u1)--(y1)
      (x1)--(y2)
      (x1)--(ymid-1)
      (x1)--(ymid+1)
      (x1)--(yk-2)
      (x1)--(yk-1);
\draw (x2)--(y1)
      (x2)--(ymid-1)
      (x2)--(ymid+1)
      (x2)--(yk-2)
      (x2)--(yk-1);
\draw (xmid-1)--(y1)
      (xmid-1)--(y2)
      (xmid-1)--(ymid+1)
      (xmid-1)--(yk-2)
      (xmid-1)--(yk-1);
\draw (xmid+1)--(y1)
      (xmid+1)--(y2)
      (xmid+1)--(ymid-1)
      (xmid+1)--(yk-2)
      (xmid+1)--(yk-1);
\draw (xk-2)--(y1)
      (xk-2)--(y2)
      (xk-2)--(ymid-1)
      (xk-2)--(ymid+1)
      (xk-2)--(yk-1);
\path (xk-1)--node(v1)[dot][yshift=-10pt][label=below:\( v_1 \)][label={[vcolour]above:0}]{} (yk-1);
\draw (xk-1)--(y1)
      (xk-1)--(y2)
      (xk-1)--(ymid-1)
      (xk-1)--(ymid+1)
      (xk-1)--(yk-2)
      (xk-1)--(v1)--(yk-1);

\draw (u1)--++(3.5,1) node(u2)[dot][label=below right:\( u_2 \)][label={[vcolour]above:\( k\text{-1} \)}]{};
\draw (y2)--(u2)
      (ymid-1)--(u2)
      (ymid+1)--(u2)
      (yk-2) to[bend right=10] (u2);
\draw (u2)--++(2.5,0) node(u3)[dot]{} node[terminal][label={below:\( u_3=w \)}][label={[vcolour]above:\( c \)}]{};

\draw (v1)--++(-3.5,-1) node(v2)[dot][label=above left:\( v_2 \)][label={[vcolour]below:\( k\text{-1} \)}]{};
\draw (x2)--(v2)
      (xmid-1)--(v2)
      (xmid+1)--(v2)
      (xk-2)--(v2);
\draw (v2)--++(-2.5,0) node(v3)[dot][label=above:\( v_3 \)][label={[vcolour]below:\( c \)}]{}; \end{tikzpicture}
\caption[A \( k \)-rs colouring scheme for the colour blocking gadget.]{A \( k \)-rs colouring scheme for the colour blocking gadget, where \( c=f(w) \). 
Note that \( 0<c<k-1 \) because \( f \) uses only colours \( 1,2,\dots,k-2 \). 
To be clear, if \( c=1 \), then \( f(x_i)=f(y_i)=i \) for \( 2\leq i\leq k-2 \) (shown in Figure~\ref{fig:eg rs colouring colour blocking gadget}). 
Similarly, if \( c=k-2 \), then \( f(x_i)=f(y_i)=i-1 \) for \( 2\leq i\leq k-2 \).}
\label{fig:rs colouring of colour blocking gadget}
\end{figure}
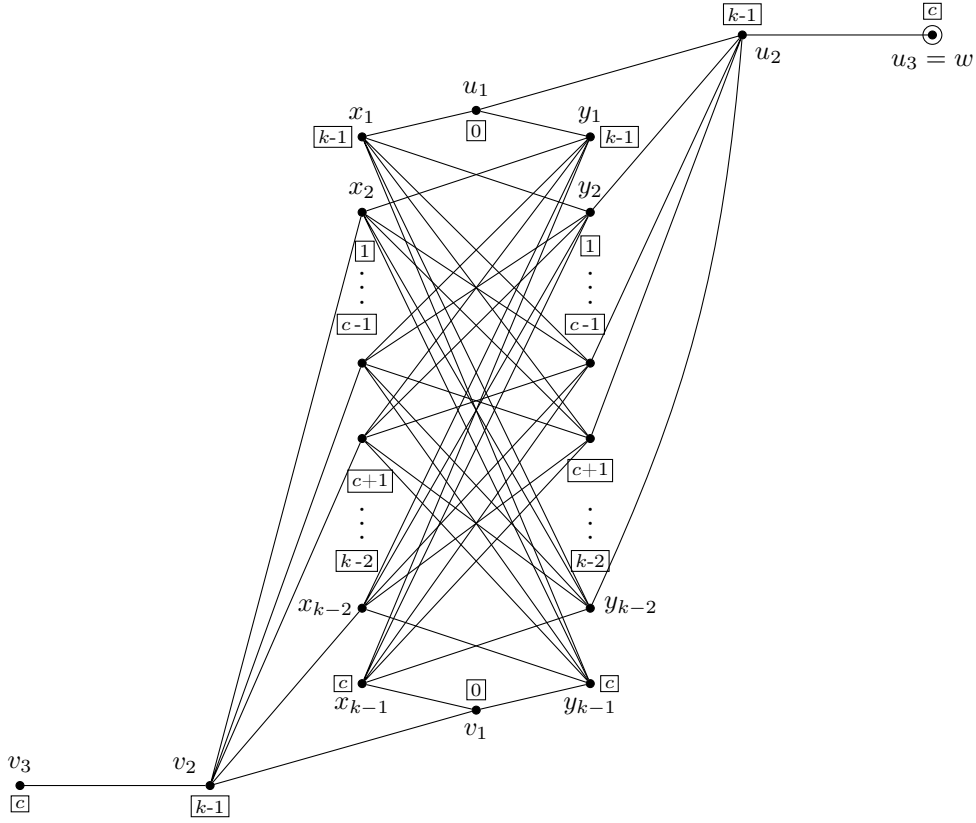
\begin{proof}[Proof of guarantee]
First, we prove that if \( G \) is \( (k-2) \)-rs colourable, then \( G' \) is \( k \)-rs colourable. 
Suppose that \( G \) admits a \( (k-2) \)-rs colouring \( f\colon V(G)\to\{1,2,\dots,k-2\} \). 
Extend \( f \) into a \( k \)-colouring \( f' \) of \( G' \) by using the scheme in Figure~\ref{fig:rs colouring of colour blocking gadget} on each colour blocking gadget. 
Observe that each bicoloured \( P_3 \) in Figure~\ref{fig:rs colouring of colour blocking gadget} has colour~0 on its middle vertex or colour~\( k-1 \) on its endvertices. 
Thus, in Figure~\ref{fig:rs colouring of colour blocking gadget}, there is no bicoloured \( P_3 \) with the higher colour on its middle vertex; i.e., the colouring scheme in Figure~\ref{fig:rs colouring of colour blocking gadget} is a \( k \)-rs colouring of the gadget. \\[3pt]
\noindent \textbf{Claim~1:} \( f' \) is a \( k \)-rs colouring of \( G' \).\\[5pt]
On the contrary, assume that there is a bicoloured 3-vertex path \( Q=x,y,z \) in \( G' \) with the higher colour on its middle vertex (i.e., \( f'(y)>f'(x)=f'(z) \)). 
We know that \( f' \) employs a \( k \)-rs colouring scheme (namely Figure~\ref{fig:rs colouring of colour blocking gadget}) on the colour blocking gadget. 
We also know that the restriction of \( f' \) to \( V(G) \) is an rs colouring of \( G \) (namely \( f \)). 
Hence, either (i)~\( Q \) contains edges from two colour blocking gadgets, or (ii)~\( Q \) contains an edge from a colour blocking gadget and an edge from the copy of \( G \) in \( G' \). 
Since \( Q \) is a 3-vertex path, in both cases, the middle vertex \( y \) of \( Q \) is a terminal of some colour blocking gadget, and \( Q \) contains an edge of the form \( u_2u_3 \) from that colour blocking gadget. 
Without loss of generality, assume that \( xy \) is the edge of the form \( u_2u_3 \). 
That is, \( y \) is the terminal (i.e., vertex \( u_3 \)) of some colour blocking gadget, and \( x \) is the neighbour of the terminal (i.e., vertex \( u_2 \)) in that colour blocking gadget. 
Due to the colouring scheme used on colour blocking gadgets, \( f'(y)=f(y)<k-1 \) and \( f'(x)=k-1 \). 
Thus, \( f'(x)>f'(y) \), which is a contradiction to the assumption that \( f'(y)>f'(x)=f'(z) \). 
This proves Claim~1, and thus \( G' \) is \( k \)-rs colourable.

Conversely, suppose that \( G' \) admits a \( k \)-rs colouring \( f':V(G')\to\{0,1,\dots,k-1\} \). 
Consider an arbitrary vertex \( w\in V(G) \).

Note that \( w \) is the terminal of at least one colour blocking gadget attached at \( w \) in \( G' \) (because \( \Delta(G)=k-2 \)). 
By Lemma~\ref{lem:colour blocking gadget}, terminals of colour blocking gadgets cannot get colour~0. 
Hence, \( f'(w)\neq 0 \). 
Since \( w\in V(G) \) is arbitrary, no vertex in \( V(G) \) is coloured~0 by \( f' \).

We claim that \( f'(w)\neq k-1 \). 
On the contrary, assume that \( f'(w)=k-1 \). 
We know that \( \deg_{G'}(w)=k-1 \). 
Owing to the definition of rs colouring, if a vertex \( v \) of degree \( k-1 \) in a graph \( H \) is coloured \( k-1 \) under a \( k \)-rs colouring of \( H \), then \( v \) has a neighbour coloured~0, a neighbour coloured~1, \dots, a neighbour coloured~\( k-2 \) in~\( H \). 
Since \( f'(w)=k-1 \) and \( \deg_{G'}(w)=k-1 \), the vertex \( w \) has a neighbour coloured~0, a neighbour coloured~1, \dots, a neighbour coloured~\( k-2 \) in \( G' \). 
In particular, \( w \) has a neighbour \( w' \) in \( G' \) coloured~0 under \( f' \). 
Since \( w' \) is a neighbour of \( w \) in \( G' \), \( w' \) is either from the copy of \( G \) (i.e., \( w'\in V(G) \)) or from a colour blocking gadget. 
But, \( w'\notin V(G) \) since \( f'(w')=0 \) and no vertex in \( V(G) \) is coloured~0 by \( f' \). 
Since \( w'\notin V(G) \), the vertex \( w' \) is in some colour blocking gadget. 
Moreover, \( w \) is the terminal of a colour blocking gadget and \( w' \) is the neighbour of the terminal in that colour blocking gadget. 
By Lemma~\ref{lem:colour blocking gadget}, the neighbour of the terminal is not coloured~0 by \( f' \) contradicting the assumption that \( f'(w')=0 \). 
Thus, \( f'(w)\neq k-1 \) by contradiction.

Since \( w\in V(G) \) is arbitrary, \( f' \) uses only colours \( 1,2,\dots,k-2 \) in \( V(G) \). 
Therefore, the restriction of \( f' \) to \( V(G) \) is a \( (k-2) \)-rs colouring of \( G \). 
Hence, \( G \) is \( (k-2) \)-rs colourable. 
\end{proof}

Note that a colour blocking gadget has only \( 2k+3 \) non-terminal vertices and \( (k-1)(k-2)+2(k-3)+8\leq k^2 \) edges. 
Hence, \( G' \) has at most \( \left((k-1)(2k+3)+1\right)n=O(n) \) vertices and at most \( m+(k-1)k^2 n=O(m+n) \) edges, where \( n=|V(G)| \) and \( m=|E(G)| \).
Hence, Construction~\ref{make:k-rs colouring} requires only time polynomial in the input size. 

For all \( k\geq 5 \), Construction~\ref{make:k-rs colouring} establishes a reduction from \textsc{\( (k-2) \)-RS Colourability} of triangle-free graphs of maximum degree \( k-2 \) to \textsc{\( k \)-RS Colourability} of triangle-free graphs of maximum degree \( k-1 \). 
Since \textsc{\( k \)-RS Colourability} of triangle-free graphs of maximum degree \( k \) is NP-complete for \( k\geq 3 \) \cite[Theorem~3]{shalu_cyriac2}, \textsc{\( k \)-RS Colourability} is NP-complete for triangle-free graphs of maximum degree \( k-1 \) for \( k\geq 5 \). 
\begin{theorem}\label{thm:k-rs npc max deg k-1}
For \( k\geq 5 \), \textsc{\( k \)-RS Colourability} is NP-complete for triangle-free graphs of maximum degree~\( k-1 \). 
\qed
\end{theorem}

By Corollary~\ref{cor:4-rs colouring planar} and Theorem~\ref{thm:k-rs npc max deg k-1}, we have the following. 
\begin{theorem}\label{thm:k-rs npc max deg k-1 gen}
For \( k\geq 4 \), \textsc{\( k \)-RS Colourability} is NP-complete for triangle-free graphs of maximum degree~\( k-1 \). 
\qed
\end{theorem}

\iftoggle{extended}
{ \subsection{Hardness Transitions}\label{sec:rs colouring hardness transitions}
\textcolor{red}{TODO. To be written}

\section{Regular Graphs}
We prove that for all \( k\geq 4 \) and \( d<k \), the complexity of \textsc{\( k \)-RS Colourability} is the same for graphs of maximum degree \( d \) and \( d \)-regular graphs.
} { } Next, we prove that for all \( k\geq 4 \) and \( d<k \), the complexity of \textsc{\( k \)-RS Colourability} is the same for graphs of maximum degree \( d \) and \( d \)-regular graphs.
First, we show this for \( d=k-1 \).

\begin{construct}\label{make:rs colouring degre k-1 vs k-1 regular}
\emph{Parameter:} An integer \( k\geq 4 \).\\ \emph{Input:} A graph \( G \) of maximum degree \( k-1 \).\\ 
\emph{Output:} A \( (k-1) \)-regular graph \( G' \).\\
\emph{Guarantee:} \( G \) is \( k \)-rs colourable if and only if \( G' \) is \( k \)-rs colourable.\\
\emph{Steps:}\\
Introduce two copies of \( G \). 
For each vertex \( v \) of \( G \), introduce \( (k-1)-\deg_G(v) \) filler gadgets (see Figure~\ref{fig:rs filler gadget d=k-1}) between the two copies of \( v \); see Figure~\ref{fig:eg filler gadget d=k-1} for an example. 

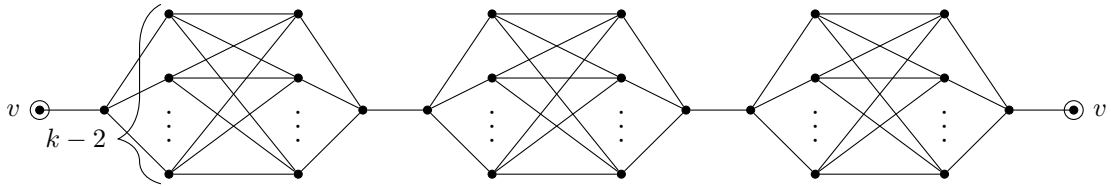
\begin{figure}[hbt]
\centering
\begin{tikzpicture}[scale=0.85]
\path (0,0) node(1st)[dot]{} ++(5,0) node(2nd)[dot]{} ++(5,0) node(3rd)[dot]{};

\path (1st) ++(1,1.5) node(x1)[dot]{} ++(0,-1) node(x2)[dot]{} ++(0,-1.5) node(xk-2)[dot]{};
\path (x2)--node[sloped,font=\large]{\( \dots \)} (xk-2);
\path (x1) ++(2,0) node(y1)[dot]{} ++(0,-1) node(y2)[dot]{} ++(0,-1.5) node(yk-2)[dot]{};
\path (y2)--node[sloped,font=\large]{\( \dots \)} (yk-2);
\path (y1) ++(1,-1.5) node(1stEnd)[dot]{};

\draw (1st)--(x1)  (1st)--(x2)  (1st)--(xk-2);
\draw (x1)--(y1)  (x1)--(y2)  (x1)--(yk-2);
\draw (x2)--(y1)  (x2)--(y2)  (x2)--(yk-2);
\draw (xk-2)--(y1)  (xk-2)--(y2)  (xk-2)--(yk-2);
\draw (1stEnd)--(y1)  (1stEnd)--(y2)  (1stEnd)--(yk-2);

\path (x1)--+(0,0.15) coordinate(brace1Start);
\path (xk-2)--+(0,-0.15) coordinate(brace1End);
\draw [opacity=0.4,decorate,decoration={brace,aspect=0.75,amplitude=16pt,raise=3pt,mirror}] (brace1Start)--  node[left=7mm,yshift=-6mm,opacity=1]{\( k-2 \)} (brace1End);

\path (2nd) ++(1,1.5) node(x1)[dot]{} ++(0,-1) node(x2)[dot]{} ++(0,-1.5) node(xk-2)[dot]{};
\path (x2)--node[sloped,font=\large]{\( \dots \)} (xk-2);
\path (x1) ++(2,0) node(y1)[dot]{} ++(0,-1) node(y2)[dot]{} ++(0,-1.5) node(yk-2)[dot]{};
\path (y2)--node[sloped,font=\large]{\( \dots \)} (yk-2);
\path (y1) ++(1,-1.5) node(2ndEnd)[dot]{};

\draw (2nd)--(x1)  (2nd)--(x2)  (2nd)--(xk-2);
\draw (x1)--(y1)  (x1)--(y2)  (x1)--(yk-2);
\draw (x2)--(y1)  (x2)--(y2)  (x2)--(yk-2);
\draw (xk-2)--(y1)  (xk-2)--(y2)  (xk-2)--(yk-2);
\draw (2ndEnd)--(y1)  (2ndEnd)--(y2)  (2ndEnd)--(yk-2);

\path (3rd) ++(1,1.5) node(x1)[dot]{} ++(0,-1) node(x2)[dot]{} ++(0,-1.5) node(xk-2)[dot]{};
\path (x2)--node[sloped,font=\large]{\( \dots \)} (xk-2);
\path (x1) ++(2,0) node(y1)[dot]{} ++(0,-1) node(y2)[dot]{} ++(0,-1.5) node(yk-2)[dot]{};
\path (y2)--node[sloped,font=\large]{\( \dots \)} (yk-2);
\path (y1) ++(1,-1.5) node(3rdEnd)[dot]{};

\draw (3rd)--(x1)  (3rd)--(x2)  (3rd)--(xk-2);
\draw (x1)--(y1)  (x1)--(y2)  (x1)--(yk-2);
\draw (x2)--(y1)  (x2)--(y2)  (x2)--(yk-2);
\draw (xk-2)--(y1)  (xk-2)--(y2)  (xk-2)--(yk-2);
\draw (3rdEnd)--(y1)  (3rdEnd)--(y2)  (3rdEnd)--(yk-2);

\draw (1stEnd)--(2nd)  (2ndEnd)--(3rd);
\draw (1st)--+(-1,0) node[dot]{} node[terminal][label=left:\( v \)]{}
(3rdEnd)--+(1,0) node[dot]{} node[terminal][label=right:\( v \)]{};
\end{tikzpicture}
\caption{Filler gadget for \( v\in V(G) \).}
\label{fig:rs filler gadget d=k-1}
\end{figure}
\end{construct}
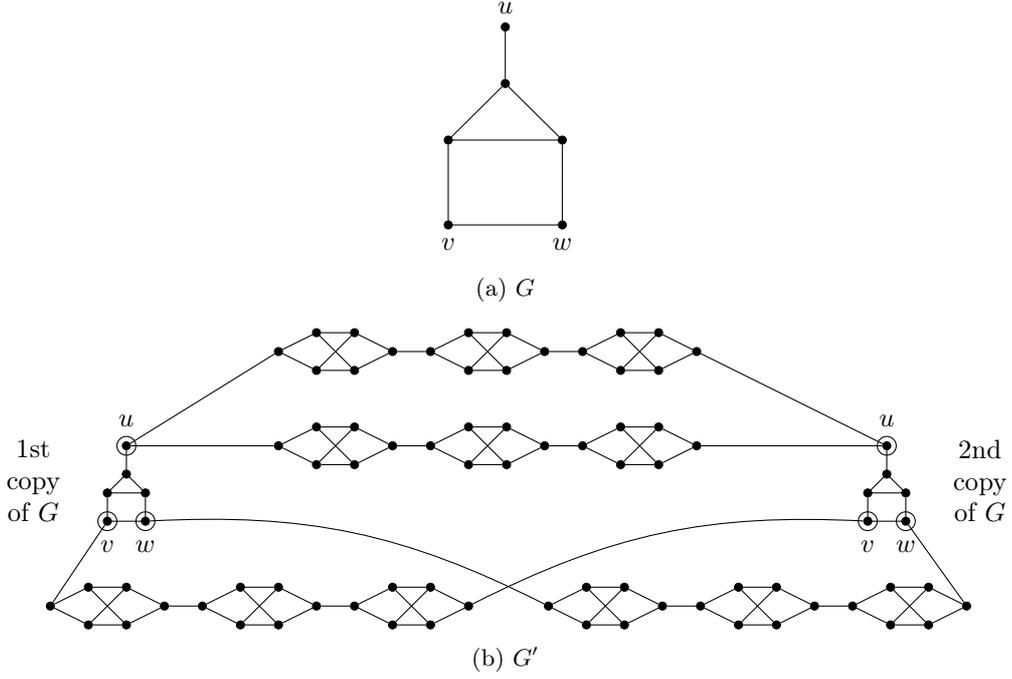
\begin{figure}[hbt]
\centering
\begin{subfigure}[c]{\textwidth}
\centering
\begin{tikzpicture}[scale=1.5]
\draw (0,0) node(u)[dot][label=\( u \)]{} --++(0,-0.5) node(x)[dot]{} --++(0.5,-0.5) node(y)[dot]{} --++(-1,0) node(z)[dot]{}--(x);
\draw (z) --++(0,-0.75) node(v)[dot][label=below:\( v \)]{} --++(1,0) node(w)[dot][label=below:\( w \)]{} --(y);
\end{tikzpicture}
\caption{\( G \)}
\end{subfigure}
\vspace*{0.25cm}

\begin{subfigure}[c]{\textwidth}
\centering
\begin{tikzpicture}[scale=0.5]
\path (0,0) coordinate(origin);

\draw (origin) ++(-10,0) node(u1)[dot]{} node[terminal][label=\( u \)]{} --++(0,-0.75) node(x1)[dot]{} --++(0.5,-0.5) node(y1)[dot]{} --++(-1,0) node(z1)[dot]{}--(x1);
\draw (z1) --++(0,-0.75) node(v1)[dot]{} node[terminal][label=below:\( v \)]{} --++(1,0) node(w1)[dot]{} node[terminal][label=below:\( w \)]{} --(y1);

\path (y1)--node[left=1.0cm,yshift=10pt,align=center]{1st\\copy\\of \( G \)} (w1);

\draw (origin) ++(10,0) node(u2)[dot]{} node[terminal][label=\( u \)]{} --++(0,-0.75) node(x2)[dot]{} --++(0.5,-0.5) node(y2)[dot]{} --++(-1,0) node(z2)[dot]{}--(x2);
\draw (z2) --++(0,-0.75) node(v2)[dot]{} node[terminal][label=below:\( v \)]{} --++(1,0) node(w2)[dot]{} node[terminal][label=below:\( w \)]{} --(y2);

\path (z2)--node[right=1.0cm,yshift=10pt,align=center]{2nd\\copy\\of \( G \)} (v2);

\path (origin) ++(-6,2.5) node(1st)[dot]{} ++(4,0) node(2nd)[dot]{} ++(4,0) node(3rd)[dot]{};

\path (1st) ++(1,0.5) node(x1)[dot]{} ++(0,-1) node(x2)[dot]{};
\path (x1) ++(1,0) node(y1)[dot]{} ++(0,-1) node(y2)[dot]{};
\path (y1) ++(1,-0.5) node(1stEnd)[dot]{};

\draw (1st)--(x1)  (1st)--(x2);
\draw (x1)--(y1)  (x1)--(y2)  ;
\draw (x2)--(y1)  (x2)--(y2)  ;
\draw (1stEnd)--(y1)  (1stEnd)--(y2);

\path (2nd) ++(1,0.5) node(x1)[dot]{} ++(0,-1) node(x2)[dot]{};
\path (x1) ++(1,0) node(y1)[dot]{} ++(0,-1) node(y2)[dot]{};
\path (y1) ++(1,-0.5) node(2ndEnd)[dot]{};

\draw (2nd)--(x1)  (2nd)--(x2);
\draw (x1)--(y1)  (x1)--(y2)  ;
\draw (x2)--(y1)  (x2)--(y2)  ;
\draw (2ndEnd)--(y1)  (2ndEnd)--(y2);

\path (3rd) ++(1,0.5) node(x1)[dot]{} ++(0,-1) node(x2)[dot]{};
\path (x1) ++(1,0) node(y1)[dot]{} ++(0,-1) node(y2)[dot]{};
\path (y1) ++(1,-0.5) node(3rdEnd)[dot]{};

\draw (3rd)--(x1)  (3rd)--(x2);
\draw (x1)--(y1)  (x1)--(y2)  ;
\draw (x2)--(y1)  (x2)--(y2)  ;
\draw (3rdEnd)--(y1)  (3rdEnd)--(y2);

\draw (1stEnd)--(2nd)  (2ndEnd)--(3rd);

\draw (1st)--(u1) (3rdEnd)--(u2);

\path (origin) ++(-6,0) node(1st)[dot]{} ++(4,0) node(2nd)[dot]{} ++(4,0) node(3rd)[dot]{};

\path (1st) ++(1,0.5) node(x1)[dot]{} ++(0,-1) node(x2)[dot]{};
\path (x1) ++(1,0) node(y1)[dot]{} ++(0,-1) node(y2)[dot]{};
\path (y1) ++(1,-0.5) node(1stEnd)[dot]{};

\draw (1st)--(x1)  (1st)--(x2);
\draw (x1)--(y1)  (x1)--(y2)  ;
\draw (x2)--(y1)  (x2)--(y2)  ;
\draw (1stEnd)--(y1)  (1stEnd)--(y2);

\path (2nd) ++(1,0.5) node(x1)[dot]{} ++(0,-1) node(x2)[dot]{};
\path (x1) ++(1,0) node(y1)[dot]{} ++(0,-1) node(y2)[dot]{};
\path (y1) ++(1,-0.5) node(2ndEnd)[dot]{};

\draw (2nd)--(x1)  (2nd)--(x2);
\draw (x1)--(y1)  (x1)--(y2)  ;
\draw (x2)--(y1)  (x2)--(y2)  ;
\draw (2ndEnd)--(y1)  (2ndEnd)--(y2);

\path (3rd) ++(1,0.5) node(x1)[dot]{} ++(0,-1) node(x2)[dot]{};
\path (x1) ++(1,0) node(y1)[dot]{} ++(0,-1) node(y2)[dot]{};
\path (y1) ++(1,-0.5) node(3rdEnd)[dot]{};

\draw (3rd)--(x1)  (3rd)--(x2);
\draw (x1)--(y1)  (x1)--(y2)  ;
\draw (x2)--(y1)  (x2)--(y2)  ;
\draw (3rdEnd)--(y1)  (3rdEnd)--(y2);

\draw (1stEnd)--(2nd)  (2ndEnd)--(3rd);

\draw (1st)--(u1) (3rdEnd)--(u2);

\path (origin) ++(-12,-4.25) node(1st)[dot]{} ++(4,0) node(2nd)[dot]{} ++(4,0) node(3rd)[dot]{};

\path (1st) ++(1,0.5) node(x1)[dot]{} ++(0,-1) node(x2)[dot]{};
\path (x1) ++(1,0) node(y1)[dot]{} ++(0,-1) node(y2)[dot]{};
\path (y1) ++(1,-0.5) node(1stEnd)[dot]{};

\draw (1st)--(x1)  (1st)--(x2);
\draw (x1)--(y1)  (x1)--(y2)  ;
\draw (x2)--(y1)  (x2)--(y2)  ;
\draw (1stEnd)--(y1)  (1stEnd)--(y2);

\path (2nd) ++(1,0.5) node(x1)[dot]{} ++(0,-1) node(x2)[dot]{};
\path (x1) ++(1,0) node(y1)[dot]{} ++(0,-1) node(y2)[dot]{};
\path (y1) ++(1,-0.5) node(2ndEnd)[dot]{};

\draw (2nd)--(x1)  (2nd)--(x2);
\draw (x1)--(y1)  (x1)--(y2)  ;
\draw (x2)--(y1)  (x2)--(y2)  ;
\draw (2ndEnd)--(y1)  (2ndEnd)--(y2);

\path (3rd) ++(1,0.5) node(x1)[dot]{} ++(0,-1) node(x2)[dot]{};
\path (x1) ++(1,0) node(y1)[dot]{} ++(0,-1) node(y2)[dot]{};
\path (y1) ++(1,-0.5) node(3rdEnd)[dot]{};

\draw (3rd)--(x1)  (3rd)--(x2);
\draw (x1)--(y1)  (x1)--(y2)  ;
\draw (x2)--(y1)  (x2)--(y2)  ;
\draw (3rdEnd)--(y1)  (3rdEnd)--(y2);

\draw (1stEnd)--(2nd)  (2ndEnd)--(3rd);

\draw (1st)--(v1) (3rdEnd) to[bend left=15] (v2);

\path (origin) ++(1.1,-4.25) node(1st)[dot]{} ++(4,0) node(2nd)[dot]{} ++(4,0) node(3rd)[dot]{};

\path (1st) ++(1,0.5) node(x1)[dot]{} ++(0,-1) node(x2)[dot]{};
\path (x1) ++(1,0) node(y1)[dot]{} ++(0,-1) node(y2)[dot]{};
\path (y1) ++(1,-0.5) node(1stEnd)[dot]{};

\draw (1st)--(x1)  (1st)--(x2);
\draw (x1)--(y1)  (x1)--(y2)  ;
\draw (x2)--(y1)  (x2)--(y2)  ;
\draw (1stEnd)--(y1)  (1stEnd)--(y2);

\path (2nd) ++(1,0.5) node(x1)[dot]{} ++(0,-1) node(x2)[dot]{};
\path (x1) ++(1,0) node(y1)[dot]{} ++(0,-1) node(y2)[dot]{};
\path (y1) ++(1,-0.5) node(2ndEnd)[dot]{};

\draw (2nd)--(x1)  (2nd)--(x2);
\draw (x1)--(y1)  (x1)--(y2)  ;
\draw (x2)--(y1)  (x2)--(y2)  ;
\draw (2ndEnd)--(y1)  (2ndEnd)--(y2);

\path (3rd) ++(1,0.5) node(x1)[dot]{} ++(0,-1) node(x2)[dot]{};
\path (x1) ++(1,0) node(y1)[dot]{} ++(0,-1) node(y2)[dot]{};
\path (y1) ++(1,-0.5) node(3rdEnd)[dot]{};

\draw (3rd)--(x1)  (3rd)--(x2);
\draw (x1)--(y1)  (x1)--(y2)  ;
\draw (x2)--(y1)  (x2)--(y2)  ;
\draw (3rdEnd)--(y1)  (3rdEnd)--(y2);

\draw (1stEnd)--(2nd)  (2ndEnd)--(3rd);

\draw (1st) to[bend right=15] (w1) (3rdEnd)--(w2);
\end{tikzpicture}
\caption{\( G' \)}
\end{subfigure}
\caption[Example of Construction~\ref{make:rs colouring degre k-1 vs k-1 regular}.]{Example of Construction~\ref{make:rs colouring degre k-1 vs k-1 regular} (here, \( k=4 \)).}
\label{fig:eg filler gadget d=k-1}
\end{figure}

Each non-terminal vertex of a filler gadget has degree \( k-1 \). 
For each \( v\in V(G) \), both copies of \( v \) in \( G' \) have degree \( k-1 \) (because there are exactly \( (k-1)-\deg_G(v) \) filler gadgets between the two copies of \( v \)). 
Therefore, \( G' \) is \( (k-1) \)-regular. 

\begin{proof}[Proof of guarantee]
If \( G' \) is \( k \)-rs colourable, then \( G \) is \( k \)-rs colourable (because \( G \) is a subgraph of~\( G' \)). 
Conversely, suppose that \( G \) admits a \( k \)-rs colouring \( f:V(G)\to\{0,1,\dots,k-1\} \). 
We produce a \( k \)-colouring \( f' \) of \( G' \) as follows. 
The copies of \( G \) are coloured first, followed by the filler gadgets. 
Colour both copies of \( G \) using \( f \). 
For each vertex \( v \) of \( G \), the filler gadgets for \( v \) are coloured by various \( k \)-rs colouring schemes depending on the colour of \( v \) under \( f \). 
If \( f(v)<k-1 \), we employ the following \( k \)-rs colouring scheme on each filler gadget for \( v \) which ensures that the neighbour of the terminal in the gadget has a higher colour compared to the terminal \( v \):  
(i)~if \( f(v)=0 \), colour the filler gadgets for \( v \) by the \( k \)-rs colouring scheme in Figure~\ref{fig:scheme 1 filler gadget d=k-1}, 
(ii)~if \( 0<f(v)<k-1 \), colour the filler gadgets for \( v \) by the \( k \)-rs colouring scheme in Figure~\ref{fig:scheme 2 filler gadget d=k-1}. 
If \( f(v)=k-1 \), colour each filer gadget for~\( v \), one by one, as follows: choose a colour \( j \) not yet used in the neighbourhood of (copy of) \( v \) in \( G' \), and colour the filler gadget by the \( k \)-rs colouring scheme in Figure~\ref{fig:scheme 3 filler gadget d=k-1} (note that by the colouring scheme used on the filler gadgets, the colours present on the neighbourhood of the fist copy of \( v \) in \( G' \) are the same as the colours present on the neighbourhood of the second copy of \( v \) in \( G' \)). 
See Figure~\ref{fig:eg colouring filler gadget d=k-1} for an example. 

Clearly, \( f' \) is a \( k \)-colouring of \( G' \).\\[3pt]
\noindent \textbf{Claim 1:} \( f' \) is a \( k \)-rs colouring of \( G' \).\\[3pt]
We know that the copies of \( G \) and the filler gadgets in \( G' \) are coloured by \( k \)-rs colouring schemes. 
Hence, to prove Claim~1, it suffices to show that no terminal \( y \) in \( G' \) has two neighbours \( x \) and \( z \) such that \( f'(y)>f'(x)=f'(z) \). 
On the contrary, assume that there exists a terminal \( y \) with neighbours \( x \) and \( z \) in \( G' \) such that \( \bm{f'(y)>f'(x)=f'(z)} \). 

Obviously, \( y \) is a vertex in a copy of \( G \) (in \( G' \)). 
Since \( f' \) restricted to this copy of \( G \) is a \( k \)-rs colouring (namely \( f \)), \( x \) and/or \( z \) must be in a filler gadget. 
Without loss of generality, assume that \( z \) is in a filler gadget \( F_z \). 
Clearly, \( y \) is the terminal of the filler gadget \( F_z \) and \( z \) is the neighbour of the terminal in the filler gadget \( F_z \). 
Recall that unless \( f'(y)=k-1 \), the colouring scheme used on the filler gadget \( F_z \) ensures that \( f'(z)>f'(y) \) (i.e., the neighbour of the terminal in the gadget has a higher colour compared to the terminal). 
Since \( f'(y)>f'(z) \), we have \( f'(y)=k-1 \). 
As a result, the colouring scheme in Figure~\ref{fig:scheme 3 filler gadget d=k-1} is used on the filler gadgets attached at \( y \) and in particular on \( F_z \). 
When the filler gadget \( F_z \) was coloured, a colour \( j \) not yet present in the neighbourhood of \( y \) in \( G' \) was chosen, and then the colouring scheme in Figure~\ref{fig:scheme 3 filler gadget d=k-1} was applied on \( F_z \). 
This means that \( j=f'(z) \). 
We have two cases. 

\noindent \textit{Case 1:} \( x \) is in a copy of \( G \) in \( G' \).\\[5pt]
Clearly, \( x \) was coloured before the filler gadget \( F_z \) was coloured. 
Hence, the colour \( f'(x) \) was present in the neighbourhood of \( y \) in \( G' \) before \( F_z \) was coloured. 
As a result, \( j\neq f'(x) \) by the choice of colour~\( j \). 
This is a contradiction since \( j=f'(z)=f'(x) \). 

\noindent \textit{Case 2:} \( x \) is in a filler gadget, say \( F_x \).\\[5pt]
Without loss of generality, assume that the gadget \( F_x \) was coloured first and the gadget \( F_z \) was coloured later. 
Consequently, \( x \) was coloured before the filler gadget \( F_z \) was coloured. 
Thus, the colour \( f'(x) \) was present in the neighbourhood of \( y \) in \( G' \) before \( F_z \) was coloured, and thus \( j\neq f'(x) \) by the choice of colour~\( j \). 
This is a contradiction since \( j=f'(z)=f'(x) \). 

Since we have a contradiction in both cases, Claim~1 is proved. 
Therefore, \( G' \) is \( k \)-rs colourable. 
\end{proof}

\begin{figure}[hbtp]
\centering
\begin{subfigure}[b]{\textwidth}
\centering
\begin{tikzpicture}[scale=0.85]
\path (0,0) node(1st)[dot][label={[vcolour,xshift=-5pt,yshift=2pt]\( k\text{-1} \)}]{} ++(5,0) node(2nd)[dot][label={[vcolour,xshift=-2pt,yshift=1pt]1}]{} ++(5,0) node(3rd)[dot][label={[vcolour,xshift=-4pt,yshift=1pt]\( k\text{-1} \)}]{};

\path (1st) ++(1,1.5) node(x1)[dot][label={[vcolour]1}]{} ++(0,-1) node(x2)[dot][label={[vcolour]2}]{} ++(0,-1.5) node(xk-2)[dot][label={[vcolour]below:\( k\text{-2} \)}]{};
\path (x2)--node[sloped,font=\large]{\( \dots \)} (xk-2);
\path (x1) ++(2,0) node(y1)[dot][label={[vcolour]\( k\text{-1} \)}]{} ++(0,-1) node(y2)[dot][label={[vcolour]\( k\text{-1} \)}]{} ++(0,-1.5) node(yk-2)[dot][label={[vcolour]below:\( k\text{-1} \)}]{};
\path (y2)--node[sloped,font=\large]{\( \dots \)} (yk-2);
\path (y1) ++(1,-1.5) node(1stEnd)[dot][label={[vcolour,xshift=2pt,yshift=1pt]0}]{};

\draw (1st)--(x1)  (1st)--(x2)  (1st)--(xk-2);
\draw (x1)--(y1)  (x1)--(y2)  (x1)--(yk-2);
\draw (x2)--(y1)  (x2)--(y2)  (x2)--(yk-2);
\draw (xk-2)--(y1)  (xk-2)--(y2)  (xk-2)--(yk-2);
\draw (1stEnd)--(y1)  (1stEnd)--(y2)  (1stEnd)--(yk-2);

\path (2nd) ++(1,1.5) node(x1)[dot][label={[vcolour]\( k\text{-1} \)}]{} ++(0,-1) node(x2)[dot][label={[vcolour]\( k\text{-1} \)}]{} ++(0,-1.5) node(xk-2)[dot][label={[vcolour]below:\( k\text{-1} \)}]{};
\path (x2)--node[sloped,font=\large]{\( \dots \)} (xk-2);
\path (x1) ++(2,0) node(y1)[dot][label={[vcolour]0}]{} ++(0,-1) node(y2)[dot][label={[vcolour]2}]{} ++(0,-1.5) node(yk-2)[dot][label={[vcolour]below:\( k\text{-2} \)}]{};
\path (y2)--node[sloped,font=\large]{\( \dots \)} (yk-2);
\path (y1) ++(1,-1.5) node(2ndEnd)[dot][label={[vcolour,xshift=2pt,yshift=1pt]1}]{};

\draw (2nd)--(x1)  (2nd)--(x2)  (2nd)--(xk-2);
\draw (x1)--(y1)  (x1)--(y2)  (x1)--(yk-2);
\draw (x2)--(y1)  (x2)--(y2)  (x2)--(yk-2);
\draw (xk-2)--(y1)  (xk-2)--(y2)  (xk-2)--(yk-2);
\draw (2ndEnd)--(y1)  (2ndEnd)--(y2)  (2ndEnd)--(yk-2);

\path (3rd) ++(1,1.5) node(x1)[dot][label={[vcolour]0}]{} ++(0,-1) node(x2)[dot][label={[vcolour]2}]{} ++(0,-1.5) node(xk-2)[dot][label={[vcolour]below:\( k\text{-2} \)}]{};
\path (x2)--node[sloped,font=\large]{\( \dots \)} (xk-2);
\path (x1) ++(2,0) node(y1)[dot][label={[vcolour]\( k\text{-1} \)}]{} ++(0,-1) node(y2)[dot][label={[vcolour]\( k\text{-1} \)}]{} ++(0,-1.5) node(yk-2)[dot][label={[vcolour]below:\( k\text{-1} \)}]{};
\path (y2)--node[sloped,font=\large]{\( \dots \)} (yk-2);
\path (y1) ++(1,-1.5) node(3rdEnd)[dot][label={[vcolour,xshift=2pt,yshift=1pt]1}]{};

\draw (3rd)--(x1)  (3rd)--(x2)  (3rd)--(xk-2);
\draw (x1)--(y1)  (x1)--(y2)  (x1)--(yk-2);
\draw (x2)--(y1)  (x2)--(y2)  (x2)--(yk-2);
\draw (xk-2)--(y1)  (xk-2)--(y2)  (xk-2)--(yk-2);
\draw (3rdEnd)--(y1)  (3rdEnd)--(y2)  (3rdEnd)--(yk-2);

\draw (1stEnd)--(2nd)  (2ndEnd)--(3rd);
\draw (1st)--+(-1,0) node[dot]{} node[terminal][label=left:\( v \)][label={[vcolour]0}]{}
(3rdEnd)--+(1,0) node[dot]{} node[terminal][label=right:\( v \)][label={[vcolour]0}]{};
\end{tikzpicture}
\caption{When \( f(v)=0 \).}
\label{fig:scheme 1 filler gadget d=k-1}
\end{subfigure}
\vspace*{5pt}

\begin{subfigure}[b]{\textwidth}
\centering
\begin{tikzpicture}[scale=0.85]
\path (0,0) node(1st)[dot][label={[vcolour,xshift=-5pt,yshift=1pt]\( k\text{-1} \)}]{} ++(5,0) node(2nd)[dot][label={[vcolour,xshift=-2pt,yshift=1pt]0}]{} ++(5,0) node(3rd)[dot][label={[vcolour,xshift=-2pt,yshift=1pt]\( i \)}]{};

\path (1st) ++(1,1.5) node(x1)[dot][label={[vcolour]1}]{} ++(0,-1) node(x2)[dot][label={[vcolour]2}]{} ++(0,-1.5) node(xk-2)[dot][label={[vcolour]below:\( k\text{-2} \)}]{};
\path (x2)--(xk-2) node(i-1label)[pos=0.35][vcolour]{\( i\text{-1} \)} node(i+1label)[pos=0.65][vcolour]{\( i\text{\smaller{+}1} \)};
\path [sloped,font=\tiny]
(x2)--node{...} (i-1label)
(xk-2)--node{...} (i+1label);
\path (x1) ++(2,0) node(y1)[dot][label={[vcolour]\( k\text{-1} \)}]{} ++(0,-1) node(y2)[dot][label={[vcolour]\( k\text{-1} \)}]{} ++(0,-1.5) node(yk-2)[dot][label={[vcolour]below:\( k\text{-1} \)}]{};
\path (y2)--node[sloped,font=\large]{\( \dots \)} (yk-2);
\path (y1) ++(1,-1.5) node(1stEnd)[dot][label={[vcolour,xshift=2pt,yshift=1pt]\( i \)}]{};

\draw (1st)--(x1)  (1st)--(x2)  (1st)--(xk-2);
\draw (x1)--(y1)  (x1)--(y2)  (x1)--(yk-2);
\draw (x2)--(y1)  (x2)--(y2)  (x2)--(yk-2);
\draw (xk-2)--(y1)  (xk-2)--(y2)  (xk-2)--(yk-2);
\draw (1stEnd)--(y1)  (1stEnd)--(y2)  (1stEnd)--(yk-2);

\path (2nd) ++(1,1.5) node(x1)[dot][label={[vcolour]\( k\text{-1} \)}]{} ++(0,-1) node(x2)[dot][label={[vcolour]\( k\text{-1} \)}]{} ++(0,-1.5) node(xk-2)[dot][label={[vcolour]below:\( k\text{-1} \)}]{};
\path (x2)--node[sloped,font=\large]{\( \dots \)} (xk-2);
\path (x1) ++(2,0) node(y1)[dot][label={[vcolour]1}]{} ++(0,-1) node(y2)[dot][label={[vcolour]2}]{} ++(0,-1.5) node(yk-2)[dot][label={[vcolour]below:\( k\text{-2} \)}]{};
\path (y2)--node(ilabel)[vcolour]{\( i \)} (yk-2);
\path[sloped,font=\scriptsize] (y2) --node[pos=0.20]{...} node[pos=0.80]{...} (yk-2);
\path (y1) ++(1,-1.5) node(2ndEnd)[dot][label={[vcolour,xshift=2pt,yshift=1pt]0}]{};

\draw (2nd)--(x1)  (2nd)--(x2)  (2nd)--(xk-2);
\draw (x1)--(y1)  (x1)--(y2)  (x1)--(yk-2);
\draw (x2)--(y1)  (x2)--(y2)  (x2)--(yk-2);
\draw (xk-2)--(y1)  (xk-2)--(y2)  (xk-2)--(yk-2);
\draw (2ndEnd)--(y1)  (2ndEnd)--(y2)  (2ndEnd)--(yk-2);
\draw [line width=0.35mm] (2ndEnd)--(ilabel); 

\path (3rd) ++(1,1.5) node(x1)[dot][label={[vcolour]\( k\text{-1} \)}]{} ++(0,-1) node(x2)[dot][label={[vcolour]\( k\text{-1} \)}]{} ++(0,-1.5) node(xk-2)[dot][label={[vcolour]below:\( k\text{-1} \)}]{};
\path (x2)--node[sloped,font=\large]{\( \dots \)} (xk-2);
\path (x1) ++(2,0) node(y1)[dot][label={[vcolour]0}]{} ++(0,-1) node(y2)[dot][label={[vcolour]1}]{} ++(0,-1.5) node(yk-2)[dot][label={[vcolour]below:\( k\text{-2} \)}]{};
\path (y2)--(yk-2) node(i-1label)[pos=0.35][vcolour]{\( i\text{-1} \)} node(i+1label)[pos=0.65][vcolour]{\( i\text{\smaller{+}1} \)};
\path [sloped,font=\tiny]
(y2)--node{...} (i-1label)
(yk-2)--node{...} (i+1label);
\path (y1) ++(1,-1.5) node(3rdEnd)[dot][label={[vcolour,xshift=5pt,yshift=2pt]\( k\text{-1} \)}]{};
\draw (3rd)--(x1)  (3rd)--(x2)  (3rd)--(xk-2);
\draw (x1)--(y1)  (x1)--(y2)  (x1)--(yk-2);
\draw (x2)--(y1)  (x2)--(y2)  (x2)--(yk-2);
\draw (xk-2)--(y1)  (xk-2)--(y2)  (xk-2)--(yk-2);
\draw (3rdEnd)--(y1)  (3rdEnd)--(y2)  (3rdEnd)--(yk-2);

\draw (1stEnd)--(2nd);
\draw [line width=0.35mm]  (2ndEnd)--(3rd);
\draw (1st)--+(-1,0) node[dot]{} node[terminal][label=left:\( v \)][label={[vcolour]\( i \)}]{}
(3rdEnd)--+(1,0) node[dot]{} node[terminal][label=right:\( v \)][label={[vcolour]\( i \)}]{};
\end{tikzpicture}
\caption{When \( f(v)=i \) and \( 0<i<k-1 \).}
\label{fig:scheme 2 filler gadget d=k-1}
\end{subfigure}

\caption[Colouring scheme for filler gadget for \( v \) when \( f(v)<k-1 \).]
{Colouring scheme for filler gadget for \( v\in V(G) \), provided \( f(v)<k-1 \). 
The colourings displayed are \( k \)-rs colourings (the highlighted bicoloured \( P_3 \) has colour~0 on its middle vertex, and every other bicoloured \( P_3 \) has colour~\( k-1 \) on its endvertices).}
\label{fig:schemes 1N2 filler gadget d=k-1}
\end{figure}
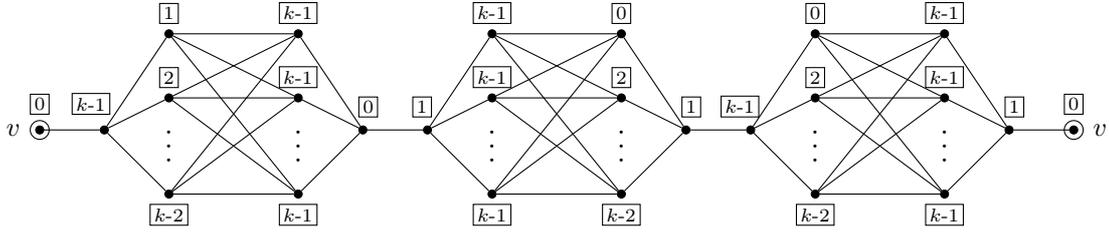
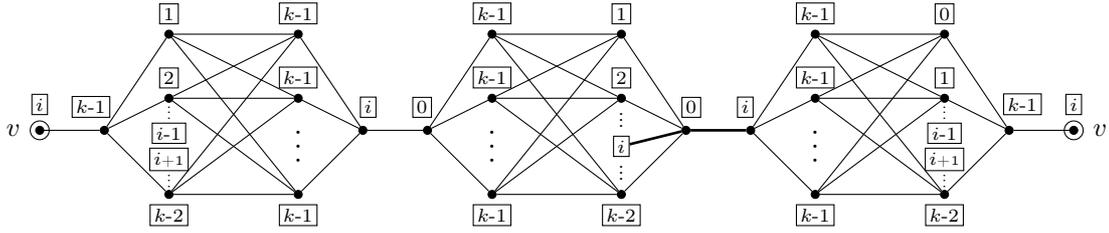

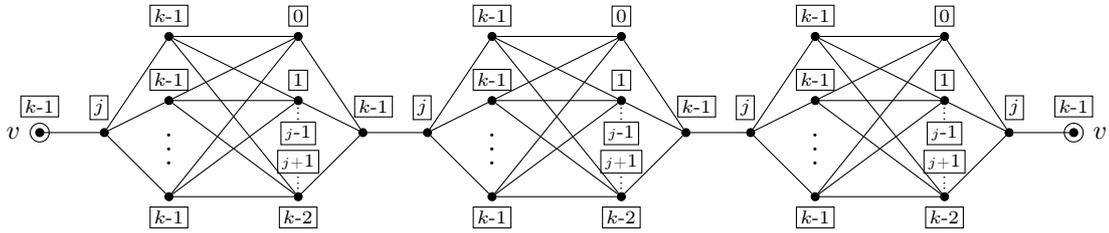
\begin{figure}[hbtp]
\centering
\begin{tikzpicture}[scale=0.85]
\path (0,0) node(1st)[dot][label={[vcolour,xshift=-2pt,yshift=1pt]\( j \)}]{} ++(5,0) node(2nd)[dot][label={[vcolour,xshift=-2pt,yshift=1pt]\( j \)}]{} ++(5,0) node(3rd)[dot][label={[vcolour,xshift=-2pt,yshift=1pt]\( j \)}]{};

\path (1st) ++(1,1.5) node(x1)[dot][label={[vcolour]\( k\text{-1} \)}]{} ++(0,-1) node(x2)[dot][label={[vcolour]\( k\text{-1} \)}]{} ++(0,-1.5) node(xk-2)[dot][label={[vcolour]below:\( k\text{-1} \)}]{};
\path (x2)--node[sloped,font=\large]{\( \dots \)} (xk-2);
\path (x1) ++(2,0) node(y1)[dot][label={[vcolour]0}]{} ++(0,-1) node(y2)[dot][label={[vcolour]1}]{} ++(0,-1.5) node(yk-2)[dot][label={[vcolour]below:\( k\text{-2} \)}]{};
\path (y2)--(yk-2) node(j-1label)[pos=0.325][vcolour]{\( \mathsmaller{j}\text{-1} \)} node(j+1label)[pos=0.65][vcolour]{\( \mathsmaller{j}\text{\smaller{+}}\text{1} \)};
\path [sloped,font=\tiny]
(y2)--node{...} (j-1label)
(yk-2)--node{...} (j+1label);
\path (y1) ++(1,-1.5) node(3rdEnd)[dot][label={[vcolour,xshift=4pt,yshift=2pt]\( k\text{-1} \)}]{};

\draw (1st)--(x1)  (1st)--(x2)  (1st)--(xk-2);
\draw (x1)--(y1)  (x1)--(y2)  (x1)--(yk-2);
\draw (x2)--(y1)  (x2)--(y2)  (x2)--(yk-2);
\draw (xk-2)--(y1)  (xk-2)--(y2)  (xk-2)--(yk-2);
\draw (1stEnd)--(y1)  (1stEnd)--(y2)  (1stEnd)--(yk-2);

\path (2nd) ++(1,1.5) node(x1)[dot][label={[vcolour]\( k\text{-1} \)}]{} ++(0,-1) node(x2)[dot][label={[vcolour]\( k\text{-1} \)}]{} ++(0,-1.5) node(xk-2)[dot][label={[vcolour]below:\( k\text{-1} \)}]{};
\path (x2)--node[sloped,font=\large]{\( \dots \)} (xk-2);
\path (x1) ++(2,0) node(y1)[dot][label={[vcolour]0}]{} ++(0,-1) node(y2)[dot][label={[vcolour]1}]{} ++(0,-1.5) node(yk-2)[dot][label={[vcolour]below:\( k\text{-2} \)}]{};
\path (y2)--(yk-2) node(j-1label)[pos=0.325][vcolour]{\( \mathsmaller{j}\text{-1} \)} node(j+1label)[pos=0.65][vcolour]{\( \mathsmaller{j}\text{\smaller{+}}\text{1} \)};
\path [sloped,font=\tiny]
(y2)--node{...} (j-1label)
(yk-2)--node{...} (j+1label);
\path (y1) ++(1,-1.5) node(2ndEnd)[dot][label={[vcolour,xshift=4pt,yshift=2pt]\( k\text{-1} \)}]{};

\draw (2nd)--(x1)  (2nd)--(x2)  (2nd)--(xk-2);
\draw (x1)--(y1)  (x1)--(y2)  (x1)--(yk-2);
\draw (x2)--(y1)  (x2)--(y2)  (x2)--(yk-2);
\draw (xk-2)--(y1)  (xk-2)--(y2)  (xk-2)--(yk-2);
\draw (2ndEnd)--(y1)  (2ndEnd)--(y2)  (2ndEnd)--(yk-2);

\path (3rd) ++(1,1.5) node(x1)[dot][label={[vcolour]\( k\text{-1} \)}]{} ++(0,-1) node(x2)[dot][label={[vcolour]\( k\text{-1} \)}]{} ++(0,-1.5) node(xk-2)[dot][label={[vcolour]below:\( k\text{-1} \)}]{};
\path (x2)--node[sloped,font=\large]{\( \dots \)} (xk-2);
\path (x1) ++(2,0) node(y1)[dot][label={[vcolour]0}]{} ++(0,-1) node(y2)[dot][label={[vcolour]1}]{} ++(0,-1.5) node(yk-2)[dot][label={[vcolour]below:\( k\text{-2} \)}]{};
\path (y2)--(yk-2) node(j-1label)[pos=0.325][vcolour]{\( \mathsmaller{j}\text{-1} \)} node(j+1label)[pos=0.65][vcolour]{\( \mathsmaller{j}\text{\smaller{+}}\text{1} \)};
\path [sloped,font=\tiny]
(y2)--node{...} (j-1label)
(yk-2)--node{...} (j+1label);
\path (y1) ++(1,-1.5) node(3rdEnd)[dot][label={[vcolour,xshift=2pt,yshift=1pt]\( j \)}]{};

\draw (3rd)--(x1)  (3rd)--(x2)  (3rd)--(xk-2);
\draw (x1)--(y1)  (x1)--(y2)  (x1)--(yk-2);
\draw (x2)--(y1)  (x2)--(y2)  (x2)--(yk-2);
\draw (xk-2)--(y1)  (xk-2)--(y2)  (xk-2)--(yk-2);
\draw (3rdEnd)--(y1)  (3rdEnd)--(y2)  (3rdEnd)--(yk-2);

\draw (1stEnd)--(2nd)  (2ndEnd)--(3rd);
\draw (1st)--+(-1,0) node[dot]{} node[terminal][label=left:\( v \)][label={[vcolour]\( k\text{-1} \)}]{}
(3rdEnd)--+(1,0) node[dot]{} node[terminal][label=right:\( v \)][label={[vcolour]\( k\text{-1} \)}]{};
\end{tikzpicture}
\caption[A \( k \)-rs colouring scheme for filler gadget for \( v \) when \( f(v)=k-1 \).]{A \( k \)-rs colouring scheme for filler gadget for \( v\in V(G) \) when \( f(v)=k-1 \) (observe that every bicoloured \( P_3 \) has colour~\( k-1 \) on its endvertices).}
\label{fig:scheme 3 filler gadget d=k-1}
\end{figure}

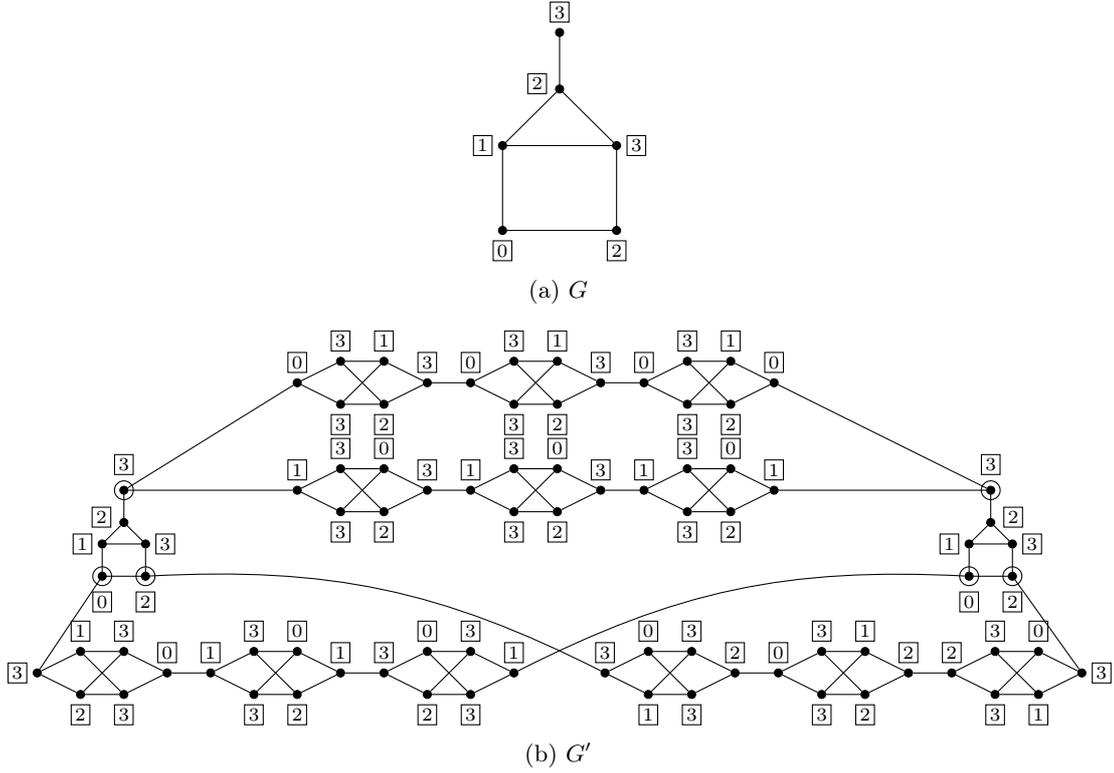
\begin{figure}[hbtp]
\centering
\begin{subfigure}[c]{\textwidth}
\centering
\begin{tikzpicture}[scale=1.5]
\draw (0,0) node(u)[dot][label={[vcolour]\( 3 \)}]{} --++(0,-0.5) node(x)[dot][label={[vcolour,label distance=3pt,yshift=2pt]left:\( 2 \)}]{} --++(0.5,-0.5) node(y)[dot][label={[vcolour]right:\( 3 \)}]{} --++(-1,0) node(z)[dot][label={[vcolour]left:\( 1 \)}]{}--(x);
\draw (z) --++(0,-0.75) node(v)[dot][label={[vcolour]below:\( 0 \)}]{} --++(1,0) node(w)[dot][label={[vcolour]below:\( 2 \)}]{} --(y);
\end{tikzpicture}
\caption{\( G \)}
\end{subfigure}
\vspace*{0.25cm}

\begin{subfigure}[c]{\textwidth}
\centering
\begin{tikzpicture}[scale=0.57]
\path (0,0) coordinate(origin);

\draw (origin) ++(-10,0) node(u1)[dot]{} node[terminal][label={[vcolour]\( 3 \)}]{} --++(0,-0.75) node(x1)[dot][label={[vcolour,label distance=3pt,yshift=2pt]left:\( 2 \)}]{} --++(0.5,-0.5) node(y1)[dot][label={[vcolour]right:\( 3 \)}]{} --++(-1,0) node(z1)[dot][label={[vcolour]left:\( 1 \)}]{}--(x1);
\draw (z1) --++(0,-0.75) node(v1)[dot]{} node[terminal][label={[vcolour]below:\( 0 \)}]{} --++(1,0) node(w1)[dot]{} node[terminal][label={[vcolour]below:\( 2 \)}]{} --(y1);

\draw (origin) ++(10,0) node(u2)[dot]{} node[terminal][label={[vcolour]\( 3 \)}]{} --++(0,-0.75) node(x2)[dot][label={[vcolour,label distance=3pt,yshift=2pt]right:\( 2 \)}]{} --++(0.5,-0.5) node(y2)[dot][label={[vcolour]right:\( 3 \)}]{} --++(-1,0) node(z2)[dot][label={[vcolour]left:\( 1 \)}]{}--(x2);
\draw (z2) --++(0,-0.75) node(v2)[dot]{} node[terminal][label={[vcolour]below:\( 0 \)}]{} --++(1,0) node(w2)[dot]{} node[terminal][label={[vcolour]below:\( 2 \)}]{} --(y2);

\path (origin) ++(-6,2.5) node(1st)[dot][label={[vcolour]above:0}]{} ++(4,0) node(2nd)[dot][label={[vcolour]above:0}]{} ++(4,0) node(3rd)[dot][label={[vcolour]above:0}]{};

\path (1st) ++(1,0.5) node(x1)[dot][label={[vcolour]above:3}]{} ++(0,-1) node(x2)[dot][label={[vcolour]below:3}]{};
\path (x1) ++(1,0) node(y1)[dot][label={[vcolour]above:1}]{} ++(0,-1) node(y2)[dot][label={[vcolour]below:2}]{};
\path (y1) ++(1,-0.5) node(1stEnd)[dot][label={[vcolour]above:3}]{};

\draw (1st)--(x1)  (1st)--(x2);
\draw (x1)--(y1)  (x1)--(y2)  ;
\draw (x2)--(y1)  (x2)--(y2)  ;
\draw (1stEnd)--(y1)  (1stEnd)--(y2);

\path (2nd) ++(1,0.5) node(x1)[dot][label={[vcolour]above:3}]{} ++(0,-1) node(x2)[dot][label={[vcolour]below:3}]{};
\path (x1) ++(1,0) node(y1)[dot][label={[vcolour]above:1}]{} ++(0,-1) node(y2)[dot][label={[vcolour]below:2}]{};
\path (y1) ++(1,-0.5) node(2ndEnd)[dot][label={[vcolour]above:3}]{};

\draw (2nd)--(x1)  (2nd)--(x2);
\draw (x1)--(y1)  (x1)--(y2)  ;
\draw (x2)--(y1)  (x2)--(y2)  ;
\draw (2ndEnd)--(y1)  (2ndEnd)--(y2);

\path (3rd) ++(1,0.5) node(x1)[dot][label={[vcolour]above:3}]{} ++(0,-1) node(x2)[dot][label={[vcolour]below:3}]{};
\path (x1) ++(1,0) node(y1)[dot][label={[vcolour]above:1}]{} ++(0,-1) node(y2)[dot][label={[vcolour]below:2}]{};
\path (y1) ++(1,-0.5) node(3rdEnd)[dot][label={[vcolour]above:0}]{};

\draw (3rd)--(x1)  (3rd)--(x2);
\draw (x1)--(y1)  (x1)--(y2)  ;
\draw (x2)--(y1)  (x2)--(y2)  ;
\draw (3rdEnd)--(y1)  (3rdEnd)--(y2);

\draw (1stEnd)--(2nd)  (2ndEnd)--(3rd);

\draw (1st)--(u1) (3rdEnd)--(u2);

\path (origin) ++(-6,0) node(1st)[dot][label={[vcolour]above:1}]{} ++(4,0) node(2nd)[dot][label={[vcolour]above:1}]{} ++(4,0) node(3rd)[dot][label={[vcolour]above:1}]{};

\path (1st) ++(1,0.5) node(x1)[dot][label={[vcolour]above:3}]{} ++(0,-1) node(x2)[dot][label={[vcolour]below:3}]{};
\path (x1) ++(1,0) node(y1)[dot][label={[vcolour]above:0}]{} ++(0,-1) node(y2)[dot][label={[vcolour]below:2}]{};
\path (y1) ++(1,-0.5) node(1stEnd)[dot][label={[vcolour]above:3}]{};

\draw (1st)--(x1)  (1st)--(x2);
\draw (x1)--(y1)  (x1)--(y2)  ;
\draw (x2)--(y1)  (x2)--(y2)  ;
\draw (1stEnd)--(y1)  (1stEnd)--(y2);

\path (2nd) ++(1,0.5) node(x1)[dot][label={[vcolour]above:3}]{} ++(0,-1) node(x2)[dot][label={[vcolour]below:3}]{};
\path (x1) ++(1,0) node(y1)[dot][label={[vcolour]above:0}]{} ++(0,-1) node(y2)[dot][label={[vcolour]below:2}]{};
\path (y1) ++(1,-0.5) node(2ndEnd)[dot][label={[vcolour]above:3}]{};

\draw (2nd)--(x1)  (2nd)--(x2);
\draw (x1)--(y1)  (x1)--(y2)  ;
\draw (x2)--(y1)  (x2)--(y2)  ;
\draw (2ndEnd)--(y1)  (2ndEnd)--(y2);

\path (3rd) ++(1,0.5) node(x1)[dot][label={[vcolour]above:3}]{} ++(0,-1) node(x2)[dot][label={[vcolour]below:3}]{};
\path (x1) ++(1,0) node(y1)[dot][label={[vcolour]above:0}]{} ++(0,-1) node(y2)[dot][label={[vcolour]below:2}]{};
\path (y1) ++(1,-0.5) node(3rdEnd)[dot][label={[vcolour]above:1}]{};

\draw (3rd)--(x1)  (3rd)--(x2);
\draw (x1)--(y1)  (x1)--(y2)  ;
\draw (x2)--(y1)  (x2)--(y2)  ;
\draw (3rdEnd)--(y1)  (3rdEnd)--(y2);

\draw (1stEnd)--(2nd)  (2ndEnd)--(3rd);

\draw (1st)--(u1) (3rdEnd)--(u2);

\path (origin) ++(-12,-4.25) node(1st)[dot][label={[vcolour]left:3}]{} ++(4,0) node(2nd)[dot][label={[vcolour]above:1}]{} ++(4,0) node(3rd)[dot][label={[vcolour]above:3}]{};

\path (1st) ++(1,0.5) node(x1)[dot][label={[vcolour]above:1}]{} ++(0,-1) node(x2)[dot][label={[vcolour]below:2}]{};
\path (x1) ++(1,0) node(y1)[dot][label={[vcolour]above:3}]{} ++(0,-1) node(y2)[dot][label={[vcolour]below:3}]{};
\path (y1) ++(1,-0.5) node(1stEnd)[dot][label={[vcolour]above:0}]{};

\draw (1st)--(x1)  (1st)--(x2);
\draw (x1)--(y1)  (x1)--(y2)  ;
\draw (x2)--(y1)  (x2)--(y2)  ;
\draw (1stEnd)--(y1)  (1stEnd)--(y2);

\path (2nd) ++(1,0.5) node(x1)[dot][label={[vcolour]above:3}]{} ++(0,-1) node(x2)[dot][label={[vcolour]below:3}]{};
\path (x1) ++(1,0) node(y1)[dot][label={[vcolour]above:0}]{} ++(0,-1) node(y2)[dot][label={[vcolour]below:2}]{};
\path (y1) ++(1,-0.5) node(2ndEnd)[dot][label={[vcolour]above:1}]{};

\draw (2nd)--(x1)  (2nd)--(x2);
\draw (x1)--(y1)  (x1)--(y2)  ;
\draw (x2)--(y1)  (x2)--(y2)  ;
\draw (2ndEnd)--(y1)  (2ndEnd)--(y2);

\path (3rd) ++(1,0.5) node(x1)[dot][label={[vcolour]above:0}]{} ++(0,-1) node(x2)[dot][label={[vcolour]below:2}]{};
\path (x1) ++(1,0) node(y1)[dot][label={[vcolour]above:3}]{} ++(0,-1) node(y2)[dot][label={[vcolour]below:3}]{};
\path (y1) ++(1,-0.5) node(3rdEnd)[dot][label={[vcolour]above:1}]{};

\draw (3rd)--(x1)  (3rd)--(x2);
\draw (x1)--(y1)  (x1)--(y2)  ;
\draw (x2)--(y1)  (x2)--(y2)  ;
\draw (3rdEnd)--(y1)  (3rdEnd)--(y2);

\draw (1stEnd)--(2nd)  (2ndEnd)--(3rd);

\draw (1st)--(v1) (3rdEnd) to[bend left=15] (v2);

\path (origin) ++(1.1,-4.25) node(1st)[dot][label={[vcolour]above:3}]{} ++(4,0) node(2nd)[dot][label={[vcolour]above:0}]{} ++(4,0) node(3rd)[dot][label={[vcolour]above:2}]{};

\path (1st) ++(1,0.5) node(x1)[dot][label={[vcolour]above:0}]{} ++(0,-1) node(x2)[dot][label={[vcolour]below:1}]{};
\path (x1) ++(1,0) node(y1)[dot][label={[vcolour]above:3}]{} ++(0,-1) node(y2)[dot][label={[vcolour]below:3}]{};
\path (y1) ++(1,-0.5) node(1stEnd)[dot][label={[vcolour]above:2}]{};

\draw (1st)--(x1)  (1st)--(x2);
\draw (x1)--(y1)  (x1)--(y2)  ;
\draw (x2)--(y1)  (x2)--(y2)  ;
\draw (1stEnd)--(y1)  (1stEnd)--(y2);

\path (2nd) ++(1,0.5) node(x1)[dot][label={[vcolour]above:3}]{} ++(0,-1) node(x2)[dot][label={[vcolour]below:3}]{};
\path (x1) ++(1,0) node(y1)[dot][label={[vcolour]above:1}]{} ++(0,-1) node(y2)[dot][label={[vcolour]below:2}]{};
\path (y1) ++(1,-0.5) node(2ndEnd)[dot][label={[vcolour]above:2}]{};

\draw (2nd)--(x1)  (2nd)--(x2);
\draw (x1)--(y1)  (x1)--(y2)  ;
\draw (x2)--(y1)  (x2)--(y2)  ;
\draw (2ndEnd)--(y1)  (2ndEnd)--(y2);

\path (3rd) ++(1,0.5) node(x1)[dot][label={[vcolour]above:3}]{} ++(0,-1) node(x2)[dot][label={[vcolour]below:3}]{};
\path (x1) ++(1,0) node(y1)[dot][label={[vcolour]above:0}]{} ++(0,-1) node(y2)[dot][label={[vcolour]below:1}]{};
\path (y1) ++(1,-0.5) node(3rdEnd)[dot][label={[vcolour]right:3}]{};

\draw (3rd)--(x1)  (3rd)--(x2);
\draw (x1)--(y1)  (x1)--(y2)  ;
\draw (x2)--(y1)  (x2)--(y2)  ;
\draw (3rdEnd)--(y1)  (3rdEnd)--(y2);

\draw (1stEnd)--(2nd)  (2ndEnd)--(3rd);

\draw (1st) to[bend right=15] (w1) (3rdEnd)--(w2);
\end{tikzpicture}
\caption{\( G' \)}
\end{subfigure}
\caption[Example of producing \( f' \) from \( f \) in Construction~\ref{make:rs colouring degre k-1 vs k-1 regular}.]{Example of producing \( f' \) from \( f \) in Construction~\ref{make:rs colouring degre k-1 vs k-1 regular}. (a) a graph \( G \) with a 4-rs colouring \( f \), and (b) graph \( G' \) with the corresponding 4-rs colouring \( f' \).}
\label{fig:eg colouring filler gadget d=k-1}
\end{figure}

\FloatBarrier 

Next, we generalise Construction~\ref{make:rs colouring degre k-1 vs k-1 regular}. 

\begin{construct}\label{make:rs colouring bdd degre vs regular}
\emph{Parameters:} Integers \( k\geq 4 \) and \( d\leq k-1 \).\\
\emph{Input:} A graph \( G \) of maximum degree \( d \).\\ 
\emph{Output:} A \( d \)-regular graph \( G^* \).\\
\emph{Guarantee:} \( G\) is \( k \)-rs colourable if and only if \( G^* \) is \( k \)-rs colourable.\\
\emph{Steps:}\\
Introduce two copies of \( G \). 
For each vertex \( v \) of \( G \), introduce \( d-\deg_G(v) \) filler gadgets (see Figure~\ref{fig:rs filler gadget d<k}) between the two copies of \( v \).
\begin{figure}[hbt]
\centering
\begin{tikzpicture}[scale=0.85]
\path (0,0) node(1st)[dot]{} ++(5,0) node(2nd)[dot]{} ++(5,0) node(3rd)[dot]{};

\path (1st) ++(1,1.5) node(x1)[dot]{} ++(0,-1) node(x2)[dot]{} ++(0,-1.5) node(xk-2)[dot]{};
\path (x2)--node[sloped,font=\large]{\( \dots \)} (xk-2);
\path (x1) ++(2,0) node(y1)[dot]{} ++(0,-1) node(y2)[dot]{} ++(0,-1.5) node(yk-2)[dot]{};
\path (y2)--node[sloped,font=\large]{\( \dots \)} (yk-2);
\path (y1) ++(1,-1.5) node(1stEnd)[dot]{};

\draw (1st)--(x1)  (1st)--(x2)  (1st)--(xk-2);
\draw (x1)--(y1)  (x1)--(y2)  (x1)--(yk-2);
\draw (x2)--(y1)  (x2)--(y2)  (x2)--(yk-2);
\draw (xk-2)--(y1)  (xk-2)--(y2)  (xk-2)--(yk-2);
\draw (1stEnd)--(y1)  (1stEnd)--(y2)  (1stEnd)--(yk-2);

\path (x1)--+(0,0.15) coordinate(brace1Start);
\path (xk-2)--+(0,-0.15) coordinate(brace1End);
\draw [opacity=0.4,decorate,decoration={brace,aspect=0.75,amplitude=16pt,raise=3pt,mirror}] (brace1Start)--  node[left=7mm,yshift=-6mm,opacity=1]{\( d-1 \)} (brace1End);

\path (2nd) ++(1,1.5) node(x1)[dot]{} ++(0,-1) node(x2)[dot]{} ++(0,-1.5) node(xk-2)[dot]{};
\path (x2)--node[sloped,font=\large]{\( \dots \)} (xk-2);
\path (x1) ++(2,0) node(y1)[dot]{} ++(0,-1) node(y2)[dot]{} ++(0,-1.5) node(yk-2)[dot]{};
\path (y2)--node[sloped,font=\large]{\( \dots \)} (yk-2);
\path (y1) ++(1,-1.5) node(2ndEnd)[dot]{};

\draw (2nd)--(x1)  (2nd)--(x2)  (2nd)--(xk-2);
\draw (x1)--(y1)  (x1)--(y2)  (x1)--(yk-2);
\draw (x2)--(y1)  (x2)--(y2)  (x2)--(yk-2);
\draw (xk-2)--(y1)  (xk-2)--(y2)  (xk-2)--(yk-2);
\draw (2ndEnd)--(y1)  (2ndEnd)--(y2)  (2ndEnd)--(yk-2);

\path (3rd) ++(1,1.5) node(x1)[dot]{} ++(0,-1) node(x2)[dot]{} ++(0,-1.5) node(xk-2)[dot]{};
\path (x2)--node[sloped,font=\large]{\( \dots \)} (xk-2);
\path (x1) ++(2,0) node(y1)[dot]{} ++(0,-1) node(y2)[dot]{} ++(0,-1.5) node(yk-2)[dot]{};
\path (y2)--node[sloped,font=\large]{\( \dots \)} (yk-2);
\path (y1) ++(1,-1.5) node(3rdEnd)[dot]{};

\draw (3rd)--(x1)  (3rd)--(x2)  (3rd)--(xk-2);
\draw (x1)--(y1)  (x1)--(y2)  (x1)--(yk-2);
\draw (x2)--(y1)  (x2)--(y2)  (x2)--(yk-2);
\draw (xk-2)--(y1)  (xk-2)--(y2)  (xk-2)--(yk-2);
\draw (3rdEnd)--(y1)  (3rdEnd)--(y2)  (3rdEnd)--(yk-2);

\draw (1stEnd)--(2nd)  (2ndEnd)--(3rd);
\draw (1st)--+(-1,0) node[dot]{} node[terminal][label=left:\( v \)]{}
(3rdEnd)--+(1,0) node[dot]{} node[terminal][label=right:\( v \)]{};
\end{tikzpicture}
\caption{Filler gadget for \( v\in V(G) \) in Construction~\ref{make:rs colouring bdd degre vs regular}.}
\label{fig:rs filler gadget d<k}
\end{figure}
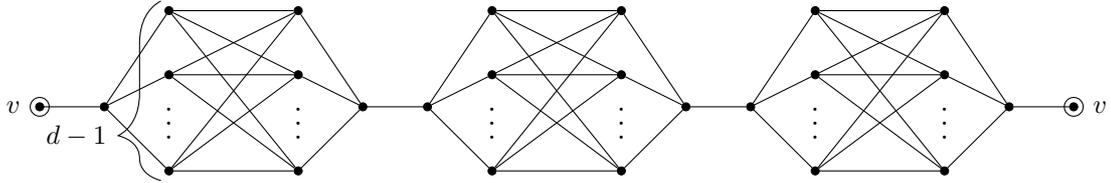
\end{construct}
\begin{proof}[Proof of guarantee]
Observe that since \( d\leq k-1 \), the filler gadget in Construction~\ref{make:rs colouring bdd degre vs regular} (i.e., Figure~\ref{fig:rs filler gadget d<k}) is as subgraph of the filler gadget in Construction~\ref{make:rs colouring degre k-1 vs k-1 regular} (i.e., Figure~\ref{fig:rs filler gadget d=k-1}). 
Hence, \( G^* \) is a subgraph of the output graph \( G' \) of Construction~\ref{make:rs colouring degre k-1 vs k-1 regular}. 
Since \( G \) is a subgraph of \( G^* \), one direction is obvious. 
To prove the other direction, assume that \( G \) admits a \( k \)-rs colouring \( f\colon V(G)\to \{0,1,\dots,k-1\} \). 
By the guarantee in Construction~\ref{make:rs colouring degre k-1 vs k-1 regular}, \( G' \) is \( k \)-rs colourable. 
Since \( G^* \) is a subgraph of~\( G' \),\; \( G^* \) is \( k \)-rs colourable as well. 
This completes the proof of the other direction. 
\end{proof}

Note that the filler gadget in Construction~\ref{make:rs colouring bdd degre vs regular} has \( 6d \) non-terminal vertices and \( 3d(d-1)+4=O(d^2) \) edges. 
Hence, \( G^* \) has only \( (2+6d)n=O(n) \) vertices and \( 2m+O(d^2)n=O(m+n) \) edges. 
Thus, Construction~\ref{make:rs colouring bdd degre vs regular} requires only time polynomial in the input size. 

By Theorem~\ref{thm:k-rs npc max deg k-1}, for all \( k\geq 4 \), \textsc{\( k \)-RS Colourability} is NP-complete for graphs of maximum degree~\( k-1 \). 
For \( k\geq 4 \) and \( d\leq k-1 \), Construction~\ref{make:rs colouring bdd degre vs regular} establishes a reduction from \textsc{\( k \)-RS Colourability}(\( \Delta=d \)) to \textsc{\( k \)-RS Colourability}(\( d \)-regular). 
Hence, for \( k\geq 4 \) and \( d\leq k-1 \), if \textsc{\( k \)-RS Colourability} is NP-complete for graphs of maximum degree \( d \), then \textsc{\( k \)-RS Colourability} is NP-complete for \( d \)-regular graphs. 
Clearly, if \textsc{\( k \)-RS Colourability} is NP-complete for \( d \)-regular graphs, then \textsc{\( k \)-RS Colourability} is NP-complete for graphs of maximum degree \( d \). 
Thus, we have the following theorem. 
\begin{theorem}\label{thm:rs colouring bdd degree to regular}
For all \( k\geq 4 \) and \( d\leq k-1 \), \textsc{\( k \)-RS Colourability} is NP-complete for graphs of maximum degree \( d \) if and only if \textsc{\( k \)-RS Colourability} is NP-complete for \( d \)-regular graphs. 
In particular, for all \( k\geq 4 \), \textsc{\( k \)-RS Colourability} is NP-complete for \( (k-1) \)-regular graphs. 
\qed
\end{theorem}

On the other hand, for all \( k\geq 4 \) and \( d\geq k \), \textsc{\( k \)-RS Colourability} is NP-complete for graphs of maximum degree \( d \) whereas it is trivially in P for \( d \)-regular graphs (because the answer is always no \cite{almeter}).\\

\iftoggle{extended}
{ \subsection{Hardness Transitions}
} { \subsection[Results on \( L_{rs}^{(k)} \) and RS Colouring of Regular Graphs]{\boldmath Results on \( L_{rs}^{(k)} \) and RS Colouring of Regular Graphs}\label{sec:rs colouring points of hardness transition}
} Recall that for \( k\geq 3 \), \( L_{rs}^{(k)} \) is the least integer \( d \) such that \textsc{\( k \)-RS Colourability} in graphs of maximum degree \( d \) is NP-complete. 
Bear in mind that we assume P \( \neq \) NP throughout this paper; thus, NP is partitioned into three classes: P, NPC and NPI~\cite{paschos}. 
If a problem in NP is not NP-complete (i.e., not in NPC), then it is either in P or in NPI. 
By the definition of \( L_{rs}^{(k)} \), \textsc{\( k \)-RS Colourability}(\( \Delta=d \)) is not NP-complete for \( d<L_{rs}^{(k)} \), which means that the problem is either in P or in NPI (we do not know which is the case).

Let \( G \) be a graph of maximum degree \( d \). 
If \( d\leq 2 \), then \( G \) is a disjoint union of paths and cycles, and thus the rs chromatic number of \( G \) can be computed in polynomial time. 
Since \textsc{3-RS Colourability} is NP-complete for graphs of maximum degree~3~\cite[Theorem~1]{shalu_cyriac2}, we have \( L_{rs}^{(3)}=3 \).  
Theorem~\ref{thm:k-rs npc max deg k-1 gen} proved that for \( k\geq 4 \), \textsc{\( k \)-RS Colourability} is NP-complete for graphs of maximum degree \( k-1 \), and thus \( L_{rs}^{(k)}\leq k-1 \).

Next, we show that \( L_{rs}^{(k)}>\sqrt{k} \). 
Let \( G \) be a graph of maximum degree \( d \). 
Each distance-two colouring of \( G \) is an rs colouring of \( G \)~\cite{almeter}. 
Moreover, \( G \) admits a distance-two colouring  (i.e., a colouring of the square graph \( G^2 \)) with \( \Delta(G^2)+1=d^2+1 \) colours. 
That is, \( \chi(G^2)\leq d^2+1 \). 
Furthermore, \( \chi(G^2)\leq d^2 \) unless \( G^2\cong K_{d^2+1} \), which is true only if \( G \) is a Moore graph of diameter~2~\cite{cranston2023}. 
Using properties of Moore graphs, one can easily show that \( G \) is \( d^2 \)-rs colourable (that is, \( \chi_{rs}(G)\leq d^2 \)). 
\begin{observation}\label{obs:G is d_square rs col}
\( \chi_{rs}(G)\leq d^2 \) for every graph \( G \) of maximum degree \( d \).
\qed
\end{observation}
\noindent See the supplement for a proof of Observation~\ref{obs:G is d_square rs col}. \\

Consider the problem \textsc{\( k \)-RS Colourability} in graphs of maximum degree \( d \). 
When \( k\geq d^2 \), we have \( \chi_{rs}(G)\leq d^2\leq k \) by Observation~\ref{obs:G is d_square rs col}; that is, \( G \) is \( k \)-rs colourable. 
In other words, for \( k\in \mathbb{N} \) and \( d\leq \sqrt{k} \), every graph of maximum degree \( d \) is \( k \)-rs colourable, and thus \textsc{\( k \)-RS Colourability}(\( \Delta=d \)) is polynomial-time solvable. 
Therefore, \( L_{rs}^{(k)}>\sqrt{k} \). 
\begin{observation}\label{obs:bounds on L_rs_k}
For \( k\geq 4 \), \( \sqrt{k}<L_{rs}^{(k)}\leq k-1 \). 
\qed
\end{observation}

Next, let us consider regular graphs. 
It is known that \( \chi_{rs}(G)\geq d+1 \) for every \( d \)-regular graph \( G \)~\cite{almeter}. 
Hence, for a fixed \( k\geq 3 \), \textsc{\( k \)-RS Colourability} in \( d \)\nobreakdash-regular graphs is polynomial-time solvable for each \( d\geq  k \) (because the answer is always `no'). 
In particular, \textsc{3-RS Colourability} in \( d \)-regular graphs is polynomial-time solvable for all \( d\in \mathbb{N} \).

Theorem~\ref{thm:rs colouring bdd degree to regular} proved that for \( k\geq 4 \) and \( d\leq k-1 \), \textsc{\( k \)-RS Colourability} in graphs of maximum degree \( d \) is NP-complete if and only if \textsc{\( k \)-RS Colourability} in \( d \)-regular graphs is NP-complete. 
For \( k\geq 4 \), by the definition of \( L_{rs}^{(k)} \), \textsc{\( k \)-RS Colourability} in graphs of maximum degree \( d \) is NP-complete for \( d=L_{rs}^{(k)} \), and not NP-complete for \( d<L_{rs}^{(k)} \). 
Hence, for \( d<L_{rs}^{(k)} \), we have \( d<L_{rs}^{(k)}\leq k-1 \) by Observation~\ref{obs:bounds on L_rs_k}, and thus \textsc{\( k \)-RS Colourability} in d-regular graphs is not NP-complete by Theorem~\ref{thm:rs colouring bdd degree to regular}.   
We know that \textsc{\( k \)-RS Colourability} in graphs of maximum degree \( d \) is NP-complete for \( d\geq L_{rs}^{(k)} \). 
 As a result, for \( d \) in the range \( L_{rs}^{(k)}\leq d\leq k-1 \), \textsc{\( k \)-RS Colourability} in d-regular graphs is also NP-complete by Theorem~\ref{thm:rs colouring bdd degree to regular}. 
Moreover, for \( d\geq  k \), \textsc{\( k \)-RS Colourability} in \( d \)\nobreakdash-regular graphs is polynomial-time solvable (see the previous paragraph). 
Thus, we have the following theorem. 
\begin{theorem}
For \( k\geq 4 \), \textsc{\( k \)\nobreakdash-RS Colourability} is NP-complete for \( d \)-regular graphs if and only if \( L_{rs}^{(k)}\leq d\leq k-1 \). 
\qed 
\end{theorem}

\section{Conclusion and Open Problems}\label{sec:conclusion}
We conclude the paper with this final section (see Sections~\ref{sec:intro hardness transitions}, \ref{sec:star colouring points of hardness transition} and \ref{sec:rs colouring points of hardness transition} for details). 
For \( k\geq 3 \), there exists an integer \( d \) such that \textsc{\( k \)-Colourability} is NP-complete for graphs of maximum degree \( d \). 
In fact, for \( k\geq 3 \), there exists a unique integer \( L^{(k)} \) such that \textsc{\( k \)-Colourability} is NP-complete for graphs of maximum degree \( d \) if and only if \( d\geq L^{(k)} \). 
Similarly, for \( k\geq 3 \), there exists a unique integer \( L_s^{(k)} \) (resp.\ \( L_{rs}^{(k)} \)) such that \textsc{\( k \)-Star Colourability} (resp.\ \textsc{\( k \)-RS Colourability}) is NP-complete for graphs of maximum degree \( d \) if and only if \( d\geq L_s^{(k)} \) (resp.\ \( d\geq L_{rs}^{(k)} \)). 
\begin{problem}For \( k\geq 3 \), determine \( L^{(k)} \), \( L_s^{(k)} \) and \( L_{rs}^{(k)} \). 
\end{problem}

For each \( k\geq 3 \), we have \( k+1\leq L^{(k)}\leq k-1+\raisebox{1.5pt}{\big\lceil}\sqrt{k}\raisebox{1.5pt}{\big\rceil} \)~\cite{emden-weinert} and for sufficiently large \( k \), we have \( L^{(k)}=k-1+\raisebox{1.5pt}{\big\lceil}\sqrt{k}\raisebox{1.5pt}{\big\rceil} \)~\cite{molloy_reed}. 
In particular, \( L^{(3)}=4 \), \( L^{(4)}=5 \) and \( 6\leq L^{(5)}\leq 7 \). 
Yet, the following is open. 
\begin{problem}[Paulusma~\cite{paulusma}]
Is \textsc{5\nobreakdash-Colourability} NP-complete for graphs of maximum degree~6?\\
\hspace*{2cm} In other words, is \( L^{(5)}=6 \)?
\end{problem}

Regarding star colouring and rs colouring, we have (i)~\( L_s^{(3)}=L_{rs}^{(3)}=3 \), (ii)~for \( k\geq 4 \), we have \( 0.33\, k^{\,2/3}<L_s^{(k)}\leq k \) and \( \sqrt{k}<L_{rs}^{(k)}\leq k-1 \), and (iii)~\( L_s^{(k)}\leq k-1 \) for \( k=5 \) and \( k\geq 7 \). 
\begin{problem}\label{prob:L_s_k at most k-1 for 4 N 6}
Is \( L_s^{(k)}\leq k-1 \) for \( k\in \{4,6\} \)?
\end{problem}

Next, let us consider the class of regular graphs. 
For \( k\geq 4 \), \textsc{\( k \)\nobreakdash-RS Colourability} is NP-complete for \( d \)-regular graphs if and only if \( L_{rs}^{(k)}\leq d\leq k-1 \). 
It is unknown whether this result has a star colouring analogue. 
Hence, for values of \( k \) such that \textsc{\( k \)-Star Colourability} is NP-complete for \( d^* \)-regular graphs for some \( d^*\in \mathbb{N} \), we define \( \widetilde{L}_s^{(k)} \) (resp.\ \( \widetilde{H}_s^{(k)} \)) as the least (resp.\ highest) integer \( d \) such that \textsc{\( k \)-Star Colourability} is NP-complete for \( d \)-regular graphs. 
For other values of \( k \), let us say that \( \widetilde{L}_s^{(k)} \) and \( \widetilde{H}_s^{(k)} \) are undefined (e.g., \( \widetilde{L}_s^{(3)} \) and \( \widetilde{H}_s^{(3)} \) are undefined).

For \( k\in \{4,5,7,8,\dots \} \), \( \widetilde{L}_s^{(k)} \) and \( \widetilde{H}_s^{(k)} \) are defined, and \( \widetilde{H}_s^{(k)}\leq 2k-4 \). 
Moreover, \( L_s^{(k)}=\widetilde{L}_s^{(k)} \leq k-1\leq \widetilde{H}_s^{(k)}\leq 2k-4 \) for \( k=5 \) and \( k\geq 7 \). 
If the answer to Problem~\ref{prob:L_s_k at most k-1 for 4 N 6} is `yes' for some \( k\in \{4,6\} \), then \( L_s^{(k)}=\widetilde{L}_s^{(k)} \) (by Theorem~\ref{thm:star colouring degree d vs d-regular}). 
Since \textsc{4-Star Colourability} is NP-complete for 4-regular graphs~\cite{cyriac}, we have \( L_s^{(4)}\leq \widetilde{L}_s^{(4)} \leq 4=\widetilde{H}_s^{(k)} \). 
\begin{problem}[\cite{shalu_cyriac3}]Is \textsc{4-Star Colourability} NP-complete for 3-regular graphs? 
\end{problem}
\noindent Depending on the answer to this problem, we have either (i)~\( L_s^{(4)}=\widetilde{L}_s^{(4)}=3 \) or (ii)~\textsc{4-Star Colourability}(\( d \)-regular)\( \,\in\, \)NPC if and only if \( d=4 \).

Consider the complexity of \textsc{\( k \)-Star Colourability} in 3-regular graphs. 
Since \( 4\leq \chi_s(G)\leq 6 \) for every 3-regular graph \( G \)~\cite{chen2013,xie}, \textsc{\( k \)-Star Colourability} in 3-regular graphs is polynomial-time solvable for all \( k \) except possibly \( k\in \{4,5\} \). 
According to Conjecture~12 of Almeter et al.~\cite{almeter}, \textsc{5-Star Colourability} in 3-regular graphs is polynomial-time solvable. 
If this conjecture is true, then either (i)~\textsc{4-Star Colourability}(3-regular)\( \,\in\, \)NPC (i.e., \( L_s^{(4)}=\widetilde{L}_s^{(4)}=3 \)), or (ii)~\textsc{Star Colourability}(3\nobreakdash-regular)\( \,\notin\, \)NPC.

Whenever \( \widetilde{H}_s^{(k)} \) is defined, we have \( \widetilde{H}_s^{(k)}\leq 2k-4 \), and equality holds for \( k=4 \). 
\begin{conjecture}\label{conj:2p-regular (p+2)-star colouring NPC}
For \( k\geq 4 \), \( \widetilde{H}_s^{(k)} \) is defined, and \( \widetilde{H}_s^{(k)}=2k-4 \);\\
\hspace*{2.5cm}that is, \textsc{\( k \)-Star Colourability} is NP-complete for \( (2k-4) \)-regular graphs.
\end{conjecture}

\section*{Acknowledgement}
We thank Sounaka Mishra for suggesting \textsc{4-RS Colourability} of cubic graphs as a problem to study.

 \bibliographystyle{unsrtnat}

\end{document}